\documentclass[a4paper]{amsart}

\usepackage[T1]{fontenc}
\usepackage[utf8]{inputenc}
\usepackage[english]{babel}
\usepackage{amssymb,amsmath,amsthm,amscd}

\usepackage{caption}
\usepackage{subcaption}
\usepackage{url}
\usepackage[all]{xy}

\newcommand{\CP}{\mathbb{C}\mathbb{P}^{1}}
\newcommand{\RP}{\mathbb{R}\mathbb{P}^{1}}
\newcommand{\RPP}{\mathbb{R}\mathbb{P}^{2}}
\newcommand{\CPP}{\mathbb{C}\mathbb{P}^{2}}

\newcommand{\R}{\mathbb{R}}
\newcommand{\C}{\mathbb{C}}
\newcommand{\Z}{\mathbb{Z}}
\newcommand{\ZZ}{\mathbb{Z}/2\mathbb{Z}}
\newcommand{\N}{\mathbb{N}}

\newcommand{\lra}{\longrightarrow}
\newcommand{\lmt}{\longmapsto}

\newcommand{\Si}{\Sigma}

\renewcommand{\Im}[1]{\mathrm{Im}\, #1}

\newtheorem{thm}{Theorem}[section]
\newtheorem{lm}[thm]{Lemma}
\newtheorem{coro}[thm]{Corollary}
\newtheorem{prop}[thm]{Proposition}
\theoremstyle{definition}
\newtheorem{df}[thm]{Definition}
\theoremstyle{plain}

\newcommand{\tv}{$\times$-vertex}
\newcommand{\tvs}{$\times$-vertices}
\newcommand{\gc}{generalized cut}

\DeclareMathOperator{\deep}{dp}
\DeclareMathOperator{\Cut}{Cut}
\DeclareMathOperator{\Ver}{Ver}
\DeclareMathOperator{\Ind}{Ind}
\DeclareMathOperator{\Dssn}{Dssn}

\DeclareMathOperator{\Sing}{Sing}

\usepackage{graphicx}
\usepackage{hyperref}\usepackage{tikz}
\usetikzlibrary{arrows,decorations.pathmorphing,decorations.pathreplacing,backgrounds,positioning,fit,petri,calc,cd}

\title[Rigid isotopy classification of generic rational quintics in $\RPP$]{Rigid isotopy classification of generic rational curves of degree $5$ in the real projective plane}
\author{Andrés Jaramillo Puentes}

\begin{document}
\begin{abstract}

In this article we obtain the rigid isotopy classification of generic rational curves of degre~$5$ in~$\RPP$.
In order to study the rigid isotopy classes of nodal rational curves of degree $5$ in $\RPP$, we associate to every real rational nodal quintic curve with a marked real nodal point a nodal trigonal curve in the Hirzebruch surface $\Sigma_3$
and the corresponding nodal real dessin on~$\CP/(z\mapsto\bar{z})$.
The dessins are real versions, proposed by S. Orevkov~\cite{Orevkov}, of Grothendieck's {\it dessins d'enfants}. The dessins are graphs embedded in a topological surface and endowed with a certain additional structure.

We study the combinatorial properties and decompositions of dessins corresponding to real nodal trigonal curves~$C\subset \Sigma_n$ in real Hirzebruch surfaces~$\Sigma_n$.
Nodal dessins in the disk can be decomposed in blocks corresponding to cubic dessins in the disk~$\mathbf{D}^2$, which produces a classification of these dessins.
The classification of dessins under consideration leads to a rigid isotopy classification of real rational quintics in~$\RPP$.
\end{abstract}

\maketitle
\tableofcontents
\section*{Introduction}

The first part of Hilbert's 16th problem \cite{Hilbert} asks for a topological classification of all possible pairs $(\RPP,\R A)$,
where $\R A$ is the real point set of a non-singular curve $A$ of fixed degree $d$ in the real projective plane $\RPP$. 
The fact that any homeomorphism of $\RPP$ is isotopic to the identity implies that two topological pairs $(\RPP,\R A)$ and $(\RPP,\R B)$ are homeomorphic if and only if $\R A$ is isotopic to $\R B$ as subsets of $\RPP$. 
The moduli space $\mathbb{RP}^{\frac{d(d+3)}{2}}$ of real homogeneous polynomials of degree $d$ in three variables (up to multiplication by a non-zero real number) has open strata formed by the non-singular curves (defined by polynomials with non-vanishing gradient in $\C^3\setminus\{\bar{0}\}$) and a codimension one subset formed by the singular curves; the latter subset is an algebraic hypersurface called the \emph{discriminant}. 
The discriminant divides $\mathbb{RP}^{\frac{d(d+3)}{2}}$ into connected components, called {\it chambers}.
Two curves in the same chamber can be connected by a path that does not cross the discriminant, 
and therefore every point of the path corresponds to a non-singular curve. Such a path is called a \emph{rigid isotopy}. 
A version of Hilbert's 16th problem asks for a description of the set of chambers of the moduli space $\mathbb{RP}^{\frac{d(d+3)}{2}}$, equivalently, for a classification up to rigid isotopy of non-singular curves of degree~$d$ in $\RPP$. 
More generally, given a deformation class of complex algebraic varieties, one can ask for a description of the set of real varieties up to equivariant deformation within this class.
In the case of rational curves of degree $d$ in $\RPP$, a generic rational curve has only nodal points.
A rigid isotopy classification of real rational curves of degree 4 in $\RPP$ can be found in \cite{Gud}.

\vskip10pt 

In the case of curves of degree $5$ in $\RPP$, the rigid isotopy classification of non-singular curves and curves with one non-degenerate double point can be found in \cite{Kha81}.
Our main result is a rigid isotopy classification of real nodal rational curves of degree $5$ in $\RPP$.
(The isotopy classification of nodal rational curves of degree~$5$ in~$\RPP$ was obtained in~\cite{IMR}.) 
In the case where the studied curves have at least one real nodal point, we use a correspondence between real nodal rational curves of degree $5$ with a marked real nodal point and dessins on the disk $\CP/{(z\lmt\bar{z})}$ endowed with some extra combinatorial data.

The above correspondence makes use of trigonal curves on the Hirzebruch surface~$\Sigma_3$. In this paper, a trigonal curve is a complex curve $C$ lying in a ruled surface such that the restriction to $C$ of the projection provided by the ruling is a morphism of degree $3$. 
We consider an auxiliary morphism of $j$-invariant type in order to associate to the curve $C$ a dessin, 
a real version, proposed by S. Orevkov, of  dessins d'enfants introduced by A. Grothendieck.
A toile is a particular type of dessin corresponding the to the curves under consideration. 

\vskip10pt 

In Section \ref{ch:tri}, we introduce the basic notions and discuss some properties of trigonal curves and the associated dessins in the complex and real case.
We explain how the study of equivalence classes of dessins allows us to obtain information about the equivariant deformation classes of real trigonal curves. 

In Section \ref{ch:toi}, we proved that a dessin lying on $\mathbb{D}^2$ and having at most nodal singular vertices corresponds to a dessin having a decomposition as gluing of dessins associated to cubic curves in~$\RPP$.
These combinatorial properties are valid in a more general setting than the one needed for the case of quintic curves in $\RPP$.

In Section \ref{ch:hir}, we deal with a relation between curves in $\RPP$ and trigonal curves on the Hirzebruch surfaces $\Sigma_n$ and detail the case of plane cubic curves, which leads to the dessins that serve as building blocks for our constructions.
We prove a formula of Rokhlin's complex orientation type for nodal trigonal curves on $\Sigma_{2k+1}$.

Finally, in Section \ref{ch:deg5} we presented a rigid isotopy classification of nodal rational curves of degree $5$ in $\RPP$. 
These classes are presented in Figures~\ref{fig:b01} to~\ref{fig:maxellip} and in Figures~\ref{fig:MM2I} to~\ref{fig:M6}.

This classification is shown using two rigid isotopy invariants: the total Betti number with $\Z/2\Z$-coefficients (cf. \ref{thm:sti}) of the type~$\mathrm{I}$ perturbation of the curve (i.e., a perturbation of the curve of type~$\mathrm{I}$ with non-singular real point set, see Definition \ref{df:Mdpert}), and the total Betti number with $\Z/2\Z$-coefficients of the rational curve.

\section{Trigonal curves and dessins}\label{ch:tri}

In this chapter we introduce trigonal curves and dessins, which are the principal tool we use in order to study the rigid isotopy classification of generic rational curves of degree $5$ in $\RPP$.
The content of this chapter is based on the book~\cite{deg} and the article~\cite{DIZ}.

\subsection{Ruled surfaces and trigonal curves}

\subsubsection{Basic definitions}
A compact complex surface $\Sigma$ is a \emph{(geometrically) ruled surface} over a curve $B$ 
if $\Si$ is endowed with a projection $\pi:\Si\lra B$ of fiber~$\CP$ as well as a special section $E$ 
of non-positive self-intersection.

\begin{df}
A \emph{trigonal curve} is a reduced curve $C$ lying in a ruled surface $\Sigma$ such that $C$ contains neither the exceptional section $E$ nor a fiber 
 as component, and	the restriction $\pi|_{C}:C\lra B$ is a degree $3$ map.
 
A trigonal curve $C \subset \Sigma$ is {\it proper} if it does not intersect the exceptional section~$E$.
A {\it singular fiber} of a trigonal curve $C\subset\Si$ is a fiber~$F$ of $\Si$ intersecting $C\cup E$ geometrically in less than $4$ points.
\end{df} 

A fiber $F$ is singular if $C$ passes through $E\cap F$, or if~$C$ is tangent to $F$ or if $C$ has a singular point belonging to $F$ (those cases are not exclusive). A singular fiber~$F$ is {\it proper} if $C$ does not pass through $E\cap F$. Then, $C$ is proper if and only if all its singular fibers are.
We call a singular fiber~$F$ \emph{simple} if either $C$ is tangent to $F$ or $F$ contains a node (but none of the branches of $C$ is tangent to $F$), a cusp or an inflection point. The set $\{b\in B\mid F_b \text{ is a singular fiber }\}$ of points in the base having singular fibers is a discrete subset of the base $B$. We denote its complement by $B^{\#}=B^{\#}(C)$; it is a curve with punctures.
We denote by $F_b^0$ the complement $F_b\setminus\{F_b\cap E\}$.

\subsubsection{Deformations} 
We are interested in the study of trigonal curves up to deformation. In the real case, we consider the curves up to equivariant deformation (with respect to the action of the complex conjugation, cf.~\ref{sssec:realstruc}).

In the Kodaira-Spencer sense, a \emph{deformation} of the quintuple $(\pi\colon\Si\lra B,E,C)$ refers to an analytic space $X\lra S$ fibered over an marked open disk $S\ni o$ endowed with analytic subspaces $\mathcal{B}, \mathcal{E}, \mathcal{C}\subset X$ such that for every $s\in S$, the fiber~$X_{s}$ is diffeomorphic to $\Si$ and the intersections $\mathcal{B}_s:= X_s\cap \mathcal{B}$, $\mathcal{E}_s:= X_s\cap \mathcal{E}$ and $\mathcal{C}_s:= X_s\cap \mathcal{C}$ are diffeomorphic to $B$, $E$ and $C$, respectively, and there exists a map $\pi_s\colon X_s\lra B_s$ making $X_s$ a geometrically ruled surface over $B_s$ with exceptional section $E_s$, such that the diagram in Figure~\ref{fig:KodSpe} commutes
and $(\pi_o\colon X_o\lra B_o,E_o,C_o)=(\pi\colon\Si\lra B,E,C)$.

\begin{figure}[h]
\begin{center}
\begin{tikzcd}
	& E_s \arrow[r, "\text{diff.}"] \arrow[d, hook] & E \arrow[d, hook] & \\
C_s \arrow[rd, "\pi_s|_{C_s}"'] \arrow[r, hook]	& X_s \arrow[r, "\text{diff.}"] \arrow[d, "\pi_s"] & \Sigma	\arrow[d,"\pi"'] \arrow[r,hookleftarrow] & C \arrow[ld, "\pi|_C"]  \\
	& B_s \arrow[r, "\text{diff.}"] & B &
\end{tikzcd}
\end{center}
\caption{Commuting diagram of morphisms and diffeomorphisms of the fibers of a deformation.}
\label{fig:KodSpe}
\end{figure}

\begin{df}
 An \emph{elementary deformation} of a trigonal curve $C\subset\Si\lra B$ is a deformation of the quintuple $(\pi\colon\Si\lra B,E,C)$ in the Kodaira-Spencer sense.
\end{df}

An elementary deformation $X\lra S$ is \emph{equisingular} if for every $s\in S$ there exists a neighborhood $U_s\subset S$ of $s$ such that for every singular fiber~$F$ of $C$, there exists a neighborhood $V_{\pi(F)}\subset B$ of $\pi(F)$, where $\pi(F)$ is the only point with a singular fiber for every $t\in U_s$.
An elementary deformation over $D^2$ is a \emph{degeneration} or \emph{perturbation} if the restriction to $D^2\setminus\{0\}$ is equisingular and for a set of singular fibers~$F_i$ there exists a neighborhood $V_{\pi(F_i)}\subset B$ where there are no points with a singular fiber for every $t\neq0$. In this case we say that $C_t$ \emph{degenerates} to $C_0$ or $C_0$ is \emph{perturbed} to $C_t$, for $t\neq0$.

\subsubsection{Nagata transformations} 
One of our principal tools in this paper are the dessins, which we are going to associate to proper trigonal curves. A trigonal curve $C$ intersects the exceptional section $E$ in a finite number of points, since $C$ does not contain $E$ as component. We use the Nagata transformations in order to construct a proper trigonal curve out of a non-proper one. 

\begin{df}
 A {\it Nagata transformation} is a fiberwise birational transformation $\Si\lra\Si'$ consisting of blowing up a point $p\in\Si$ (with exceptional divisor $E''$) and contracting the strict transform of the fiber~$F_{\pi(p)}$ containing $p$. 
The new exceptional divisor $E'\subset\Si'$ is the strict transform of $E\cup E''$.
 The transformation is called {\it positive} if $p\in E$, and {\it negative} otherwise.
\end{df}
 
The result of a positive Nagata transformation is a ruled surface $\Si'$, with an exceptional divisor $E'$ such that $-E'^2=-E^2+1$. The trigonal curves $C_1$ and $C_2$ over the same base $B$ are {\it Nagata equivalent} if there exists a sequence of Nagata transformations mapping one curve to the other by strict transforms. Since all the points at the intersection $C\cap E$ can be resolved, every trigonal curve $C$ is Nagata equivalent to a proper trigonal curve $C'$ over the same base, called a {\it proper model} of $C$.
 
 \subsubsection{Weierstraß equations}
 For a trigonal curve, the Weierstraß equations are an algebraic tool which allows us to study the behavior of the trigonal curve with respect to the zero section and the exceptional one. They give rise to an auxiliary morphism of $j$-invariant type, which plays an intermediary role between trigonal curves and dessins. 
Let $C\subset\Si\lra B$ be a proper trigonal curve. Mapping a point $b\in B$ of the base to the barycenter of the points in $C\cap F_b^0$ (weighted according to their multiplicity) defines a section $B\lra Z\subset\Si$ called the {\it zero section}; it is disjoint from the exceptional section $E$.\\

The surface $\Si$ can be seen as the projectivization of a rank $2$ vector bundle, which splits as a direct sum of two line bundles such that the zero section $Z$ corresponds to the projectivization of $\mathcal{Y}$, one of the terms of this decomposition. In this context, the trigonal curve $C$ can be described by a Weierstraß equation, which in suitable affine charts has the form
\begin{equation} \label{eq:Weiertrass}
x^3+g_{2}x+g_{3}=0,
\end{equation}
where $g_{2}$, $g_{3}$ are sections of $\mathcal{Y}^2$, $\mathcal{Y}^3$ respectively, and $x$ is an affine coordinate such that $Z=\{x=0\}$ and $E=\{x=\infty\}$. For this construction, we can identify $\Si\setminus B$ with the total space of $\mathcal{Y}$ and take $x$ as a local trivialization of this bundle. Nonetheless, the sections $g_{2}$, $g_{3}$ are globally defined. The line bundle $\mathcal{Y}$ is determined by $C$. The sections $g_{2}$, $g_{3}$ are determined up to change of variable defined by 
\begin{equation*}
(g_{2}, g_{3})\lra(s^2g_{2},s^3 g_{3} ),\;s\in H^{0}(B,\mathcal{O}^{*}_{B}).
\end{equation*}
Hence, the singular fibers of the trigonal curve $C$ correspond to the points where the equation~\eqref{eq:Weiertrass} has multiple roots, i.e., the zeros of the discriminant section
\begin{equation}
 \Delta:=-4g_2^3-27g_3^2\in H^0(B,\mathcal{O}_B(\mathcal{Y}^6)).
\end{equation}
Therefore, $C$ being reduced is equivalent to $\Delta$ being identically zero.
A Nagata transformation over a point $b\in B$ changes the line bundle $\mathcal{Y}$ to $\mathcal{Y}\otimes\mathcal{O}_B(b)$ and the sections $g_2$ and $g_3$ to $s^2 g_2$ and $s^3 g_3$, where $s\in H^0(B,\mathcal{O}_B)$ is any holomorphic function having a zero at $b$.

\begin{df} \label{df:gen}
 Let $C$ be a non-singular trigonal curve with Weierstraß model determined by the sections $g_2$ and $g_3$ as in \eqref{eq:Weiertrass}. The trigonal curve $C$ is {\it almost generic} if every singular fiber corresponds to a simple root of the determinant section $\Delta=-4g_2^3-27g_3^2$ which is not a root of $g_2$ nor of $g_3$. The trigonal curve $C$ is {\it generic} if it is almost generic and the sections $g_2$ and $g_3$ have only simple roots. 
\end{df}

\subsubsection{The $j$-invariant}
The $j$-invariant describes the relative position of four points in the complex projective line $\CP$. We describe some properties of the $j$-invariant in order to use them in the description of the dessins.

\begin{df}
 Let $z_1$, $z_2$, $z_3$, $z_4\in\CP$. The $j$-{\it invariant} of a set $\{z_1, z_2, z_3, z_4\}$ is given by
\begin{equation} \label{eq:jinvariant}
\displaystyle j(z_1, z_2, z_3, z_4)=\frac{4(\lambda^2-\lambda+1)^3}{27\lambda^2(\lambda-1)^2},
\end{equation}
where $\lambda$ is the {\it cross-ratio} of the quadruple $(z_1, z_2, z_3, z_4)$ defined as

 \begin{equation*}
\displaystyle\lambda(z_1, z_2, z_3, z_4)=\frac{z_1-z_3}{z_2-z_3} : \frac{z_1-z_4}{z_2-z_4}.
\end{equation*}
 \end{df}
 
The cross-ratio depends on the order of the points while the $j$-invariant does not. Since the cross-ratio $\lambda$ is invariant under Möbius transformations, so is the $j$-invariant. When two points $z_i$, $z_j$ coincide, the cross-ratio $\lambda$ equals either $0$, $1$ or~$\infty$, and the $j$-invariant equals $\infty$. 
 
Let us consider 
a polynomial $z^3+g_2z+g_3$. We define the $j$-invariant $j(z_1, z_2, z_3)$ of its roots $z_1$, $z_2$, $z_3$ as $j(z_1, z_2, z_3,\infty)$. If $\Delta=-4g_2^3-27g_3^2$ is the discriminant of the polynomial, then
\begin{equation*}
\displaystyle j(z_1, z_2, z_3,\infty)=\frac{-4g_2^3}{\Delta}.
\end{equation*}

A subset $A$ of $\CP$ is real if $A$ is invariant under the complex conjugation. We say that $A$ has a nontrivial symmetry if there is a nontrivial permutation of its elements which extends to a linear map $z\lmt az+b$, $a\in\C^*$, $b\in\C$.

\begin{lm}[\cite{deg}]
The set $ \{z_1, z_2, z_3\}$ of roots of the polynomial $z^3+g_2z+g_3$ has a nontrivial symmetry if and only if its $j$-invariant equals $0$ (for an order 3 symmetry) or 1 (for an order 2 symmetry).
\end{lm}

\begin{prop}[\cite{deg}]
Assume that $j(z_1, z_2, z_3)\in\R$. Then, the following holds
\begin{itemize}
\item
The $j$-invariant $j(z_1, z_2, z_3)<1$ if and only if the points $z_1, z_2, z_3$ form an isosceles triangle. The special angle seen as a function of the $j$-invariant is a increasing monotone function. This angle tends to $0$ when $j$ tends to $-\infty$, equals $\frac{\pi}{3}$ at $j=0$ and tends to $\frac{\pi}{2}$ when $j$ approaches $1$. 

\item
The $j$-invariant $j(z_1, z_2, z_3)\geq1$ if and only if the points $z_1, z_2, z_3$ are collinear.The ratio between the lengths of the smallest segment and the longest segment $\overline{z_lz_k}$ seen as a function of the $j$-invariant is a decreasing monotone function. This ratio equals $1$ when $j$ equals $1$, and $0$ when $j$ approaches $\infty$.
\end{itemize}
\end{prop}

As a corollary, if the $j$-invariant of $\{z_1, z_2, z_3\}$ is not real, then the points form a triangle having side lengths pairwise different. Therefore, in this case, the points $z_1, z_2, z_3$ can be ordered according to the increasing order of side lengths of the opposite edges. E.g., for $\{z_1, z_2, z_3\}=\{\frac{i}{2},0,1\}$ the order is $1$, $\frac{i}{2}$ and $0$.

\begin{prop}[\cite{deg}]
 If $j(z_1, z_2, z_3)\notin\R$, then the above order on the points~$z_1, z_2, z_3$ is clockwise if $\Im(j(z_1, z_2, z_3))>0$ and anti-clockwise if ${\Im(j(z_1, z_2, z_3))<0}$.
\end{prop}

\subsubsection{The $j$-invariant of a trigonal curve}
Let $C$ be a proper trigonal curve. We use the $j$-invariant defined for triples of complex numbers in order to define a meromorphic map $j_C$ on the base curve $B$. The map $j_{C}$ encodes the topology of the trigonal curve $C$ with respect to the fibration $\Sigma\lmt B$. The map $j_C$ is called the $j$-\emph{invariant} of the curve $C$ and provides a correspondence between trigonal curves and dessins.

\begin{df}
 For a proper trigonal curve~$C$, we define its $j$-invariant $j_C$ as the analytic continuation of the map
 \[
\begin{array}{ccc}
 B^{\#}&\lra&\C \\
 b &\lmt &j\mbox{-invariant of } C\cap F_b^0\subset F_b^0\cong\C.
\end{array}
\]
We call the trigonal curve $C$ {\it isotrivial} if its $j$-invariant is constant.
\end{df}

If a proper trigonal curve $C$ is given by a Weierstraß equation of the form \eqref{eq:Weiertrass}, then
\begin{equation}
j_C=-\frac{4g_2^3}{\Delta}\; , \text{ where }\Delta=-4g_2^3-27g_3^2.
\end{equation}

\begin{thm}[\cite{deg}] \label{thm:Etc}
 Let $B$ be a compact curve and $j\colon B\lra\CP$ a non-constant meromorphic map. Up to Nagata equivalence, there exists a unique trigonal curve $C\subset\Si\lra B$ such that $j_{C}=j$.
\end{thm}

Following the proof of the theorem, $j_B\lra\CP$ leads to a unique \emph{minimal} proper trigonal curve $C_j$, in the sense that any other trigonal curve with the same $j$-invariant can be obtained by positive Nagata transformations from $C_j$.

An equisingular deformation $C_s$, $s\in S$, of $C$ leads to an analytic deformation of the couple $(B_s,j_{C_s})$.

\begin{coro}[\cite{deg}] \label{coro:Etc}
 Let $(B,j)$ be a couple, where $B$ is a compact curve and $j\colon B\lra\CP$ is a non-constant meromorphic map. Then, any deformation of $(B,j)$ results in a deformation of the minimal curve $C_j\subset\Si\lra B$ associated to $j$.
\end{coro}

The $j$-invariant of a generic trigonal curve $C\subset\Sigma\lra B$ has degree $\deg(j_C)=6d$, where $d=-E^2$. A positive Nagata transformation increases $d$ by one while leaving $j_C$ invariant.
The $j$-invariant of a generic trigonal curve $C$ has a ramification index equal to $3$, $2$ or $1$ at every point $b\in B$ such that $j_C(b)$ equals $0$, $1$ or $\infty$, respectively. We can assume, up to perturbation, that every critical value of $j_C$ is simple. In this case we say that $j_C$ has a {\it generic branching behavior}.

\subsubsection{Real structures}\label{sssec:realstruc}
We are mostly interested in real trigonal curves. A {\it real structure} on a complex variety $X$ is an anti-holomorphic involution $c\colon X\lra X$. We define a {\it real variety} as a couple $(X,c)$, where $c$ is a real structure on a complex variety $X$. We denote by $X_{\R}$ the fixed point set of the involution $c$ and we call $X_{\R}$ {\it the set of real points} of $c$.

\begin{thm}[Smith-Thom inequality \cite{DK}] \label{thm:sti}
If $(X,c)$ is a compact real variety, then
\begin{equation}
 \sum_{i = 0}^{+\infty} \dim_{\Z/2\Z}H_i(X_{\R}; \Z/2\Z)
\leq
 \sum_{i = 0}^{+\infty} \dim_{\Z/2\Z}H_i(X; \Z/2\Z).
\end{equation}
\end{thm}

\begin{thm}[Smith congruence \cite{DK}]
If $(X,c)$ is a compact real variety, then
\begin{equation}
 \sum_{i = 0}^{+\infty} \dim_{\Z/2\Z}H_i(X_{\R}; \Z/2\Z)
\equiv
 \sum_{i = 0}^{+\infty} \dim_{\Z/2\Z}H_i(X; \Z/2\Z) \mod 2. 
\end{equation}
\end{thm}

\begin{df}
 A real variety $(X,c)$ is {\it an $(M-a)$-variety} if 
 \[ \sum_{i = 0}^{+\infty} \dim_{\Z/2\Z}H_i(X; \Z/2\Z)-\sum_{i = 0}^{+\infty} \dim_{\Z/2\Z}H_i(X_{\R}; \Z/2\Z)=2a.\] 
\end{df}

We say that a real curve $(X,c)$ is of \emph{type~$\mathrm{I}$} if $\widetilde{X}\setminus \widetilde{X}_{\R}$ is disconnected, where $\widetilde{X}$ is the normalization of $X$.

\subsubsection{Real trigonal curves}
\label{sssec:rtc}

A geometrically ruled surface $\pi\colon\Si\lra B$ is \emph{real} if there exist real structures $c_{\Si}\colon\Si\lra\Si$ and $c_{B}\colon B\lra B$ compatible with the projection~$\pi$, i.e., such that $\pi\circ c_{\Si}=c_{B}\circ\pi$. We assume the exceptional section is {\it real} in the sense that it is invariant by conjugation, i.e., $c_{\Si}(E)=E$. 
Put $\pi_{\R}:=\pi|_{\Si_{\R}}\colon{\Si_{\R}\lra B_{\R}}$. Since the exceptional section is real, the fixed point set of every fiber is not empty, implying that the real structure on the fiber is isomorphic to the standard complex conjugation on $\CP$. Hence all the fibers of $\pi_{\R}$ are isomorphic to~$\RP$.
Thus, the map $\pi_{\R}$ establishes a bijection between the connected components of the real part $\Si_{\R}$ of the surface $\Si$ and the connected components of the real part~$B_{\R}$ of the curve $B$. Every connected component of $\Si_{\R}$ is homeomorphic either to a torus or to a Klein bottle.

The ruled surface $\Sigma$ can be seen as the fiberwise projectivization of a rank 2 vector bundle over~$B$.
Let us assume $\Sigma=\mathbf{P}(1\oplus \mathcal{Y})$, where $1$ is the trivial line bundle over~$B$ and $\mathcal{Y}\in\operatorname{Pic}(B)$. We put $\mathcal{Y}_{i}:=\mathcal{Y}_{\R}|_{B_{i}}$ for every connected component $B_{i}$ of~$B_{\R}$. Hence $\Si_{i}:=\Si_{\R}|_{B_{i}}$ is orientable if and only if $\mathcal{Y}_{i}$ is topologically trivial, i.e., its first Stiefel-Whitney class $w_{1}(\mathcal{Y}_{i})$ is zero.

\begin{df}
 A \emph{real trigonal curve} $C$ is a trigonal curve contained in a real ruled surface $(\Si,c_{\Si})\lra (B,c_{B})$ such that $C$ is $c_{\Si}$-invariant, i.e., $c_{\Si}(C)=C$.
\end{df}

The line bundle $\mathcal{Y}$ can inherit a real structure from $(\Sigma,c_{\Sigma})$. 
If a real trigonal curve is proper, then its $j$-invariant is real, seen as a morphism $j_{C}\colon(B,c_{B})\lra (\CP,z\lmt\bar{z})$, where $z\lmt\bar{z}$ denotes the standard complex conjugation on $\CP$. In addition, the sections $g_{2}$ and $g_{3}$ can be chosen real.

 Let us consider the restriction $\pi|_{C_{\R}}\colon C_{\R}\lra B_{\R}$. We put $C_{i}:=\pi|_{C_{\R}}^{-1}(B_{i})$ for every connected component $B_{i}$ of $B_{\R}$. We say that $B_{i}$ is \emph{hyperbolic} if $\pi|_{C_{i}}\colon C_{i}\lra~B_{i}$ has generically a fiber with three elements. The trigonal curve $C$ is \emph{hyperbolic} if its real part is non-empty and all the connected components of $B_{\R}$ are hyperbolic.

\begin{df} \label{df:ovzz}
 Let $C$ be a non-singular generic real trigonal curve.
A connected component of the set $\{b\in B\mid \operatorname{Card}(\pi|_{C_{\R}}^{-1}(b))\geq2\}$ is an \emph{oval} if it is not a hyperbolic component and its preimage by $\pi|_{C_{\R}}$ is disconnected. Otherwise, the connected component is called a \emph{zigzag}.
\end{df}

\subsection{Dessins}

The dessins d'enfants were introduced by A. Grothendieck (cf.~\cite{schneps})
in order to study the action of the absolute Galois group. We use a modified version of dessins d'enfants which was proposed by S. Orevkov~\cite{Orevkov}.

\subsubsection{Trichotomic graphs}

Let $S$ be a compact connected topological surface. A graph $D$ on the surface $S$ is a graph embedded into the surface and considered as a subset $D\subset S$. We denote by $\operatorname{Cut}(D)$ the \emph{cut} of $S$ along $D$, i.e., the disjoint union of the closure of connected components of $S\setminus D$.

\begin{df} \label{df:trigra}
 A \emph{trichotomic graph} on a compact surface $S$ is an embedded finite directed graph $D\subset S$ decorated with the following additional structures (referred to as \emph{colorings} of the edges and vertices of $D$, respectively):
 
\begin{itemize}
\item every edge of $D$ is color solid, bold or dotted,
\item every vertex of $D$ is black ($\bullet$), white ($\circ$), cross ($\times$) 
or monochrome (the vertices of the first three types are called \emph{essential}),
\end{itemize}
and satisfying the following conditions:
\begin{enumerate}
\item $\partial S\subset D$,
\item every essential vertex is incident to at least $2$ edges,
\item every monochrome vertex is incident to at least $3$ edges,
\item the orientations of the edges of $D$ form an orientation of the boundary $\partial\operatorname{Cut} (D)$ which is compatible with an orientation on $\operatorname{Cut} (D)$,
\item all edges incident to a monochrome vertex are of the same color,
\item $\times$-vertices are incident to incoming dotted edges and outgoing solid edges,
\item $\bullet$-vertices are incident to incoming solid edges and outgoing bold edges,
\item $\circ$-vertices are incident to incoming bold edges and outgoing dotted edges.
\end{enumerate}
\end{df}

Let $D\subset S$ be a trichotomic graph. A \emph{region} $R$ is an element of $\operatorname{Cut(D)}$. The boundary $\partial R$ of $R$ contains $n=3k$ essential vertices. 
A region with $n$ vertices on its boundary is called an \emph{$n$-gonal region}.
We denote by $D_{solid}$, $D_{bold}$, $D_{dotted}$ the monochrome parts of $D$, i.e., the sets of vertices and edges of the specific color. On the set of vertices of a specific color, we define the relation $u\preceq v$ if there is a monochrome path from $u$ to~$v$, i.e., a path formed entirely of edges and vertices of the same color. We call the graph $D$ \emph{admissible} if the relation $\preceq$ is a partial order, equivalently, if there are no directed monochrome cycles.

\begin{df}
 A trichotomic graph $D$ is a \emph{dessin} if
\begin{enumerate}
 \item $D$ is admissible;
 \item every trigonal region of $D$ is homeomorphic to a disk.
\end{enumerate}
\end{df}

The orientation of the graph $D$ is determined by the pattern of colors of the vertices on the boundary of every region.

\subsubsection{Complex and real dessins}

Let $S$ be an orientable surface. Every orientation of $S$ induces a \emph{chessboard coloring} of $\Cut(D)$, i.e., a function on $\Cut(D)$ determining if a region $R$ endowed with the orientation set by $D$ coincides with the orientation of $S$.

\begin{df} A \emph{real trichotomic graph} on a real closed surface $(S,c)$ is a trichotomic graph $D$ on $S$ which is invariant under the action of $c$. Explicitly, every vertex $v$ of $D$ has as image $c(v)$ a vertex of the same color; every edge $e$ of $D$ has as image $c(e)$ an edge of the same color.
\end{df}

Let $D$ be a real trichotomic graph on $(S,c)$. Let $\overline{S}:=S/c$ be the quotient surface and put $\overline{D}\subset\overline{S}$ as the image of $D$ by the quotient map $S\lra S/c$. The graph $\overline{D}$ is a well defined trichotomic graph on the surface $S/c$.

In the inverse sense, let $S$ be a compact surface, which can be non-orientable or can have non-empty boundary. Let $D\subset S$ be a trichotomic graph on $S$.
Consider its complex double covering $\widetilde{S}\lra S$ (cf. \cite{AG} for details), which has a real structure given by the deck transformation, and put $\widetilde{D}\subset\widetilde{S}$ the inverse image of $D$. The graph~$\widetilde{D}$ is a graph on $\widetilde{S}$ invariant by the deck transformation. We use these constructions in order to identify real trichotomic graphs on real surfaces with their images on the quotient surface.

\begin{lm}[\cite{deg}]
 Let $D$ be a real trichotomic graph on a real closed surface~$(\widetilde{S},c)$. Then, every region $R$ of $D$ is disjoint from its image $c(R)$.
\end{lm}

\begin{prop}[\cite{deg}] \label{prop:realdd}
Let $S$ be a compact surface. Given a trichotomic graph~${D\subset S}$, then its oriented double covering $\widetilde{D}\subset\widetilde{S}$ is a real trichotomic graph. Moreover, $\widetilde{D}\subset\widetilde{S}$ is a dessin if and only if so is $D\subset S$.
Conversely, if $(S,c)$ is a real compact surface and $D\subset S$ is a real trichotomic graph, then its image $\overline{D}$ in the quotient $\overline{S}:=S/c$ is a trichotomic graph. Moreover, $\overline{D}\subset\overline{S}$ is a dessin if and only if so is $D\subset S$. 
\end{prop}

\begin{df}
Let $D$ be a dessin on a compact surface $S$. Let us denote by $\Ver(D)$ the set of vertices of $D$. For a vertex $v\in\Ver(D)$, we define the \emph{index} $\Ind(v)$ of $v$ as half of the number of incident edges of $\widetilde{v}$, where $\widetilde{v}$ is a preimage of $v$ by the double complex cover of $S$ as in Proposition~\ref{prop:realdd}.

A vertex $v\in\Ver(D)$ is \emph{singular} if
\begin{itemize}
\item $v$ is black and $\Ind(v)\not\equiv0\mod3$,
\item or $v$ is white and $\Ind(v)\not\equiv0\mod2$,
\item or $v$ has color $\times$ and $\Ind(v)\geq2$.
\end{itemize}
We denote by $\Sing(D)$ the set of singular vertices of $D$. A dessin is \emph{non-singular} if none of its vertices is singular.
\end{df}

\begin{df}
 Let $B$ be a complex curve and let $j\colon B\lra\CP$ a non-constant meromorphic function, in other words, a ramified covering of the complex projective line. The dessin $D:=\Dssn(j)$ associated to $j$ is the graph given by the following construction:
\begin{itemize}
\item as a set, the dessin $D$ coincides with $j^{-1}(\RP)$, where $\RP$ is the fixed point set of the standard complex conjugation in $\CP$;
\item black vertices $(\bullet)$ are the inverse images of $0$;
\item white vertices $(\circ)$ are the inverse images of $1$;
\item vertices of color $\times$ are the inverse images of $\infty$;
\item monochrome vertices are the critical points of $j$ with critical value in $\CP\setminus\{0,1,\infty\}$;
\item solid edges are the inverse images of the interval $[\infty,0]$;
\item bold edges are the inverse images of the interval $[0,1]$;
\item dotted edges are the inverse images of the interval $[1,\infty]$;
\item orientation on edges is induced from an orientation of $\RP$.
\end{itemize}
\end{df}

\begin{lm}[\cite{deg}]
Let $S$ be an oriented connected closed surface. Let ${j\colon S\lra\CP}$ a ramified covering map. The trichotomic graph $D=\Dssn(j)\subset S$ is a dessin. Moreover, if $j$ is real with respect to an orientation-reversing involution $c\colon S\lra S$, then $D$ is $c$-invariant.
\end{lm}

Let $(S,c)$ be a compact real surface. If $j\colon (S,c)\lra (\CP,z\lra\bar{z})$ is a real map, we define $\Dssn_{c}(j):=\Dssn(j)/c\subset S/c$.

\begin{thm}[\cite{deg}]
Let $S$ be an oriented connected closed surface (and let ${c\colon S\lra S}$ a orientation-reversing involution). A (real) trichotomic graph $D\subset S$ is a (real) dessin if and only if $D=\Dssn(j)$ 
for a (real) ramified covering $j\colon S\lra\CP$. 

Moreover, $j$ is unique up to homotopy in the class of (real) ramified coverings with dessin $D$.
\end{thm}

The last theorem together with the Riemann existence theorem provides the next corollaries, for the complex and real settings.

\begin{coro}[\cite{deg}] \label{coro:Ej}
Let $D\subset S$ be a dessin on a compact closed orientable surface~$S$. Then there exists a complex structure on $S$ and a holomorphic map $j\colon S\lra\CP$ such that $\Dssn(j)=D$. Moreover, this structure is unique up to deformation of the complex structure on $S$ and the map $j$ in the Kodaira-Spencer sense.
\end{coro}

\begin{coro}[\cite{deg}] \label{coro:EjR}
Let $D\subset S$ be a dessin on a compact surface $S$. Then there exists a complex structure on its double cover $\widetilde{S}$ and a holomorphic map~$j\colon\widetilde{S}\lra\CP$ such that $j$ is real with respect to the real structure $c$ of $\widetilde{S}$ and $\Dssn_{c}(j)=D$. Moreover, this structure is unique up to equivariant deformation of the complex structure on $S$ and the map $j$ in the Kodaira-Spencer sense.
\end{coro}

\subsubsection{Deformations of dessins}

In this section we describe the notions of deformations which allow us to associate classes of non-isotrivial trigonal curves and classes of dessins, up to deformations and equivalences that we explicit.

\begin{df}
A {\it deformation of coverings} is a homotopy $S\times [0,1] \lra \CP$ within the class of (equivariant) ramified coverings. The deformation is {\it simple} if it preserves the multiplicity of the inverse images of $0$, $1$, $\infty$ and of the other real critical values.
\end{df}

Any deformation is locally simple except for a finite number of values~$t\in [0,1]$.

\begin{prop}[\cite{deg}]
 Let $j_{0}, j_{1}\colon S\lra\CP$ be ($c$-equivariant) ramified coverings. They can be connected by a simple (equivariant) deformation if and only their dessins $D(j_0)$ and $D(j_{1})$ are isotopic (respectively, $D_{c}(j_{0})$ and $D_{c}(j_{1})$).
\end{prop}

\begin{df} \label{df:equidef}
 A deformation $j_{t}\colon S\lra\CP$ of ramified coverings is \emph{equisingular} if the union of the supports 
\[\bigcup_{t\in [0,1]} \operatorname{supp}\left\{ (j^{*}_{t}(0)\mod 3)+(j^{*}_{t}(1)\mod 2)+j^{*}_{t}(\infty)\right\}\] 
considered as a subset of $S\times [0,1]$ is an isotopy. Here $^{*}$ denotes the divisorial pullback of a map $\varphi:S\lra S'$ at a point $s'\in S'$:
 \[ \varphi^{*}(s')=\displaystyle\sum_{s\in\varphi^{-1}(s')}r_{s}s, \] where $r_{s}$ if the ramification index of $\varphi$ at $s\in S$.
\end{df}

A dessin $D_1\subset S$ is called a \emph{perturbation} of a dessin $D_0\subset S$, and $D_0$ is called a \emph{degeneration} of $D_1$, if for every vertex $v\in \Ver(D_0)$ there exists a small neighboring disk $U_v\subset S$ such that $D_0\cap U_v$ only has edges incident to $v$,  $D_1\cap U_v$ contains essential vertices of at most one color, and $D_0$ and $D_1$ coincide outside of $U_v$.

\begin{thm}[\cite{deg}] \label{thm:support}
 Let $D_{0}\subset S$ be a dessin, and let $D_{1}$ be a perturbation. Then there exists a map $j_{t}\colon S\lra\CP$ such that
\begin{enumerate}
 \item $D_{0}=\Dssn(j_0)$ and $D_{1}=\Dssn(j_{1})$;
 \item $j_{t}|_{S\setminus\bigcup_{v}U_{v}}=j_{t'}|_{S\setminus\bigcup_{v}U_{v}}$ for every $t$, $t'$;
 \item the deformation restricted to $S\times (0,1]$ is simple.
\end{enumerate}
\end{thm}

\begin{coro}[\cite{deg}] \label{coro:def}
 Let $S$ be a complex compact curve, $j\colon S\lra\CP$ a non-constant holomorphic map, and let $\Dssn(j)=D_{0}, D_{1}, \dots, D_{n}$ be a chain of dessins in $S$ such that for $i=1, \dots, n$ either $D_{i}$ is a perturbation of $D_{i-1}$, or $D_{i}$ is a degeneration of $D_{i-1}$, or $D_{i}$ is isotopic to $D_{i-1}$. Then there exists a piecewise-analytic deformation $j_{t}\colon S_{t}\lra\CP$, $t\in[0,1]$, of $j_0=j$ such that $\Dssn(j_{1})=D_{n}$.
\end{coro}

\begin{coro}[\cite{deg}] \label{coro:defR}
 Let $(S,c)$ be a real compact curve, $j\colon (S,c)\lra(\CP,{z\lmt\bar{z}})$ be a real non-constant holomorphic map, and let $\Dssn_{c}(j)= D_{0}, D_{1}, \dots, D_{n}$ be a chain of real dessins in $(S,c)$ such that for $i=1, \dots, n$ either $D_{i}$ is a equivariant perturbation of $D_{i-1}$, or $D_{i}$ is a equivariant degeneration of $D_{i-1}$, or $D_{i}$ is equivariantly isotopic to $D_{i-1}$. Then there is a piecewise-analytic real deformation $j_{t}\colon (S_{t},c_{t})\lra(\CP,\bar{\cdot})$, $t\in[0,1]$, of $j_0=j$ such that $\Dssn_{c}(j_{1})=D_{n}$.
\end{coro}

Due to Theorem \ref{thm:support}, the deformation $j_t$ given by Corollaries \ref{coro:def} and \ref{coro:defR} is equisingular in the sense of Definition \ref{df:equidef} if and only if all perturbations and degenerations of the dessins on the chain $D_{0}, D_{1}, \dots, D_{n}$ are equisingular.

\subsection{Trigonal curves and their dessins}

In this section we describe an equivalence between dessins.

\subsubsection{Correspondence theorems}

Let $C\subset\Si\lra B$ be a non-isotrivial proper trigonal curve.
We associate to $C$ the dessin corresponding to its $j$-invariant $\Dssn(C):=\Dssn(j_{C})\subset B$. 
In the case when $C$ is a real trigonal curve we associate to $C$ the image of the real dessin corresponding to its $j$-invariant under the quotient map, $\Dssn_{c}(C):=\Dssn(j_{C})\subset B/{c_{B}}$, where~$c_{B}$ is the real structure of the base curve $B$.\\

So far, we have focused on one direction of the correspondences: we start with a trigonal curve $C$, 
consider its $j$-invariant and 
construct the dessin associated to it.  
Now, we study the opposite direction. 
Let us consider a dessin $D$ on a topological orientable closed surface $S$. By Corollary \ref{coro:Ej},
there exist a complex structure~$B$ on~$S$ and a holomorphic map $j_{D}\colon B\lra\CP$ such that $\Dssn(j_{D})=D$. By Theorem~\ref{thm:Etc} and Corollary~\ref{coro:Etc} there exists a trigonal curve~$C$ having~$j_{D}$ as $j$-invariant; such a curve is unique up to deformation in the class of trigonal curves with fixed dessin. Moreover, due to Corollary~\ref{coro:def}, any sequence of isotopies, perturbations and degenerations of dessins gives rise to a piecewise-analytic deformation of trigonal curves, which is equisingular if and only if all perturbations and degenerations are.

In the real framework, let $(S,c)$ a compact close oriented topological surface endowed with a orientation-reversing involution. Let $D$ be a real dessin on $(S,c)$. By Corollary~\ref{coro:EjR}, there exists a real structure $(B,c_{B})$ on $(S,c)$ and a real map \linebreak[1]$j_{D}\colon(B,c_{B})\lra(\CP,z\longmapsto\bar{z})$ such that $\Dssn_c(j_{D})=D$. By Theorem \ref{thm:Etc}, Corollary \ref{coro:Etc} and the remarks made in Section \ref{sssec:rtc}, there exists a real trigonal curve~$C$ having~$j_{D}$ as $j$-invariant; such a curve is unique up to equivariant deformation in the class of real trigonal curves with fixed dessin. Furthermore, due to Corollary~\ref{coro:defR}, any sequence of isotopies, perturbations and degenerations of dessins gives rise to a piecewise-analytic equivariant deformation of real trigonal curves, which is equisingular if and only if all perturbations and degenerations are.

\begin{df}
A dessin is \emph{reduced} if
\begin{itemize}
\item 
for every $v$ $\bullet$-vertex one has $\Ind{v}\leq3$,
\item 
for every $v$ $\circ$-vertex one has $\Ind{v}\leq2$,
\item every monochrome vertex is real and has index $2$.
\end{itemize}
A reduced dessin is {\it generic} if all its $\bullet$-vertices and $\circ$-vertices are non-singular and all its $\times$-vertices have index~$1$.
\end{df}

Any dessin admits an equisingular perturbation to a reduced dessin. The vertices with excessive index (i.e., index greater than 3 for $\bullet$-vertices or than 2 for $\circ$-vertices) can be reduced by introducing new vertices of the same color.

In order to define an equivalence relation of dessins, we introduce 
{\it elementary moves}. Consider two reduced dessins $D$, $D'\subset S$ such that they coincide outside a closed disk $V\subset S$.
If $V$ does not intersect $\partial S$ and the graphs  $D\cap V$ and $D'\cap V$ are as shown in Figure~\ref{fig:elem}(a), then we say that performing a {\it monochrome modification} on the edges intersecting $V$ produces $D'$ from $D$, or {\it vice versa}.
This is the first type of {\it elementary moves}. 
Otherwise, the boundary component inside $V$ is shown in light gray. In this setting, if the graphs $D\cap V$ and $D'\cap V$ are as shown in one of the subfigures in Figure~\ref{fig:elem}, we say that performing an \emph{elementary move} of the corresponding type on $D\cap V$ produces $D'$ from $D$, or {\it vice versa}. 

\begin{df} 
Two reduced dessins $D$, $D'\subset S$ are {\it elementary equivalent} if, after a (preserving orientation, in the complex case) homeomorphism of the underlying surface $S$ they can be connected by a sequence of isotopies and elementary moves between dessins, as described in Figure~\ref{fig:elem}.
\end{df}

\begin{figure} 
\centering
\includegraphics[width=5in]{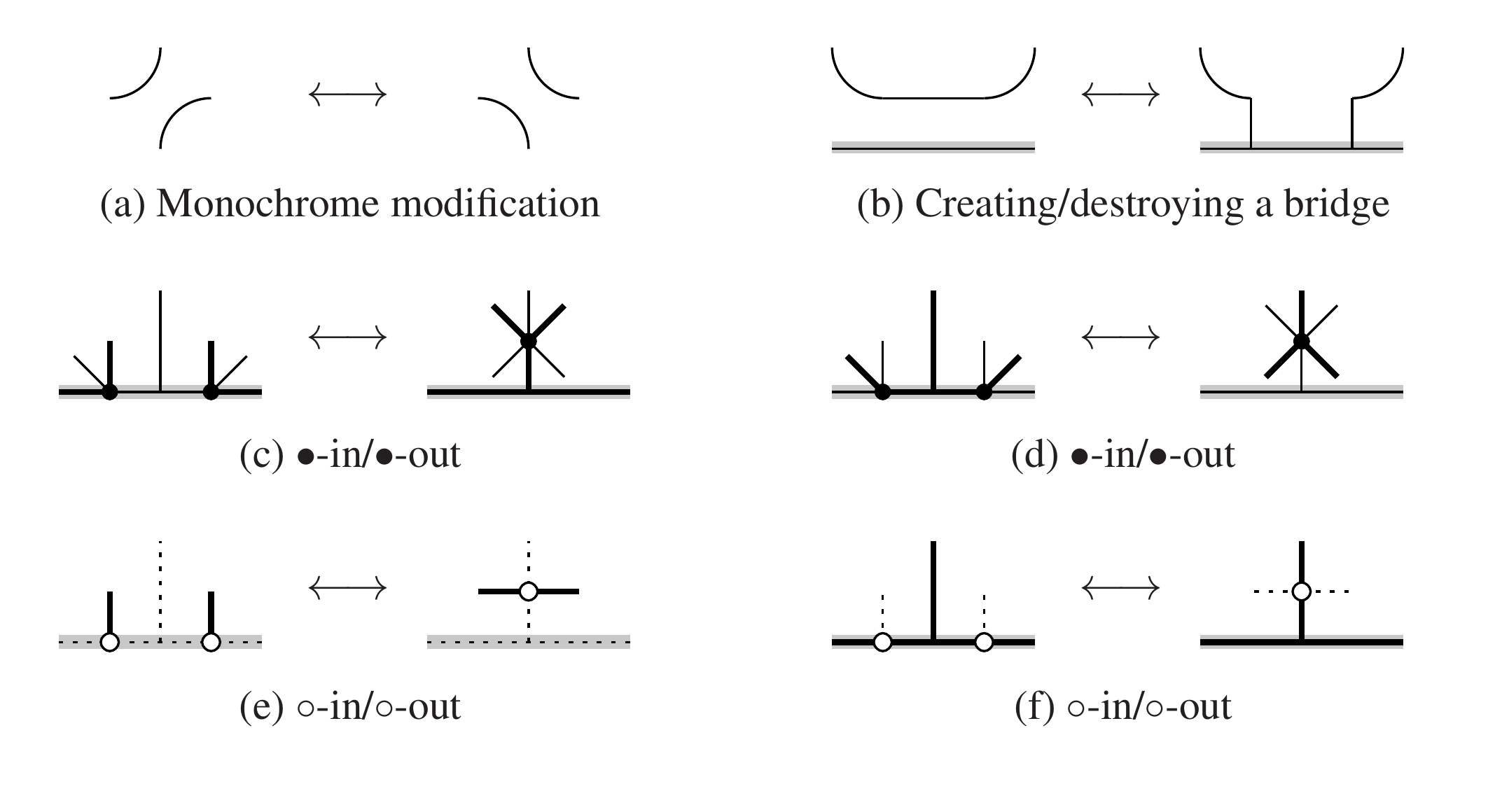}
\caption{Elementary moves.\label{fig:elem}}
\end{figure}

This definition is meant so that two reduced dessins are elementary equivalent if and only if they can be 
connected up to homeomorphism by a sequence of isotopies, equisingular perturbations and degenerations.

The following theorems establish the equivalences between the deformation classes of trigonal curves we are interested in and elementary equivalence classes of certain dessins. We use these links to obtain different classifications of curves {\it via} the combinatorial study of dessins.

\begin{thm}[\cite{deg}] \label{th:correpondance1}
 There is a one-to-one correspondence between the set of equisingular deformation classes of non-isotrivial proper trigonal curves ${C\subset\Si\lra B}$ with $\widetilde{A}$ type singular fibers only and the set of elementary equivalence classes of reduced dessins $D \subset B$.
\end{thm}

 \begin{thm}[\cite{deg}] \label{th:correpondance2}
 There is a one-to-one correspondence between the set of equivariant equisingular deformation classes of non-isotrivial proper real trigonal\linebreak[1] curves $C\subset\Si\lra(B,c)$ with $\widetilde{A}$ type singular fibers only and the set of elementary equivalence classes of reduced real dessins $D \subset B/c$.
\end{thm}

\begin{thm}[\cite{deg}] \label{th:correpondance3}
 There is a one-to-one correspondence between the set of equivariant equisingular deformation classes of almost generic real trigonal\linebreak[1] curves $C \subset\Si\lra(B,c)$ and the set of elementary equivalence classes of generic real dessins $D \subset B/c$.
\end{thm}

This correspondences can be extended to trigonal curves with more general singular fibers (see \cite{deg}).

\begin{df} 
Let $C\subset\Si\lra B$ be a proper trigonal curve. 
We define the {\it degree} of the curve $C$ as $\deg(C):=-3E^2$ where $E$ is the exceptional section of~$\Si$. For a dessin $D$, we define its {\it degree} as $\deg(D)=\deg(C)$ where $C$ is a minimal proper trigonal curve such that $\Dssn(C)=D$.
\end{df}

\subsubsection{Real generic curves}

Let $C$ be a generic real trigonal curve and let $D:=\Dssn_{c}(C)$ be a generic dessin. 
The {\it real part} of $D\subset S$ is the intersection $D\cap\partial S$. For a specific color $\ast\in\{\mathrm{ solid},\mathrm{ bold},\mathrm{ dotted}\}$, $D_{\ast}$ is the subgraph of the corresponding color and its adjacent vertices. The components of $D_{\ast}\cap\partial S$ are either components of~$\partial S$, called \emph{monochrome components} of $D$, or 
segments, called {\it maximal monochrome segments} of $D$. 
We call these monochrome components or segments \emph{even} or \emph{odd} according to 
the parity of the number of $\circ$-vertices they contain.

Furthermore, we refer to the dotted monochrome components as \emph{hyperbolic components}. A dotted segment without $\times$-vertices of even index is referred to as an \emph{oval} if it is even, or as a \emph{zigzag} if it is odd.

\begin{df} 
Let $D\subset S$ be a real dessin. Assume that there is a subset of $S$ in which $D$ has a configuration of vertices and edges as in Figure~\ref{fig:zigzag}. Replacing the corresponding configuration with the alternative one defines another dessin $D'\subset S$. We say that $D'$ is obtained from $D$ by \emph{straightening/creating} a zigzag.

Two dessins $D,D'$ are \emph{weakly equivalent} if there exists a finite sequence of dessins $D=D_0, D_1, \dots,D_n=D'$ such that $D_{i+1}$ is either elementary equivalent to $D_{i}$, or $D_{i+1}$ is obtained from $D_i$ by straightening/creating a zigzag.
\end{df}

Notice that if $D'$ is obtained from $D$ by straightening/creating a zigzag and $\widetilde{D},\widetilde{D'}\subset\widetilde{S}$ are the liftings of $D$ and $D'$ in $\widetilde{S}$, the double complex of $S$, then $\widetilde{D}$ and $\widetilde{D'}$ are elementary equivalent as complex dessins. However, $D$ and $D'$ are not elementary equivalent, since the number of zigzags of a real dessin is an invariant on the elementary equivalence class of real dessins.

\begin{figure}
\centering 
\includegraphics[width=5in]{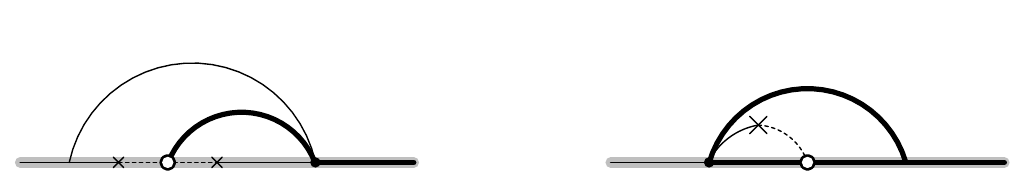}
\caption{Straightening/creating a zigzag.\label{fig:zigzag}}
\end{figure}

\subsubsection{Type of a dessin}
 Let $C\subset\Si\lra B$ be a real trigonal curve over a base curve of type~$I$. We define $C_{\mathrm{Im}}$ as the closure of the set $\pi^{-1}|_{C}(B_{\R})\setminus C_{\R}$ and let~$B_{\mathrm{\Im}}=\pi(C_{\mathrm{\Im}})$. By definition $C_{\mathrm{Im}}$ is invariant with respect to the real structure of~$C$. Moreover, $C_{\mathrm{Im}}=\emptyset$ if and only if $C$ is a hyperbolic trigonal curve.

\begin{lm}[\cite{DIZ}] 
 A trigonal curve is of type~$\mathrm{I}$ if and only if the homology class~$[C_{\mathrm{Im}}]\in H_{1}(C;\Z/2\Z)$ is zero.
\end{lm}

In view of the last lemma, we can represent a trigonal curve of type~$\mathrm{I}$ as the union of two orientable surfaces $C_{+}$ and $C_{-}$, intersecting at their boundaries ${C_{+}\cap C_{-}=\partial C_{+}=\partial C_{-}}$. Both surfaces, $C_{+}$ and $C_{-}$, are invariant under the real structure of~$C$. We define 
\[
m_{\pm}:
\begin{array}{ccc}
		B &	\lra	& \Z \\
		b & 	\lmt	&	\operatorname{Card}(\pi|_{C}^{-1}(b)\cap C_{\pm})-\chi_{B_{\mathrm{Im}}}(b), 
\end{array}
\]
where $\chi_{B_{\mathrm{Im}}}$ is the characteristic function of the set $B_{\mathrm{Im}}$. These maps are locally constant over $B^{\#}$, and since $B^{\#}$ is connected, the maps are actually constant. Moreover, $m_{+}+m_{-}=3$, so we choose the surfaces $C_{\pm}$ such that $m_{+}|_{B^{\#}}\equiv 1$ and~$m_{-}|_{B^{\#}}\equiv 2$. 

We can label each region $R\in\operatorname{Cut}(D)$ where $D=\Dssn(C)$ according to the label on $C_{+}$. Given any point $b\in R$ on the interior, the vertices of the triangle~$\pi|_{C}^{-1}(b)\subset F_b$ are labeled by 1, 2, 3, according to the
increasing order of lengths of the opposite side of the triangle. We label the region~$R$ by the label of the point~$\pi|_{C}^{-1}(b)\cap C_{+}$. 

We can also label the interior edges according to the adjacent regions in the following way:
\begin{itemize}
 \item every solid edge can be of type~$1$ (i.e., both adjacent regions are of type 1) or type~$\overline{1}$ (i.e., one region of type~$2$ and one of type~$3$);
 \item every bold edge can be of type~$3$ or type~$\overline{3}$;
 \item every dotted edge can be of type~$1$, $2$ or $3$.
\end{itemize}
We use the same rule for the real edges of $D$. Note that there are no real solid edges of type~$1$ nor real bold edges of type~$3$ (otherwise the morphism $C_{+}\lra B$ would have two layers over the regions of $D$ adjacent to the edge).

\begin{thm}[\cite{DIZ}]\label{df:type}
 A generic non-hyperbolic curve $C$ is of type~$\mathrm{I}$ if and only if the regions of $D$ admit a labeling which satisfies the conditions described above.
\end{thm}

\section{Toiles}\label{ch:toi}

Generic trigonal curves are smooth.
Smooth proper trigonal curves have non-singular dessins.
Singular proper trigonal curves have singular dessins and the singular points are represented by singular vertices. 
A generic singular trigonal curve $C$ has exactly one singular point, which is a non-degenerate double point ({\it node}). 
Moreover, if $C$ is a proper trigonal curve, then the double point on it is represented by a $\times$-vertex of index $2$ on its dessin.
In addition, if $C$ has a real structure,
the double point is real and so is its corresponding vertex, leading to the cases where the $\times$-vertex of index $2$ has dotted real edges (representing the intersection of two real branches) or has solid real edges (representing one isolated real point, which is the intersection of two complex conjugated branches).

\begin{df} 
Let $D\subset S$ be a dessin on a compact surface $S$. A {\it nodal vertex} 
({\it node})  
of $D$ is a $\times$-vertex of index $2$.
The dessin $D$ is called \emph{nodal} if all its singular vertices are nodal vertices.
We call a \emph{toile} a non-hyperbolic real nodal dessin on $(\CP,z\lmt\bar{z})$. 
\end{df}

Since a real dessin on $(\CP,z\lmt\bar{z})$ descends to the quotient, we represent toiles on the disk.

In a real dessin, there are two types of real nodal vertices, namely, vertices having either real solid edges and interior dotted edges, or dotted real edges and interior solid edges. We call \emph{isolated nodes} of a dessin $D$ thoses $\times$-vertices of index $2$ corresponding to the former case and \emph{non-isolated nodes} those corresponding to the latter. 

\begin{df} 
Let $D\subset S$ be a real dessin.
A {\it bridge} of $D$ is an edge $e$ contained in a connected component of the boundary $\partial S$ having more than two vertices, such that $e$ connects two monochrome vertices. The dessin $D$ is called {\it bridge-free} if it has no bridges.
The dessin $D$ is called {\it peripheral} if it has no inner vertices other than $\times$-vertices.
\end{df}

For non-singular dessins, combinatorial statements analogous to Lemma~\ref{lm:bf} and Proposition~\ref{prop:deep} are proved in \cite{DIK}.

\begin{lm} \label{lm:bf} \label{cr:bf}
  A nodal dessin $D$ is elementary equivalent to a bridge-free dessin $D'$ having the same number of inner essential vertices and real essential vertices.
\end{lm}

\begin{proof}
Let $D$ be a dessin on $S$. Let $e$ be a bridge of $D$ lying on a connected component of $\partial S$. Let $u$ and $v$ be the endpoints of $e$. Since $e$ is a bridge, there exists at least one real vertex $w\neq u$ of $D$ adjacent to $v$. If $w$ is an essential vertex, destroying the bridge is an admissible elementary move. 
Otherwise, the vertex $w$ is monochrome of the same type as $u$ and $v$, and the edge connecting $v$ and $w$ is another bridge $e'$ of $D$. 
The fact that every region of the dessin contains on its boundary essential vertices implies that after destroying the bridge $e$ the regions of the new graph have an oriented boundary with essential vertices. Therefore the resulting graph is a dessin and destroying that bridge is admissible.
All the elementary moves used to construct the elementary equivalent dessin $D'$ from $D$ are destructions of bridges. Since destroying bridges does not change the nature of essential vertices, $D$ and $D'$ have the same amount of inner essential vertices and real essential vertices.
\end{proof}

A real singular $\times$-vertex $v$ of index $2$ in a dessin can be perturbed within the class of real dessins 
on the same surface in two different ways. Locally, when $v$ is isolated,
the real part of the corresponding real trigonal curve has an isolated point as connected component of its real part, which 
can be perturbed to a topological circle or disappears, when $v$ becomes two real $\times$-vertices of index $1$ or one pair of complex conjugated interior $\times$-vertices of index $1$, respectively.
When $v$ is non-isolated, the real part of the corresponding real trigonal curve has a double point, which
can be perturbed to two real branches without ramification or 
leaving a segment of the third branch being one-to-one with respect to the projection $\pi$ while being three-to-one after two vertical tangents (see Figure~\ref{fig:deformsnodal}).  

\begin{figure}
\centering
    {\includegraphics[width=4in]{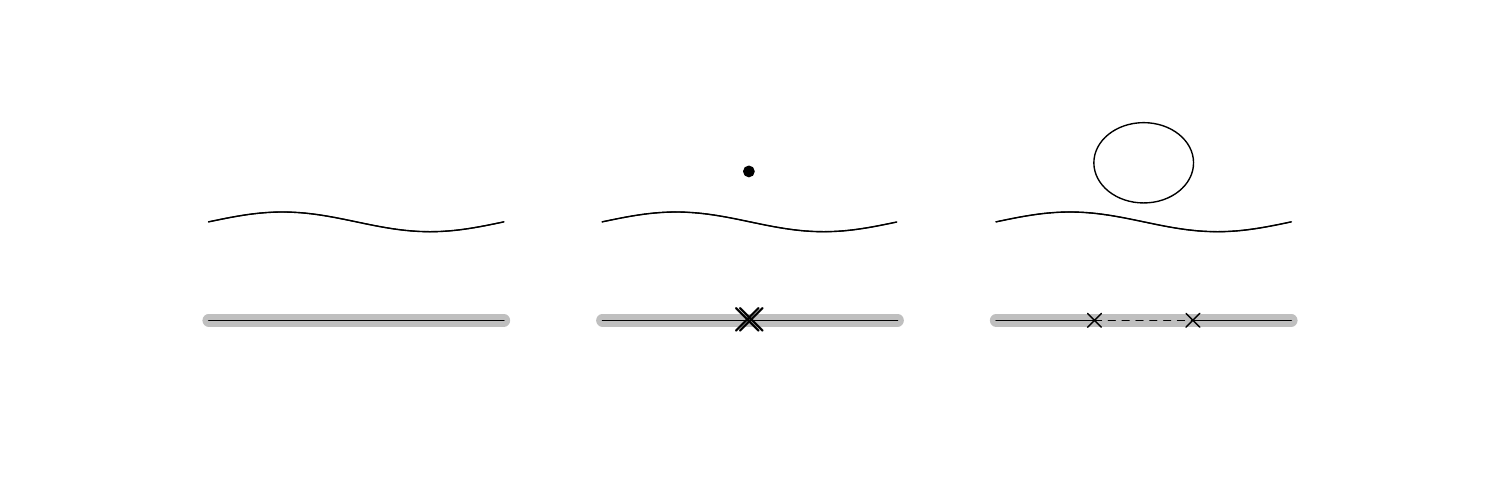}} \\
    {\includegraphics[width=4in]{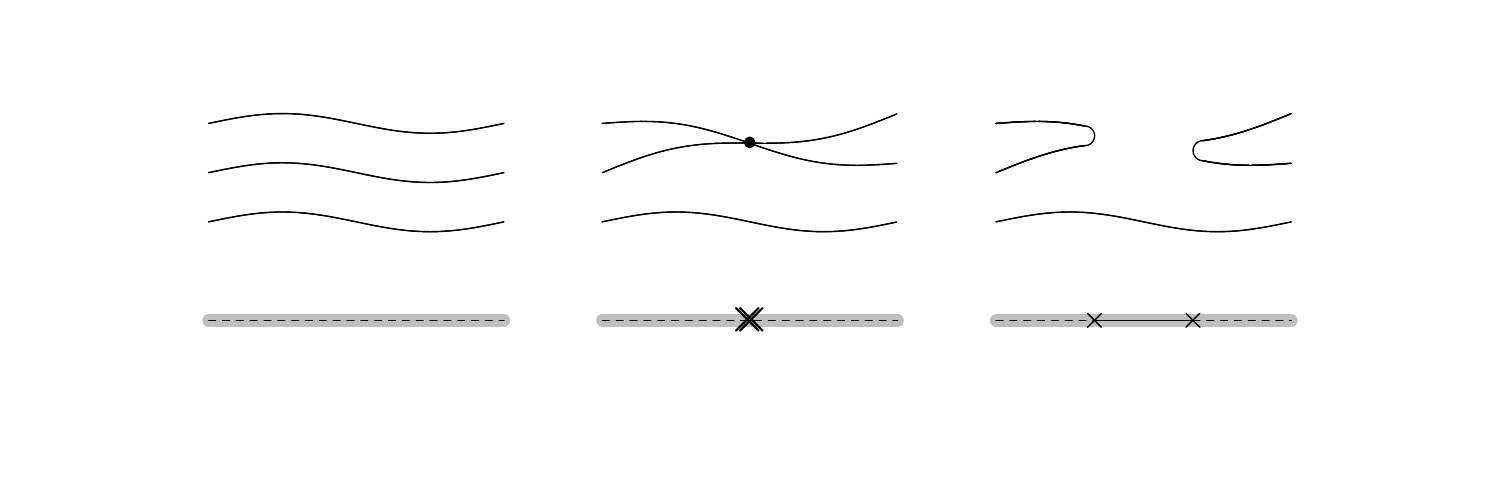}}
    \caption{Perturbations of a real $\times$-vertex of index $2$.}
    \label{fig:deformsnodal}
\end{figure}

\begin{df} 
Given a dessin $D$, a subgraph $\Gamma\subset D$ is a \emph{cut} if it consists of a single interior edge connecting two real monochrome vertices. An {\it axe} is an interior edge of a dessin connecting a real $\times$-vertex of index $2$ and a real monochrome vertex.
\end{df}
 
Let us consider a dessin $D$ lying on a surface $S$ having a cut or an axe~$T$. 
Assume that $T$ divides $S$ and consider the connected components $S_{1}$ and $S_2$ of~$S\setminus T$. 
Then, we can define two dessins $D_1$, $D_2$, each lying on the compact surface $\overline{S_i}\subset S$, respectively for $i=1, 2$, and determined by $D_i:=(D\cap S_i)\cup\{T\}$. 
If $S\setminus T$ is connected, we define the surface $S'=(S\setminus T)\sqcup T_1 \sqcup T_2/\varphi_1, \varphi_2$, where $\varphi_i:T_i\lra S$ is the inclusion of one copy $T_i$ of $T$ into $S$, and the dessin $D':=(D\setminus T)\sqcup T_1 \sqcup T_2/\varphi_1, \varphi_2$.
 
By these means, a dessin having a cut or an axe determines either two other dessins of smaller degree or a dessin lying on a surface with a smaller fundamental group. 
Moreover, in the case of an axe, the resulting dessins have one singular vertex less. 
Considering the inverse process, we call $D$ the \emph{gluing} of $D_1$ and $D_2$ along $T$ or the \emph{gluing} of $D'$ with itself along $T_1$ and $T_2$. 

\label{sec:toiles}

\begin{df} 
Given a real dessin $D$ and a vertex $v\in\Ver(D)$, we call the \emph{depth} of~$v$ the minimal number $n$ such that there exists an undirected inner chain $v_0,\dots, v_n$ in $D$ from $v_0=v$ to a real vertex $v_n$ and we denote the depth of $v$ by $\operatorname{dp}(v)$.
The depth of a dessin $D$ is defined as the maximum of the depth of the black and white vertices of $D$ and it is denoted by $\operatorname{dp}(D)$.
\end{df}

\begin{df} 
A {\it {\gc}} of a dessin $D$ is an inner undirected chain formed entirely of inner edges of the same color, either dotted or solid, connecting two distinct real nodal or monochrome vertices.
\end{df}

Analogously to a cut, cutting a dessin $D\subset S$ by a {\gc} produces two dessins of lower degree or a dessin of the same degree in a surface with a simpler topology, depending on whether the inner chain divides or not the surface $S$.

\begin{prop} \label{prop:deep}
Let $D\subset \mathbb{D}^2$ be a toile of degree greater than $3$. Then, there exists a toile $D'$ weakly equivalent to $D$ such that either $D'$ has depth $1$ or $D'$ has a {\gc}.
\end{prop}


\begin{proof}
Within the class of elementary equivalence of our initial toile $D$, let us choose a toile $D_0$ having minimal depth. Due to Corollary \ref{cr:bf} we can choose $D_0$ bridge-free.
If $\operatorname{dp}(D_0)\geq2$, then there is a vertex $v_0\in\Ver(D_0)$ having $\operatorname{dp}(v_0)\geq2$, hence there exists an undirected inner chain $v_0, v_1,\dots, v_n$, with $n={\operatorname{dp}(v_0)}$, of minimal size connecting $v_0$ to the boundary of $D^2$. By definition, $\operatorname{dp}(v_{n-2})=2$. Put $u=v_{n-2}, v=v_{n-1}, w=v_{n}, e=[u,v], f=[v,w]$. We study the possible configurations of the chain $u,v,w$ in order to show that, unless there exists a {\gc}, we can decreases the depth of every vertex having depth at least $2$, contradicting minimality assumption. We assume that the toile $D_0$ does not have dotted cuts, since otherwise the Proposition follows trivially.

\vskip5pt
{\it Case 0:} whenever $v$ is black or white and connected to a monochrome vertex, an elementary move of type $\bullet$-out or $\circ$-out, respectively, reduces the depth of $u$. For simplicity, from now on we assume that $v$ is only connected to essential vertices.

\vskip5pt
{\it Case 1.1:} the vertex $u$ is white and the vertex $v$ is black.

{\it Case 1.1.1:} the vertex $w$ is white.
Let $g$ be the solid edge adjacent to $v$ sharing a region with the edges $e$ and $f$.
Let $w_1$ be the real neighbor vertex to $w$ in the region $R$ determined by $f$ and $g$.
If $w_1$ is a simple {\tv}, it determines a solid segment where the construction of a bridge with the edge $g$ bring us to Case~0.
If $w_1$ is a nodal {\tv}, up to a monochrome modification it is connected to $v$, and then the creation of a bridge with the inner dotted edge of $u$ beside $w_1$ decreases the depth of $u$.
Otherwise, $w_1$ is monochrome. If it is connected to an inner simple {\tv} $w_2$, up to a monochrome modification the vertex $w_2$ is connected to $v$, and as before, the creation of a bridge with the inner dotted edge of $u$ beside $w_1$ decreases the depth of $u$.
If $w_1$ is connected to an inner nodal {\tv} $w_2$, up to a monochrome modifications the vertex $w_2$ is connected to $u$ and $v$. Since the dessin is a bridge-free toile, $w_1$ has a white real neighbor vertex $w_3\neq w$.

If $w_3$ is connected to a monochrome vertex by an inner edge, then in the region determined by the vertices $w_1$ and $w_3$ there is a black vertex $w_4$ and up to monochrome modifications the vertices $u$ and $w_2$ are connected to $w_4$, reducing the depth of $u$.
If $w_3$ is connected to a real black vertex $w_4$, then there are two cases: in the region determined by the vertices $w_1$ and $w_3$, the vertex $w_4$ is adjacent to a inner or real solid edge $h$. When the edge $h$ is inner, up to a monochrome modification, the vertices $v$ and $w_4$ are connected, then the creation of a bridge with a inner bold edge of $u$ beside $w_4$ bring us to Case~0.

When the edge $h$ is real, connecting $w_4$ with a {\tv} $w_5$, then the creation of a dotted bridge with the edge adjacent to $w_1$ beside $w_5$ or a monochrome modification connecting $w_1$ and $w_5$, respectively if it is a simple or nodal {\tv}, produces a {\gc}.
When the edge $h$ connects $w_4$ with a monochrome vertex$w_5$, sine the toile is bridge free, there exist a real black vertex $w_6\neq w_4$ connected to $w_5$. Monochrome modifications connect $w_2$ with $w_5$ and $u$ with $w_6$ reducing the depth of $u$.
Finally, if the vertex $w_3$ is connected to an inner black vertex $w_4$, which up to monochrome modifications is connected to $u$ and $w_2$, we consider the real neighbor vertex $w_5\neq w_1$.
If $w_5$ is a simple {\tv}, it determines a solid segment where the creation of a bridge with the inner solid edge of $v$ in the region bring us to Case~0.
If $w_5$ is a nodal {\tv}, up to a monochrome modification it is connected to $v$. Let $w_6\neq u,w$ be the vertex connected to $v$ by a bold edge.
If $w_6$ is an inner white vertex, the creation of a dotted bridge followed by an elementary move of type $\circ$-out lead us to the next consideration. 
When the vertex $w_6$ is a real white vertex, let us defined $w':=w_6$ and consider instead the chain $u,v,w'$ and the considerations made in this algorithm. 
If the algorithm cycles back to this configuration and we denote by $w_i'$ the vertices on the second iteration, the edges $[w_1,w_2],[w_2,u],[u,w_2'],[w_2',w_1']$ form a {\gc}.

{\it Case~1.1.2:} the vertex $w$ is nodal.
If $w$ has a real white neighboring vertex $w'$, we consider the chain $u,v,w'$ as in Case~1.1.1.
Otherwise $w$ has a real monochrome neighboring vertex $w_1$.
If $w_1$ is connected to an inner white vertex $w_2$, an elementary move of type $\circ$-out with it at $w_2$ creates real white neighboring vertex of $w$ and we consider the previous case.
If $w_1$ is connected real white vertex $w_2$, then the creation of a bold bridge with an inner bold edge of $v$ beside $w_2$ followed by an elementary move of type $\bullet$-out bring us to a configuration in which $u$ is connected to a real vertex, so its depth is reduced.

\vskip5pt
{\it Case~1.2:} the vertex $u$ is white and the vertex $v$ is a simple {\tv}.

{\it Case~1.2.1:} the vertex $w$ is black.
The vertex $w$ has a real solid bold edge in the same region as $u$. Then, the creation of a bridge with an inner bold edge of $u$ beside $w$ followed by an elementary move of type $\circ$-out transfers the vertex $u$ to the boundary in a bridge-free toile.

{\it Case~1.2.2:} the vertex $w$ is monochrome.
Since the toile is bridge-free, $w$ has a real black neighboring vertex which is connected to $u$ after a monochrome modification reducing the depth of $u$.

\vskip5pt
{\it Case~1.3:} the vertex $u$ is white and the vertex $v$ is a nodal {\tv}.

{\it Case~1.3.1:} the vertex $w$ is a monochrome dotted vertex.
Let $w_1$ be a vertex connected to $v$ by a solid edge.
If $w_1$ is real, then either its bold inner edge can be connected to $u$ after a monochrome modification or its bold real edge allows us to create a bridge with the inner bold edge of $u$ and then an elementary move of type $\circ$-out transfers the vertex $u$ to the boundary in a bridge-free toile.
If $w_1$ is a monochrome vertex, since the toile is bridge-free, then it has has a real black neighboring vertex $w_2$. A monochrome modification connects $u$ to $w_2$ decreasing the depth of $u$ in a bridge-free toile.
Otherwise, the vertex $w_1$ is an inner black vertex. Since the the toile is bridge-free, there exist a white vertex $w_2$ real neighbor of $w$ in the same region as $w_1$. A monochrome modification connects $w_1$ to $w_2$. We consider instead the chain $u,w_1,w_2$ as in Case~1.1.1.

{\it Case~1.3.2:} the vertex $w$ is a monochrome solid vertex.
This configuration was considered in Case~1.3.1.

{\it Case~1.3.3:} the vertex $w$ is white.
Let $w_1$ be a vertex connected to $v$ by a solid edge.
If $w_1$ is a real vertex, the considerations made on the Case~1.3.1 apply.
Otherwise, $w_1$ is an inner black vertex. Up to monochrome modifications 
The creation of a bold bridge beside $w$ with an inner bold edge adjacent to $w_1$ followed by an elementary move of type $\bullet$-out decreases the depth of $u$ in a bridge-free toile.

{\it Case~1.3.4:} the vertex $w$ is black.
This case correspond to the configuration in Case~1.3.1 when $w_1$ is a real vertex.

\vskip5pt
{\it Case~2.1:} the vertex $u$ is black and the vertex $v$ is white.

{\it Case~2.1.1:} the vertex $w$ is black.
Let $w_1$ be a vertex connected to $v$ by a dotted edge.
If $w_1$ is a real nodal {\tv}, then the creation of a bridge beside $w_1$ with an inner solid edge of $u$ in the region followed by an elementary move of type $\bullet$-out transfers the vertex $u$ to the boundary in a bridge-free toile.
If $w_1$ is an inner simple {\tv}, then we consider the inner solid edge $g$ adjacent to $w$. If $g$ and $w_1$ belong to the same region, then we can assume $g$ is adjacent to $w_1$ up to a monochrome modification. In this setting, the creation of a bridge with the edge $e$ beside $w$ followed by an elementary move of type $\bullet$-out transfers the vertex $u$ to the boundary in a bridge-free toile.
If $g$ and $w_1$ do not share any region, then we can assume $w_1$ is connected to $u$ up to a monochrome modification. In this setting the creation of a bridge with the inner solid edge connecting $w_1$ and $u$ beside $w$ followed by an elementary move of type $\bullet$-out transfers the vertex $u$ to the boundary in a bridge-free toile.
Lastly, if $w_1$ is a nodal {\tv}, up to monochrome modifications it is connected to $u$ and $w$.
Let $w_2\neq v$ be the vertex connected to $w_2$ by a dotted edge.
If $w_2$ is a real white vertex, then the creation of a bridge with an inner bold edge adjacent to $u$ beside $w_2$ followed by an elementary move of type $\circ$-out transfers the vertex $u$ to the boundary in a bridge-free toile.
If $w_2$ is an inner white vertex, up to a monochrome modification it is connected to $u$, then the creation of a bridge with an inner bold edge adjacent to $w_2$ beside $w$ followed by elementary moves of type $\circ$-out at $w_2$ and $\bullet$-out at $u$ transfers the vertex $u$ to the boundary in a bridge-free toile.
If $w_2$ is a monochrome vertex, it has two real white neighboring vertices, then an elementary move of type $\circ$-in at $w_2$ brings us to the previous consideration.

{\it Case~2.1.2:} the vertex $w$ is a nodal {\tv}.
This case corresponds to the configuration in Case~2.1.1 when $w_1$ is a real nodal {\tv}.

\vskip5pt
{\it Case~2.2:} the vertex $u$ is black and the vertex $v$ is a simple {\tv}.

{\it Case~2.2.1:} the vertex $w$ is monochrome.
Since the toile is bridge-free, the vertex $w$ has two different real white neighboring vertices. A monochrome modification connects $u$ to one of those vertices decreasing the depth of $u$.

{\it Case~2.2.2:} the vertex $w$ is white.
In this setting, the creation of a bold bridge beside $w$ with an inner bold edge adjacent to $v$ followed by an elementary move of type $\bullet$-out transfers the vertex $u$ to the boundary in a bridge-free toile.

\vskip5pt
{\it Case~2.3:} the vertex $u$ is black and the vertex $v$ is a nodal {\tv}.

{\it Case~2.3.1:} the vertex $w$ is monochrome dotted.
The vertex $u$ is connected to an inner white vertex $w_1$ on the same region as the edge $f$. The creation of a dotted bridge beside $w$ with an inner dotted edge of $w_1$ followed by an elementary move of type $\circ$-out decreasing the depth of $u$.

{\it Case~2.3.2:} the vertex $w$ is white.
In this setting, the creation of a bold bridge beside $w$ with an inner bold edge adjacent to $v$ follow by an elementary move of type $\bullet$-out transfers the vertex $u$ to the boundary in a bridge-free toile.

{\it Case~2.3.3:} the vertex $w$ is monochrome solid.
Let $w_1$ be a vertex connected to $v$ by a dotted edge.
If $w_1$ is a real monochrome vertex or a white vertex, we consider instead the chain $u,v,w_1$ as in Case~2.3.1 and Case~2.3.2, respectively.
If $w_1$ is an inner white vertex, let $w_2$ be a real black vertex neighbor of $w$ sharing a region with $w_1$. Up to a monochrome modification $w_1$ is connected to $w_2$. We consider instead the chain $u,w_1,w_2$ as in Case~2.1.1.

{\it Case~2.3.4:} the vertex $w$ is black.
Let $w_1$ be a vertex connected to $v$ by a dotted edge.
If $w_1$ is a real monochrome vertex or a white vertex, we consider instead the chain $u,v,w_1$ as in Case~2.3.1 and Case~2.3.2, respectively.
If $w_1$ is an inner white vertex, then up to a monochrome modification $w_1$ is connected to $w$. We consider instead the chain $u,w_1,w$ as in Case~2.1.1.

\vskip5pt
{\it Case~3.1:} the vertex $u$ is a simple {\tv} and the vertex $v$ is white.

{\it Case~3.1.1:} the vertex $w$ is black.
Let $g$ be the solid edge adjacent to $w$ sharing a region with the vertex $u$.
If $g$ is an inner edge, up to monochrome modification between $g$ and the solid edge adjacent to $u$ decreases the depth of $u$.
If $g$ is a real edge, let $w_1$ be the vertex connected to $u$ by a solid edge. Since the depth of $u$ is two, the vertex $w_1$ is an inner black vertex. The creation of a bridge on $g$ with the edge $[u,w_1]$ followed by an elementary move of type $\bullet$-out at $w_1$ decreases the depth of $u$ in a bridge-free toile.

{\it Case~3.1.2:} the vertex $w$ is a nodal {\tv}.
Let $w_1$ be a real neighbor vertex of $w$.
If $w_1$ is a monochrome vertex determining a solid cut, let $w_2$ be the monochrome vertex connected to $w_1$ through the cut and let $w_3$ be the black vertex neighbor to $w_2$ sharing a region with $u$. A monochrome modification between the inner solid edges adjacent to $u$ and $w_2$ connects these two vertices, reducing the depth of $u$.
If $w_1$ is a monochrome vertex connected to a real black vertex $w_2$, then in the region determined by $v$ and $w_2$ the vertex $w_2$ either has an inner bold edge and a monochrome modification allows us to consider instead the chain $u,v,w_2$ as in Case~3.1.1 or it has a real bold edge where the creation of a bridge with the inner bold edge adjacent to $v$ in the region followed by an elementary move of type $\circ$-out at $v$ decreases the depth of $u$ in a bridge-free toile.
If $w_1$ is black, up to a monochrome modification $v$ and $w_1$ are connected and we consider instead the chain $u,v,w_1$ as in Case~3.1.1.
Lastly, if $w_1$ is a monochrome vertex connected to an inner black vertex $w_2$, an elementary move of type $\bullet$-out at $w_2$ bring us to the previous consideration.

\vskip5pt
{\it Case~3.2:} the vertex $u$ is a simple {\tv} and the vertex $v$ is black.

{\it Case~3.2.1:} the vertex $w$ is white.
Since $\deep(u)=2$, the vertex $u$ is connected to an inner white vertex $w_1$.
If the vertices $u$ and $w$ share a region, the creation of a dotted bridge beside $w$ with the dotted edge adjacent to $u$ followed by an elementary move of type $\circ$-out at $w_2$ decreases the depth of $u$ in a bridge-free toile.
If the vertices $u$ and $w$ do not share a region, let $w_2, w_2' \neq u$ be the vertices connected to $v$ by a solid edge. We can assume up to monochrome modifications that $v$ and $w_1$ are connected by two different bold edges.
If one of the vertices $w_2$ or $w_2'$ is an inner simple {\tv}, up to a monochrome modifications it is connected to $w_1$, and then the creation of a bridge beside $w$ with the inner dotted edge adjacent to $w_1$ followed by an elementary move of type $\circ$-out at $w_1$ decreases the depth of $u$ in a bridge-free toile.
Otherwise, the vertices $w_2$ and $w_2'$ are nodal {\tvs}. We can assume up to one monochrome modification that $w_2$ and $w_1$ are connected.
If both $w_2$ and $w_2'$ are real, 
the creation of two bridges with an inner dotted edge adjacent to $w_1$, one beside $w_2$ and one beside $w_2'$, produces a dotted cut.
If both $w_2$ and $w_2'$ are inner, the creation of two bridges beside $w$, one in every side, with inner dotted edges adjacent to $w_2$ and $w_2'$ respectively, followed by the creation of an inner monochrome vertex between the inner dotted edges of $w_2$ and $w_2'$ sharing a region with $w_1$, create a {\gc}.
It between the vertices $w_2$ and $w_2'$ one is real and one is inner, then the creation of dotted bridges and an inner monochrome vertex as in the previous considerations allow us to create a {\gc}.

{\it Case~3.2.2:} the vertex $w$ is a nodal {\tv}. 
Since $\deep(u)=2$, the vertex $u$ it is connected to an inner white vertex $w_1$, which we can assume connected to $v$ by two different bold edges up to monochrome modification.
In this setting, the creation of a bridge beside $w$ with the inner dotted edge adjacent to $w_1$ followed by an elementary move of type $\circ$-out decreases the depth of $u$ in a bridge-free toile.

\vskip5pt
{\it Case~4.1:} the vertex $u$ is a nodal {\tv} and the vertex $v$ is white.

{\it Case~4.1.1:} the vertex $w$ is black.
If it was an inner solid edge $g$ sharing a region with~$u$, then the vertices $u$ and $w$ can be connected by a monochrome modification between the edges $e$ and $g$ decreasing the depth of $u$. 
Otherwise, the vertex $w$ has a real solid edge sharing a region with $u$. Since $\deep(u)=2$, the vertex $u$ is connected to an inner black vertex $w_1$. 
In this setting, the creation of a bridge beside $w$ with a solid inner edge adjacent to $w_1$ followed by an elementary move of type $\bullet$-out at $w_1$ decreases the depth of~$u$ in a bridge-free toile.

{\it Case~4.1.2:} the vertex $w$ is a nodal {\tv}.
Since $\deep(u)=2$, the vertex $u$ is connected to an inner black vertex $w_1$, which up to a monochrome modification it is connected to $v$. Then, the creation of a bridge beside $w$ with a solid inner edge adjacent to $w_1$ followed by an elementary move of type $\bullet$-out at $w_1$ decreases the depth of $u$ in a bridge-free toile.

\vskip5pt
{\it Case~4.2:} the vertex $u$ is a nodal {\tv} and the vertex $v$ is black.

{\it Case~4.2.1:} the vertex $w$ is a nodal {\tv}.
Since $\deep(u)=2$, the vertex $u$ is connected to an inner white vertex $w_1$ sharing a region with $w$, which we can assume connected to $v$ up to a monochrome modification.
In this setting, the creation of a bridge beside $w$ with an inner dotted edge adjacent to $w_1$ followed by an elementary move of type $\circ$-out at $w_1$ decreases the depth of $u$ in a bridge-free toile.

{\it Case~4.2.2:} the vertex $w$ is white.
Since $\deep(u)=2$, the vertex $u$ is connected to an inner white vertex $w_1$.
Then, if the vertices $u$ and $w$ belong to the same region, we can choose $w_1$ belonging to the same region as them. Then, the creation of a bridge beside $w$ with an inner dotted edge adjacent to $w_1$ followed by an elementary move of type $\circ$-out at $w_1$ decreases the depth of $u$ in a bridge-free toile.
Otherwise, let $w_2\neq u$ be a vertex connected to $v$ by a solid edge sharing a region with $w_1$.
If $w_2$ is a real nodal {\tv}, we consider instead the chain $u,v,w_2$ as in Case~4.2.1.
If $w_2$ is an inner simple {\tv}, up to monochrome modification it is connected to $w_1$, and then, the creation of a bridge beside $w$ with an inner dotted edge adjacent to $w_1$ followed by an elementary move of type $\circ$-out at $w_1$ decreases the depth of $u$ in a bridge-free toile.
Finally, if $w_2$ is an inner nodal {\tv}, we consider instead the white vertex $w_1'\neq w_1$ connected to $u$ and the {\tv} $w_2'\neq u,w_2$ connected to $v$. If the aforementioned consideration cycle to this configuration, then the creation of two bridges beside $w$, one at every side, with inner dotted edges adjacent to $w_2$ and $w_2'$, respectively, produces a {\gc}.
\end{proof}




 

%

\begin{coro} \label{cr:deepbf}
Let $D$ be as in Proposition \ref{prop:deep}. If there exists a toile $D'$ weakly equivalent to $D$ with depth $1$, then $D'$ can be chosen bridge-free.
\end{coro}


%

\begin{prop} \label{prop:decomptoi}
Let $D$ be a toile of degree at least $6$ and depth at most $1$.
Then, there exists a toile $D'$ weakly equivalent to $D$ such that $D'$ has a {\gc}.
Moreover, if $D$ has isolated real nodal {\tvs}, the {\gc} is dotted or a solid axe.
\end{prop}

\begin{proof}
Let $D_0$ be a dessin within the weak equivalence class of $D$, maximal with respect to the number of zigzags.
Due to Proposition~\ref{prop:deep} and Corollary \ref{cr:deepbf}, we can choose $D_0$ within the weak equivalence class of $D_0$ such that $\deep(D_0)\leq 1$ and $D_0$ is bridge-free. For simplicity we assume that there are no black or white inner vertices connected to monochrome vertices. We assume that there are no dotted cuts, since otherwise the Proposition follows trivially.
If there is a bold cut $[w,w']$ in a bridge-free toile, assuming the vertex $w$ has black neighbor real vertices, then an elementary move of type $\bullet$-in followed by an elementary move of type $\bullet$-out at the vertices $w$ and $w'$, respectively, eliminates the bold cut. This way we can assume there are no bold cuts on the dessin without breaking the bridge-free property nor changing its depth.
Let us start by the case when $D_0$ has singular vertices on the boundary of the disk $\mathbb{D}^2$. Let $v$ be a real nodal {\tv}.

\vskip5pt
{\it Case 1:} the vertex $v$ is isolated, being connected to a white vertex $u$.
If the vertex $u$ is real, the edge $e:=[v,u]$ is dividing. Let $R$ be the region containing $e$ on the connected component of $D_0\setminus e$ with a maximal number of white vertices. Let $w$ be the real neighbor vertex of $u$ in the region $R$. Let $S$ be the bold segment containing $u$.

{\it Case 1.1:} the vertex $w$ is monochrome.
If $w$ is connected to a real black vertex $w_1$, then there are two cases: in the region $R$, the vertex $w_1$ is adjacent to an inner or real solid edge $f$.

If the edge $f$ is inner, let $w_2$ be the real vertex connected to $w_1$ by a solid real edge (see Figure~\ref{fig:toilesii01}).
If $w_2$ is a {\tv}, let $w_3\neq u$ be a white real vertex connected to $w$. 
If $w_2$ is a simple {\tv}, then, the creation of a bridge with the inner dotted edge adjacent to $w_3$ beside $w_2$, followed by elementary moves of type $\circ$-in at $w$ and $\circ$-out at the bridge, produces a dotted axe in a bridge-free dessin (see Figure~\ref{fig:toilesii02}).
If $w_2$ is a nodal {\tv}, up to a monochrome modification it is connected to $w_3$. An elementary move of type $\circ$-out at $w$ produces a {\gc} (see Figure~\ref{fig:toilesii03}).
If $w_2$ is a monochrome vertex, then an elementary move of type $\bullet$-in at $w_2$ followed by an elementary move of type $\bullet$-out at $w$ bring us to a configuration we study within the Case 1.2.

\begin{figure}[h]
\begin{center}
\begin{subfigure}{0.3\textwidth}
\centering\includegraphics[width=2.5in]{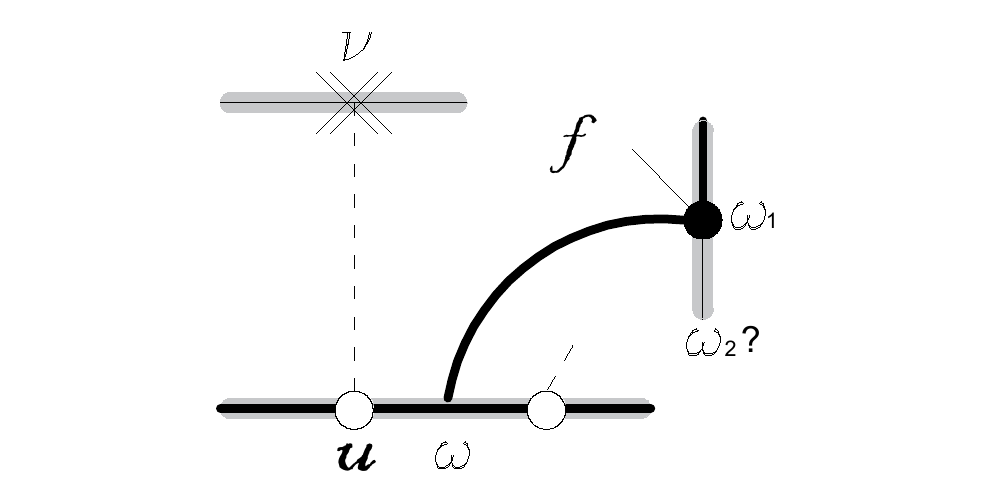}
\caption{\label{fig:toilesii01}}
\end{subfigure}
\begin{subfigure}{0.3\textwidth}
\centering\includegraphics[width=2.5in]{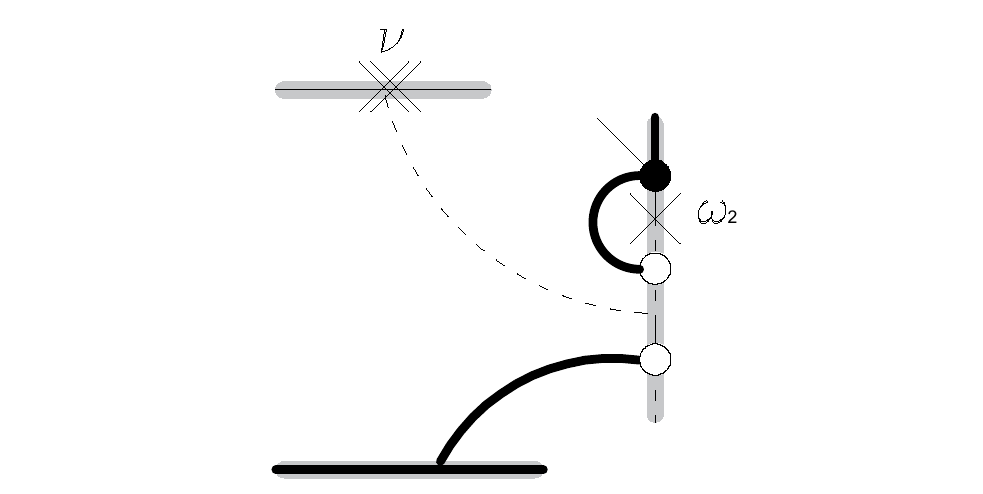}
\caption{\label{fig:toilesii02}}
\end{subfigure}
\begin{subfigure}{0.3\textwidth}
\centering\includegraphics[width=2.5in]{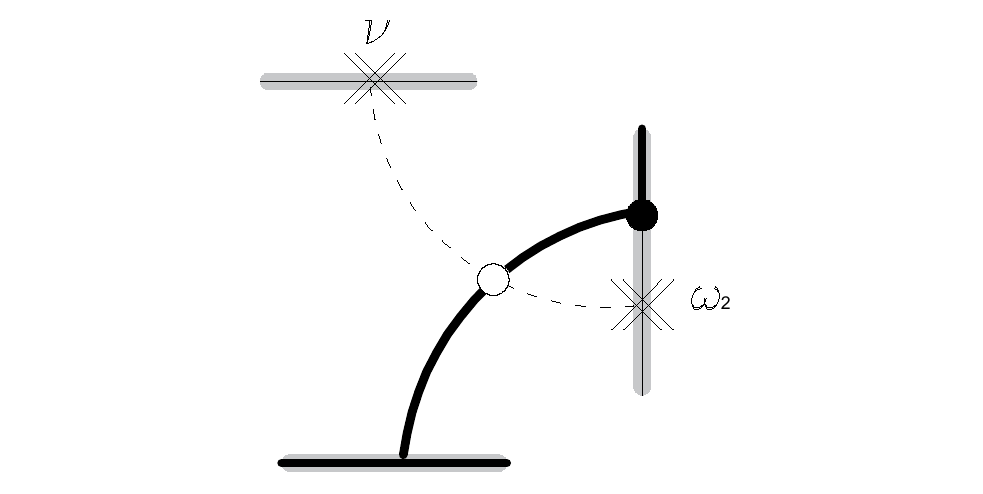}
\caption{\label{fig:toilesii03}}
\end{subfigure}
\end{center}
\caption{}
\end{figure}

If the edge $f$ is real, let $w_2$ be the real vertex connected to $w_1$ by the edge $f$ (see Figure~\ref{fig:toilesii04}).
If $w_2$ is a simple {\tv}, it determines a real dotted segment where the creation of a bridge with the edge $e$ produces a dotted axe (see Figure~\ref{fig:toilesii05}).
If $w_2$ is a monochrome vertex, then an elementary move of type $\bullet$-in at $w_2$ followed by an elementary move of type $\bullet$-out at $w$ bring us to a configuration we study within the Case 1.2.
In the case when $w_2$ is a nodal {\tv} different from $v$, and in this case the creation of an inner monochrome vertex with the edge $e$ and the inner dotted edge adjacent to $w_2$ produces a {\gc} (see Figure~\ref{fig:toilesii06}).

\begin{figure}[h]
\begin{center}
\begin{subfigure}{0.3\textwidth}
\centering\includegraphics[width=2.5in]{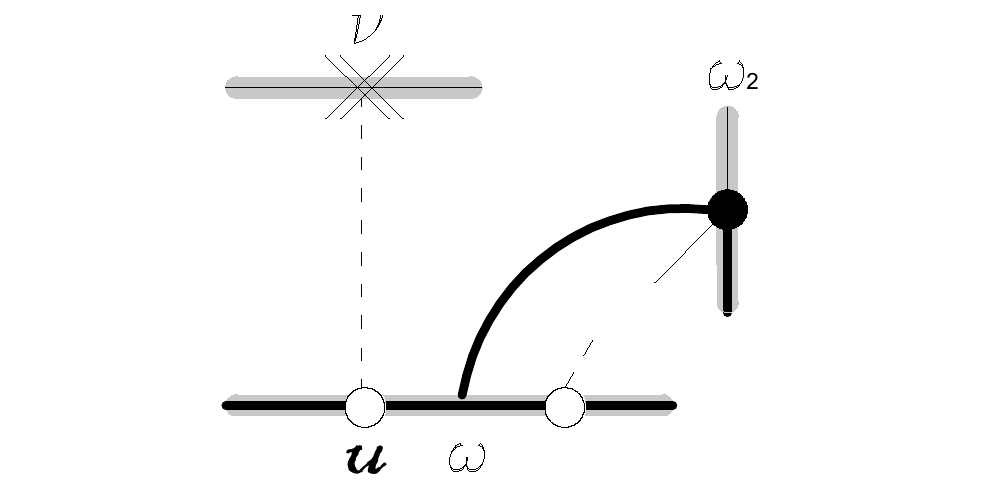}
\caption{\label{fig:toilesii04}}
\end{subfigure}
\begin{subfigure}{0.3\textwidth}
\centering\includegraphics[width=2.5in]{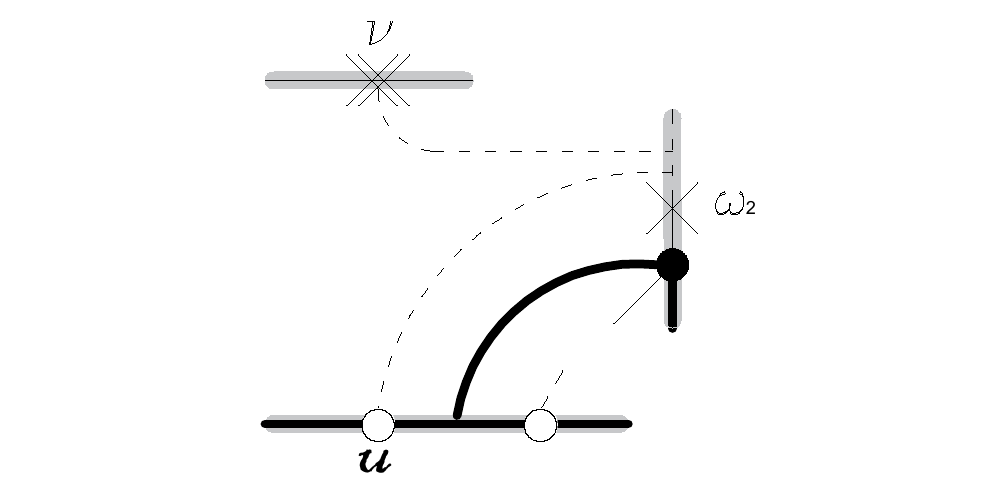}
\caption{\label{fig:toilesii05}}
\end{subfigure}
\begin{subfigure}{0.3\textwidth}
\centering\includegraphics[width=2.5in]{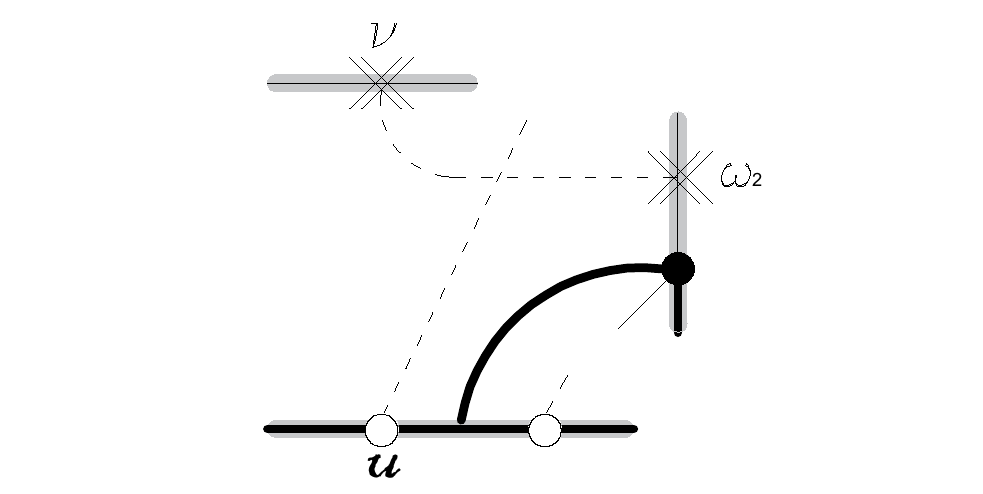}
\caption{\label{fig:toilesii06}}
\end{subfigure}
\end{center}
\caption{}
\end{figure}

A special case is when $w_2=v$. Let $w_3\neq u$ be a white real vertex connected to $w$ and let $w_4$ be the vertex connected to $w_3$ by an inner dotted edge.
If $w_4$ is a monochrome vertex or a real nodal {\tv}, an elementary move of type $\circ$-in produces a {\gc} (see Figure~\ref{fig:toilesii07} and Figure~\ref{fig:toilesii08}).
If $w_4$ is an inner simple {\tv}, it is connected to $w_1$ up to a monochrome modification. An elementary move of type $\circ$-in at $w$ followed by the creation of a bold bridge beside $w_1$ and an elementary move of type $\circ$-out bring us to a configuration where we can create a zigzag, contradicting the maximality assumption (see Figure~\ref{fig:toilesii09}). 

\begin{figure}[h]
\begin{center}
\begin{subfigure}{0.3\textwidth}
\centering\includegraphics[width=2.5in]{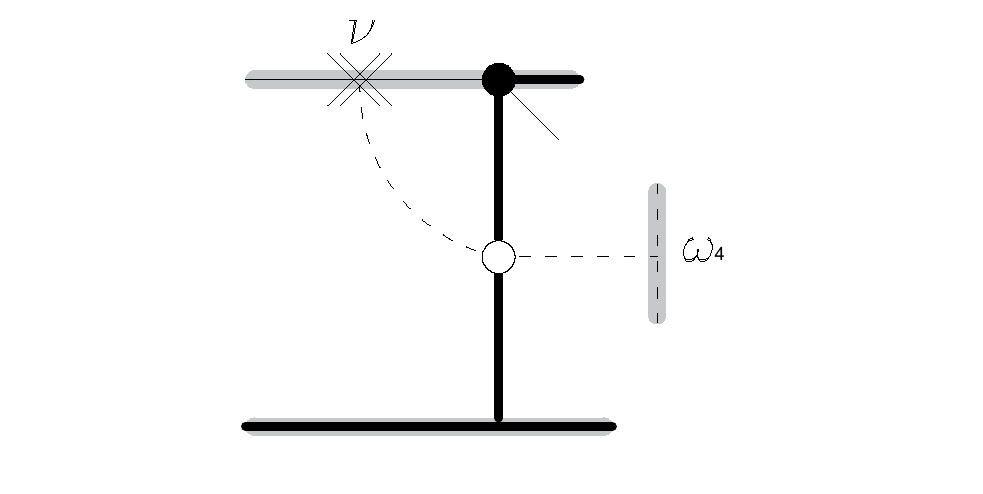}
\caption{\label{fig:toilesii07}}
\end{subfigure}
\begin{subfigure}{0.3\textwidth}
\centering\includegraphics[width=2.5in]{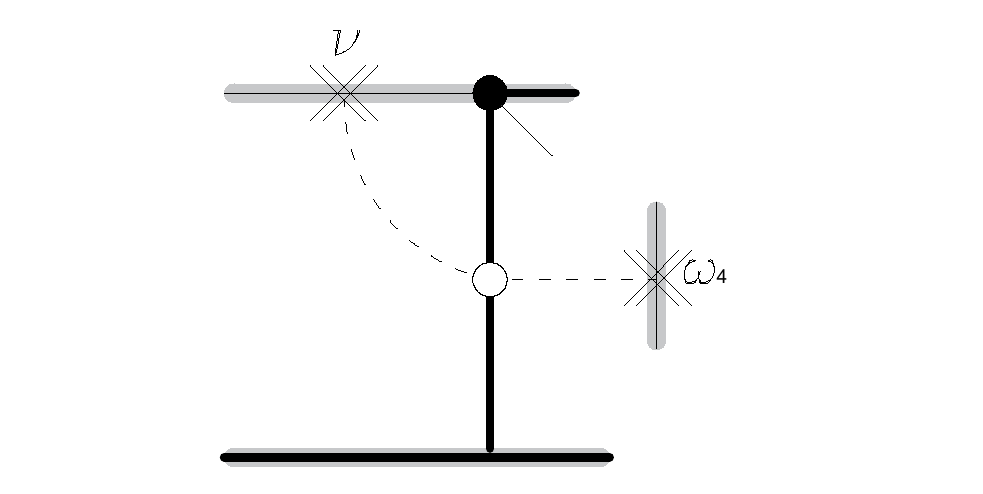}
\caption{\label{fig:toilesii08}}
\end{subfigure}
\begin{subfigure}{0.3\textwidth}
\centering\includegraphics[width=2.5in]{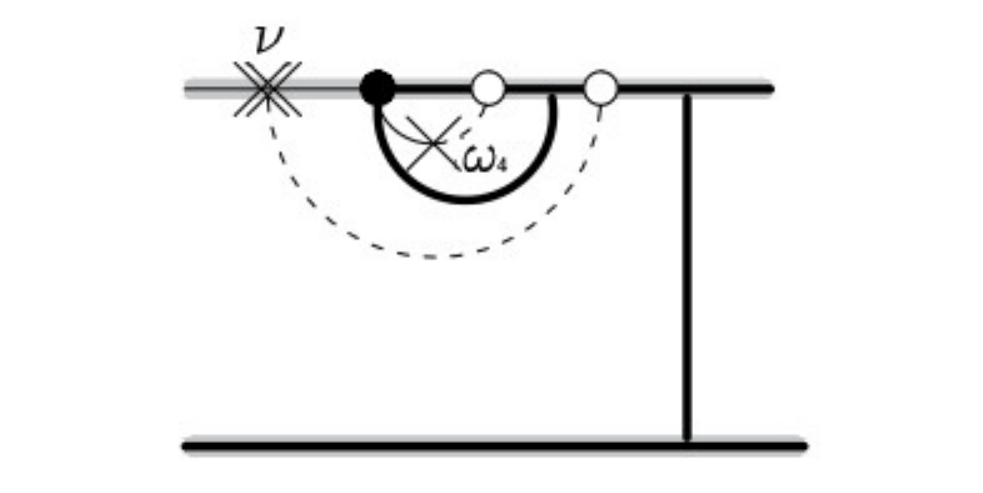}
\caption{\label{fig:toilesii09}}
\end{subfigure}
\end{center}
\caption{}
\end{figure}

If $w_4$ is an inner nodal {\tv}, up to a monochrome modification it is connected to $w_1$.
Let $w_5\neq w_3$ be a vertex connected to $w_4$ by a dotted edge.
If $w_5$ is a monochrome vertex, then an elementary move of type $\circ$-in at $w$ creates an inner white vertex $w'$ such that the chain $v,w',w_4,w_5$ is a {\gc} (see Figure~\ref{fig:toilesii11}).

If $w_5$ is a white inner vertex, the creation of a bold bridge beside $w_1$ with a bold edge of $w_5$ followed by an elementary move of type $\circ$-out allows us to consider $w_5$ as a real white vertex.

If $w_5$ is a white real vertex, let $w_6\neq w_1$ be a vertex connected to $w_4$ by a solid edge.
When $w_6$ is a real black vertex, there are two cases: in the region $R'$ determined by $w_4$, $w_5$ and $w_6$ the bold edge $g$ adjacent to $w_6$ is either real or inner.
We perform an elementary move of type $\circ$-in at $w$ producing a white inner vertex $w'$, and destroy the potential residual bold bridge. 

If $g$ is a real bold edge, let $w_7$ be the vertex connected to $w_6$ by a real solid edge. Up to a monochrome modification the vertex $w'$ is connected to $w_6$. 
If $w_7$ is a simple {\tv}, the creation of a bridge beside $w_7$ with an inner dotted edge incident to $w'$ produces a {\gc} (see Figure~\ref{fig:toilesii13}).
If $w_7$ is a nodal {\tv}, the creation of an inner monochrome vertex with the edge $e$ and the inner dotted edge adjacent to $w_7$ produces a {\gc} (see Figure~\ref{fig:toilesii14}).
If $w_7$ is a monochrome vertex, let $w_8$ be the vertex connected to $w_7$ by an inner solid edge.

\begin{figure}[h]
\begin{center}
\begin{subfigure}{0.3\textwidth}
\centering\includegraphics[width=2.5in]{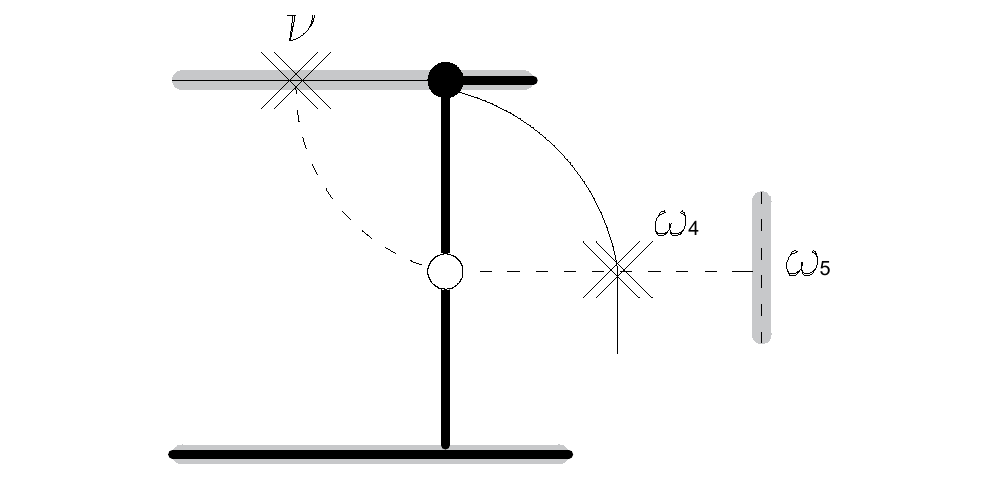}
\caption{\label{fig:toilesii11}}
\end{subfigure}
\begin{subfigure}{0.3\textwidth}
\centering\includegraphics[width=2.5in]{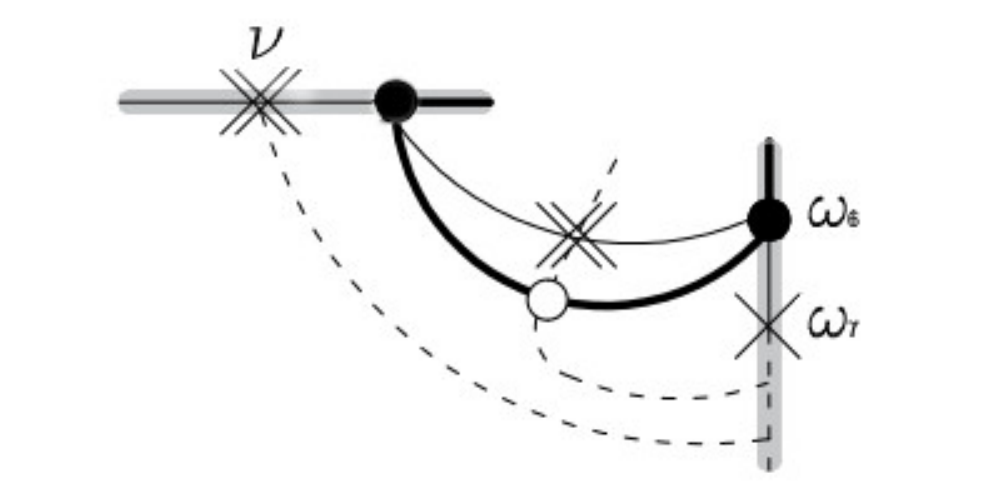}
\caption{\label{fig:toilesii13}}
\end{subfigure}
\begin{subfigure}{0.3\textwidth}
\centering\includegraphics[width=2.5in]{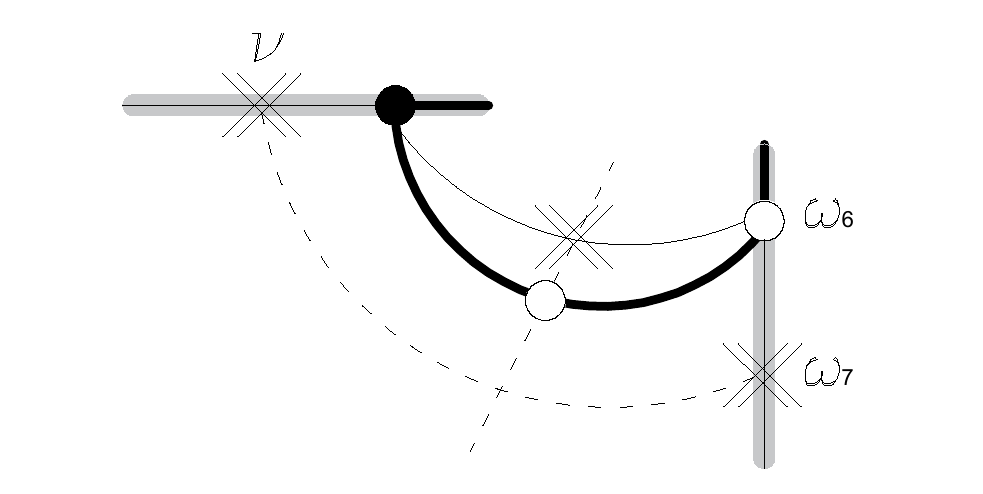}
\caption{\label{fig:toilesii14}}
\end{subfigure}
\end{center}
\caption{}
\end{figure}

If $w_8$ is a monochrome vertex, let $w_9$ be the vertex connected to $w_8$ in the region determined by $w_6$, $w_7$ and $w_8$.
If $w_9$ is a simple {\tv}, 
the creation of a bridge beside $w_9$ with 
the edge $e$ produces an axe (see Figure~\ref{fig:toilesii16}).
Otherwise $w_9$ is a nodal {\tv}, the creation of an inner monochrome vertex with the edge $e$ and the inner dotted edge adjacent to $w_9$ produces a {\gc} (see Figure~\ref{fig:toilesii17}). 
If $w_8$ is a real nodal {\tv}, the creation of a bridge beside it with the edge $e$ produces an axe (see Figure~\ref{fig:toilesii18}).
 
 \begin{figure}[h]
\begin{center}
\begin{subfigure}{0.3\textwidth}
\centering\includegraphics[width=2.5in]{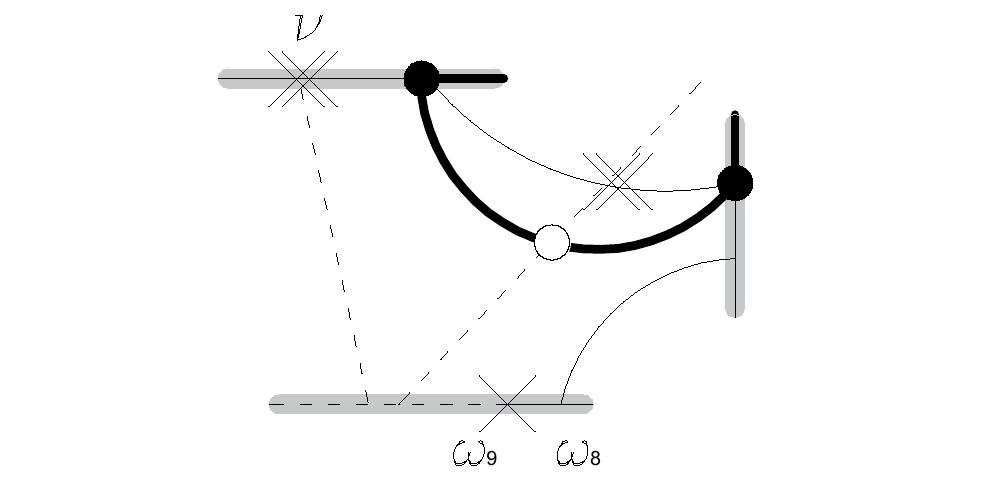}
\caption{\label{fig:toilesii16}}
\end{subfigure}
\begin{subfigure}{0.3\textwidth}
\centering\includegraphics[width=2.5in]{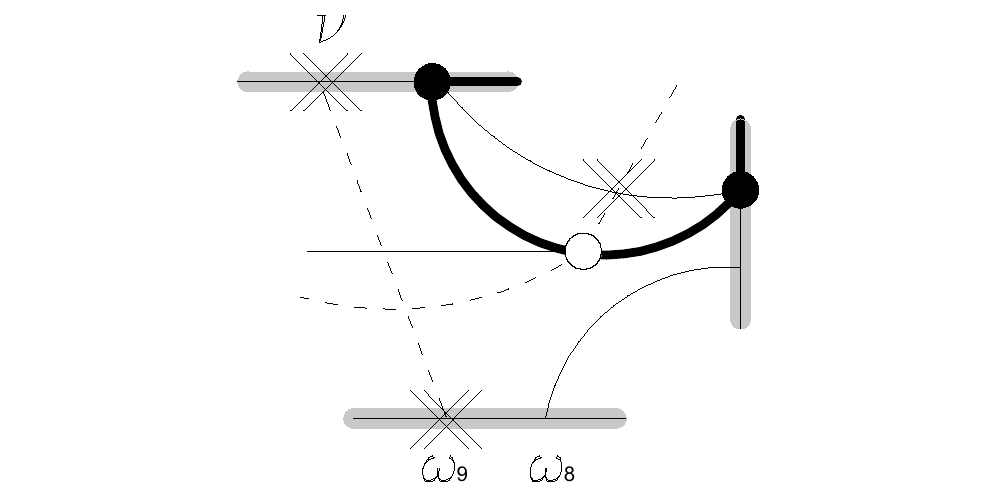}
\caption{\label{fig:toilesii17}}
\end{subfigure}
\begin{subfigure}{0.3\textwidth}
\centering\includegraphics[width=2.5in]{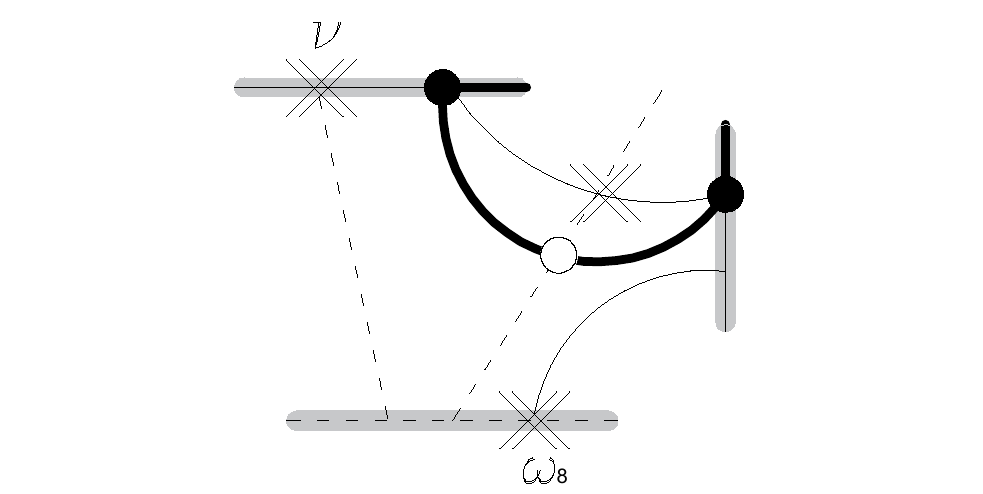}
\caption{\label{fig:toilesii18}}
\end{subfigure}
\end{center}
\caption{}
\end{figure}

If $w_8$ is an inner nodal {\tv}, 
we make monochrome modification in order to connect $w_8$ to $w'$.
Then, the creation of a bride $h$ beside $w_1$ with an inner solid edge adjacent to $w_8$ produces a solid {\gc} and cutting by it produces two different toiles (see Figure~\ref{fig:toilesii20}). Let $D_0'$ be the resulting toile containing $v$.
The toile $D_0'$ is a toile of degree strictly greater than $3$ since there are no nodal cubic toiles having two isolated nodes (cf. Section~\ref{sec:cubics}). Since $\deep(D_0')\leq1$, we can restart the algorithm with the toile $D_0'$ and it does not cycle back to this consideration. A dotted {\gc} in $D'$ having $w_8$ as an end can be extended to a {\gc} in $D_0$ by deleting the solid bridge $h$ and creating an inner monochrome dotted vertex with the edge $e$ and the inner dotted edge of $w_8$ (see Figure~\ref{fig:toilesii21}).

\begin{figure}[h]
\begin{center}
\begin{subfigure}{0.3\textwidth}
\centering\includegraphics[width=2.5in]{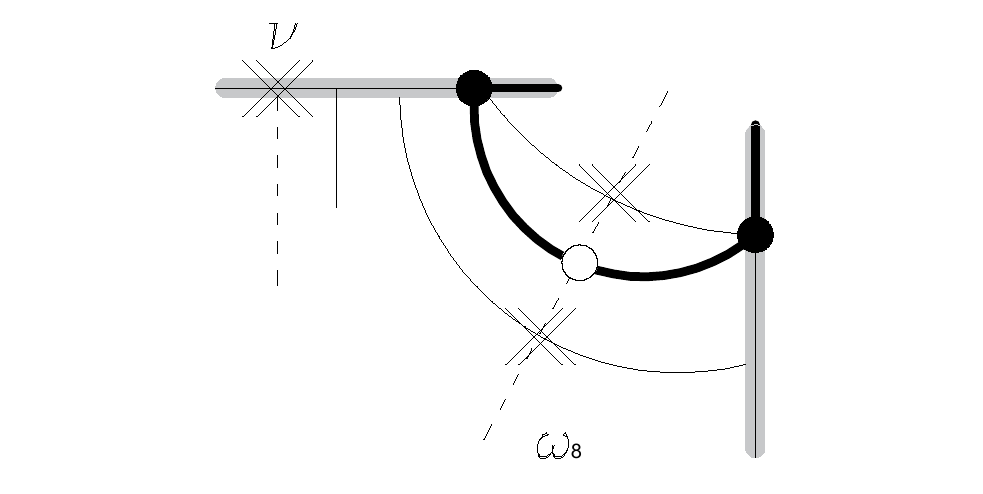}
\caption{\label{fig:toilesii20}}
\end{subfigure}
\begin{subfigure}{0.3\textwidth}
\centering\includegraphics[width=2.5in]{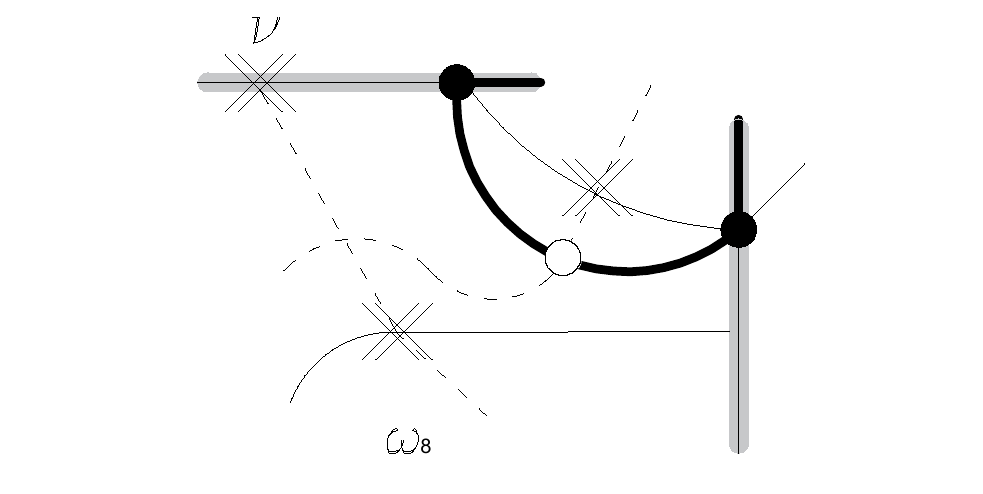}
\caption{\label{fig:toilesii21}}
\end{subfigure}
\end{center}
\caption{}
\end{figure}

If $w_8$ is a simple {\tv}, let $w_9$ be the vertex connected to it by a dotted edge.
If $w_9$ is a monochrome vertex, a monochrome modification between the inner dotted edge adjacent to $w_9$ and the edge $e$ produces an axe (see Figure~\ref{fig:toilesii22}).
If $w_9$ is a real white vertex, then an elementary move of type $\bullet$-in at $w_7$ followed by an elementary move of type $\bullet$-out beside $w_9$ allows us to create a zigzag with $w_8$, contradicting the maximality assumption (see Figure~\ref{fig:toilesii23}).
Lastly, if $w_9$ is an inner white vertex, the creation of a bridge in the segment $S$ followed by an elementary move of type $\circ$-out allows us to consider $w_9$ as a real white vertex.

\begin{figure}[h]
\begin{center}
\begin{subfigure}{0.3\textwidth}
\centering\includegraphics[width=2.5in]{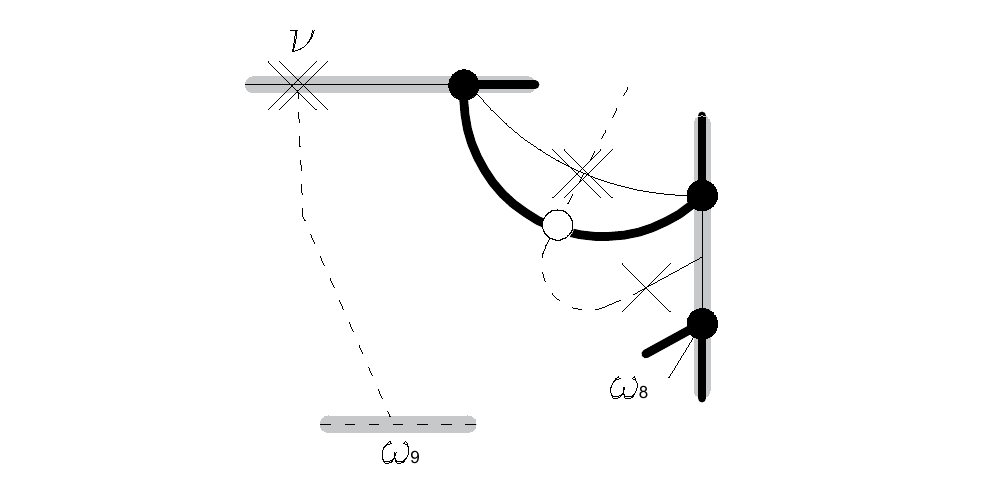}
\caption{\label{fig:toilesii22}}
\end{subfigure}
\begin{subfigure}{0.3\textwidth}
\centering\includegraphics[width=2.5in]{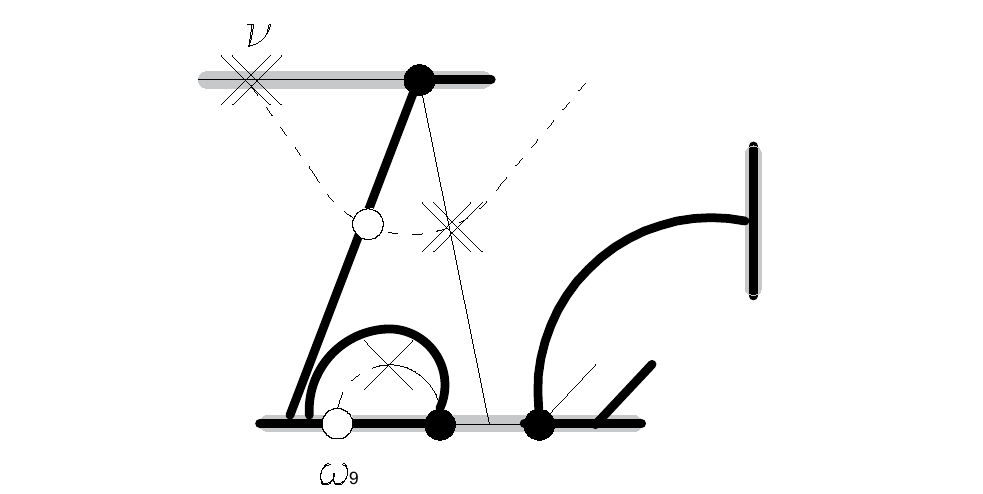}
\caption{\label{fig:toilesii23}}
\end{subfigure}
\end{center}
\caption{}
\end{figure}

If $g$ is an inner bold edge, let $w_7$ be the vertex connected to $w_6$ by an inner bold edge.
If $w_7$ is a real white vertex, the creation of a bridge beside it with an inner dotted edge adjacent to $w_4$ produces a {\gc} (see Figure~\ref{fig:toilesii25}).
If $w_7$ is a monochrome vertex, an elementary move of type $\circ$-in allows us to consider it as an inner white vertex.
If $w_7$ is an inner white vertex, up to a monochrome modification it is connected to $w_4$. Let $w_8$ be the vertex connected to $w_6$ by a real solid edge.
If $w_8$ is a simple {\tv}, the creation of a dotted bridge beside it with an inner dotted edge adjacent to $w_4$ produces a {\gc} (see Figure~\ref{fig:toilesii26}).
If $w_8$ is a nodal {\tv}, up to a monochrome modification it is connected to $w_7$ determining a {\gc} (see Figure~\ref{fig:toilesii27}).
If $w_8$ is a monochrome vertex, the creation of a bold bridge beside $w_1$ with an inner bold edge adjacent to $w_7$ followed by an elementary move of type $\circ$-out at $w_7$, an elementary move of type $\bullet$-in at $w_8$ and an elementary move of type $\bullet$-out at the bridge bring us to the configuration when the edge $g$ was a real bold edge.

\begin{figure}[h]
\begin{center}
\begin{tabular}{lcr}
\begin{subfigure}{0.3\textwidth}
\centering\includegraphics[width=2.5in]{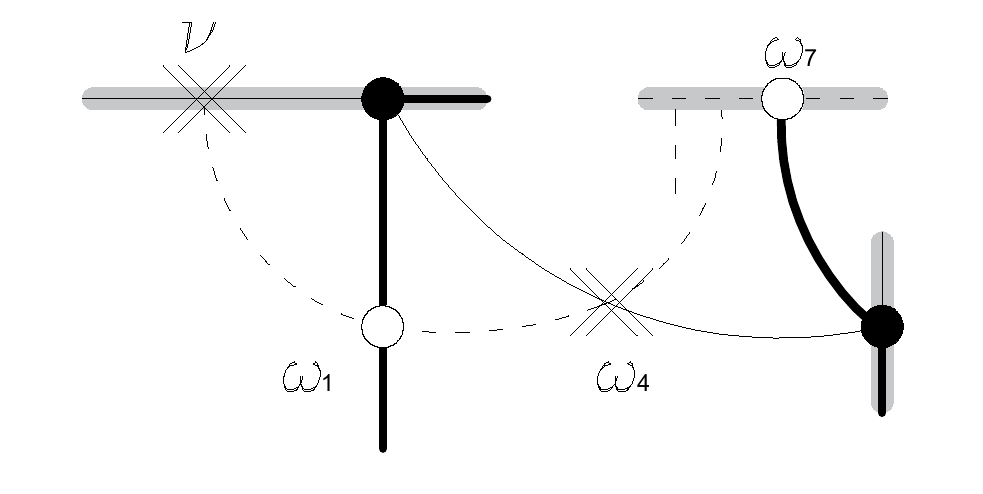}
\caption{\label{fig:toilesii25}}
\end{subfigure}
&\hspace{1cm} &
\begin{subfigure}{0.3\textwidth}
\centering\includegraphics[width=2.5in]{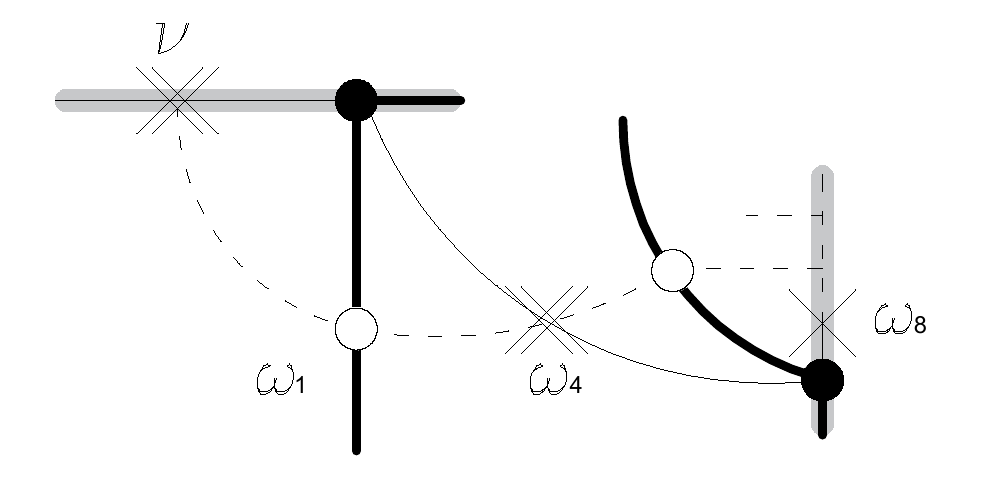}
\caption{\label{fig:toilesii26}}
\end{subfigure}
\end{tabular}

\begin{subfigure}{0.3\textwidth}
\centering\includegraphics[width=2.5in]{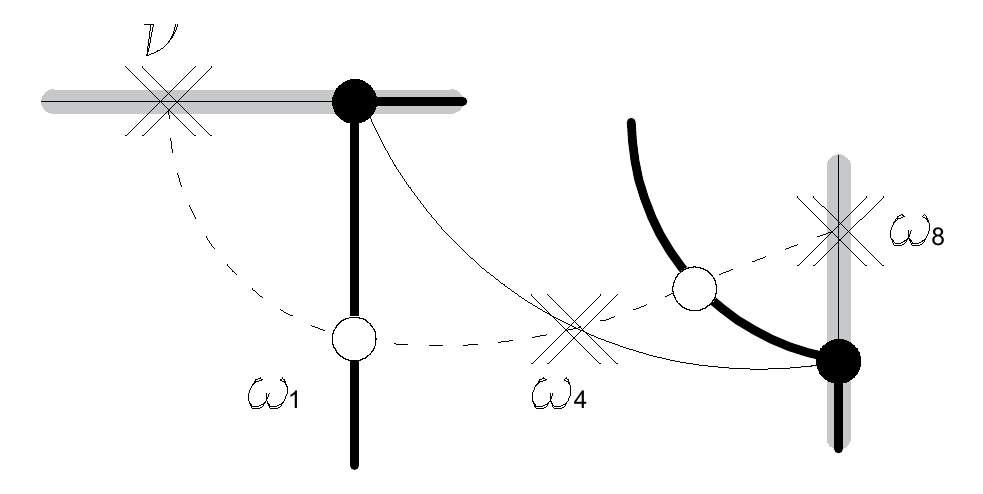}
\caption{\label{fig:toilesii27}}
\end{subfigure}
\end{center}
\caption{}
\end{figure}

In the case when $w_6$ is a monochrome vertex, an elementary move of type $\bullet$-in at $w_6$ produces an inner black vertex $w_6'$. We destroy any possible remaining bridge. Then, the creation of a bridge beside $w_5$ with an inner bold edge adjacent to $w_6'$ followed by an elementary move of type $\bullet$-out at $w_6'$ bring us to the previous consideration when $w_6$ was a black vertex. The same applies to the case when $w_6$ is an inner black vertex.

\begin{figure}[h]
\begin{center}
\begin{subfigure}{0.3\textwidth}
\centering\includegraphics[width=2.5in]{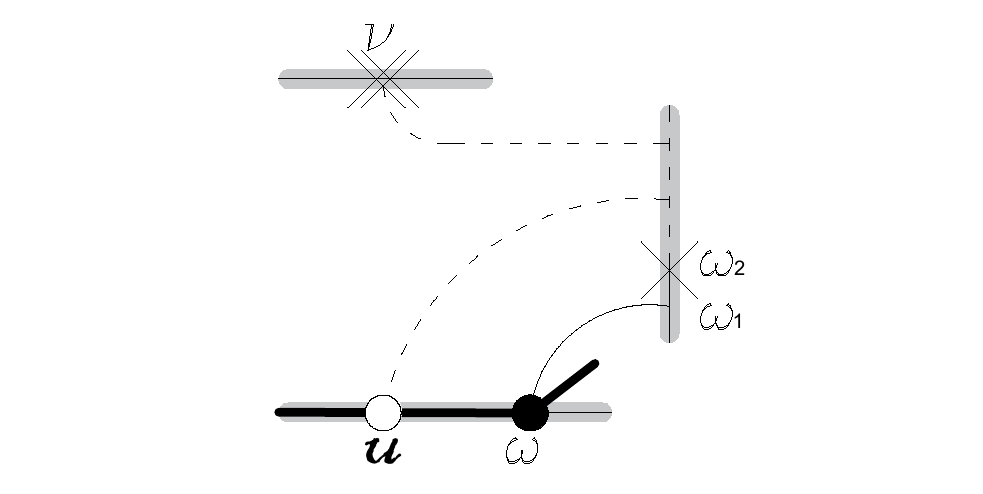}
\caption{\label{fig:toilesii31}}
\end{subfigure}\hspace{1cm}
\begin{subfigure}{0.3\textwidth}
\centering\includegraphics[width=2.5in]{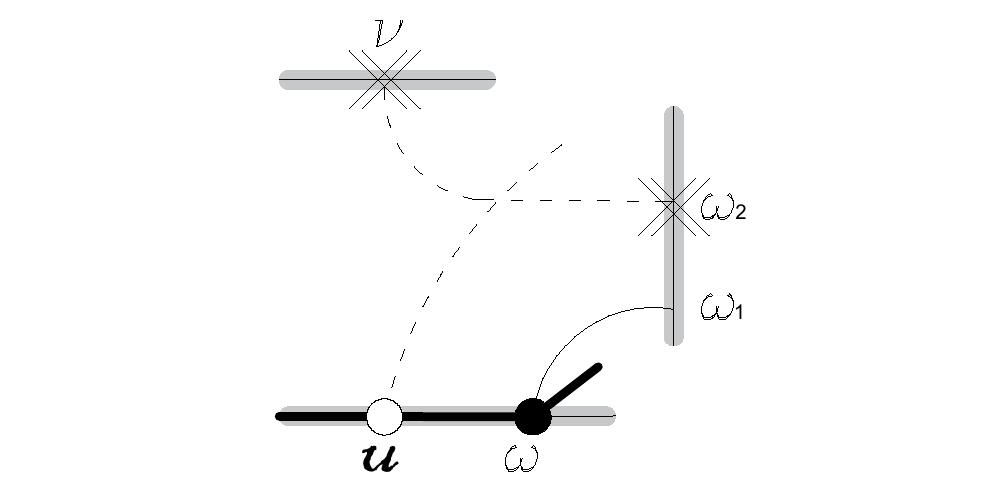}
\caption{\label{fig:toilesii32}}
\end{subfigure}
\end{center}
\caption{}
\end{figure}

{\it Case 1.2:} the vertex $w$ is black.
Let $w_1$ be the vertex connected to $w$ by an inner solid edge.
If $w_1$ is a monochrome vertex in a different solid segment that the one containing $v$, let $w_2$ be the vertex connected to $w_1$ on the region $R$.
If $w_2$ is a simple {\tv}, the creation of a bridge beside it with the edge $e$ produces a {\gc} (see Figure~\ref{fig:toilesii31}). Otherwise, the vertex $w_2$ is a nodal {\tv} and the creation of an inner dotted monochrome vertex with the inner edge of $w_2$ and the edge $e$ produces a {\gc} (see Figure~\ref{fig:toilesii32}).
If $w_1$ is a monochrome vertex connected to $v$, let $w_2$ be the vertex connected to $w$ by a real solid edge.
If $w_2$ is a monochrome vertex, we do an elementary move of type $\bullet$-in at $w_2$ followed by an elementary move $\bullet$-out at $w_1$.
In the case when there is no resulting bold bridge, this configuration has been consider in the Case 1.1.
Otherwise, we destroy the bold bridge and let $w_2'$ be the real black vertex connected to $v$ and let $w_3$ be the vertex connected to it by an inner bold edge.
If $w_3$ is an inner white vertex, the creation of a bridge beside $u$ followed by an elementary move of type $\circ$-out sets the configuration as when there was no bold bridge.
If $w_3$ is a monochrome vertex, an elementary move of type $\circ$-in allows us to consider it as an inner white vertex.
If $w_3$ is a real white vertex, the creation of a bridge beside it with the edge $e$ produces an axe (see Figure~\ref{fig:toilesii36}).

In the case when the vertex $w_2$ is a simple {\tv}, let $w_3$ be the vertex connected to $w_2$ by a dotted real edge.
If $w_3$ is a white vertex, it is connected to $w$ up to a monochrome modification. Let $w_4\neq v$ be the real neighbor vertex of $w_1$.

When the vertex $w_4$ is a simple {\tv}, if it is connected to $w_3$, then the toile would be a cubic. Hence the vertices $w_3$ and $w_4$ are not neighbors.
Let $w_5\neq w_2$ be the real neighbor vertex of $w_3$. 

\begin{figure}[h]
\begin{center}
\begin{subfigure}{0.3\textwidth}
\centering\includegraphics[width=2.5in]{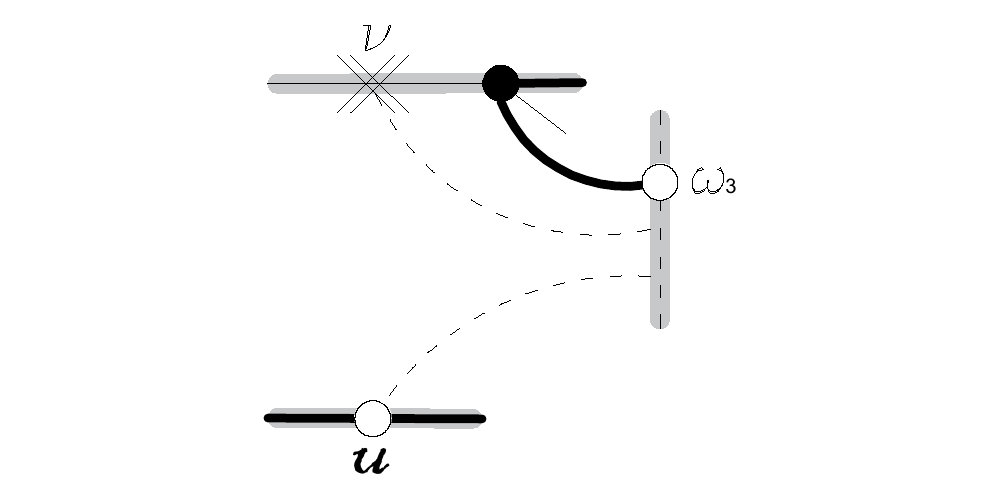}
\caption{\label{fig:toilesii36}}
\end{subfigure}\hspace{1cm}
\begin{subfigure}{0.3\textwidth}
\centering\includegraphics[width=2.5in]{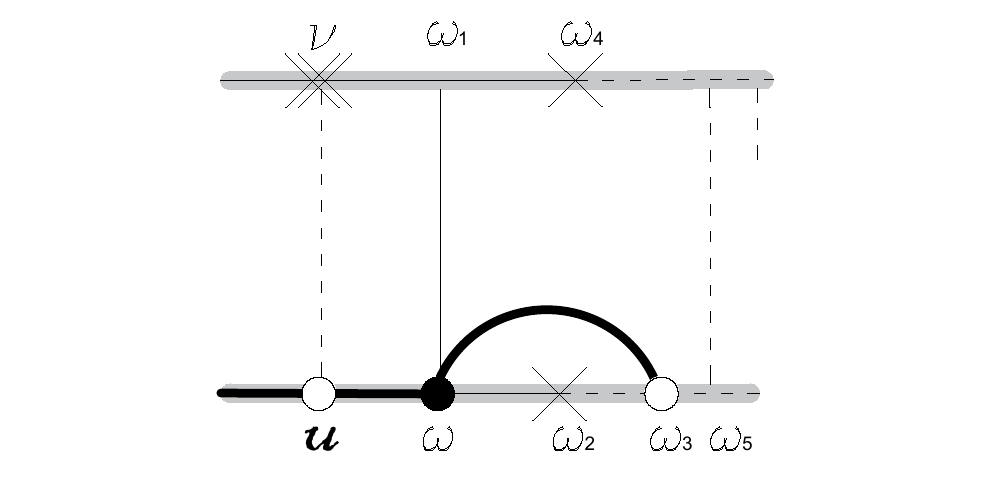}
\caption{\label{fig:toilesii37}}
\end{subfigure}
\end{center}
\caption{}
\end{figure}

When $w_5$ is a monochrome vertex, the creation of a bridge with its inner dotted edge beside $w_4$ produces a cut (see Figure~\ref{fig:toilesii37}).
If $w_5$ is a simple {\tv} having a black neighboring vertex, then we create a bridge beside $w_5$ with the inner solid edge adjacent to $w$ creating a solid cut, we destroy the correspondent zigzag, make an elementary move of type $\bullet$-it followed by an elementary move of type $\bullet$-out, make an elementary move of type $\circ$-in with $u$ and $w_5$ in order to create a bold bridge and perform an elementary move of type $\circ$-out. Then, the creation of a zigzag brings us to a configuration considered in Case~1.1 without breaking the maximality assumption (see Figure~\ref{fig:toilesii38}). 

\begin{figure}[h]
\begin{center}
\begin{subfigure}{0.3\textwidth}
\centering\includegraphics[width=2.5in]{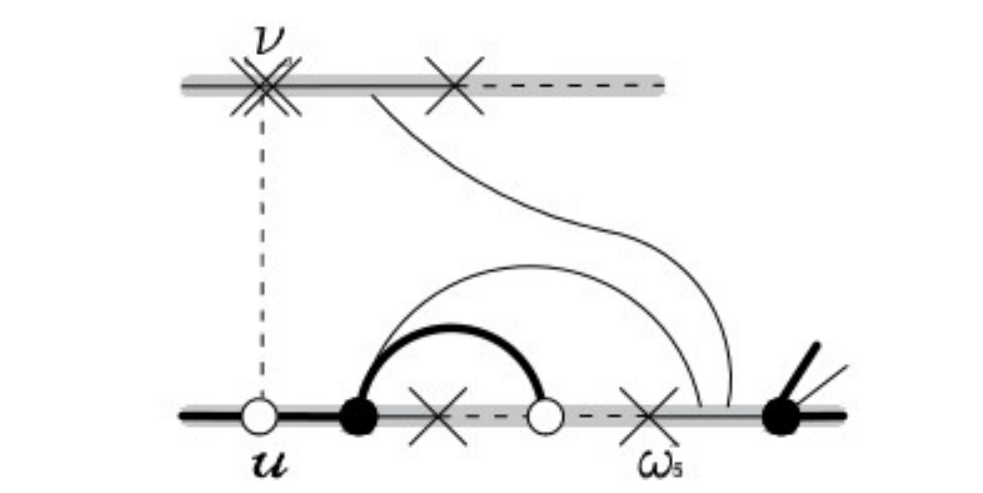}
\end{subfigure}\hspace{1cm}$\cong$
\begin{subfigure}{0.5\textwidth}
\centering\includegraphics[width=2.5in]{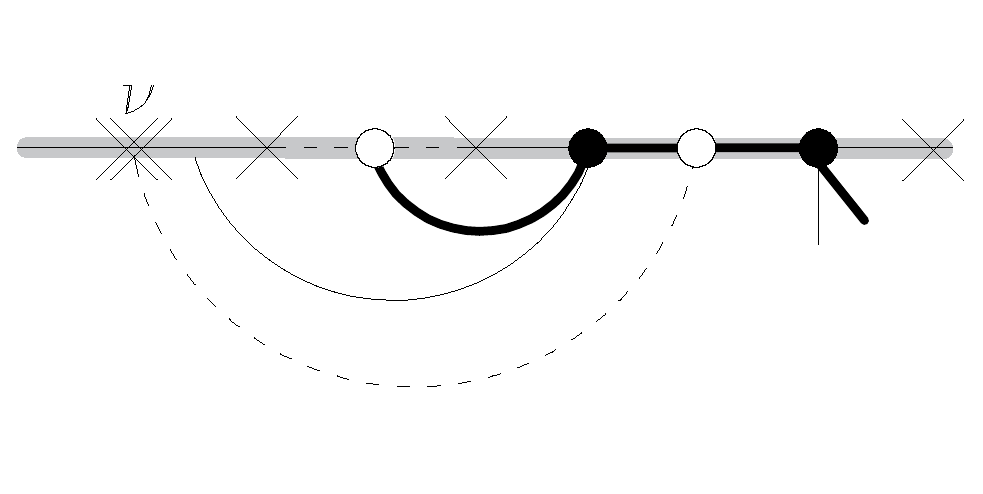}
\end{subfigure}
\end{center}
\caption{\label{fig:toilesii38}}
\end{figure}

If $w_5$ is a simple {\tv} having a solid neighboring monochrome vertex $w_6$, we perform a monochrome modification to connect $w_6$ with $w$. 
Let $w_7\neq w_5$ be the real neighboring vertex of $w_6$.
If $w_7$ is a simple {\tv}, the creation of a dotted bridge beside with the edge $e$ produces an axe (see Figure~\ref{fig:toilesii43}).
If $w_7$ is a nodal {\tv}, the creation of an inner monochrome vertex between the edge $e$ and the inner dotted edge adjacent to $w_7$ produces a {\gc} (see Figure~\ref{fig:toilesii44}).

If $w_5$ is a nodal {\tv}, a monochrome modification connect it to $w$ and then, the creation of a bridge beside $w_5$ with the edge $e$ produces an axe (see Figure~\ref{fig:toilesii45}).

\begin{figure}[h]
\begin{center}
\begin{subfigure}{0.3\textwidth}
\centering\includegraphics[width=2.5in]{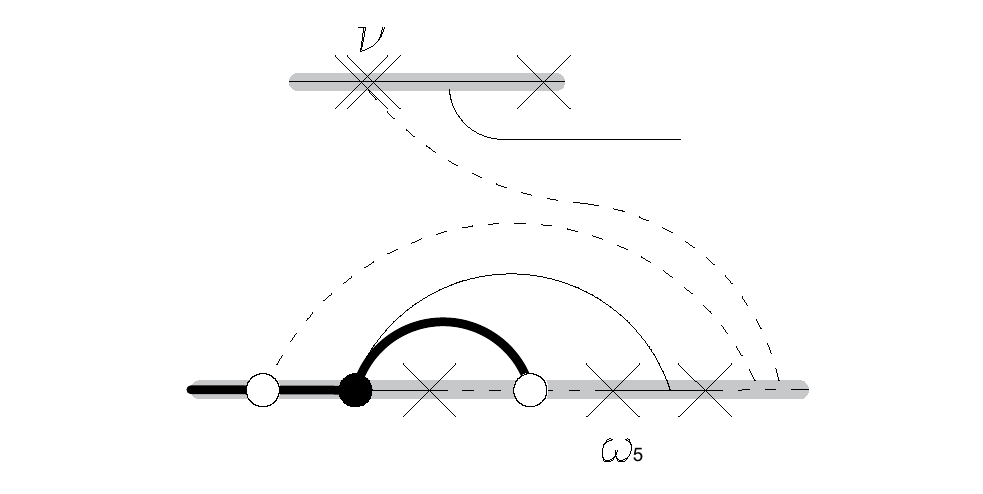}
\caption{\label{fig:toilesii43}}
\end{subfigure}
\hspace{4mm} 
\begin{subfigure}{0.3\textwidth}
\centering\includegraphics[width=2.5in]{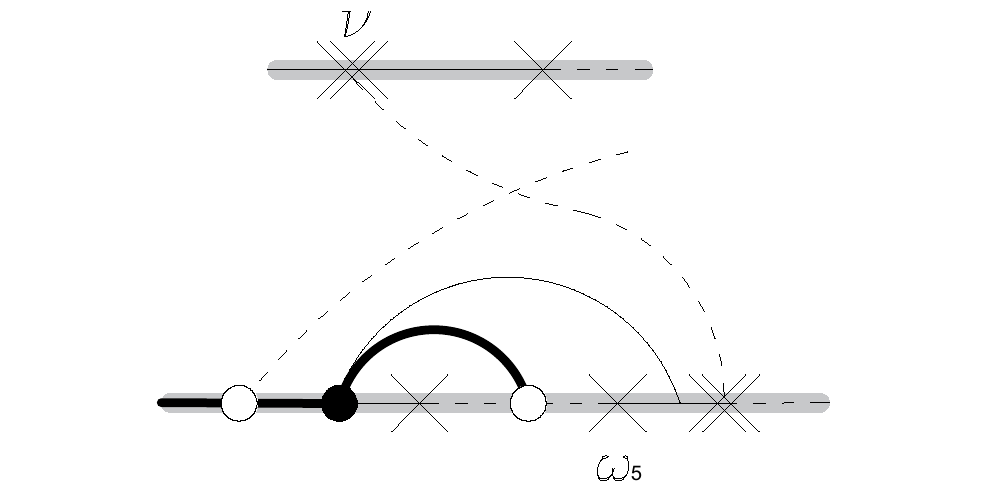}
\caption{\label{fig:toilesii44}}
\end{subfigure}\hspace{4mm} 
\begin{subfigure}{0.3\textwidth}
\centering\includegraphics[width=2.5in]{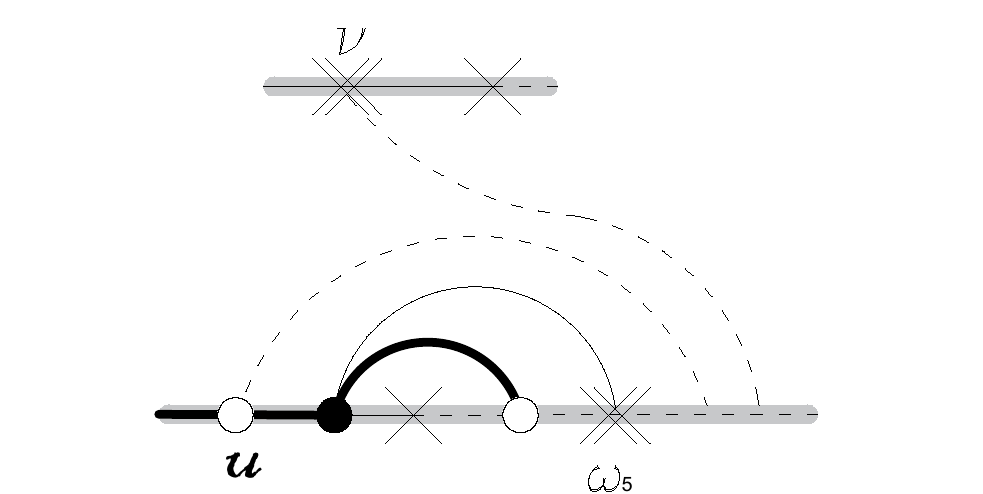}
\caption{\label{fig:toilesii45}}
\end{subfigure}
\end{center}
\caption{}
\end{figure}

If $w_4$ is a nodal {\tv}, the creation of a dotted bridge beside $w_3$ with the inner edge adjacent to $w_4$ produces an axe (see Figure~\ref{fig:toilesii46}).

If $w_3$ is a monochrome vertex, let $w_4$ be the vertex connected to $w$ by an inner bold edge.
If $w_4$ is a real white vertex, then the creation of a bridge beside it with the inner dotted edge adjacent to $w_3$ produces a cut (see Figure~\ref{fig:toilesii47}).
If $w_4$ is a monochrome vertex, an elementary move of type $\circ$-in allows us to consider it as an inner white vertex.
If $w_4$ is an inner white vertex, up to a monochrome modification it is connected to $w_3$ and then an elementary move of type $\circ$-out bring us to the configuration where $w_3$ was a white vertex.

\begin{figure}[h]
\begin{center}
\begin{subfigure}{0.3\textwidth}
\centering\includegraphics[width=2.5in]{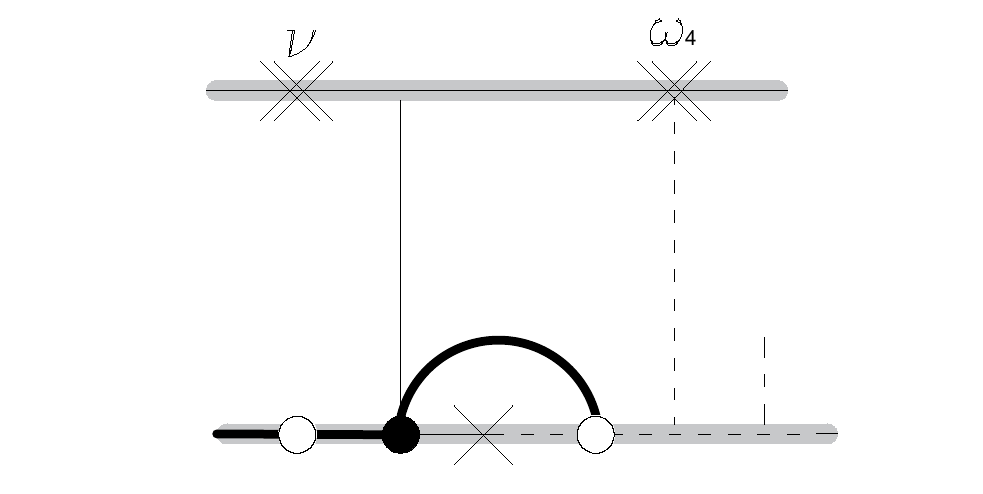}
\caption{\label{fig:toilesii46}}
\end{subfigure}\hspace{4mm} 
\begin{subfigure}{0.3\textwidth}
\centering\includegraphics[width=2.5in]{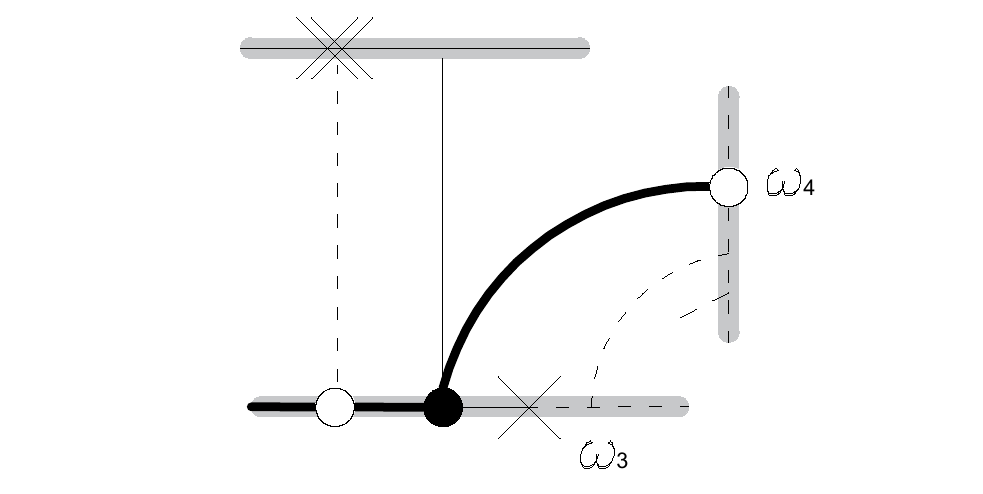}
\caption{\label{fig:toilesii47}}
\end{subfigure}\hspace{4mm} 
\begin{subfigure}{0.3\textwidth}
\centering\includegraphics[width=2.5in]{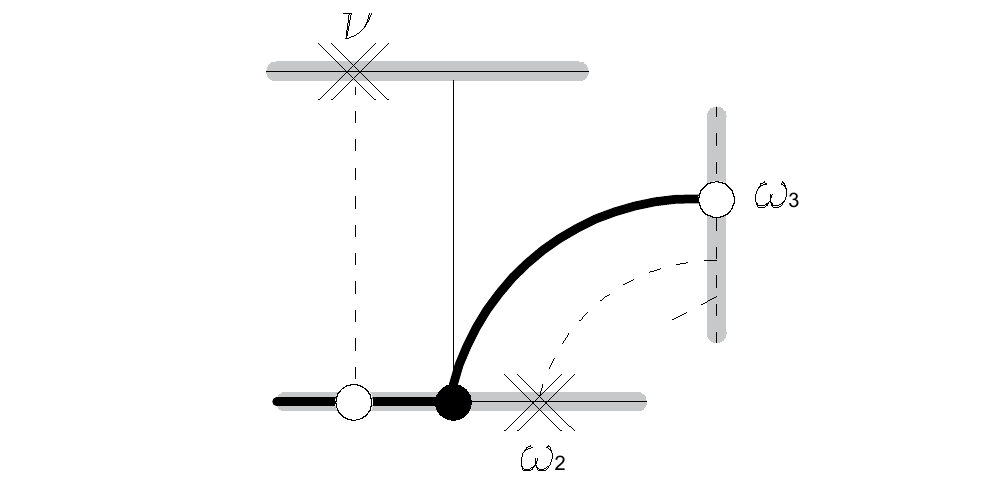}
\caption{\label{fig:toilesii48}}
\end{subfigure}
\end{center}
\caption{}
\end{figure}

Otherwise $w_2$ is a nodal {\tv}. Let $w_3$ be the vertex connected to $w$ by an inner bold edge.
If $w_3$ is a real white vertex, the creation of a bridge with the inner dotted edge adjacent to $w_2$ beside $w_3$ produces an axe (see Figure~\ref{fig:toilesii48}).
If $w_3$ is a monochrome vertex, an elementary move of type $\circ$-in allows us to consider it as an inner white vertex.
Lastly, if $w_3$ is an inner white vertex, up to a monochrome modification it is connected to $w_2$. Let $w_4\neq v$ be the vertex connected to $w_1$ by a real edge.
If $w_4$ is a simple {\tv}, up to the creation of a dotted bridge beside it with an inner dotted edge of $w_3$, we can perform an elementary move of type $\circ$-out producing an axe (see Figure~\ref{fig:toilesii49}).
If $w_4$ is a nodal {\tv}, up to a monochrome modification it is connected to $w_3$ producing a {\gc} (see Figure~\ref{fig:toilesii50}).

\begin{figure}[h]
\begin{center}
\begin{subfigure}{0.3\textwidth}
\centering\includegraphics[width=2.5in]{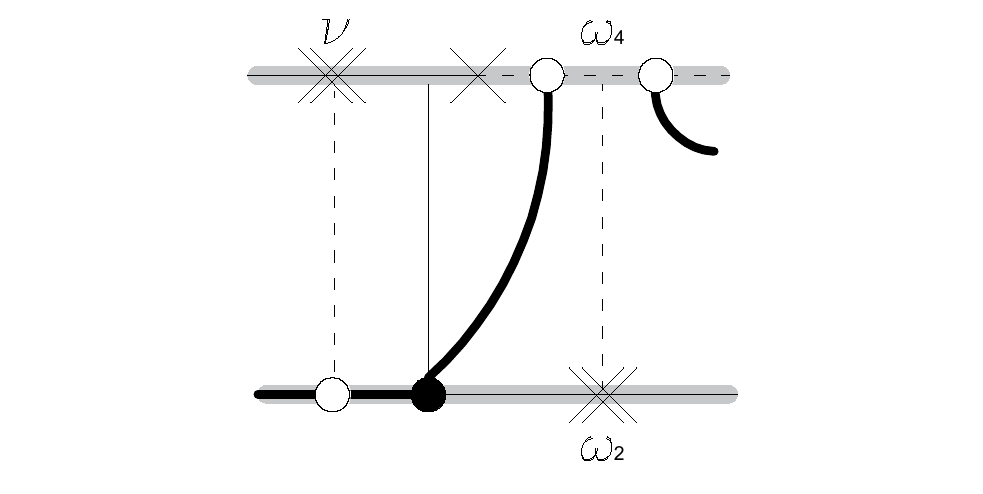}
\caption{\label{fig:toilesii49}}
\end{subfigure}\hspace{2mm} 
\begin{subfigure}{0.3\textwidth}
\centering\includegraphics[width=2.5in]{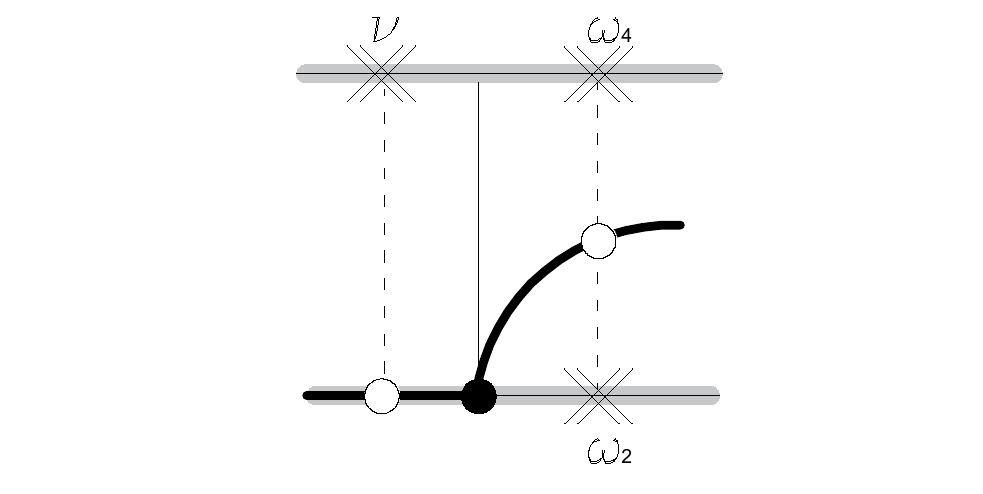}
\caption{\label{fig:toilesii50}}
\end{subfigure}
\end{center}
\caption{}
\end{figure}

If $w_1$ is an inner simple {\tv}, we can performs a monochrome modification so it is connected to $u$ in order to create a zigzag, but this contradicts the maximality assumption.

Otherwise, the vertex $w_1$ is an inner nodal {\tv}.
Let $w_2\notin R$ be the vertex connected to $w_1$ by a dotted edge. 
If $w_2$ is an inner white vertex, it is connected to $w$ up to a monochrome modification. Let $w_3$ be the vertex connected to $w$ by a real solid edge.
If $w_3$ is a simple {\tv}, the creation of a bridge beside $w_3$ with an inner dotted edge adjacent to $w_2$ and a monochrome vertex with the edges $e$ and an inner dotted edge adjacent to $w_1$ produce a {\gc} (see Figure~\ref{fig:toilesii51}).
If $w_3$ is a nodal {\tv}, up to a monochrome modification it is connected to $w_2$. Then, the creation of an inner monochrome vertex with the edges $e$ and an inner dotted edge adjacent to $w_1$ produce a {\gc} (see Figure~\ref{fig:toilesii52}).
Lastly, if $w_3$ is a monochrome vertex, we create a bridge beside $v$ with the edge $[w,w_1]$ and we perform an elementary move of type $\bullet$-in at $w_3$ followed by an elementary move of type $\bullet$-out at the bridge, bringing us to a configuration considered in Case 1.1.

\begin{figure}[h]
\begin{center}
\begin{subfigure}{0.3\textwidth}
\centering\includegraphics[width=2.5in]{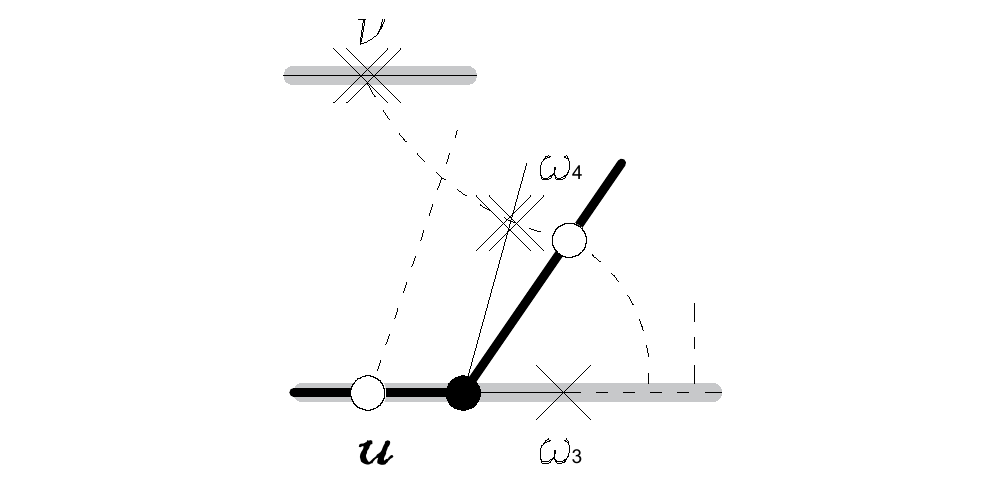}
\caption{\label{fig:toilesii51}}
\end{subfigure}\hspace{2mm} 
\begin{subfigure}{0.3\textwidth}
\centering\includegraphics[width=2.5in]{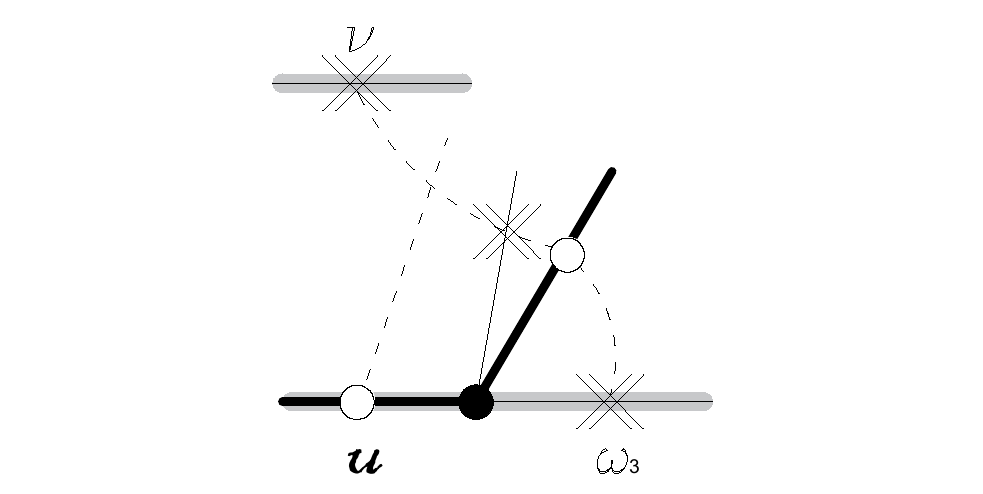}
\caption{\label{fig:toilesii52}}
\end{subfigure}
\end{center}
\caption{}
\end{figure}

If $w_2$ is a real white vertex, let $w_3$ be the vertex connected to $w$ by a bold inner edge.
If $w_3$ is a monochrome vertex, an elementary move of type $\circ$-in allows us to consider it as an inner white vertex.
If $w_3$ is an inner white vertex, a monochrome modification connects it to $w_1$ and that corresponds to the configuration where $w_2$ was an inner white vertex.
Finally, if $w_3$ is a real white vertex, the creation of a bridge beside it with an inner dotted edge adjacent to $w_1$ and a monochrome vertex with the edges $e$ and an inner dotted edge adjacent to $w_1$ produce a {\gc} (see Figure~\ref{fig:toilesii55}).

\begin{figure}[h]
\begin{center}
\begin{subfigure}{0.3\textwidth}
\centering\includegraphics[width=2.5in]{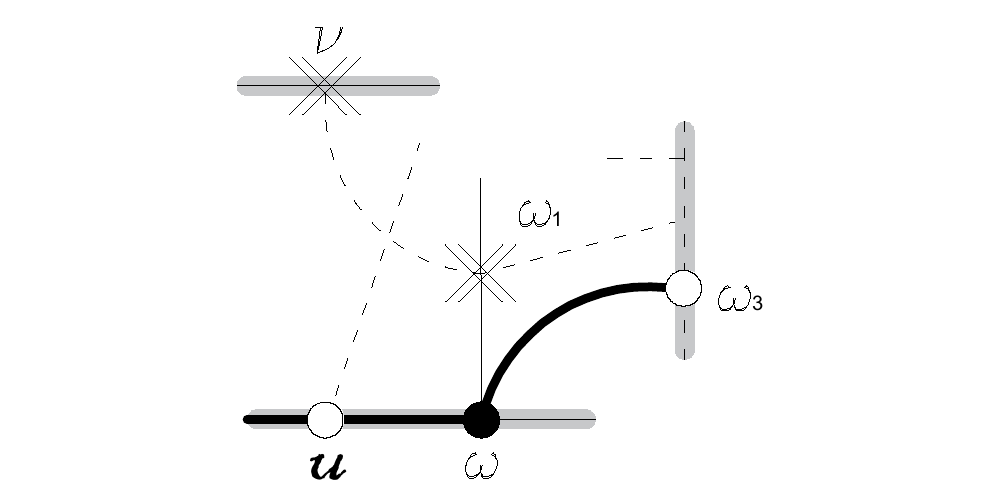}
\caption{\label{fig:toilesii55}}
\end{subfigure}\hspace{2mm} 
\begin{subfigure}{0.3\textwidth}
\includegraphics[width=2.5in]{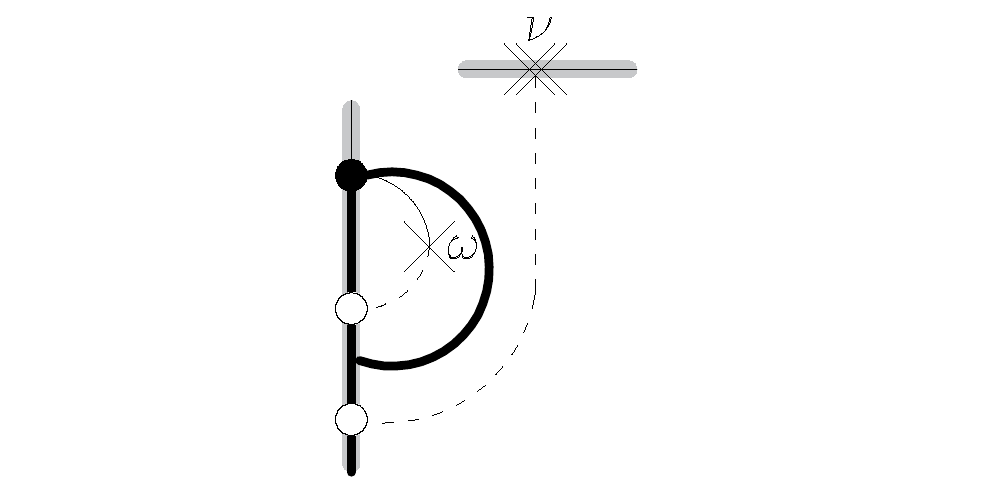}
\caption{\label{fig:toilesii57}}
\end{subfigure}
\end{center}
\caption{}
\end{figure}

{\it Case 1.3:} the vertex $u$ is an inner white vertex. If $u$ is connected to another nodal {\tv} or a monochrome dotted vertex, it determines a {\gc}.
Otherwise $u$ is connected to an inner {\tv} $w$.
Let us assume that $u$ is an inner simple {\tv}.

If $u$ is connected to a monochrome bold vertex, an elementary move of type $\circ$-out bring us to the previous cases.
If $u$ is connected to an inner black vertex, the creation of a bridge beside $v$ with an inner solid edge adjacent to the black vertex and an elementary move of type $\bullet$-out allow us to consider $u$ connected to a real black vertex.
Thus, the vertex $u$ must be connected to a black real vertex $b$.
Let $R$ to be the region determined by $u$, $v$ and $b$, and let $f\neq[u,b]$ be the edge adjacent to $b$ in $R$.
If $f$ is a real edge, then up to a monochrome modification the vertex $b$ and $w$ are connected by an inner solid edge, and then, the creation of a bold bridge beside $b$ with an inner edge adjacent to $u$ followed by an elementary move of type $\circ$-out allows us to create a zigzag, contradicting the maximality assumption (see Figure~\ref{fig:toilesii57}).

If $f$ is an inner edge, let $m$ be the vertex connected to $b$ by a real solid edge.
If $m$ is a {\tv}, then either the creation of a bridge beside $m$ with the edge $[u,w]$ or a monochrome modification connecting $u$ and $m$, if $m$ is simple or nodal, respectively, produces a {\gc}.
If $m$ is a monochrome vertex, it is connected to $w$ and then, an elementary move of type $\bullet$-in at $m$ followed by the creation of a solid bridge beside $v$ with the edge $f$ and an elementary move of type $\bullet$-out at the bridge allows us to create a zigzag, contradicting the maximality assumption (see Figure~\ref{fig:toilesii58}).

\begin{figure}[h]
\begin{center}
\begin{subfigure}{0.3\textwidth}
\centering\includegraphics[width=2.5in]{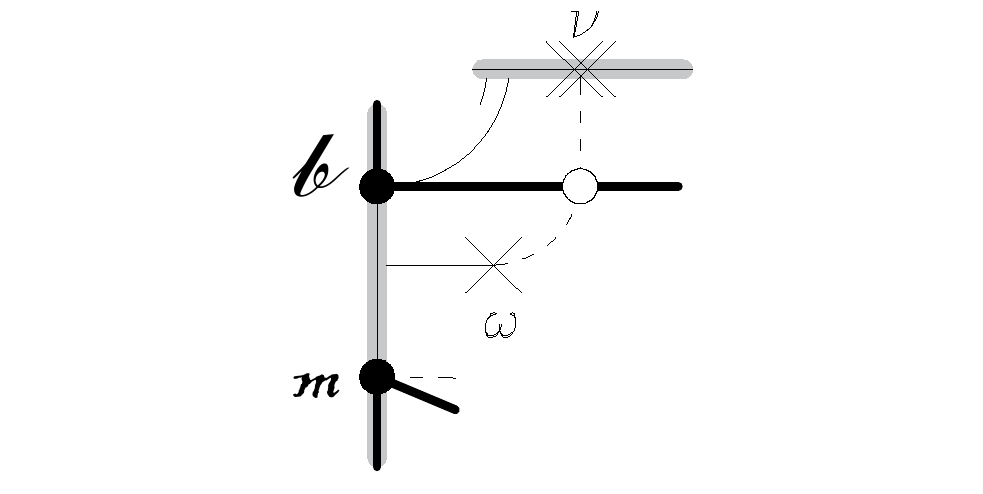}
\caption{\label{fig:toilesii58}}
\end{subfigure}$\cong$
\begin{subfigure}{0.3\textwidth}
\centering\includegraphics[width=2.5in]{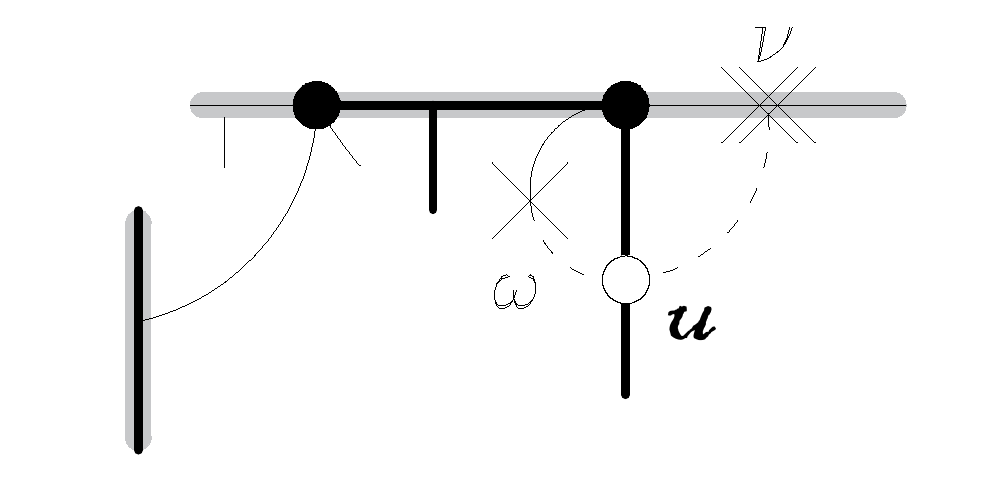}
\caption{\label{fig:toilesii59}}
\end{subfigure}
\end{center}
\caption{}
\end{figure}

The missing case is when the vertex $u$ is connected to an inner nodal {\tv} $w$. Due to the aforementioned considerations, if any of the vertices $u$ or $w$ share a region with a real {\tv}, the creation of a dotted bridge or an inner monochrome vertex produce a {\gc}. If the vertex $u$ is connected to black vertices having neighboring real solid monochrome vertices, then the creation of solid bridges beside $v$ followed by a elementary moves of type $\bullet$-in and $\bullet$-out bring us to the configuration where the vertex $v$ has two neighboring black vertices $b_1$ and $b_2$, which are connected to $u$ and $w$ up to monochrome modifications.
Let $S$ be the bold segments containing $b_1$.
If the segment $S$ has no white vertices, it does have a monochrome vertex connected to a real white vertex determining a dotted segment where the construction of a bridge with the edge $[u,w]$ produces a {\gc} (see Figure~\ref{fig:toilesii60}).
If the segment $S$ has exactly one white vertex $w_1$, it has two neighboring black vertices $b_1$ and $b_1'$. Up to monochrome modifications, the vertices $b_1$ and $b_1'$ are connected to $w$ and $u$. Then, if $b_1'$ has a neighboring real {\tv}, then either the creation of a bridge beside it with the edge $[v,u]$ or the creation of an inner monochrome vertex with the edge $[v,u]$ and the inner dotted edge adjacent to the nodal {\tv} produces a {\gc} (see Figures~\ref{fig:toilesii61} and \ref{fig:toilesii62}).
Otherwise, the vertex $b_1'$ is connected to a solid monochrome vertex $m$, in which case, the creation of a bridge beside $v$ with the inner solid edge adjacent to $m$ produces a solid cut where one of the resulting dessins is a cubic of type $I^{***}$ (a cubic with an inner nodal {\tv} and an isolated real nodal {\tv}) (see Figure~\ref{fig:toilesii63}). 
If the segment $S$ has at least two white vertices, by means of elementary moves of type $\circ$-in and $\circ$-out, we can transfer pairs of white vertices from the segment $S$ to the bold segment containing the vertex $b_2$ bring us to the case when $S$ has none or one white vertex.

\begin{figure}[h]
\begin{center}
\begin{tabular}{lp{1cm}r}
\begin{subfigure}{0.3\textwidth}
\centering\includegraphics[width=2.5in]{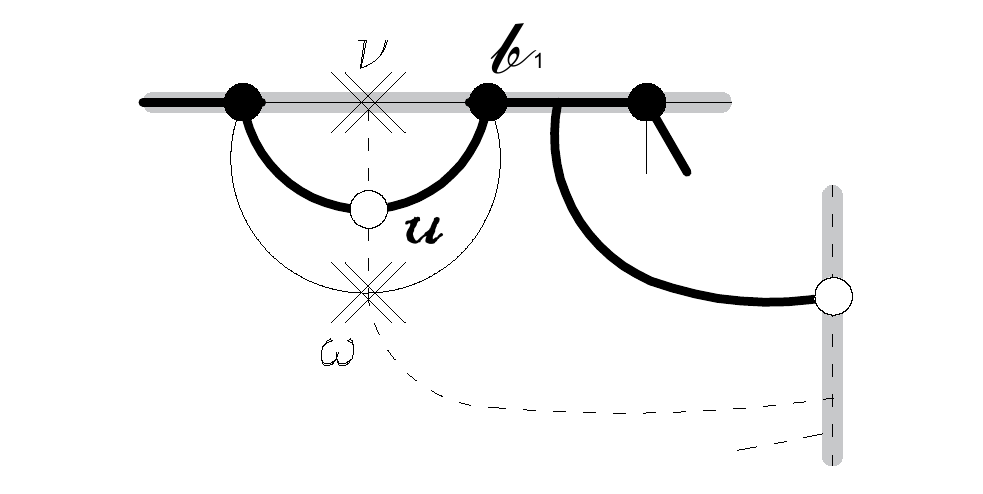}
\caption{\label{fig:toilesii60}}
\end{subfigure}
& &
\begin{subfigure}{0.3\textwidth}
\centering\includegraphics[width=2.5in]{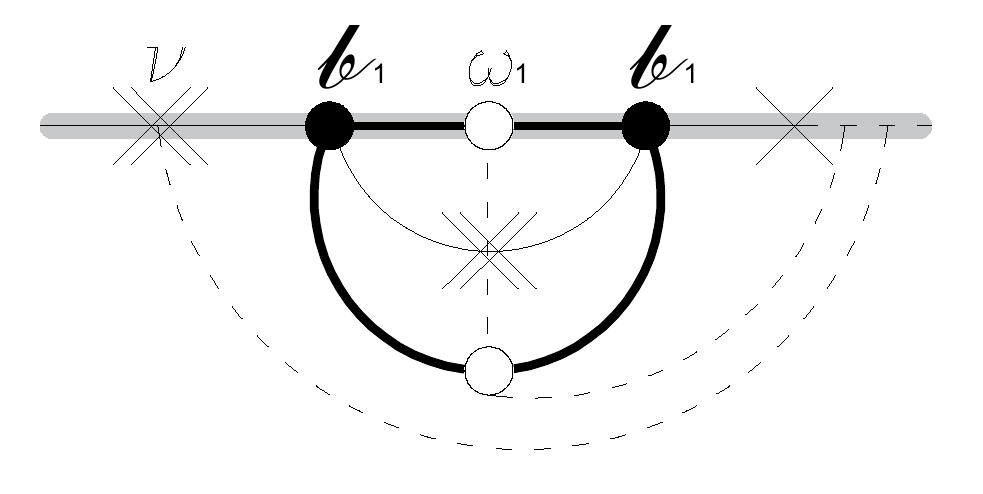}
\caption{\label{fig:toilesii61}}
\end{subfigure}
\\
\begin{subfigure}{0.3\textwidth}
\centering\includegraphics[width=2.5in]{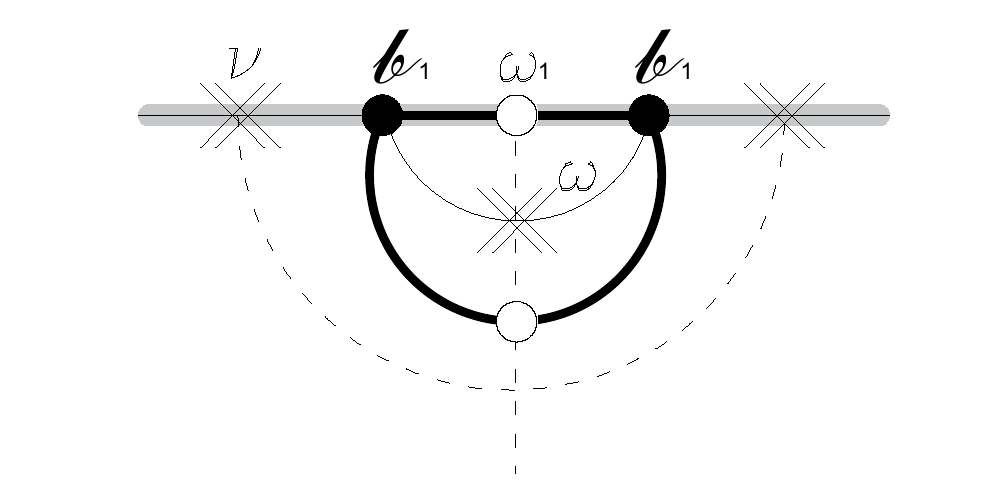}
\caption{\label{fig:toilesii62}}
\end{subfigure}
& &
\begin{subfigure}{0.3\textwidth}
\centering\includegraphics[width=2.5in]{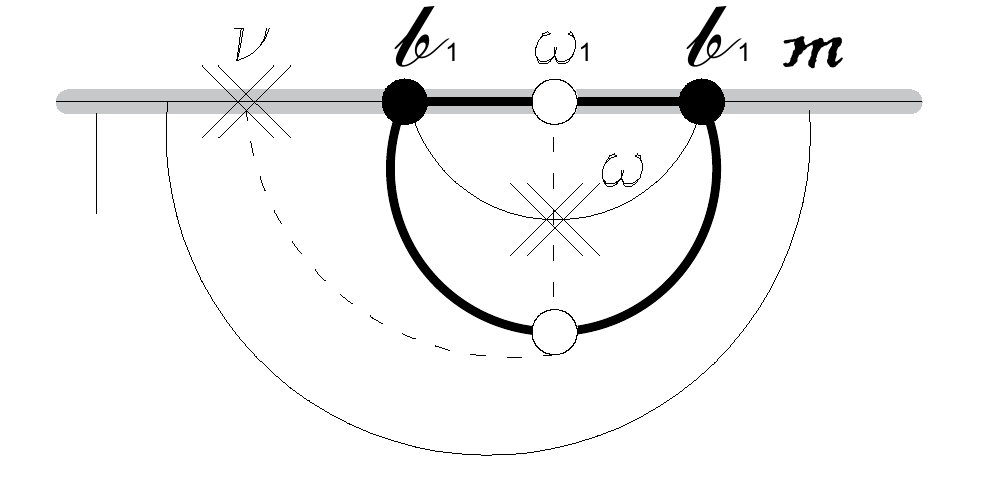}
\caption{\label{fig:toilesii63}}
\end{subfigure}
\end{tabular}
\end{center}
\caption{}
\end{figure}

\vskip5pt
{\it Case 2:} the vertex $v$ is non-isolated, being connected to a black vertex $u$.
Due to Case~1, we can assume that there are no isolated real nodal {\tvs}.
If the vertex $u$ is real, the edge $e:=[v,u]$ is dividing. Let $w$ be vertex connected to $u$ by a bold real edge. Let $R$ be the region determined by $v$, $u$ and $w$.

\begin{figure}[h]
\begin{center}
\begin{subfigure}{0.3\textwidth}
\centering\includegraphics[width=2.5in]{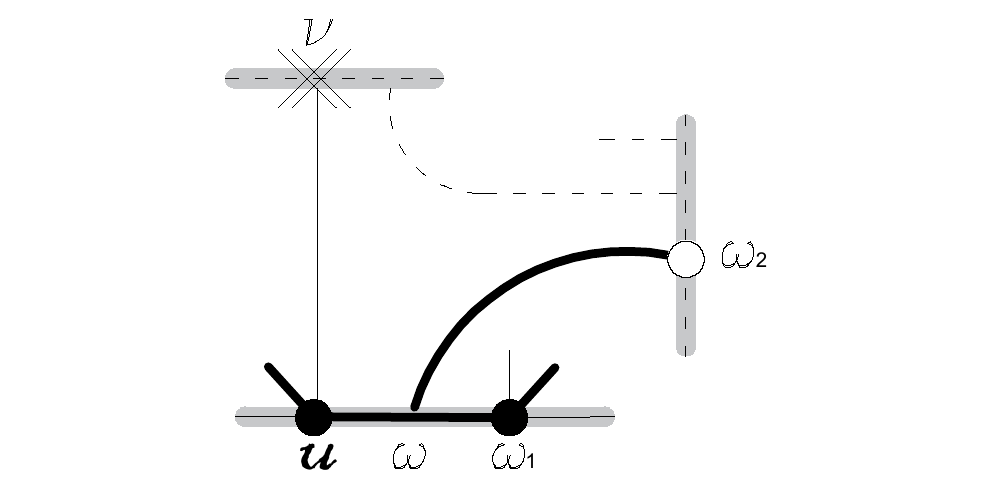}
\caption{\label{fig:toilesii64}}
\end{subfigure}\hspace{2mm} 
\begin{subfigure}{0.3\textwidth}
\centering\includegraphics[width=2.5in]{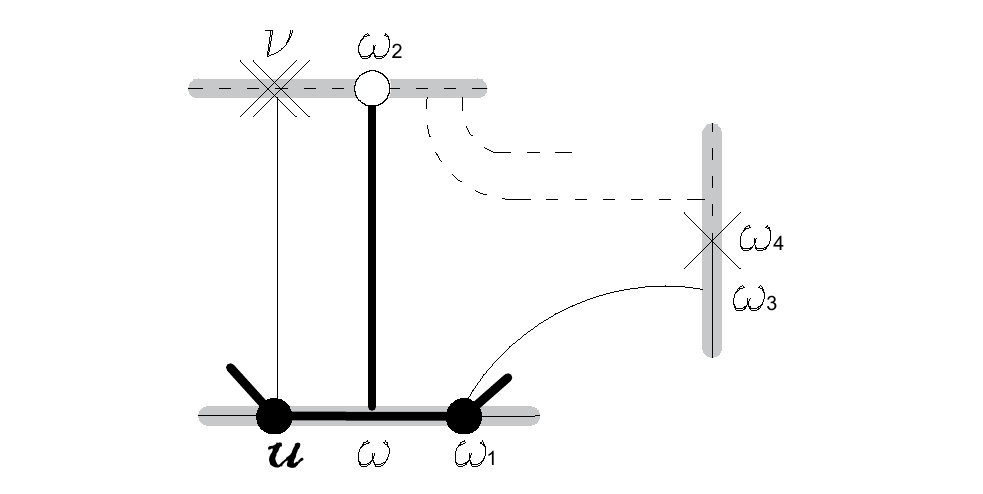}
\caption{\label{fig:toilesii66}}
\end{subfigure}\hspace{2mm} 
\begin{subfigure}{0.3\textwidth}
\includegraphics[width=2.5in]{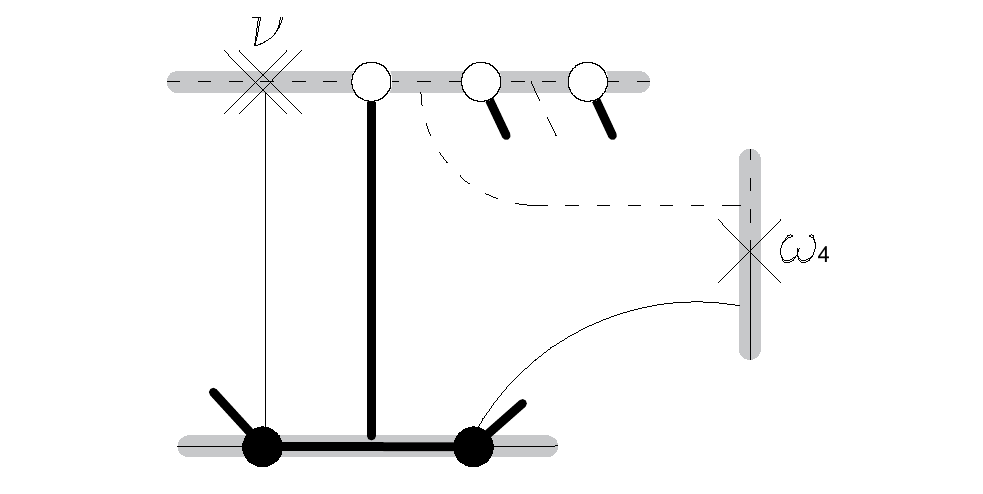}
\caption{\label{fig:toilesii68}}
\end{subfigure}
\end{center}
\caption{}
\end{figure}

{\it Case 2.1:} the vertex $w$ is a monochrome vertex. Since the toile is bridge-free, the vertex $w$ is connected to a real black vertex $w_1\neq u$.
Let $w_2$ be the vertex connected to $w$ by an inner bold edge.
If $w_2$ is a monochrome vertex, an elementary move of type $\circ$-in followed by the possible creation of a bridge beside $v$ and an elementary move of type $\circ$-out allows us to consider $w_2$ as a real white vertex.
In the case when the vertex $w_2$ is a real white vertex, if it is not a neighboring vertex of $v$ and the region $R$ contains an inner dotted edge, up to the creation of bridges beside $v$ and $w_2$ we can produce a dotted cut (see Figure~\ref{fig:toilesii64}).

Otherwise, we can assume the vertices $w_2$ and $v$ to be real neighboring vertices.
Let $w_3$ be the vertex connected to $w_1$ by an inner solid edge.
If $w_3$ is a monochrome vertex, it has two real simple {\tvs} since the toile is bridge-free and there are no isolated nodal {\tvs}.
Let $w_4$ be the simple {\tv} connected to $w_3$ sharing a region with the vertices $w$ and $w_1$. Let $w_4'\neq w_4$ be the other simple {\tv} connected to $w_3$.
All the following considerations apply to the case when $w_3$ is a nodal {\tv}, in that case, the calls to $w_4$ and $w_4'$ all refer to $w_3$.
If $w_4$ is connected to a dotted monochrome vertex $w_5$, then, the creation of a dotted bridge beside $w_2$ with the inner edge adjacent to $w_5$ produces a cut (see Figure~\ref{fig:toilesii66}).

If $w_4$ is connected to a real white vertex $w_5$ having a real neighboring monochrome vertex $w_6$, by applying a monochrome modification we can connect $w_5$ to $w$ and then create a bridge beside $v$ with the inner edge adjacent to $w_6$ in order to produce a dotted cut (see Figure~\ref{fig:toilesii68}).

Otherwise, the vertex $w_4$ is connected to a real white vertex $w_5$ having a real neighboring {\tv} $w_6\neq w_4$.
If $w_6$ is a simple {\tv}, a monochrome modification connecting $w_5$ to $w$ and the creation of a bridge beside $w_6$ with the edge $e$ produces a solid axe (see Figure~\ref{fig:toilesii69}).

If $w_6$ is a nodal {\tv}, monochrome modifications connecting $w_5$ to $w$ and $w_6$ to $u$ bring us to a configuration equivalent to consider the vertices $w_4$ and $w_2$ as real neighbors.
In the case where the vertices $w_4$ and $w_2$ are real neighbors, let us consider $w_5$ the vertex connected to $w_1$ by an inner bold edge. 
If $w_5$ is a monochrome vertex, an elementary move of type $\circ$-in at $w_5$ followed by the creation of a bridge beside $w_4'$ and an elementary move of type $\circ$-out allows to consider the vertex $w_5$ as a white real vertex.
The same applies if $w_5$ is an inner white vertex.
Let us assume the vertex $w_5$ is a white real vertex.
If $w_4'$ has a monochrome dotted neighboring vertex, then the creation of a bridge with the corresponding inner edge beside $w_5$ produces a dotted cut (see Figure~\ref{fig:toilesii70}).
Otherwise, up to a monochrome modification the vertices $w_4'$ and $w_5$ are neighbors.

\begin{figure}[h]
\begin{center}
\begin{subfigure}{0.3\textwidth}
\centering\includegraphics[width=2.5in]{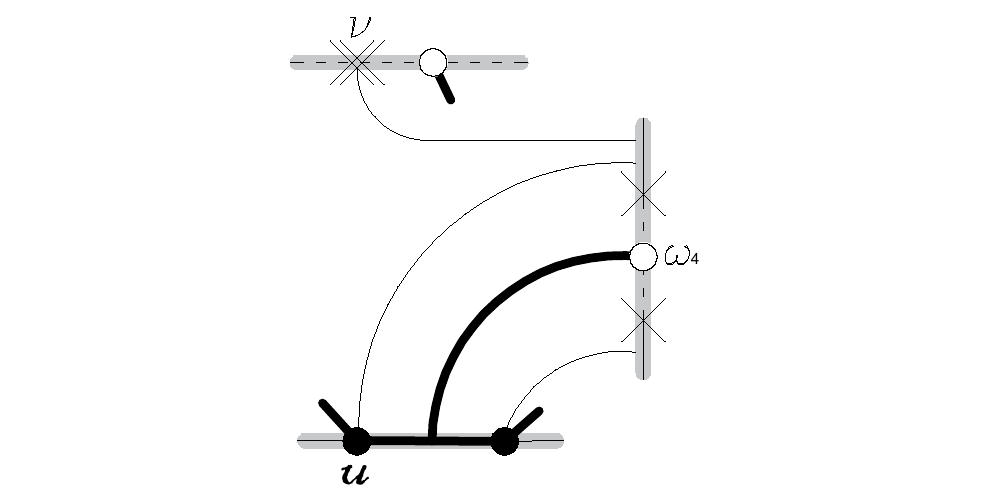}
\caption{\label{fig:toilesii69}}
\end{subfigure}\hspace{2mm} 
\begin{subfigure}{0.3\textwidth}
\includegraphics[width=2.5in]{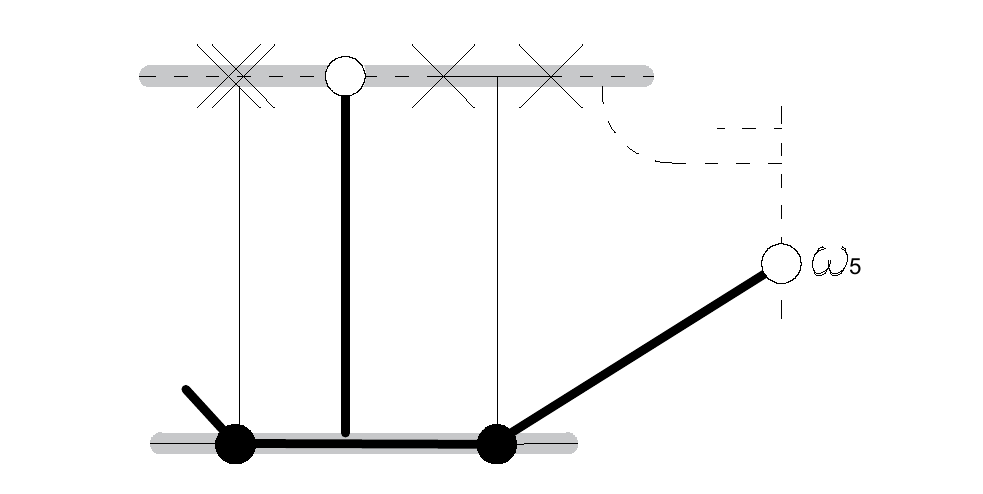}
\caption{\label{fig:toilesii70}}
\end{subfigure}
\end{center}
\caption{}
\end{figure}

Let $w_6$ be the vertex connected to $w_1$ by a real solid edge.
If $w_6$ is monochrome, it is connected to a real black vertex $w_1'\neq w_1$. Let $w_7$ be the vertex connected to $w_6$ by an inner solid edge.
If $w_7$ is a monochrome vertex, it determines a solid cut.
If $w_7$ is a real nodal {\tv}, it determines an axe.
If $w_7$ is an inner {\tv}, 
let $w_8$ 
be the vertex connected to $u$ by an inner bold edge. 
If $w_8$ is a monochrome vertex, an elementary move of type $\circ$-in at it allows us to consider $w_8$ as an inner white vertex.
If $w_8$ is an inner white vertex, up to the creation of a dotted bridge is it connected to a monochrome vertex beside $v$ and then an elementary move of type $\circ$-out bring us to the case when $w_8$ is a real white vertex.
If $w_8$ is a real white vertex, we perform an elementary move of type $\bullet$-in at $w$ and the destruction of the resulting solid bridge.
If $w_7$ is a simple {\tv}, up to the creation of bridges beside the vertices $w_5$ and $w_8$ with the inner dotted edge adjacent to $w_7$ we construct a cut (see Figure~\ref{fig:toilesii72}).
In the case when $w_7$ is an inner nodal {\tv}, up to the construction of bridges beside the vertices $w_5$ and $w_8$ with the inner dotted edges adjacent to $w_7$, respectively, produces a {\gc} (see Figure~\ref{fig:toilesii73}).

\begin{figure}[h]
\begin{center}
\begin{subfigure}{0.45\textwidth}
\centering\includegraphics[width=2.5in]{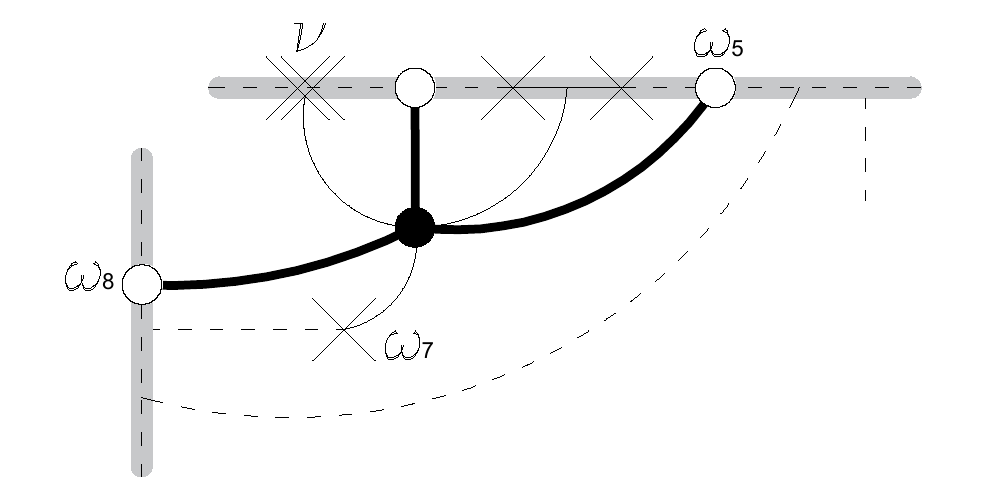}
\caption{\label{fig:toilesii72}}
\end{subfigure}\hspace{6mm} 
\begin{subfigure}{0.45\textwidth}
\centering\includegraphics[width=2.5in]{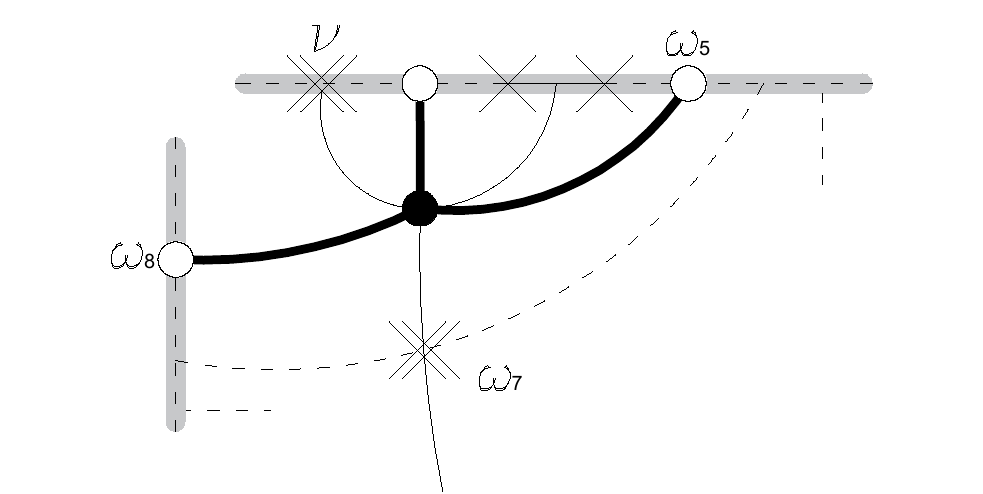}
\caption{\label{fig:toilesii73}}
\end{subfigure}
\end{center}
\caption{}
\end{figure}

If $w_6$ is a real simple {\tv}, let $w_7\neq w_4'$ be the vertex connected to $w_5$.
If $w_7$ is a monochrome vertex, the creation of a bridge with the inner dotted edge adjacent to it beside $w_6$ produces a cut (see Figure~\ref{fig:toilesii74}).
If $w_7$ is a nodal {\tv}, the the creation of a bridge with the inner solid edge adjacent to it beside $w_6$ produces an axe (see Figure~\ref{fig:toilesii75}).
In the case when $w_7\neq w_6$ is a simple {\tv}, let $w_8$ be the vertex connected to $w_6$ by a real dotted edge.
If $w_8$ is a monochrome vertex, the creation of a bridge with the inner dotted edge adjacent to it beside $w_5$ produces a cut (see Figure~\ref{fig:toilesii76}).

\begin{figure}[h]
\begin{center}
\begin{subfigure}{0.3\textwidth}
\centering\includegraphics[width=2.5in]{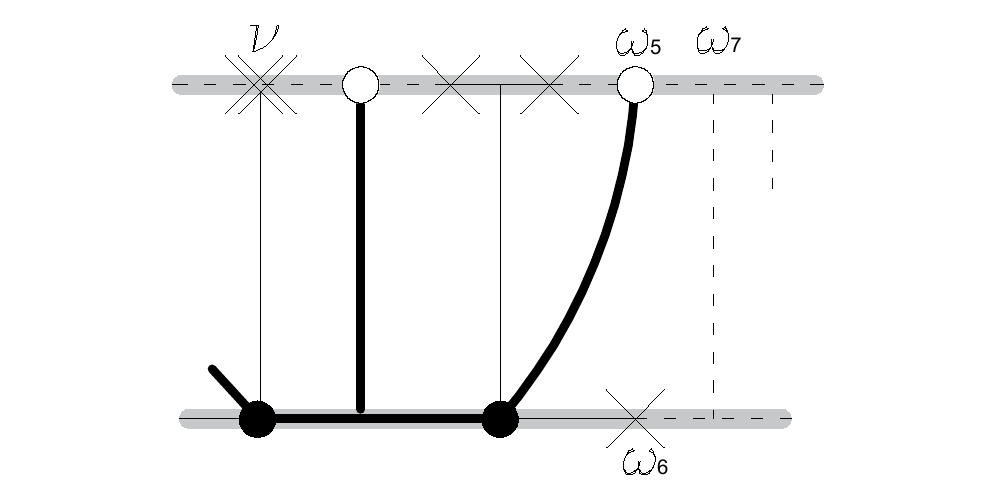}
\caption{\label{fig:toilesii74}}
\end{subfigure}\hspace{2mm} 
\begin{subfigure}{0.3\textwidth}
\centering\includegraphics[width=2.5in]{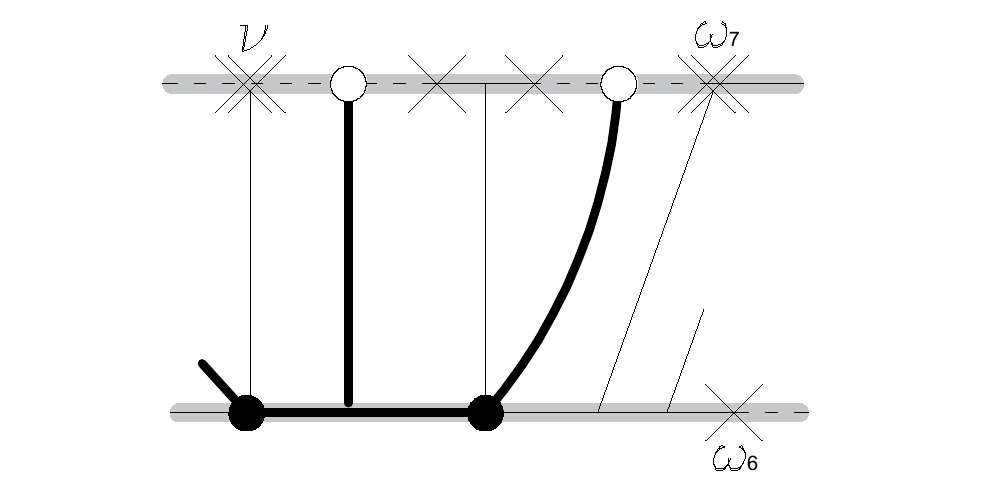}
\caption{\label{fig:toilesii75}}
\end{subfigure}\hspace{2mm}
\begin{subfigure}{0.3\textwidth}
\centering\includegraphics[width=2.5in]{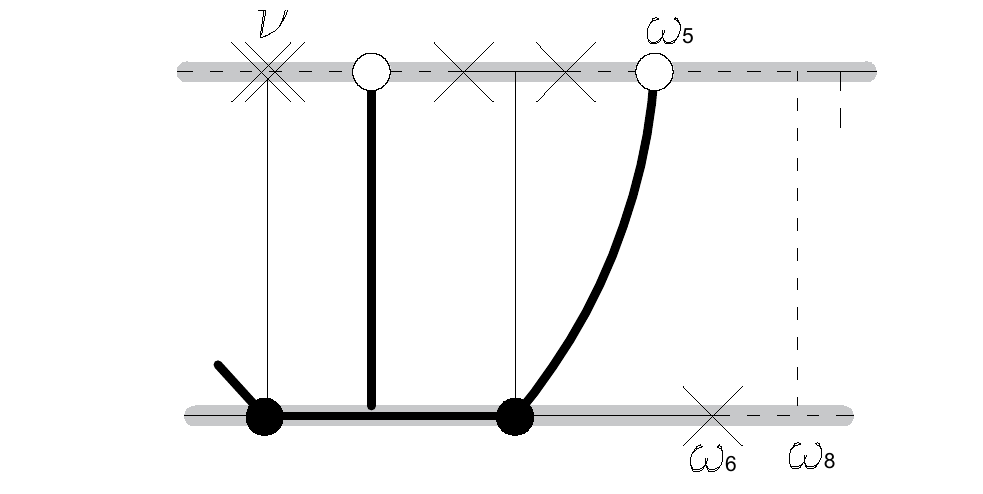}
\caption{\label{fig:toilesii76}}
\end{subfigure}
\end{center}
\caption{}
\end{figure}

If $w_8$ is a real white vertex, it is connected to $w_1$ after a monochrome modification. Let $w_9\neq w_6$ be the vertex connected to $w_8$ by a real dotted edge.
If $w_9$ is a monochrome vertex, 
the creation of a bridge with the inner dotted edge adjacent to $w_9$ beside $w_5$ produce a cut (see Figure~\ref{fig:toilesii77}).
If $w_9$ is a nodal {\tv}, the creation of an inner solid monochrome vertex with the inner edge adjacent to $w_9$ and the edge $[w_1,w_3]$ produces a solid {\gc} (see Figure~\ref{fig:toilesii78}).
If $w_9$ is a simple {\tv}, the creation of a bridge beside $w_9$ with the edge $[w_1,w_3]$ produces a cut (see Figure~\ref{fig:toilesii79}).

\begin{figure}[h]
\begin{center}
\begin{subfigure}{0.45\textwidth}
\centering\includegraphics[width=2.5in]{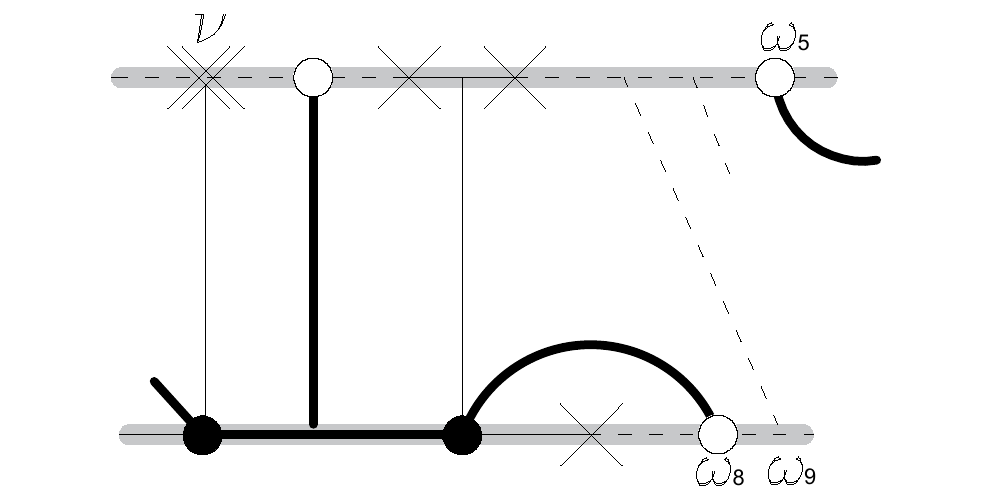}
\caption{\label{fig:toilesii77}}
\end{subfigure}\hspace{2mm} 
\begin{subfigure}{0.45\textwidth}
\centering\includegraphics[width=2.5in]{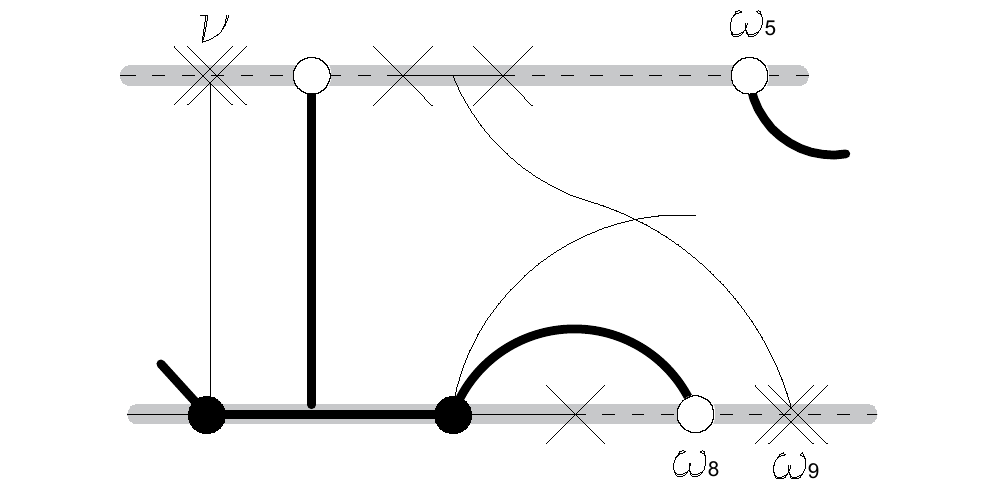}
\caption{\label{fig:toilesii78}}
\end{subfigure}

\begin{subfigure}{0.45\textwidth}
\centering\includegraphics[width=2.5in]{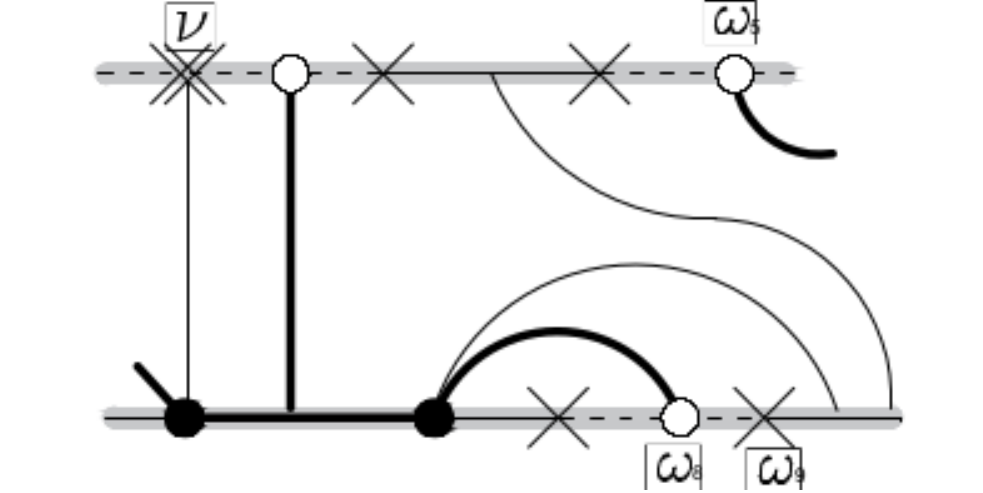}
\caption{\label{fig:toilesii79}}
\end{subfigure}
\end{center}
\caption{}
\end{figure}


A special case is when the vertices $w_7$ and $w_6$ are equal.
In this setting, we can consider the vertex $w_3$ as a nodal {\tv} and made the aforementioned considerations applied to the vertices $w_3$ and $w_1$ as $v$ and $u$, respectively.
Since the degree of the toile is greater than $3$, the process cannot cycle back to this case (see Figure~\ref{fig:toilesii80}).

\begin{figure}[b]
\begin{center}
\begin{subfigure}{0.3\textwidth}
\centering\includegraphics[width=2.5in]{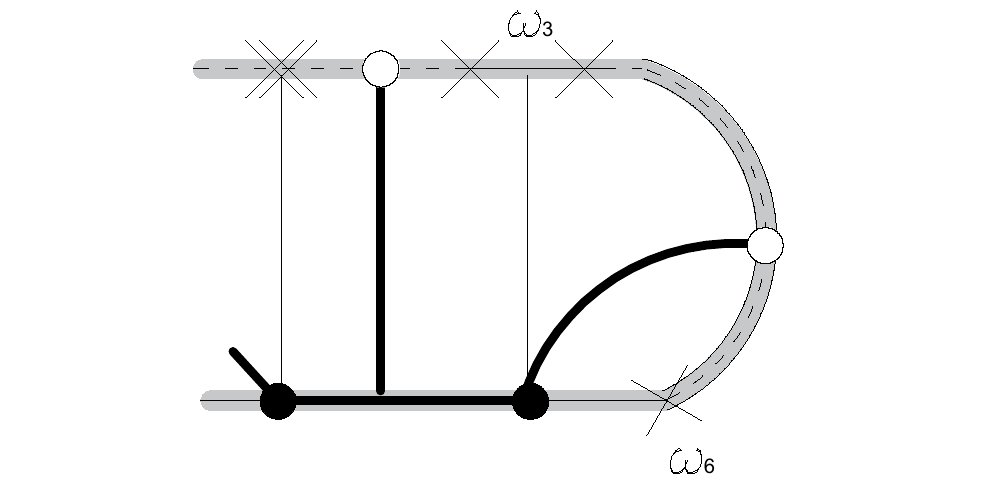}
\caption{\label{fig:toilesii80}}
\end{subfigure}\hspace{2mm} 
\begin{subfigure}{0.3\textwidth}
\centering\includegraphics[width=2.5in]{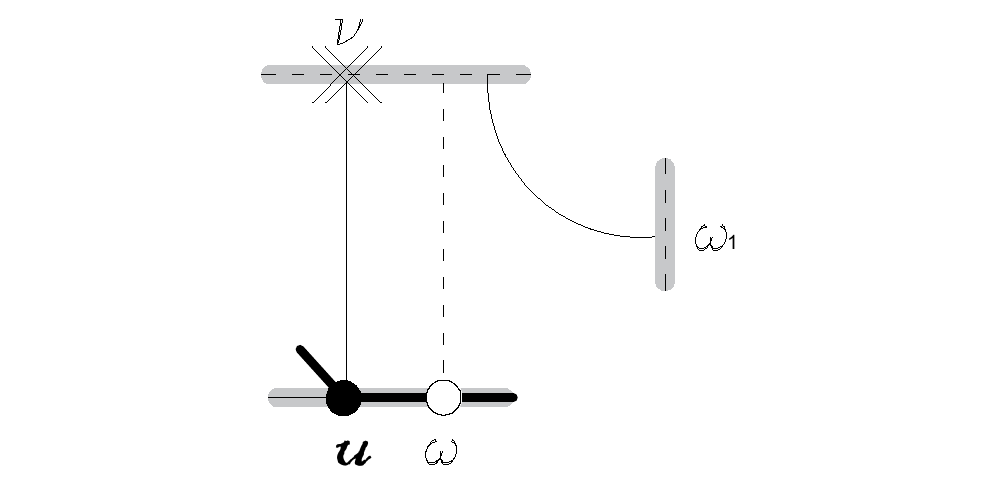}
\caption{\label{fig:toilesii81}}
\end{subfigure}
\end{center}
\caption{}
\end{figure}


\vskip5pt
{\it Case 2.2:} the vertex $w$ is a white vertex.
Let $w_1$ be the vertex connected to $w$ by an inner dotted edge and let $w_2\neq u$ be the vertex connected to $w$ by a real bold edge.
If $w_1$ is a monochrome vertex and it is not a real neighboring vertex of $v$, then the creation of a dotted bridge beside the vertex $v$ produces a cut (see Figure~\ref{fig:toilesii81}).
Otherwise, let us assume the vertex $w_1$ to be monochrome and a real neighboring vertex of $v$.
If $w_2$ is a monochrome vertex, an elementary move of type $\circ$-in at it followed by an elementary move of type $\circ$-out at $w_1$ and the destruction of a possible resulting bold bridge brings us either to Case 2.1 or to a different configuration within this case.
From now on, we assume the vertex $w_2$ to be black. Let $w_3\neq v$ be the vertex connected to $w_1$ by a real dotted edge.

\begin{figure}[h]
\begin{center}
\begin{subfigure}{0.3\textwidth}
\centering\includegraphics[width=2.5in]{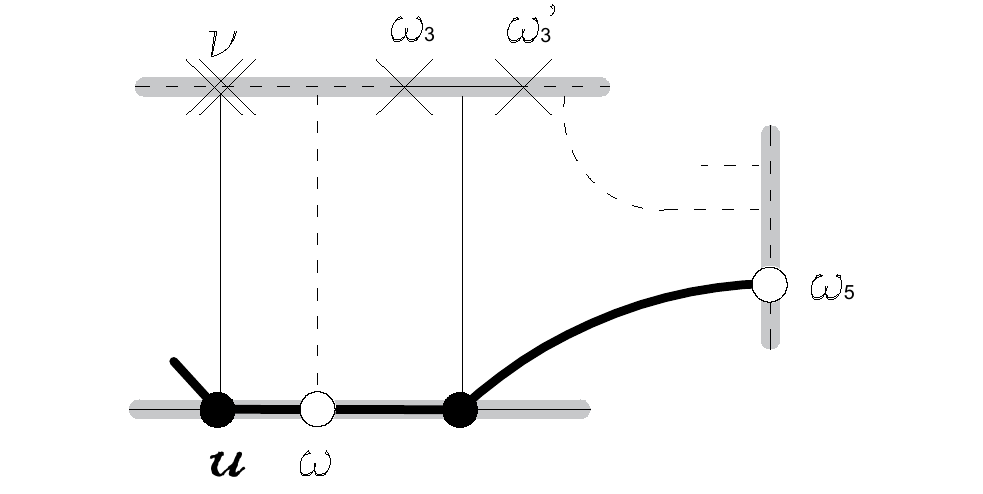}
\caption{\label{fig:toilesii82}}
\end{subfigure}\hspace{2mm} 
\begin{subfigure}{0.3\textwidth}
\centering\includegraphics[width=2.5in]{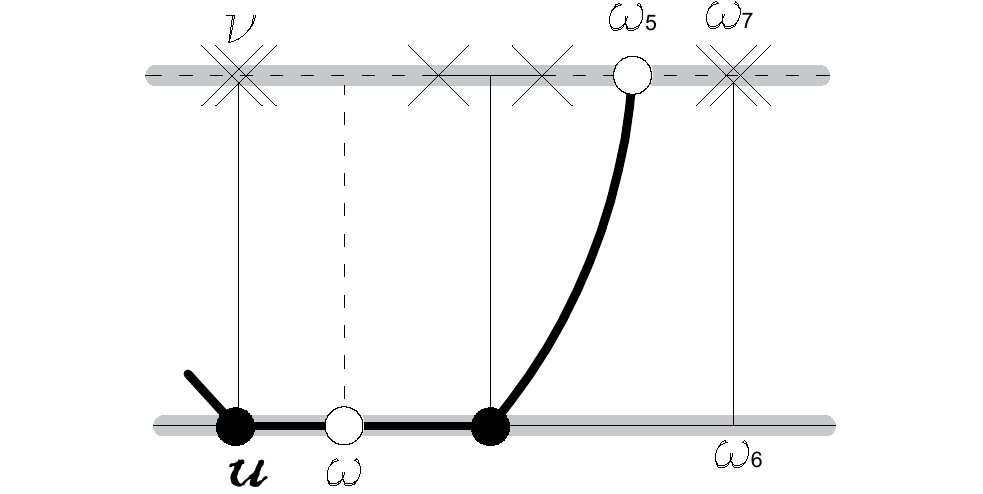}
\caption{\label{fig:toilesii83}}
\end{subfigure}\hspace{2mm}
\begin{subfigure}{0.3\textwidth}
\centering\includegraphics[width=2.5in]{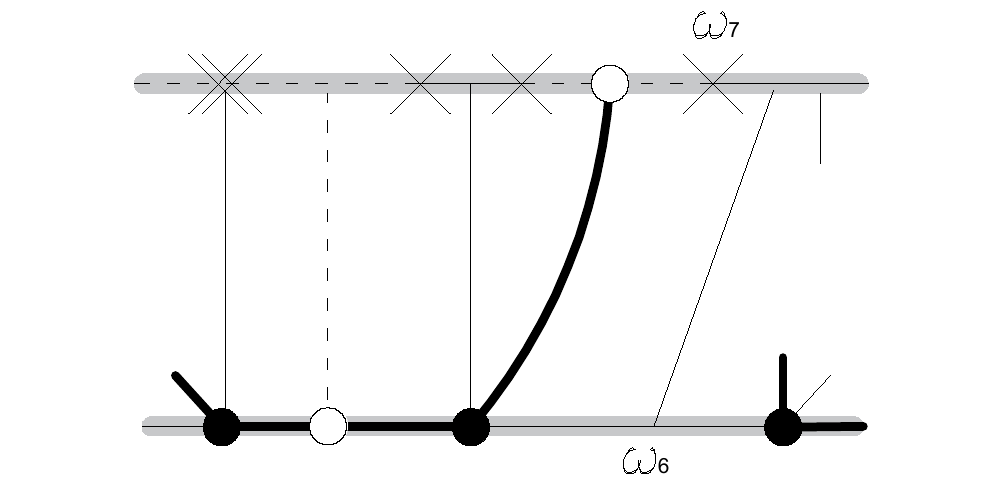}
\caption{\label{fig:toilesii84}}
\end{subfigure}
\end{center}
\caption{}
\end{figure}

If $w_3$ is a simple {\tv} followed by a monochrome vertex $w_4$, which up to a monochrome modification is connected to $w_2$, let $w_3'\neq w_3$ be the simple {\tv} connected to $w_4$ by a real solid edge.
The following considerations apply when the vertex $w_3$ is a nodal {\tv}, in which case we treat the vertices $w_3'$ and $w_4$ equal to $w_3$.
By symmetry of this configuration and since the degree is greater than $3$, we can assume the dividing edge $[w_2,w_4]$ does not divide the toile resulting in a graph with only two triangular regions.
Let $w_5$ be the vertex connected to $w_2$ by an inner bold edge.
If $w_5$ is a monochrome vertex or an inner white vertex, up to an elementary move of type $\circ$-out at it, the creation of a dotted bridge beside $w_3'$ and an elementary move of type $\circ$-out allow us to consider it as a real white vertex.
If $w_5$ is a real white vertex and the vertex $w_3'$ has a real neighboring dotted monochrome vertex, the creation of a bridge with its inner dotted edge beside $w_5$ produces a cut (see Figure~\ref{fig:toilesii82}). Otherwise, up to a monochrome modification the vertices $w_3'$ and $w_5$ are neighbors.
Let $w_6$ be the vertex connected to $w_2$ by a solid real edge.

If $w_6$ is a monochrome vertex, let $w_7\neq w_3'$ be the vertex connected to $w_5$ by a real dotted edge.
If $w_7$ is a real nodal {\tv}, up to a monochrome modification it is connected to $w_6$ defining an axe (see Figure~\ref{fig:toilesii83}).
If $w_7$ is a real simple {\tv}, the creation of a solid bridge beside $w_7$ with the inner solid edge adjacent to $w_6$ defines a solid cut (see Figure~\ref{fig:toilesii84}). 

Otherwise, the vertex $w_7$ is monochrome. Let $w_5'\neq w_5$ be the vertex connected to $w_7$ by a real dotted edge and let $w_2'\neq w_2$ be the vertex connected to $w_6$ by a real solid edge.
Let $w_8$ be the vertex connected to $w_6$ by an inner solid edge, let $w_9$ be the vertex connected to $w_2'$ by a real bold edge and let $w_{10}$ be the vertex connected to $w_2'$ by an inner solid edge (see Figure~\ref{fig:toilesii85}).

\begin{figure}[h]
\begin{center}
\begin{subfigure}{0.45\textwidth}
\centering\includegraphics[width=2.5in]{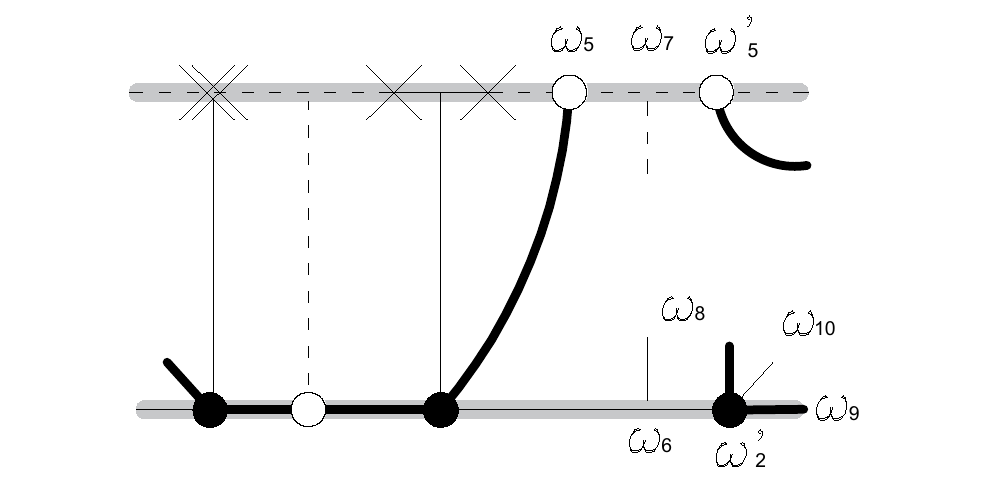}
\caption{\label{fig:toilesii85}}
\end{subfigure}\hspace{6mm} 
\begin{subfigure}{0.45\textwidth}
\centering\includegraphics[width=2.5in]{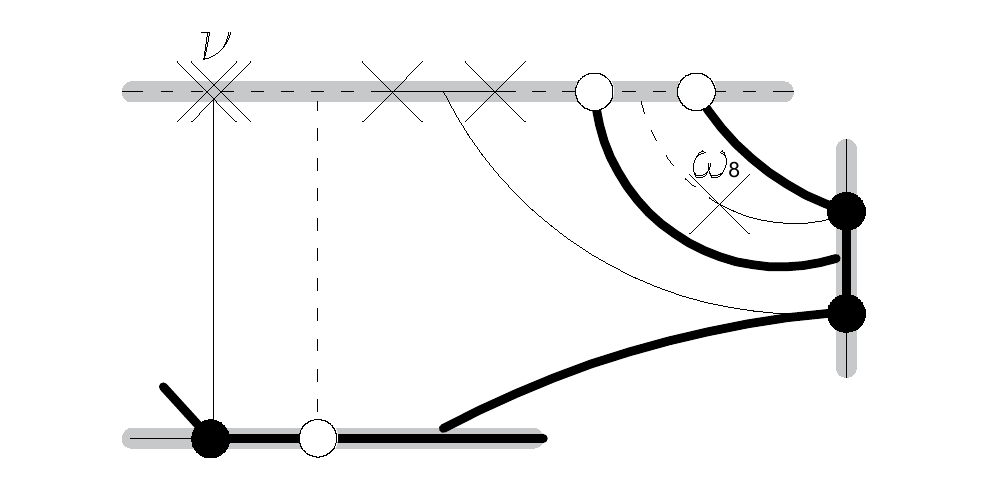}
\caption{\label{fig:toilesii87}}
\end{subfigure}
\end{center}
\caption{}
\end{figure}

If the vertex $w_8$ is monochrome, it determines a cut. 

If the vertex $w_8$ is an inner simple {\tv}, it is connected to $w_7$ up to a monochrome modification. In this setting, the vertices $w_5'$ and $w_2'$ are connected up to a monochrome modification. 
If $w_{10}$ is a monochrome vertex, an elementary move of type $\bullet$-in at $w_6$ followed by an elementary move of type $\bullet$-out at $w_{10}$ and an elementary move of type $\circ$-in at $w_7$ followed by an elementary move of type $\circ$-out allow us to create a zigzag, contradicting the maximality assumption (see Figure~\ref{fig:toilesii87}).
If $w_{10}$ is an inner simple {\tv}, an elementary move of type $\circ$-in at $w_7$ produces an inner white vertex $w'$ which is connected to $w_{10}$ up to a monochrome modification. Then, up to the creation of a bold bridge beside $w_2'$ with an inner bold edge adjacent to $w'$, an elementary move of type $\circ$-out allows us to create a zigzag, contradicting the maximality assumption (see Figure~\ref{fig:toilesii88}).
If $w_{10}$ is a real nodal {\tv}, we consider the vertex $w_9$.
If $w_9$ is a monochrome vertex, the edge $[w_{10}, w_2']$ corresponds to the configuration in the Case 2.1.
Otherwise, the vertex $w_9$ is a real white vertex, then, up to the creation of a dotted bridge beside $w_{10}$ with the inner dotted edge adjacent to $w_9$, an elementary move of type $\bullet$-in at $w_6$ followed by an elementary move of type $\circ$-in produces a {\gc} (see Figure~\ref{fig:toilesii89}).

\begin{figure}[h]
\begin{center}
\begin{subfigure}{0.45\textwidth}
\centering\includegraphics[width=2.5in]{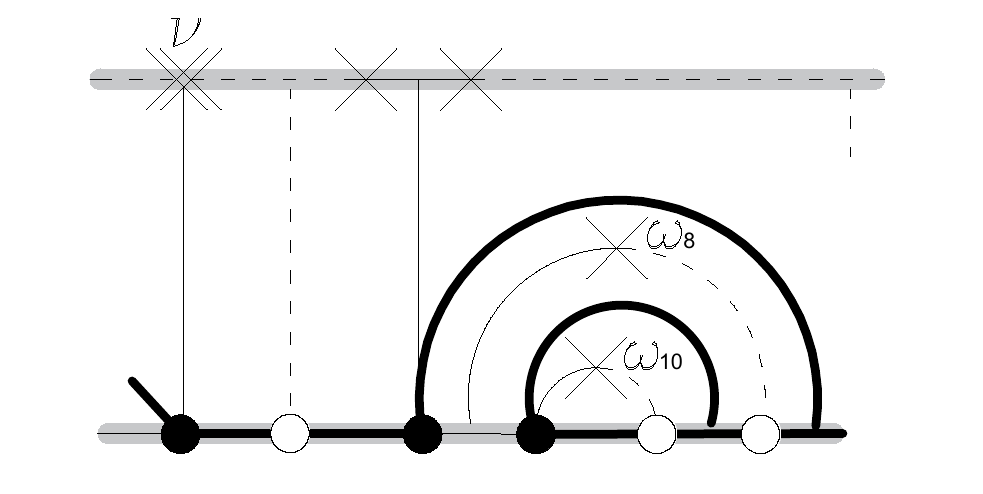}
\caption{\label{fig:toilesii88}}
\end{subfigure}\hspace{1cm} 
\begin{subfigure}{0.45\textwidth}
\centering\includegraphics[width=2.5in]{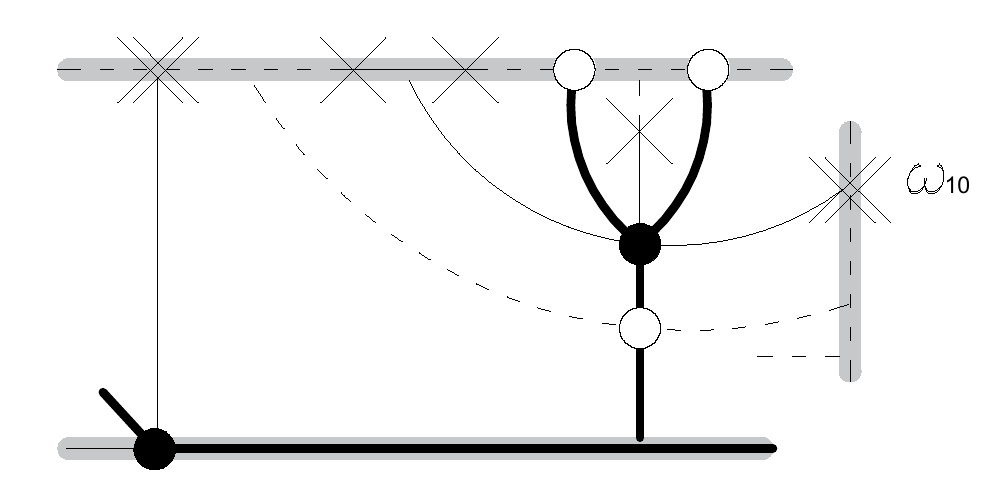}
\caption{\label{fig:toilesii89}}
\end{subfigure}
\end{center}
\caption{}
\end{figure}

Otherwise $w_{10}$ is an inner nodal {\tv}.
If $w_9$ is a real white vertex, up to a monochrome modification it is connected to $w_{10}$. Up to the creation of a bridge beside $w_5'$ with an inner dotted edge adjacent to $w_{10}$, an elementary move of type $\bullet$-in at $w_6$ followed by an elementary move of type $\circ$-in at the resulting bold monochrome vertex produces a {\gc} (see Figure~\ref{fig:toilesii90}).

If $w_9$ is a monochrome vertex, it is connected to a real white vertex $w_{11}$.
Then, the creation of bridges beside the vertices $w_5'$ and $w_{11}$ with the corresponding inner dotted edges adjacent to $w_{10}$ produces a {\gc} (see Figure~\ref{fig:toilesii91}).

\begin{figure}[h]
\begin{center}
\begin{subfigure}{0.4\textwidth}
\centering\includegraphics[width=2.5in]{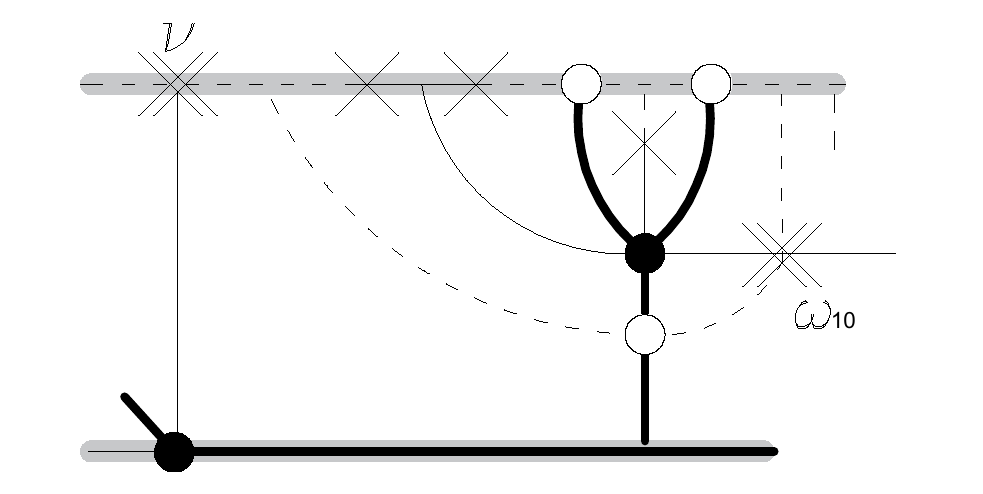}
\caption{\label{fig:toilesii90}}
\end{subfigure}\hspace{1cm} 
\begin{subfigure}{0.4\textwidth}
\centering\includegraphics[width=2.5in]{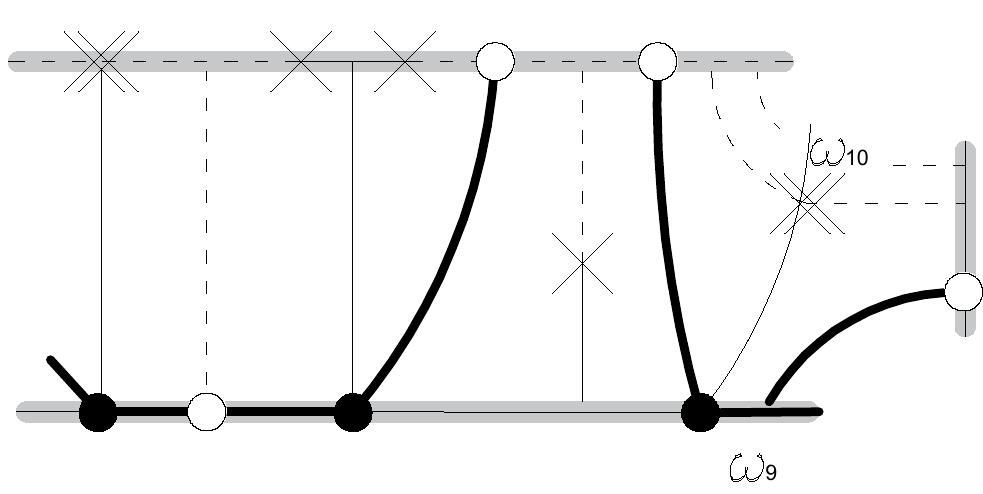}
\caption{\label{fig:toilesii91}}
\end{subfigure}
\end{center}
\caption{}
\end{figure}

If $w_8$ is an inner nodal {\tv}, up to a monochrome modification it is connected to $w_7$.
Let $w_{11}\neq w_7$ be the vertex connected to $w_8$ by an inner dotted edge.
If $w_{11}$ is a monochrome vertex, it determines a {\gc}.
If $w_{11}$ is a real white vertex, up to the creation of a bold bridge beside $w_{11}$ with the inner bold edge adjacent to $w_2'$ an elementary move of type $\bullet$-in at $w_6$ followed by an elementary move of type $\bullet$-out at the bridge and the destruction of a possible resulting bridge
produce a {\gc} (see Figure~\ref{fig:toilesii92}).

If $w_{11}$ is an inner white vertex, it is connected to $w_2'$ up to a monochrome modification. Let $w_{12}\neq w_8$ be the vertex connected to $w_{11}$ by an inner dotted edge.
If $w_{12}$ is a monochrome vertex, it determines a {\gc}.
If $w_{12}$ is an inner simple {\tv}, up to a monochrome modification it is connected to $w_2'$ and up to the creation of a bold bride beside $w_2'$ and an elementary move of type $\circ$-out at the vertex $w_{11}$ we can create a zigzag, contradicting the maximality assumption.
If $w_{12}$ is an inner nodal {\tv}, up to a monochrome modification is it connected to $w_2'$.

If $w_9$ is a monochrome vertex, it is connected to a real white vertex $w_{11}$.
Then, the creation of a bridge beside $w_{11}$ with an inner dotted edge adjacent to $w_{10}$ produces a {\gc} (see Figure~\ref{fig:toilesii93}).

\begin{figure}[h]
\begin{center}
\begin{subfigure}{0.3\textwidth}
\centering\includegraphics[width=2.5in]{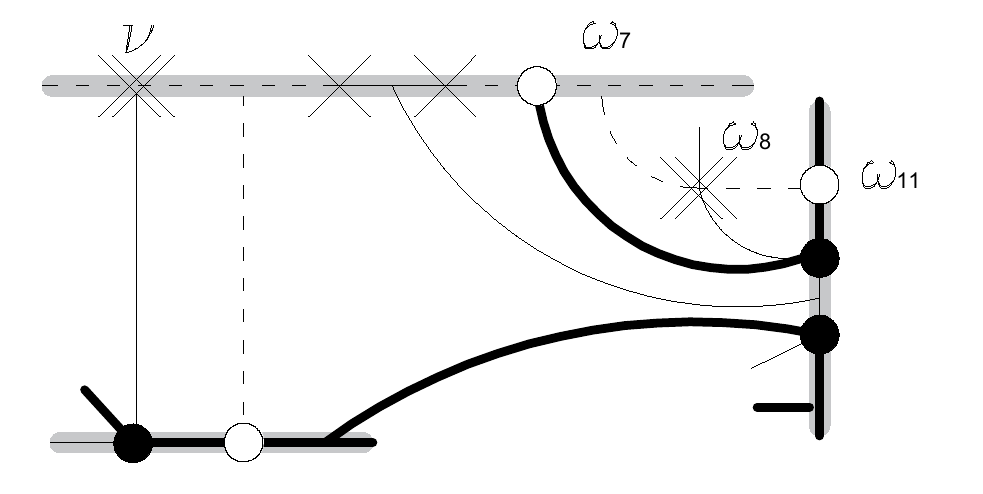}
\caption{\label{fig:toilesii92}}
\end{subfigure}\hspace{2mm} 
\begin{subfigure}{0.3\textwidth}
\centering\includegraphics[width=2.5in]{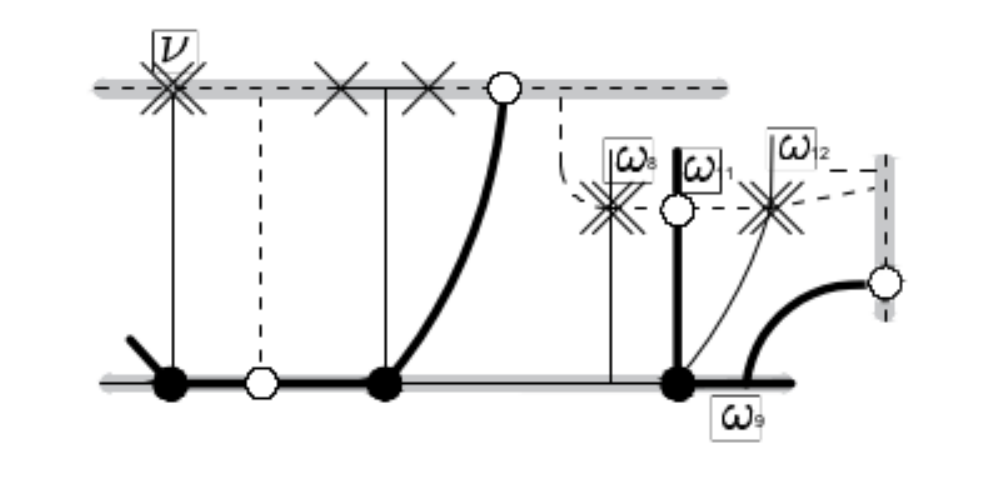}
\caption{\label{fig:toilesii93}}
\end{subfigure}\hspace{2mm}

\begin{subfigure}{0.3\textwidth}
\centering\includegraphics[width=2.5in]{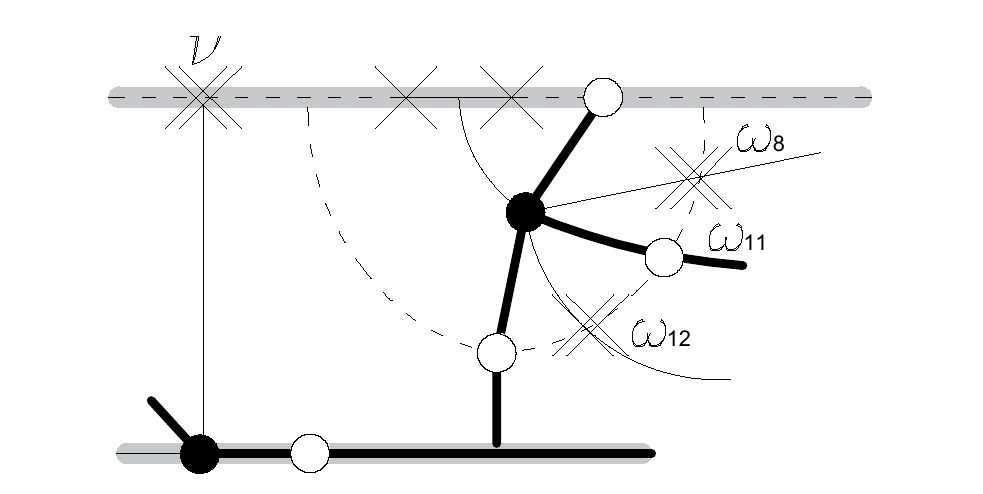}
\caption{\label{fig:toilesii94}}
\end{subfigure}
\end{center}
\caption{}
\end{figure}

If $w_9$ is a real white vertex, up to a monochrome modification it is connected to $w_{10}$. Then, an elementary move of type $\bullet$-in at $w_6$ followed by an elementary move of type $\circ$-in at the resulting bold monochrome vertex produces a {\gc} (see Figure~\ref{fig:toilesii94}).


If $w_6$ is a simple {\tv}, let $w_7\neq w_3'$ be the vertex connected to $w_5$ by a real dotted edge.
If $w_7$ is a monochrome vertex or a real nodal {\tv}, the creation of a bridge beside $w_6$ with the inner edge adjacent to $w_7$ produces a cut or an axe, respectively.
If $w_7$ is a simple {\tv}, let $w_8$ be the vertex connected to $w_6$ by a real dotted edge.
If $w_8$ is a monochrome vertex, the creation of a bridge beside $w_7$ with the inner edge adjacent to $w_8$ produces a cut.
If $w_8$ is a real white vertex, let $w_9$ be the vertex connected to $w_8$ by a real dotted edge.
We perform a monochrome modification connecting $w_2$ to $w_8$.
If $w_9$ is a monochrome vertex or a nodal {\tv}, then either the creation of a bridge beside $w_5$ with the inner edge adjacent to $w_9$ or the creation of an inner solid monochrome vertex with the edge $[w_2,w_4]$ and the inner edge adjacent to $w_9$ produces a {\gc}.
If $w_9$ is a simple {\tv}, the creation of a bridge beside it with the edge $[w_2,w_4]$ produces a {\gc}.


In the case when $w_3$ is a simple {\tv} followed by a black vertex $w_4$, let $w_5$ be the vertex connected to $w_4$ by an inner bold edge.
If $w_5$ is a monochrome vertex, an elementary move of type $\circ$-in at it allows us to consider $w_5$ as an inner white vertex.
If $w_5$ is a real white vertex, the creation of a bridge beside it with the edge $[w,w_1]$ produces a cut.
If $w_5$ is an inner white vertex, the creation of a bridge beside $w_3$ with an inner dotted edge adjacent to $w_3$ followed by an elementary move of type $\circ$-out brings us to a configuration already considered.

\begin{figure}[h]
\begin{center}
\begin{subfigure}{0.3\textwidth}
\centering\includegraphics[width=2.5in]{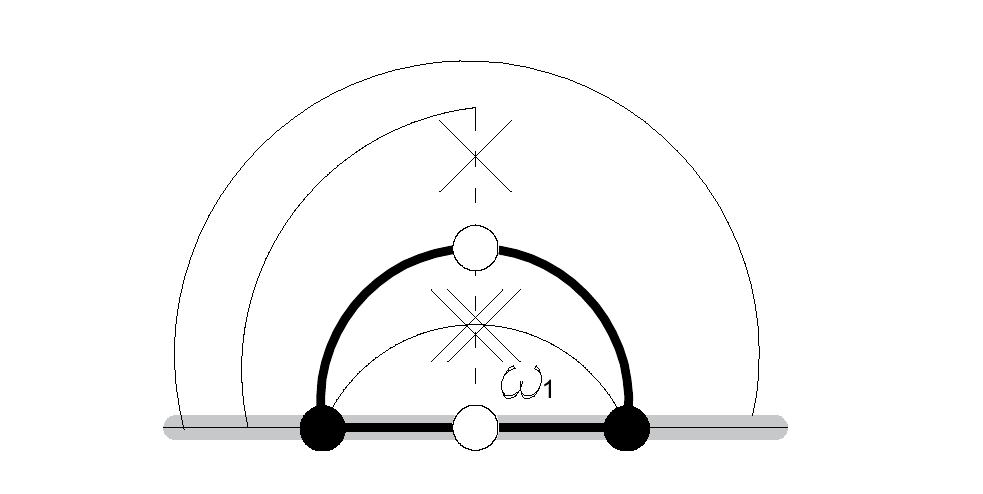}
\caption{\label{fig:toilesii96}}
\end{subfigure}\hspace{2mm} 
\begin{subfigure}{0.3\textwidth}
\centering\includegraphics[width=2.5in]{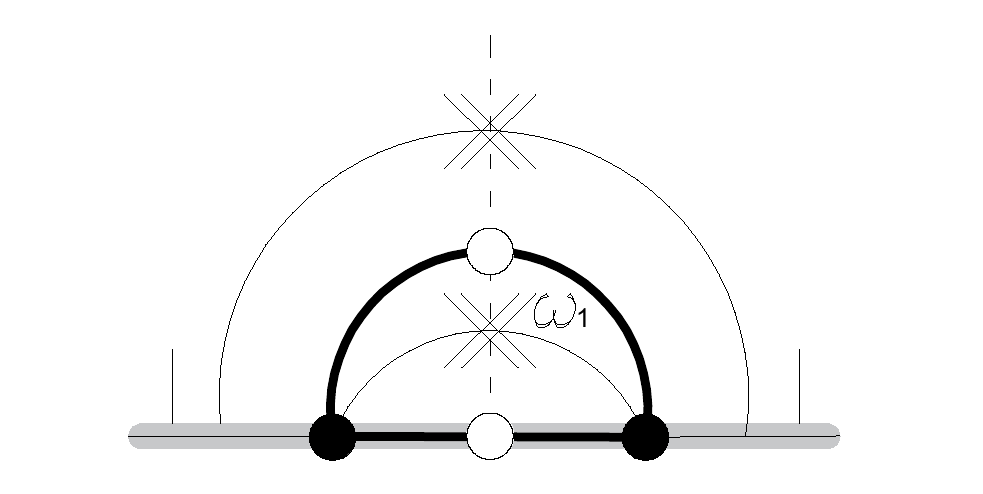}
\caption{\label{fig:toilesii97}}
\end{subfigure}
\end{center}
\caption{}
\end{figure}

If $w_1$ is an inner simple {\tv}, a monochrome modification connecting it to the vertex $u$ allows us to create a zigzag, contradicting the maximality assumption.

In the case when the vertex $w_1$ is an inner nodal {\tv}, it is connected to $w_2$ up to a monochrome modification.
Let $w_3\neq w$ be the vertex connected to $w_1$ by an inner dotted edge.
If $w_3$ is a monochrome vertex, an elementary move of type $\circ$-in allows us to consider it as an inner white vertex.
If $w_3$ is an inner white vertex, it is connected to $w_2$ up to a monochrome modification.
Let $w_4$ be the vertex connected to $w_2$ by a real solid edge.
If $w_4$ is a simple {\tv}, the creation of a dotted bridge beside it with an inner dotted edge adjacent to $w_3$ produces a {\gc}. 
If $w_4$ is a monochrome vertex, let $w_5$ be the vertex connected to it by an inner solid edge.
If $w_5$ is a real vertex, it determines a cut or an axe.
If $w_5$ is an inner {\tv}, it is connected to $w_3$ up to a monochrome modification.
Then, we perform monochrome modifications connecting $u$ to the vertices $w_1$ and $w_3$.
In this setting, the creation of a bridge beside $u$ with an inner solid edge adjacent to $w_5$ produces a {\gc} (see Figures \ref{fig:toilesii96} and \ref{fig:toilesii97} for when $w_5$ is simple or nodal, respectively).
Otherwise, the vertex $w_3$ is a real white vertex.
Let $w_4$ be the vertex connected to $w_2$ by an inner bold edge.
If $w_4$ is a real white vertex, the creation of bridge beside it with the edge $[w_1,w_3]$ and a bridge beside $v$ with the edge $[w,w_1]$ produces a {\gc}.
If $w_4$ is a monochrome vertex, an elementary move of type $\circ$-in allows us to consider it as an inner white vertex.
Lastly, if $w_4$ is an inner white vertex, a monochrome modification connecting it to $w_2$ brings us to a configuration already considered.

\vskip5pt
{\it Case 2.3:} the vertex $u$ is an inner black vertex.
If it is connected to a real monochrome vertex, an elementary move of type $\bullet$-out either produces an axe or bring us to a considered case.
Let $w_1$, $w_1'$ be the vertices connected to $u$ by an inner bold edge sharing a region with $v$.
If $w_1$ or $w_1'$ is an inner white vertex, the creation of a bridge beside $v$ with the inner dotted edge adjacent to it followed by an elementary move allows us to consider it as a real white vertex.
Let $w_2$, $w_2'\neq v$ be the {\tvs} connected to $u$ by an inner solid edge.
Let $w\neq w_1,w_1'$ be the white vertex connected to $u$ (see Figure~\ref{fig:toilesii98}).

\begin{figure}[h]
\begin{center}
\begin{subfigure}{0.3\textwidth}
\centering\includegraphics[width=2.5in]{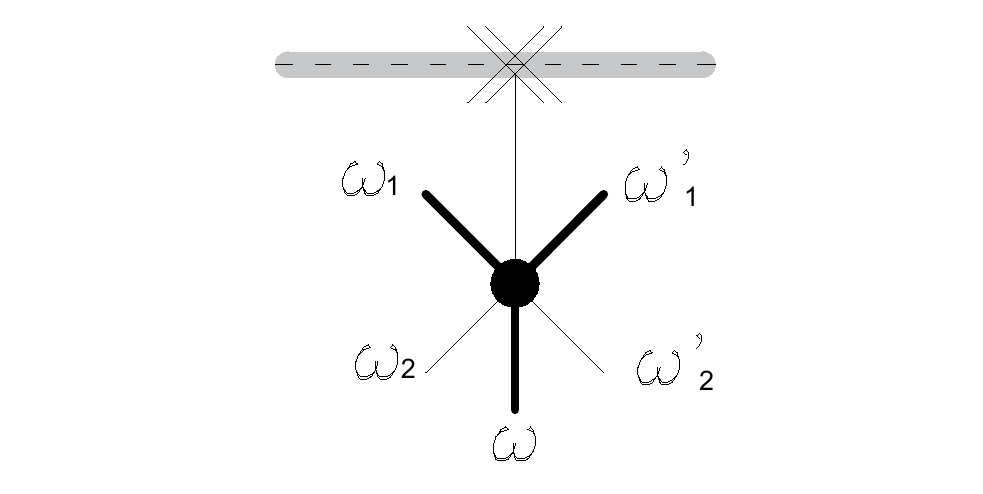}
\caption{\label{fig:toilesii98}}
\end{subfigure}\hspace{2mm} 
\begin{subfigure}{0.3\textwidth}
\centering\includegraphics[width=2.5in]{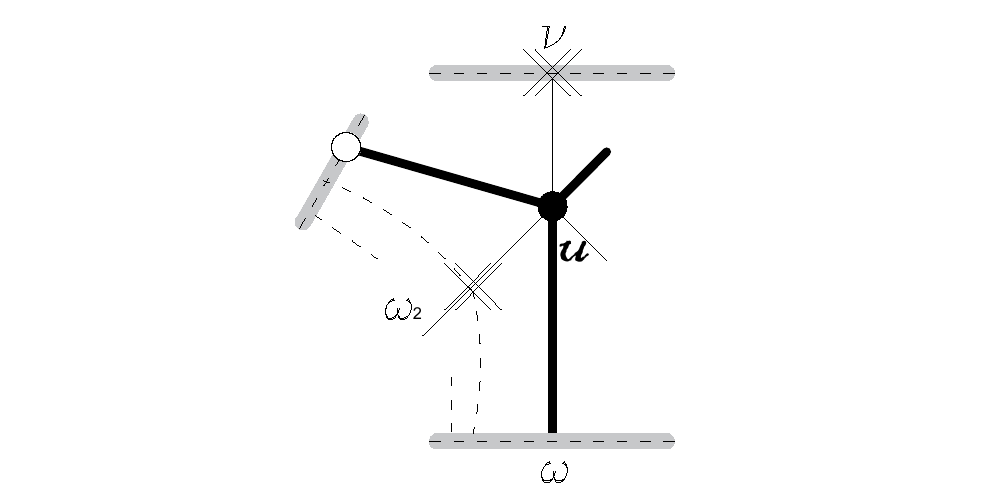}
\caption{\label{fig:toilesii99}}
\end{subfigure}\hspace{2mm}
\begin{subfigure}{0.3\textwidth}
\centering\includegraphics[width=2.5in]{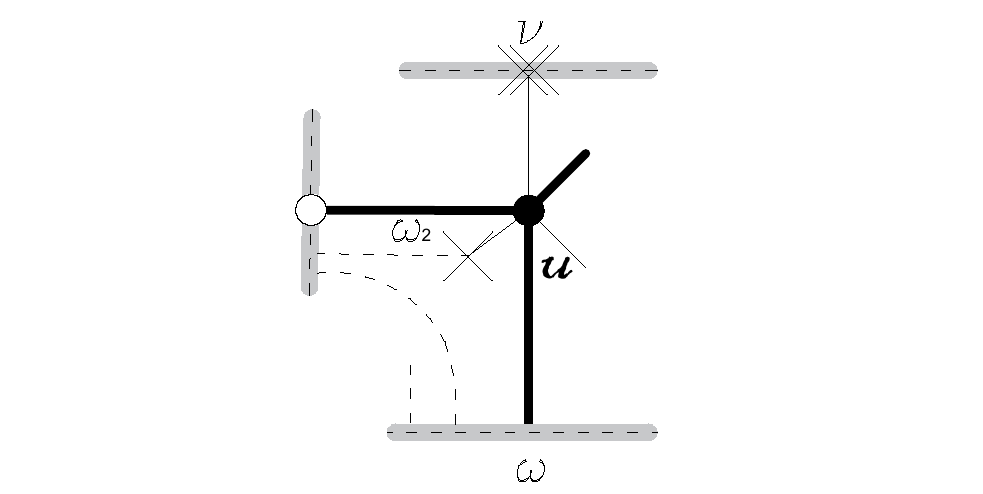}
\caption{\label{fig:toilesii100}}
\end{subfigure}
\end{center}
\caption{}
\end{figure}

In the case when every vertex connected to $u$ is a real vertex, since we assumed the dessin do not have cuts and do have degree greater than $3$, it is not possible that every region containing $u$ is triangular.
Let $R$ be one region containing $u$ which is not triangular.
Put $v'$ and $w'$ the {\tv} and white vertex connected to $u$ in $R$, respectively.
Let $w_3$ be the neighboring vertex to $w'$ in $R$.
If $w_3$ is a monochrome vertex, the creation of a bridge beside $v'$ with the inner edge adjacent to $w_3$ produces a cut.
If $w_3$ is a nodal {\tv}, the creation of an inner solid monochrome vertex with the inner edge adjacent to it and the edge $[v',u]$ produces a {\gc}.
If $w_3$ is a simple {\tv}, the creation of a bridge beside it with the edge $[v',u]$ produces an axe.
Due to these considerations, we can assume the regions determined by $v$, $u$, $w_1$ and $v$, $u$, $w_1'$ are triangular.

In the case when the vertex $w$ is a real white vertex and $w_2$ is an inner {\tv}, we have two different cases.
If $w_2$ is an inner nodal {\tv}, the creation of bridges beside $w_1$ and $w$ with the inner dotted edges adjacent to $w_2$ produces a {\gc} (see Figure~\ref{fig:toilesii99}).
If $w_2$ is a simple {\tv}, the creation of bridges beside $w_1$ and $w$ with the inner dotted edge adjacent to $w_2$ produces a cut (see Figure~\ref{fig:toilesii100}), unless this is an inadmissible elementary move, i.e., unless the vertices $w_1$, $w$ and $w_2$ are connected to the same monochrome vertex.

In this setting, if $w_2'$ it a real nodal {\tv}, one of the regions determined by $w_1'$, $u$, $w_2'$ or $w$, $u$, $w_2'$ is not triangular and the aforementioned considerations apply.
Otherwise, the vertex $w_2'$ is an inner {\tv}.
If $w_2'$ is a nodal {\tv}, the creation of bridges beside $w_1'$ and $w$ with the inner dotted edges adjacent to $w_2$ produces a {\gc}.
If $w_2'$ is a simple {\tv}, the creation of bridges beside $w_1'$ and $w$ with the inner dotted edge adjacent to $w_2'$ produces a cut, and since the degree is greater than $3$, this is an admissible elementary move.

If $w$ is an inner white vertex and at least one the vertices $w_2$ and $w_2'$ is a real nodal {\tv}, the creation of a bridge beside the real {\tv} with an inner dotted edge adjacent to $w$ and an elementary move of type $\circ$-out bring us to the configuration where $w$ is a real vertex.

If $w$ is an inner white vertex and at least one the vertices $w_2$ and $w_2'$ is an inner simple {\tv}, up to monochrome modification we can assume the vertex $w$ is connected to the {\tv} in which case, the creation of a bridge beside $w_1$ or $w_1'$ with the inner dotted edge adjacent to the simple {\tv} and an elementary move of type $\circ$-out bring us to the configuration where $w$ is a real vertex.

Finally, if the vertex $w$ is an inner white vertex and the vertices $w_2$ and $w_2'$ are inner nodal {\tvs}, which up to monochrome modifications are connected to $w$, then the creation of dotted bridges beside $w_1$ and $w_1'$ produces a cut (see Figure~\ref{fig:toilesii101}).

\begin{figure}[h]
\begin{center}
\begin{subfigure}{0.3\textwidth}
\centering\includegraphics[width=2.5in]{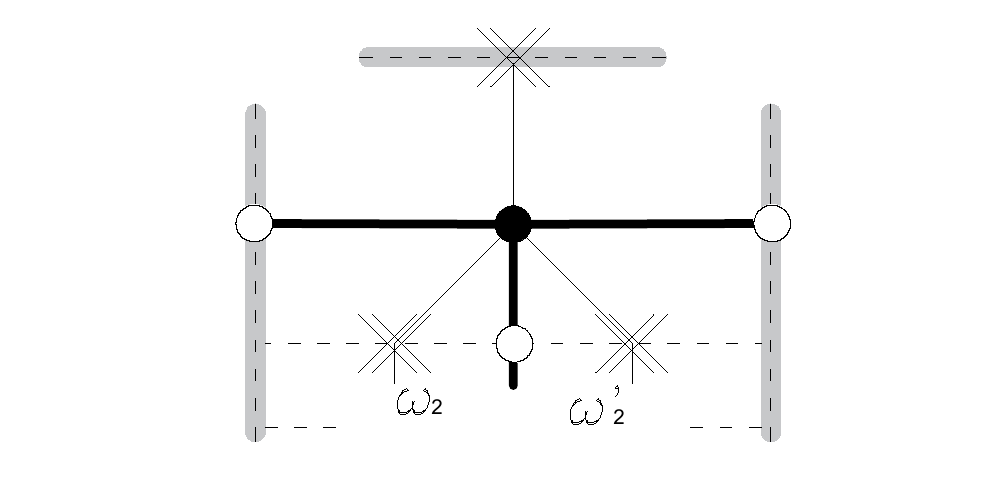}
\caption{\label{fig:toilesii101}}
\end{subfigure}\hspace{2mm} 
\begin{subfigure}{0.3\textwidth}
\centering\includegraphics[width=2.5in]{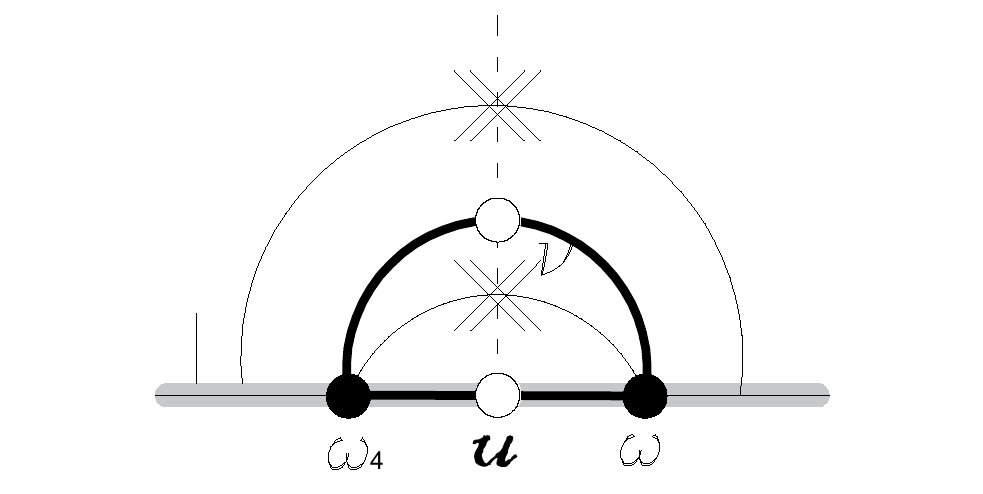}
\caption{\label{fig:toilesii102}}
\end{subfigure}
\end{center}
\caption{}
\end{figure}

\vskip5pt
{\it Case 3:} let $v$ be an inner nodal {\tv}. Due to previous considerations we can assume there are no real nodal {\tvs}. If $v$ is connected to two different dotted monochrome vertices or solid vertices, it defines a {\gc}.
Since $\deep(D)\leq1$, the vertex $v$ is connected to a real vertex $u$. Let us assume that $u$ is a white vertex. If $v$ is connected to inner black vertices, up to the creation of bold bridges beside $u$ and elementary moves of type $\bullet$-out we can assume $v$ is not connected to inner black vertices. Let $w$ be a real vertex connected to $u$.

\vskip5pt
{\it Case 3.1:} the vertex $w$ is a black vertex. Up to a monochrome modification it is connected to the vertex $v$.
Let $w_1$ be the vertex connected to $w$ by an inner bold edge.
If $w_1$ is a monochrome vertex, let $w_2$ be the vertex connected to $w$ by a real solid edge.
If $w_2$ is a monochrome vertex, let $w_3$ be the vertex connected to $w_2$ by an inner solid edge.
If $w_3$ is a monochrome vertex, it determines a cut.
If $w_3$ is an inner simple {\tv}, up to a monochrome modification it is connected to a real white vertex neighboring $w_1$. Then, an elementary move of type $\bullet$-in at $w_2$ followed by an elementary move of type $\bullet$-out at $w_1$ allows us to construct a zigzag, contradicting the maximality assumption.
 
If $w_3$ is an inner nodal {\tv}, let $w_4\neq w$ be the vertex connected to the vertex $u$ by a real bold edge.
If $w_4$ is a black vertex, up to a monochrome modification it is connected to $v$. We perform an elementary move of type $\circ$-in at $w_1$ producing an inner white vertex $w'$ which up to monochrome modifications is connected to the vertices $v$, $w_3$ and $w_4$.
Then, the creation of a bridge beside $w_4$ with an inner solid edge adjacent to $w_3$ produces a {\gc} (see Figure~\ref{fig:toilesii102}).

If $w_4$ is a monochrome vertex, it is connected to a real black vertex $w_5$ by an inner bold edge and to a white vertex $u'\neq u$ by a real bold edge.
If $w_5$ is connected to a solid monochrome vertex, an elementary move of type $\bullet$-in at it followed by an elementary move of type $\bullet$-out at $w_4$ bring us to the consideration where $w_4$ was a black vertex.
Otherwise, the vertex $w_5$ is connected to a simple {\tv} $w_6$. We perform an elementary move of type $\circ$-in at $w_1$.
If $w_6$ shares a region with the vertex $v$, the creation of a bridge beside it with the edge $[v,u]$ connects $v$ to a monochrome vertex by a chain of inner dotted edges (see Figure~\ref{fig:toilesii103}).
If $w_6$ does not share a region with the vertex $v$, the creation of a bridge beside it with the inner dotted edge adjacent to $u'$, and elementary moves of type $\circ$-in at $w_4$ and $w_1$ connects $v$ to a monochrome vertex by a chain of inner dotted edges (see Figure~\ref{fig:toilesii104}).

\begin{figure}[h]
\begin{center}
\begin{subfigure}{0.3\textwidth}
\centering\includegraphics[width=2.5in]{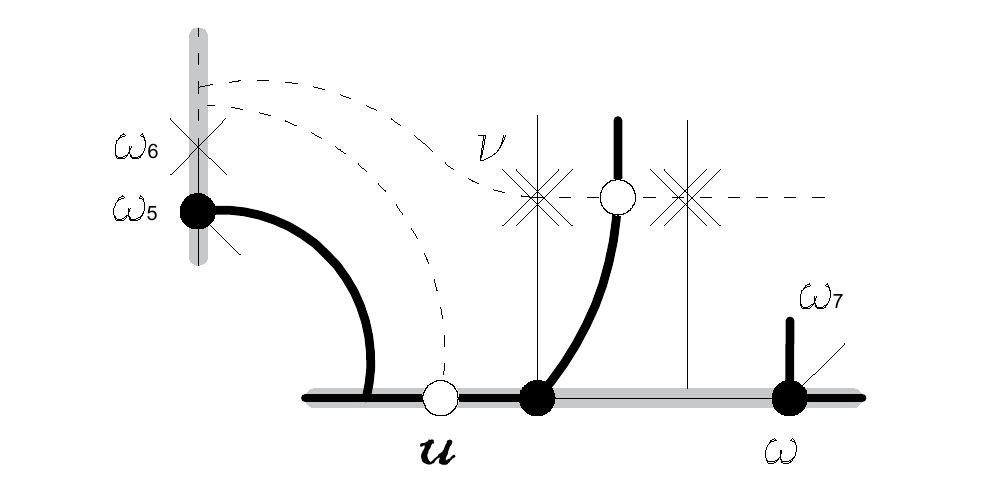}
\caption{\label{fig:toilesii103}}
\end{subfigure}\hspace{1cm} 
\begin{subfigure}{0.3\textwidth}
\centering\includegraphics[width=2.5in]{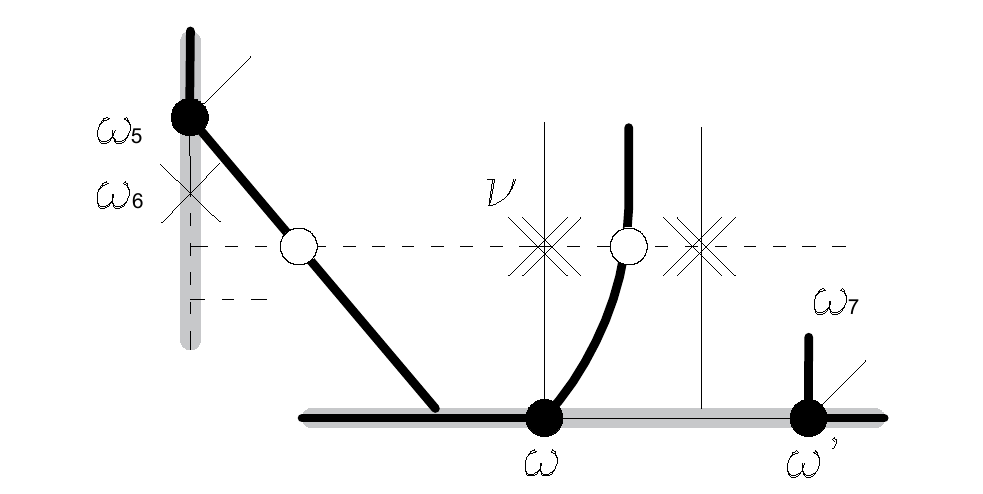}
\caption{\label{fig:toilesii104}}
\end{subfigure}
\end{center}
\caption{}
\end{figure}

Let $w'\neq w$ be the vertex connected to $w_2$ by a real solid edge and let $w_7$ be the vertex connected to $w'$ by an inner bold edge.
If $w_7$ is a real white vertex, the creation of a bridge beside it with an inner dotted edge adjacent to $w_3$ produces a {\gc}.
If $w_7$ is a monochrome vertex, an elementary move of type $\circ$-in allows us to consider it as an inner white vertex.
If $w_7$ is an inner white vertex, up to a monochrome modification it is connected to $w_3$.
Let $w_8$ be the vertex connected to $w'$ by an inner solid edge. 
If $w_8$ is an inner simple {\tv}, it is connected to $w_7$ up to a monochrome modification, then the creation of a bridge beside $u'$ with an inner bold edge adjacent to $w_7$ followed by an elementary move of type $\circ$-out allows us to create a zigzag, contradicting the maximality assumption. 
If $w_8$ is a monochrome vertex, it is connected to a simple {\tv} sharing a region with $w_7$. Then, the creation of a bridge beside this simple {\tv} with a dotted edge adjacent to $w_7$ produces a {\gc} (see Figure~\ref{fig:toilesii105}).
If $w_8$ is a real nodal {\tv}, the creation of a bridge beside it with a dotted edge adjacent to $w_7$ produces a {\gc} (see~Figure~\ref{fig:toilesii106}).
If $w_8$ is an inner nodal {\tv}, it is connected to $w_7$ up to a monochrome modification.
Let $w_9$ be the vertex connected to $w'$ by a real bold edge.
If $w_9$ is a monochrome vertex, it is connected to a real white vertex determining a dotted segment, where the creation of a bridge with an inner dotted edge adjacent to $w_8$ produces a {\gc} (see Figure~\ref{fig:toilesii107}).
If $w_9$ is a white vertex, it is connected to $w_8$ up to a monochrome modification.
We iterate the considerations starting with $w_8$ as the vertex $v$, producing a {\gc}.

\begin{figure}[h]
\begin{center}
\begin{subfigure}{0.3\textwidth}
\centering\includegraphics[width=2.5in]{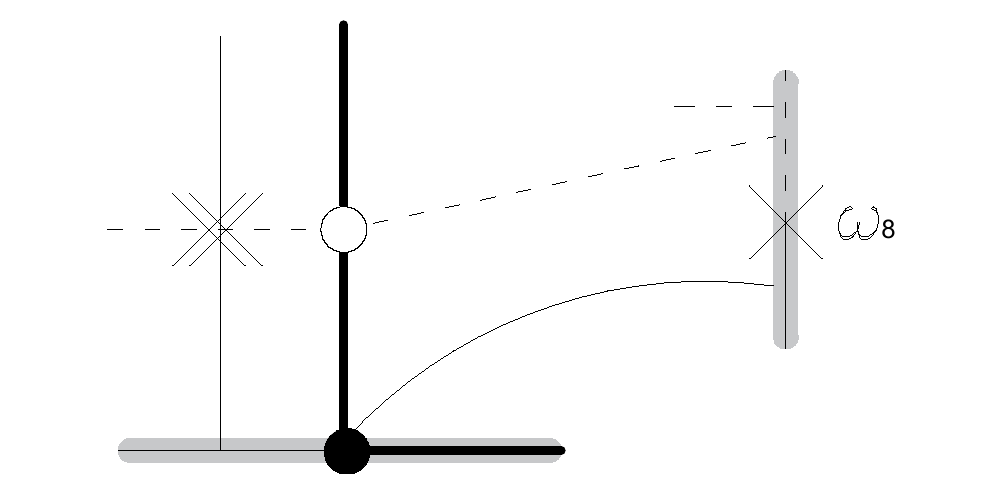}
\caption{\label{fig:toilesii105}}
\end{subfigure}\hspace{2mm} 
\begin{subfigure}{0.3\textwidth}
\centering\includegraphics[width=2.5in]{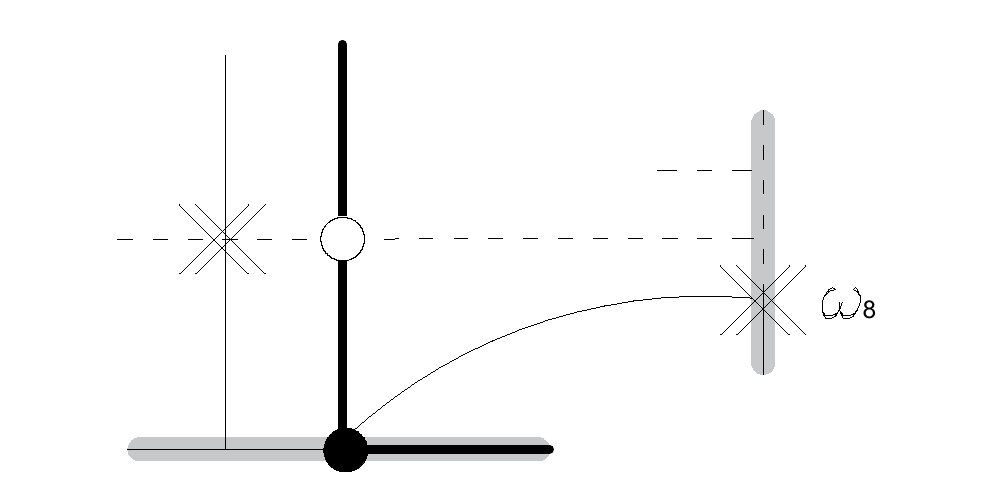}
\caption{\label{fig:toilesii106}}
\end{subfigure}\hspace{2mm}
\begin{subfigure}{0.3\textwidth}
\centering\includegraphics[width=2.5in]{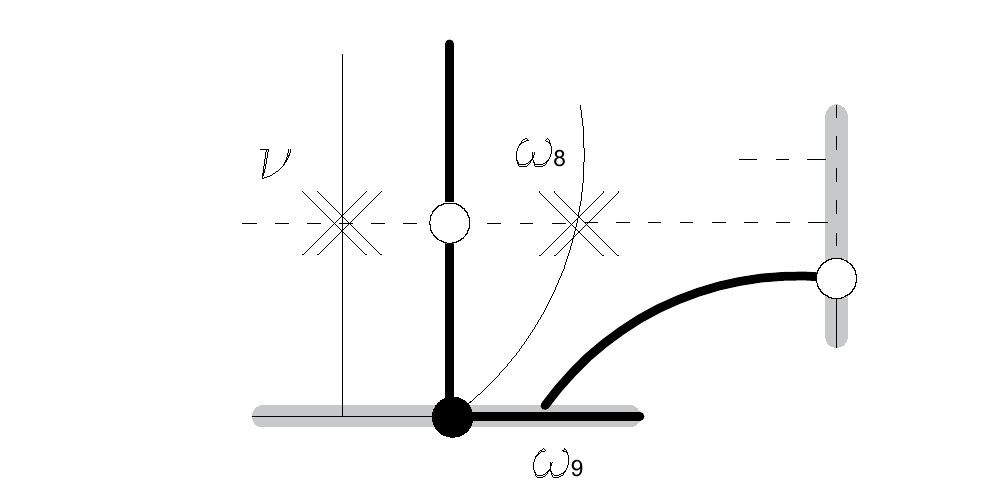}
\caption{\label{fig:toilesii107}}
\end{subfigure}
\end{center}
\caption{}
\end{figure}

If $w_2$ is a simple {\tv}, up to the creation of a bridge beside it with the inner dotted edge of a real white vertex connected to $w_1$, an elementary move of type $\circ$-in at $w_1$ followed by an elementary move of type $\circ$-out beside $w_2$ allow as to consider $w_1$ as a real white vertex.
If $w_1$ is a real white vertex, up to a monochrome modification or the creation of a dotted bridge, the vertex $v$ is connected to a monochrome vertex beside $w_1$. Then, we consider the other real vertex connected to $u$.
If $w_1$ is an inner white vertex, it is connected to $v$ up to a monochrome modification. Let $w_2$ the vertex connected to $w$ by a real solid edge.
If $w_2$ is a simple {\tv}, the creation of a bridge beside it with an inner dotted edge adjacent to $w_1$ followed by an elementary move of type $\circ$-out at the bridge bring us to the configuration when $w_1$ is a real white vertex.
If $w_2$ is a monochrome vertex, let $w'\neq w$ the black vertex connected to $w_2$ by a solid real edge and let $w_3$ be the vertex connected to $w_2$ by an inner solid edge.
If $w_3$ is a monochrome vertex, it determines a cut.
If $w_3$ is an inner nodal {\tv}, up to a monochrome modification it is connected to $w_1$, corresponding to a previous configuration already considered.

If $w_3$ is an inner simple {\tv}, up to monochrome modifications it is connected to~$w_1$ and the vertex $u'$ is connected to the vertices $w_1$ and $v$.
Let $S$ be the bold segment containing $w'$.
If the segment $S$ contains no white vertices, an elementary move of type $\bullet$-in followed by the creation of a bold bridge beside $u$ and an elementary move of type $\bullet$-out produces a {\gc} (see Figure~\ref{fig:toilesii108}).
If the segment $S$ containing exactly one white vertex $w_4$, which is connected to a black vertex $u''\neq u'$. Up to monochrome modification the vertex $w_4$ is connected to the vertex $v$ and the vertex $u''$ is connected to $v$ and $w_1$ corresponding to a configuration already considered.
Up to elementary moves of type $\circ$-in at monochrome vertex in the bold segment $S$ containing $u'$ followed by elementary moves of type $\circ$-out at the bold segment containing $u$ we reduce to the case when the segment $S$ contains one or none white vertices.

\begin{figure}[h]
\begin{center}
\begin{subfigure}{0.3\textwidth}
\centering\includegraphics[width=2.5in]{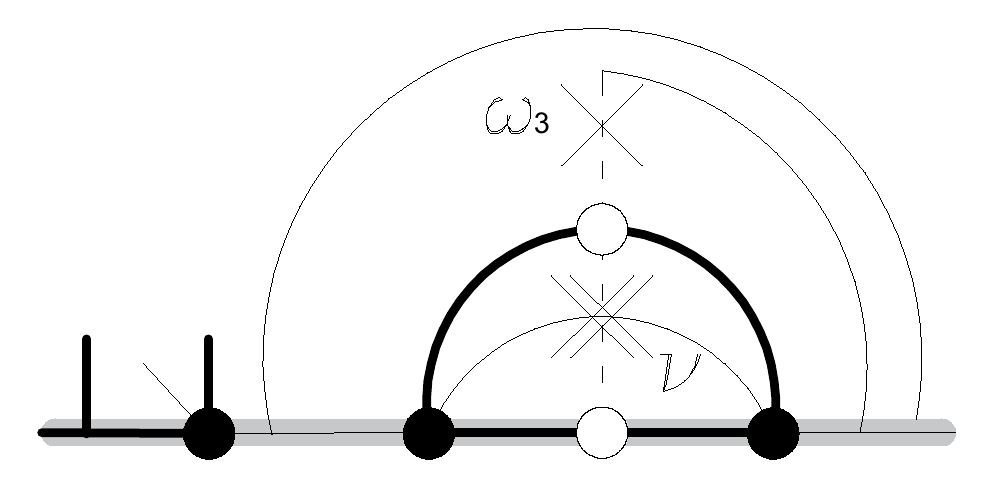}
\caption{\label{fig:toilesii108}}
\end{subfigure}\hspace{2mm} 
\begin{subfigure}{0.3\textwidth}
\centering\includegraphics[width=2.5in]{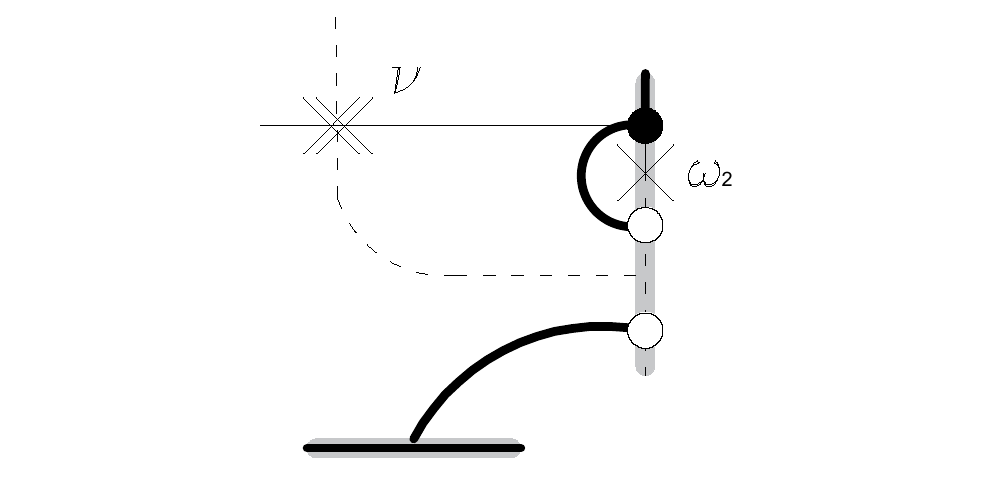}
\caption{\label{fig:toilesii109}}
\end{subfigure}
\end{center}
\caption{}
\end{figure}

\vskip5pt
{\it Case 3.2:} the vertex $w$ is a monochrome vertex. Let $w_1$ be the real black vertex connected to $w$ by an inner bold edge. 
Let $f$ be the solid real edge adjacent to $w_1$ and let $w_2$ be the vertex connected to $w_1$ by $f$.
If $f$ does not share a region with the vertex $v$, up to a monochrome modification $v$ and $w_1$ are connected.
If $w_2$ is a monochrome vertex, an elementary move of type $\bullet$-in at $w_2$ followed by an elementary move of type $\bullet$-out at $w$ brings us to the previous case.
If $w_2$ is a simple {\tv}, up to the creation of a bridge beside it with the inner dotted edge adjacent to a neighboring white vertex of $w$, an elementary move of type $\circ$-in at $w$ followed by an elementary move of type $\circ$-out beside $w_2$ connects the vertex $v$ to a monochrome vertex (see Figure~\ref{fig:toilesii108}). Then, we consider the other vertex connected to $v$ by an inner dotted edge.

Otherwise, the edge $f$ shares a region with the vertex $v$.
If $w_2$ is a monochrome vertex, up to a monochrome modification it is connected to $v$. Then, an elementary move of type $\bullet$-in at $w_2$ followed by an elementary move of type $\bullet$-out at $w$ brings us to Case~3.1.
Lastly, if $w_2$ is a simple {\tv}, the creation of a bridge beside it with the edge $[v,u]$ connects $v$ with a monochrome vertex (see Figure~\ref{fig:toilesii109}). Then, we consider the other real vertex connected to $u$.

\vskip5pt
{\it Case 3.3:} the vertex $u$ is a real black vertex. 
Let $f$ be the inner bold edge adjacent to $u$.
Let $w$ be an inner white vertex connected to $v$.
If the edge $f$ does not share a region with the vertex $w$, up to the creation of a bridge beside $u$ with an inner bold edge adjacent to $w$, an elementary move of type $\circ$-out bring us to the previous case.
If the edge $f$ shares a region with the vertex $w$, up to a monochrome modification it does connects the vertex $w$ to the vertex $u$.
Let $w_1$ be the vertex connected to $u$ by a real solid edge.
If $w_1$ is a monochrome vertex, it is connected to a black vertex $u'\neq u$. Let $w_2$ be the vertex connected to $w_1$ by an inner solid edge.
If $w_2$ is a monochrome vertex, it defines a solid cut.
If $w_2$ is an inner simple {\tv}, up to monochrome modifications the vertex $w_2$ is connected to $w$ and the vertex $u'$ is connected to the vertices $w$ and $v$, bringing us to a configuration already considered.
If $w_2$ is an inner nodal {\tv}, up to a monochrome modification it is connected to the vertex $w$ corresponding to a configuration considered in Case~3.1.
If $w_1$ is a simple {\tv}, up to the creation of a bridge beside $w_1$ with an inner dotted edge adjacent to $w$, an elementary move of type $\circ$-out connects $v$ to a monochrome vertex.

\end{proof}


\section{Curves in $\RPP$ and Hirzebruch
surfaces}
\label{ch:hir}

\subsection{Hirzebruch surfaces}

From now on we consider ruled surfaces with the Riemann sphere $B\cong\CP$ as base curve. The Hirzebruch surfaces $\Sigma_n$ are toric surfaces geometrically ruled over the Riemann sphere, determined up to isomorphism of complex surfaces by a parameter $n\in\N$. They are minimal except for $\Sigma_1$, which is isomorphic to $\CPP$ blown up in a point. The surface $\Sigma_n$ is defined by the local charts $U_0\times\CP$ and $U_1\times\CP$ where $U_0=\{[z_0:z_1]\in\CP\mid z_0\neq0\}$, $U_1=\{[z_0:z_1]\in\CP\mid z_1\neq0\}$ glued along $\C^*$ via the map 
\[
\displaystyle \begin{array}{ccc} \C^*\times\CP & \lra & \C^*\times\CP \\ {(z,w)} & \lmt & \left(\frac{1}{z}, \frac{w}{z^n}\right) \end{array}.
\]

The exceptional section is the section at infinity $E$ such that $S_{\infty}^2=-n\,(n\geq0)$. The second homology group $H_2 (\Sigma_n,\ZZ)$ of the Hirzebruch surface $Sigma_n$ is generated by the homology classes $[Z]$ and $[f]$ of the null section $Z$ and of a fiber $f$, respectively. The intersection form is determined by the Gram matrix $$\left( \begin{array}{cc} n & 1 \\   1 &  0 \end{array}  \right)$$ with respect to the base $\{[Z],[f]\}$. The homology class of the exceptional section is given by $[E]=[Z]-n[f].$
Performing a positive Nagata transformation on $\Sigma_n$ results in a geometrically ruled surface isomorphic to $\Sigma_{n+1}$ (since the exceptional divisor decreases its self-intersection by one). Likewise, a negative Nagata transformation on $\Sigma_n\,(n>0)$ results in a geometrically ruled surface isomorphic to $\Sigma_{n-1}$.

Setting $(\C^*)^2\subset\Sigma$, the trigonal curve $C$ can be described by a polynomial in two variables $f(z,w)=q_0(z)w^3+q_1(z)w^2+q_2(z)w+q_3(z)$ in which $q_0$ determines the intersection with the exceptional fiber. If the trigonal curve $C$ is proper, then $q_0$ must be constant. We can suppose $q_0$ equal to $1$. Up to affine transformations of $\C$, we can set the sum of the roots of $f(z,\cdot)$ equal to $0$, resulting in the Weiertra{\ss} equation $$w^3+q_2(z)w+q_3(z).$$

Since $C$ is a trigonal curve, then $[C]=3[Z]+0[f]$. Therefore, the intersection product $[C]\cdot[Z]=3n$ equals the degree of $q_3(z)$. Since this explicit description must be invariant by the coordinate change ${(z,w)} \lmt (\frac{1}{z}, \frac{w}{z^n})$, the degree of $q_2$ must be $2n$. Hence, the $j$-invariant $$\displaystyle j=\frac{-4q_2^3}{\Delta}, \;\Delta=-4q_2^3-27q_3^2$$ is a rational function of degree $6n$ if the curve is generic (i.e., the polynomials $q_2$ and $q_3$ have no zeros in common).

\subsubsection{Relation with plane curves}

Let $A\subset\CPP$ be a reduced algebraic curve with a distinguished point of multiplicity $\deg(A)-3$ such that $A$ does not have linear components passing through $P$. The blow-up of~$\CPP$ at~$P$ is isomorphic to $\Sigma_1$. The strict transform of $A$ is a trigonal curve $C_{A}:=\widetilde{A}\subset\Sigma_1$, called the \emph{trigonal model} of the curve $A$. A \emph{minimal proper model} of $A$ consists of a proper model of $C_{A}$ and markings corresponding to the images of the improper fibers of $C_{A}$ by the Nagata transformations.

\subsubsection{Real algebraic plane curves of degree 3}\label{sec:cubics}

\begin{figure}
\centering
\begin{tabular}{cc}
\includegraphics[width=2in]{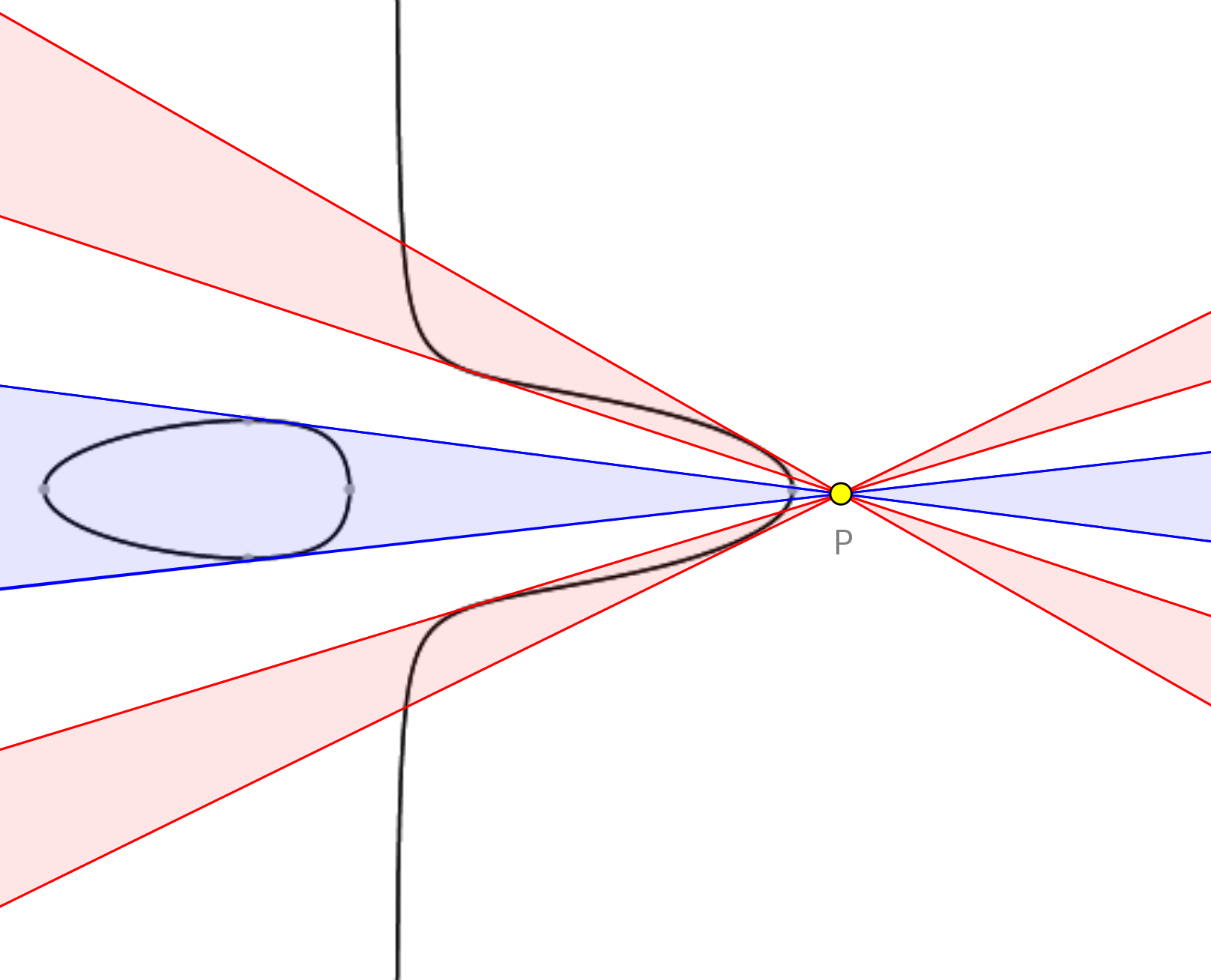}&
\includegraphics[width=3in]{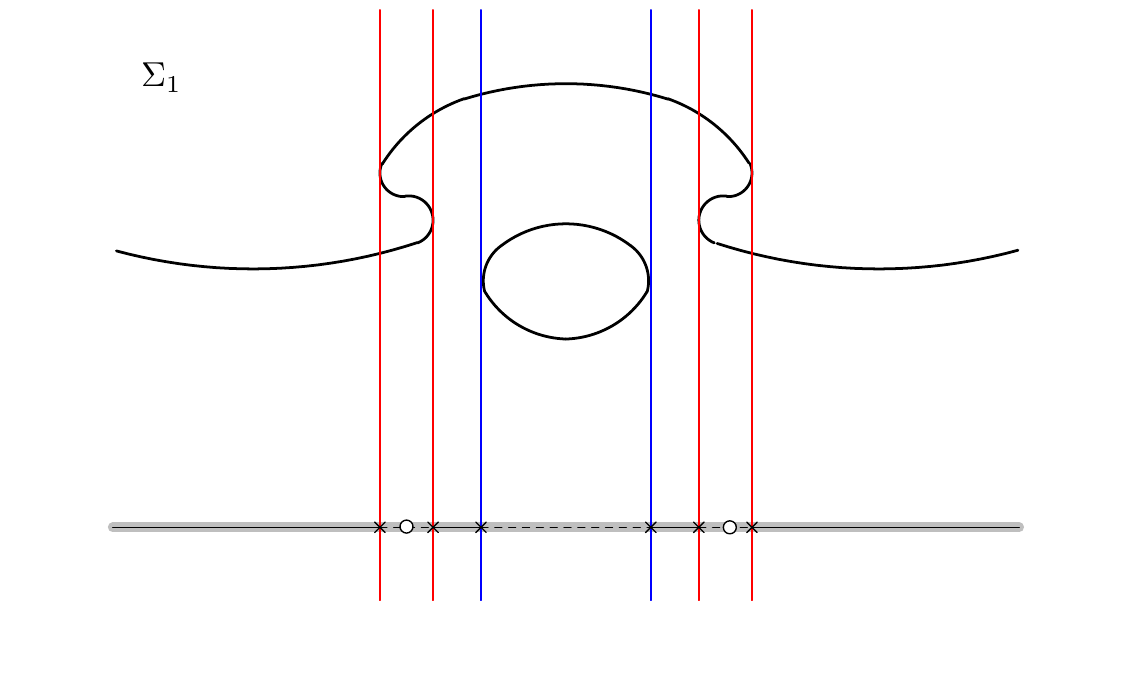}
\end{tabular}
\caption{A real plane cubic with two zigzags and one oval, and its pencil of lines seen in $\Sigma_1$.\label{fig:planeCubic}}
\end{figure}

Let $A\subset\CPP$ be a real smooth cubic. Let $p\in\CPP$ a real point which does not belong to~$A$ and let $B\cong\CP\subset\CPP$ be a real line which does not pass through $p$. The pencil of lines passing through $p$ can be parametrized by $B$, mapping every line $L\ni p$ to $L\cap B$. The blow-up of $\CPP$ at~$p$ is isomorphic to a real geometrically ruled surface over $B$. The strict transform of~$A$ is a real proper trigonal curve $C\subset\Sigma_1$ (since it is proper, it is already a minimal proper model for $A$). Since the real structures are naturally compatible, we associate to $C$ its real dessin $\Dssn(C)_{c}$ on the quotient of $B\cong\CP$ by the complex conjugation. 
Up to elementary equivalence, all the possible dessins are shown on Figure~\ref{fig:cubiques}. They are named by either their type (cf. Definition~\ref{df:type}) and the number of zigzags they possess (in the non-hyperbolic case, cf. Definition~\ref{df:ovzz}) or by $H$ in the case of the \emph{hyperbolic cubic}. 
Up to weak equivalence there are only three classes of cubics, namely the ones of type $\mathrm{I}$, type $\mathrm{II}$ and the hyperbolic one, corresponding to the rigid isotopy classification of couples $(A,p)$ of real cubic curves and one additional point of the plane outside $A$.

\begin{figure} 
\centering
\includegraphics[width=5in]{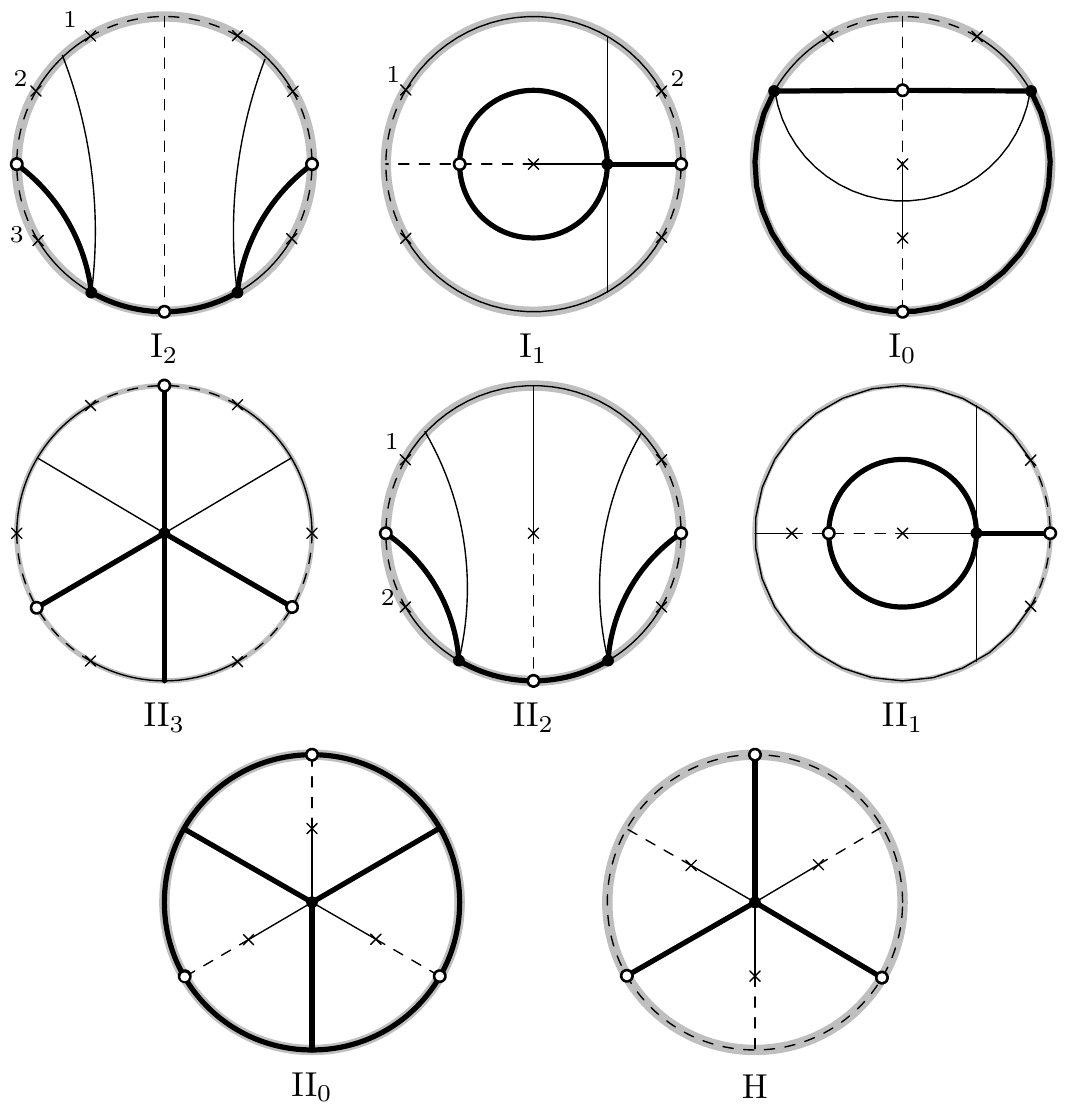}
\caption{Classes of cubic dessins up to elementary equivalence.\label{fig:cubiques}}
\end{figure}

\subsection{Positive and negative inner nodal {\tvs}}

Given an irreducible real nodal curve $X$ of type $\mathrm{I}$, we denote by $c$ the number of nodal points in~${\C X\setminus\R X}$. 
Let $\displaystyle\widehat{X}$ be the normalization of~$X$ and let~$\widehat{X}_+$ and~$\widehat{X}_-$ the connected components of $\displaystyle\C \widehat{X}\setminus\R \widehat{X}$.
We put $\sigma(X)=\#(X_+\cap X_-)$, where~$X_{\pm}\subset\C X$ is the image of~$\widehat{X}_{\pm}$ under the normalization map $\C\widehat{X}\longrightarrow\C X$.
An orientation of the real part $\R X$ induced from $X_{\pm}$ is called a \emph{complex orientation}.
A double point of $X$ is called \emph{elliptic} (resp., \emph{hyperbolic}) if it is given in local coordinates $(x,y)$ by the equation $x^2+y^2=0$ (resp.,~$x^2-y^2=0$).
A \emph{type~$\mathrm{I}$ perturbation} of~$X$ consists of a curve $X_0$ having a non-singular real part $\R X_0$ where every elliptic node has been perturbed to an oval and every hyperbolic node has been perturbed according to a complex orientation of the crossing branches of~$\R X$. 
Perturbing the real nodal points of the curve in any other way does not produce a type~$\mathrm{I}$ curve.
Moreover, if the curve $X$ is an $M$-curve, its type~$\mathrm{I}$ perturbation is non-singular.

Let $X$ be an irreducible real nodal curve $X$ of type $\mathrm{I}$ of odd degree in $\RPP$ having non-singular real part $\R X$.
Consider~$\R X$ endowed with a complex orientation.
An oval of $\R X$ is called \emph{positive} (resp. \emph{negative}) if it defines in its interior the orientation opposite to (resp., the orientation coinciding with) the one given by the orientation of $\RPP\setminus J$, where $J$ is the pseudoline of $\R X$.
A pair of ovals of $\R X$ is called \emph{injective} if one of these ovals is contained in the interior of the other one. An injective pair of ovals is called \emph{positive} if there exists an orientation of the annulus bounded by these ovals inducing a complex orientation of the ovals. Otherwise, the injective pair is called \emph{negative}.

We define the \emph{index} $i_{\mathbb{R}X}(x)$ of a point $x\in\RPP\setminus\R X$ with respect to the curve~$X$ as the absolute value of the image of~$[\R X]$ under an isomorphism between the group~$H_1(\RPP\setminus\{x\},\Z)$ and $\frac{1}{2}\Z$.

Let $X$ be an irreducible type~$\mathrm{I}$ nodal curve of odd degree in $\RPP$.
We denote by $\Lambda^+(X)$ (resp.,~$\Lambda^-(X)$) the number of positive (resp., negative) ovals and by $\Pi^+(X)$ (resp.,~$\Pi^-(X)$) the number of positive (resp., negative) injective pairs of a type~$\mathrm{I}$ perturbation of~$X$.

\begin{thm}[Rokhlin’s complex orientation formula for nodal plane curves in $\RPP$ of odd degree \cite{IMR}]
Let $X$ be a nodal type $\mathrm{I}$ plane curve of degree $d=2k+1$. Then $$\Lambda^+(X)-\Lambda^-(X)+2(\Pi^+(X)-\Pi^-(X))=l-k(k+1)+\sigma(X),$$ where $l$ is the number of ovals in a type~$\mathrm{I}$ perturbation of~$X$. 
\end{thm}

We say that an elliptic nodal point of a type $\mathrm{I}$ plane curve $X$ is \emph{positive} (resp., \emph{negative}) if its corresponding oval in the type $\mathrm{I}$ perturbation of $X$ is positive (respectively, negative).
In the same manner, we say that an isolated nodal {\tv} of a type~$\mathrm{I}$ dessin $D$ is \emph{positive} (resp., \emph{negative}) if its corresponding oval in the type~$\mathrm{I}$ perturbation of~$D$ is positive (resp., negative).

Motivated by Proposition \ref{prop:maxell} we make the following definition.

\begin{df}
Let $v$ be an inner nodal {\tv} of a type $\mathrm{I}$ dessin. We say that $v$ is {\it positive} if all the regions containing $v$ are labeled $1$. Otherwise, we say that $v$ is {\it negative}.

Given a dessin $D$, we denote by $c(D)$ twice the number of the inner nodal {\tvs} of $D$ and by $\sigma(D)$ twice the number of the negative inner nodal {\tvs} of $D$.
\end{df}

\begin{prop}\label{prop:sigmatrig}
If $X\subset\Sigma$ is a type $\mathrm{I}$ real trigonal curve with associated real dessin~$D$, then \[\sigma(X)=\sigma(D).\]
\end{prop}

\begin{proof}
Every pair $n,\bar{n}$ of complex conjugated nodal points of $X$ corresponds to an inner nodal {\tv} $v_n$, and vice versa.
By definition, the number $\sigma(X)$ counts the nodal points of the curve $X$ lying in $\C X\setminus\R X$ such that they belong to both the images of the halves $\C\widehat{X}_{\pm}$ of the normalization of $X$ on $X$, each imposing an obstruction for $X$ to be perturbed to a type $\mathrm{I}$ non-singular curve. 
In the case when the nodal point $n$ does not contribute to $\sigma(X)$, it has a local type~$\mathrm{I}$ perturbation and so does its corresponding vertex $v_n$. Since every inner simple {\tv} of type~$\mathrm{I}$ is labeled by $1$, so are all the regions containing $v_n$ and hence, it is positive.
\end{proof}

\begin{coro}\label{coro:nodalpairs}
Let $C$ be a nodal rational curve of degree $5$ whose type $\mathrm{I}$ perturbation has $l$ ovals. Assume that $C$ has an elliptic nodal point $p\in\R C$. Then, the type $\mathrm{I}$ perturbation of the marked toile $(D_C,v_{T_1},v_{T_2})$ associated to $(C,p)$ has $l-1$ ovals and $c(D_C)=c(C)+2$.
\end{coro}

\begin{proof}
The elliptic nodal point $p$ is perturbed to an oval in the type $\mathrm{I}$ perturbation of $C$, and every other oval different from the one arisen from $p$ appears in the type~$\mathrm{I}$ perturbation of $D_C$. Every inner nodal {\tv} of $D_C$ represents either a pair of complex conjugated nodal points of the curve $C$ or the pair of complex conjugated tangent lines of $C$ at $p$.
\end{proof}

Recall that for a proper non-isotrivial trigonal curve $X\subset\Sigma_n\longrightarrow\CP$, 
a real white vertex $v$ in the associated toile corresponds to an intersection of the curve $X$ 
with the zero section $Z\subset\Sigma_n$.
When $n$ is even, the real point set $\R\Sigma_n$ is homeomorphic to a torus and the long component of $X$ together with the exceptional divisor divide $\R\Sigma_n$ into two connected components.
When $n$ is odd, the real point set $\R\Sigma_n$ is homeomorphic to a Klein bottle,
and the dividing property does not hold.
Assume that $n$ is odd, the curve $X$ is of type $\mathrm{I}$, the real point set $\R X$ is non-singular and has a zigzag $z$ delimitated by vertices $u_1$ and $u_2$ in $D_X$.
 
Denote by $s=z\setminus\{u_1,u_2\}$ the interior of the zigzag.
We say that a simple {\tv} in $D_X$ is {\it on the same side} as the zigzag $z$ 
if the number of white real vertices in the connected components of $\partial D_X\setminus(s\cup \{v\})$ is even. We say that an oval is {\it on the same side} as the zigzag $z$ if its delimiting vertices are on the same side 
of the zigzag $z$. 
We denote by $\Lambda_{z}^+(X)$ (resp., $\Lambda_{z}^-(X)$ ) 
the number of positive (resp., negative) ovals on the same side 
as the zigzag $z$. 

\begin{thm}
Let $X$ be a type $\mathrm{I}$ nodal proper trigonal curve in the Hirzebruch surface $\Sigma_n\overset{\pi}{ \longrightarrow}\CP$, $n=2k+1$, and let $D_X$ be its associated toile. Assume that the real point set $\R X$ has a zigzag $z$. Then,
$$2(\Lambda_{z}^-(X)-\Lambda_{z}^+(X))+l+\sigma(X)=2n+2,$$
where $l$ is the number of connected components in the real point set of the type $\mathrm{I}$ perturbation of $X$.
Furthermore, we have
 \[\Lambda^-(X)-\Lambda^+(X)+(l-1)+\sigma(X)=2n.\]
\end{thm}

\begin{proof}
The proof follows the standard scheme of the proof of Rokhlin's formulas of complex orientations. 
Considering instead of the curve $X$ its type $\mathrm{I}$ perturbation, we can assume that the real point set $\R X$ of $X$ is non-singular.
Since $X$ is a type $\mathrm{I}$ curve, the set $\C \widehat{X}\setminus\R \widehat{X}$ has two connected components.
Choose one of the connected components of $\C \widehat{X}\setminus\R \widehat{X}$ and denote by $\C X_+\subset\C \widehat{X}$ its closure.
The orientation of $\C X_+$ induces an orientation of $\R \widehat{X}=\R X$, since $\R X$ is non-singular. 
Then, the long component $L$ of $\R X$ produces {\it via} $\pi$ an orientation of $\RP$.
Let $\C E_+$ be the closed hemisphere of the exceptional section $E\simeq\CP$ such that the induced orientation of $\R E$ produces {\it via} $\pi$ the orientation of $\RP$ opposite  to the one produced by the long component of $\R X$. 
Pick a white vertex $v$ belonging to the zigzag $z$ and consider the corresponding fiber $F=\pi^{-1}(v)$. 
Let $\C F_+$ be the closed hemisphere of $F\simeq\CP$ such that $L\cup \R E\cup\R F$ endowed with the described orientations of $L$ and $\R E$ has a type $\mathrm{I}$ perturbation with a unique connected component.
Note that if we equip $\R F$ with the opposite orientation, the set $L\cup \R E\cup\R F$ has a type $\mathrm{I}$ perturbation with three connected components.

Let $Y_+$ be the type $\mathrm{I}$ perturbation of $\C X_+\cup\C E_+\cup\C F_+$. Then, the boundary $\partial Y_+$ is a collection $I$ of contractible circles in $\R\Sigma_n$.
Put 
\begin{align*}
 w_+ & = [Y_+\cup\bigcup_{i\in I}U_i]\in H_2(\Sigma_n), \\
 w_- & = [Y_-\cup\bigcup_{i\in I}\overline{U}_i]\in H_2(\Sigma_n),
\end{align*}
where $U_i \subset \R\Sigma_n$, $i\in I$, are the disks whose oriented boundaries coincide with the connected components of $\partial Y_+$,
the set $Y_-$ is $\operatorname{conj}(Y_+)$, and $\overline{U_i}$ is the disk $U_i$ with the opposite orientation. 

In order to calculate $w_+\cdot w_-$, let us notice that $w_++w_-=
[X] + [E] + [F]$ and $\operatorname{conj}_*(w_+)=-w_-$.
Since $\operatorname{conj}_*$ acts in $H_2(\Sigma_n)$ as multiplication by $-1$, we have $w_++w_- = w_+-\operatorname{conj}_*(w_+) = 2w_+$. 
Hence, $w_+ = w_- = \frac{1}{2}
([X] + [E] + [F])$.

Since $X\subset\Sigma_n$ is a proper trigonal curve, we have 
$[X]=3[E]+
3n[F]$.
Then, 
\begin{align*}
w_+\cdot w_- & =\left(\frac{
[X] + [E] + [F]}{2}\right)^2\\
&=\left(\frac{
4[E] + (3n+1)[F]}{2}\right)^2\\
&=-4n+2(3n+1)\\
&=2n+2. 
\end{align*} 

On the other hand, we can calculate the product $w_+\cdot w_-$ in a geometrical way. 
We proceed using the arguments 
proposed in \cite{Rok}. 
Choose a smooth tangent vector field $V$ on $\R\Sigma_n$ such that 
the vector field has only nondegenerate singular points, the singular points are outside 
of $\partial Y_+$, and on $\partial Y_+$ the field is tangent to $\partial Y_+$ and directed according 
to the orientation 
induced from $Y_+$.
Extend smoothly the vector field to $Y_+$ such that it is supported on a tubular neighborhood of $\partial Y_+$.
Shift $\partial Y_+$ inside $Y_+$ along 
$\varepsilon\sqrt{-1}V$, where $\varepsilon$ is a sufficiently small positive real number, 
and extend this shift to a shift of the disks $D_i$, $i\in I$, along 
$\varepsilon\sqrt{-1}V$. 
Denote by 
$\widetilde Y$ 
the result of the shift of $Y_+\cup\bigcup_{i\in I}D_i$. By continuity of the shift, we have that 
$[{\widetilde Y}]=w_+$.
Then, we can calculate $w_+\cdot w_-$ as the intersection of the cycles 
$\widetilde Y$ 
and $Y_-\cup\bigcup_{i\in I}\overline{D}_i$, which intersect at the singular points of $V$ and the complex nodes of $X$ that contribute to $\sigma(X)$. 
At a singular point $x$ they are smooth transversal two-dimensional submanifolds, each taken with multiplicity $ -i_{\mathbb{R}X}(x)$. 
The local intersection number at $ x$ is equal to $(i_{\mathbb{R}X}(x))^2$ multiplied by the local intersection number of the submanifolds supporting the cycles. 
The latter is equal to the index of the vector field $ V$ at $ x$ multiplied by $ -1$ due to the fact that multiplication by $ \sqrt{-1}$ induces an isomorphism between the tangent and the normal fibrations of $ \mathbb{R}X$ in 
$ \mathbb{C}X$ reversing orientation.
Calculating the local intersection numbers we have that every disk $U_i$, $i\in I$, contributes $1$.
For every positive (resp., negative) injective pair the contribution is $-2$ (resp., $2$).
The cardinality of $I$ is the number $l$ of connected components of $\R X$.
The number of positive (resp., negative) injective pairs in $\partial Y_+$ correspond to the number $\Lambda_{z}^+(X)$ (resp., $\Lambda_{z}^-(X)$ ) of positive (resp., negative) ovals of $X$ on the same side as the zigzag $z$. Therefore, 

\begin{equation}\label{eq:zzigzag}
2(\Lambda_{z}^-(X)-\Lambda_{z}^+(X))+l+\sigma(X)=2n+2.
\end{equation}

Now, let us consider the other closed hemisphere $\C F_-$ of $F$.
Let $Z_+$ be the type~$\mathrm{I}$ perturbation of $\C X_+\cup\C E_+\cup\C F_-$. Then, the boundary $\partial Z_+$ is a collection $J$ of contractible circles in $\R\Sigma_n$.
Put 
\begin{align*} 
u_+ & = [Z_+\cup\bigcup_{i\in J}U'_i]\in H_2(\Sigma_n),\\
u_- & = [Z_-\cup\bigcup_{i\in J}\overline{U}'_i]\in H_2(\Sigma_n),
\end{align*}
where $U'_i \subset \R\Sigma_n$, $i\in J$, are the disks whose oriented boundaries coincide with the connected components of $\partial Z_+$, the set $Z_-$ is $\operatorname{conj}(Z_+)$, and $\overline{U'_i}$ is the disk $U'_i$ with the opposite orientation. 

In order to calculate the homological product $u_+\cdot u_-$ we proceed as before. We have that $u_+ + u_- = [X] + [E] + [F]$ and $\operatorname{conj}_*(w_+)=-w_-$, which lead us to the fact that 
$u_+ =  u_- = \frac{1}{2}([X] + [E] + [F]).$ 
Therefore, we have $u_+\cdot u_-=w_+\cdot w_-=2n+2$.

The product $u_+\cdot u_-$ can be calculated geometrically exactly in the same way as above.
The cardinality of $J$ is $l+2$.
The number of positive (resp., negative) injective pairs in $\partial Z_+$ correspond to the number $\Lambda^+(X)-\Lambda_{z}^+(X)$ (resp., $\Lambda^-(X)-\Lambda_{z}^-(X)$ ) of positive (resp., negative) ovals of $X$ which are not on the same side as the zigzag $z$. Therefore, 

\begin{equation}\label{eq:nzzigzag}
2[(\Lambda^-(X)-\Lambda_{z}^-(X))-(\Lambda^+(X)-\Lambda_{z}^+(X))]+l+2+\sigma(X)=2n+2.
\end{equation}

Hence, from Equations \eqref{eq:zzigzag} and \eqref{eq:nzzigzag}
we obtain that 
$$\Lambda^-(X)-\Lambda^+(X)+(l-1)+\sigma(X)=2n.\qedhere$$
\end{proof}

\section{Generic quintic rational curves in $\RPP$}\label{ch:deg5}

A generic rational curve in $\RPP$ has only nodal singular points.
Let $X$ be a generic real rational curve of degree~$5$ in~$\RPP$.
A double point of $X$ is called \emph{imaginary} if it belongs to~$\C X\setminus \R X$.
Since $X$ is real, the imaginary nodal points come in pairs of complex conjugated points.
Topologically, the real point set ${\mathbb R} X \subset {\mathbb R}P^2$ of the curve~$X$ is the disjoint union of a circle generically immersed in $\RPP$ and a finite set of elliptic nodes (cf. \cite{IMR}). 

Let us denote by $e$ the number of elliptic nodal points and by $h$ the number of hyperbolic nodal points of the curve $X$. Then, we have $b_0(\R X)=1+e$ and~$b_1(\R X)=1+h$.
For the complex point set of a generic rational curve, the Betti numbers depend only on the degree of the curve, namely $b_0(\C X)=1=b_2(\C X)$ and $ b_1(\C X)=\frac{1}{2}(d-1)(d-2).$ Since $X$ is of degree~$5$, applying the Smith-Thom theorem we obtain that 
$b_{*}(\R X)=2+e+h\leq b_{*}(\C X)=2+6$. The number of real nodal points of the curve $X$ is an even number less or equal than $6$.

Assume that the curve $X$ has at least one real nodal point $p\in\R X$. Such a pair $(X,p)$ is called a {\it marked real rational plane curve}. Since the point~$p$ has multiplicity $\deg(X)-3$, after blowing up the point $p\in\CPP$, the pencil of lines passing through~$p$ determines a trigonal curve $C_1\subset \Sigma_1$.
The curve $C_1$ intersects the exceptional divisor in two different points $T_1$ and $T_2$, each one corresponding to a tangent line of the curve $X$ at the point~$p$ in one of the crossing branches. If the curve is generic, each tangent line $T_i$ intersects $X\setminus\{p\}$, the complement of the point~$p$, in two different points. Then, applying positive Nagata transformations at the points~$T_i$ produces a proper trigonal curve $C_X\subset\Sigma_3$ with two additional nodal points~$n_{T_1}$ and~$n_{T_2}$.

If the nodal point $p\in\R X$ is hyperbolic, the tangent lines to the curve $X$ at the point~$p$ are real, and they transform into real nodes. Since the point~$p$ is nodal, a real tangent $T_i$ intersects the curve $X$ at the point~$p$ with multiplicity $3$. If the line~$\R T_i$ intersects $\R X\setminus\{p\}$ in two different real points, the corresponding nodal point~$n_{T_i}\in C_X$ is a hyperbolic nodal point. 
Otherwise, the line $ T_i$ intersects ${\mathbb C}X$ at two non-real points and the corresponding nodal point $n_{T_i}$ is elliptic. 

If the base nodal point $p\in\R X$ is elliptic, the tangent lines to the curve $X$ at the point~$p$ are a couple of complex conjugated lines $T$ and $\overline{T}$, each one intersecting $\C X\setminus\R X$ in two different points. Hence, the corresponding nodal points~$n_{T}$ and~$n_{\overline{T}}$ are a pair of complex conjugated nodal points and their associated vertex~$v_{T}$ is an inner nodal {\tv}. 

\begin{df}
Given a marked real rational plane curve $(X,p)$ of degree~$5$, we call the {\it marked} dessin of $(X,p)$ the toile associated to the real proper trigonal curve $C_X\subset\Sigma_3$ endowed with two nodal {\tvs} $v_{T_1}$, $v_{T_2}$ corresponding to the nodal points $n_{T_1}$ and $n_{T_2}$, respectively.
\end{df}

\begin{figure}
\begin{center}
\begin{tabular}{cc}
\includegraphics[width=2.5in]{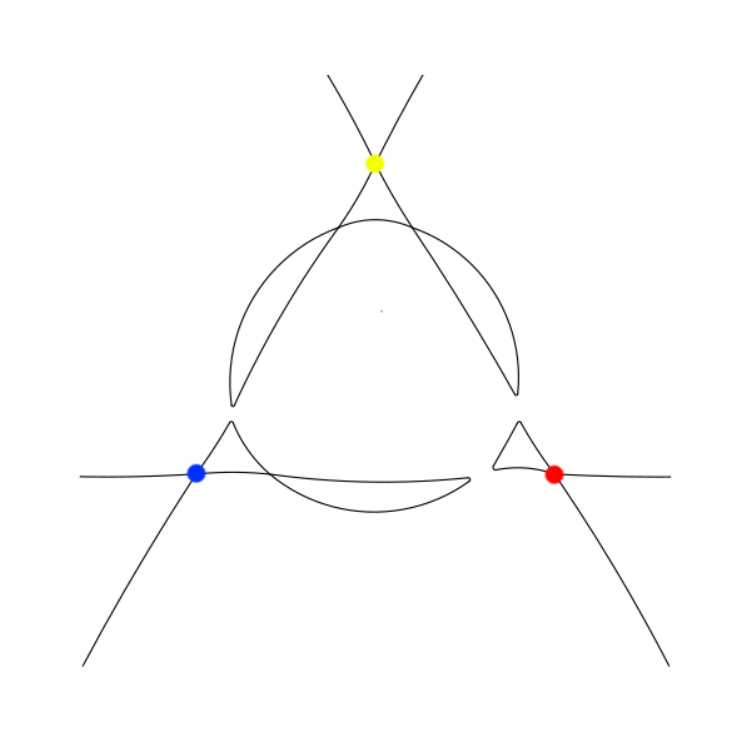}
&
\includegraphics[width=2.5in]{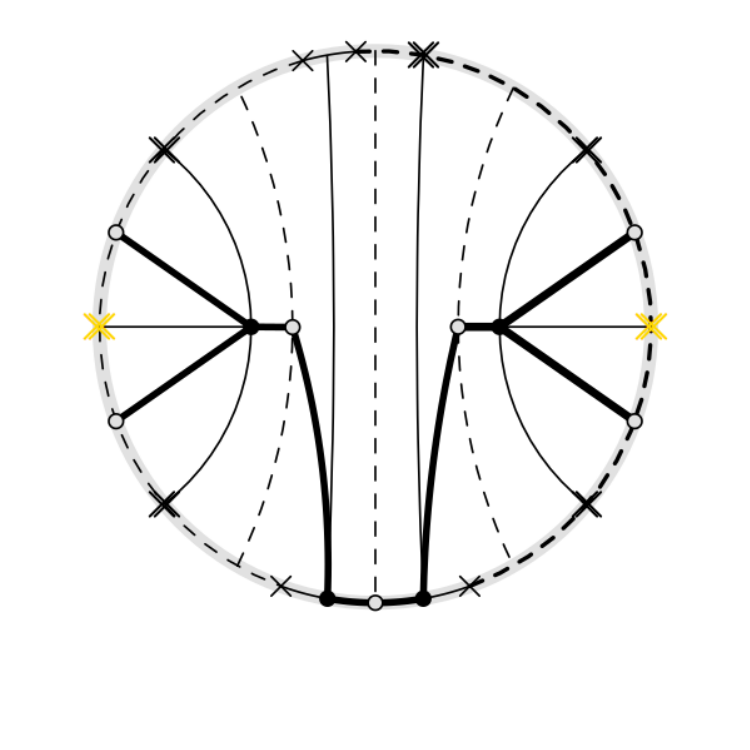}
\\
\includegraphics[width=2.5in]{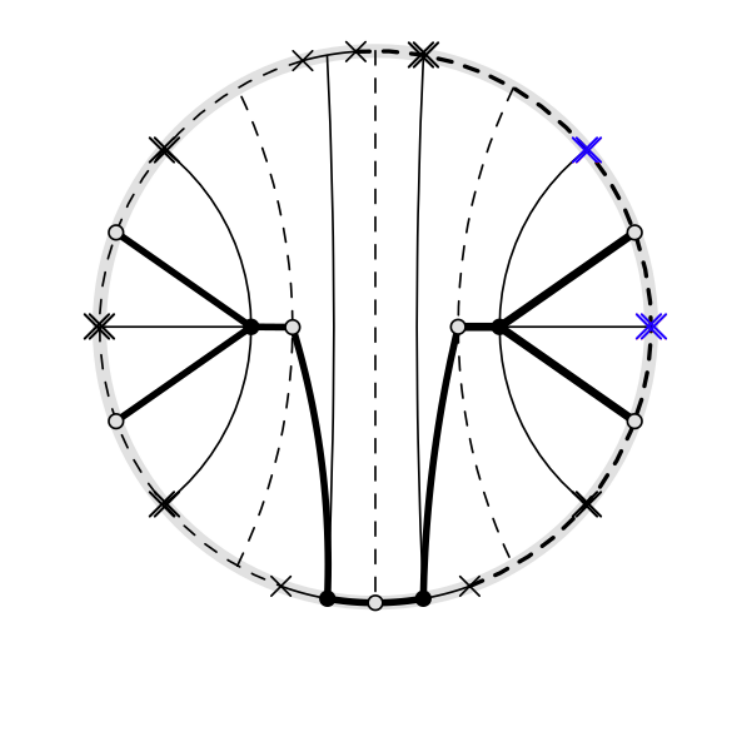}
&
\includegraphics[width=2.5in]{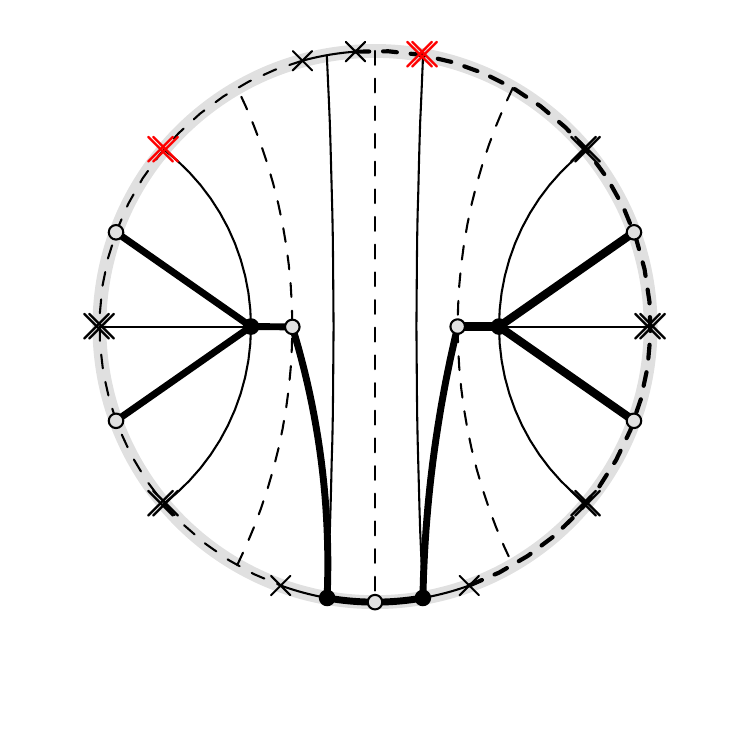}
\end{tabular} 
\end{center}
\caption{Example of a nodal rational curve in $\RPP$ with three different markings and their associated marked dessins.}
\end{figure}

On the other hand, given a marked degree~$9$ toile $(D,v_{T_1},v_{T_2})$, performing negative Nagata transformations of the associated trigonal curve $C_X\subset\Sigma_3$ at the nodes corresponding to the vertices $v_{T_1}$ and $v_{T_2}$ followed by the blow-down of the exceptional section of $\Sigma_1$, gives rise to a real degree~$5$ plane curve $X_D$ endowed with a nodal point $p\in\R X_D$.

\begin{df}
Let $D\subset S$ be a real dessin. Let us assume there is a subset of $S$ in which $D$ has a configuration of vertices and edges as in Figure~\ref{fig:xzigzag}. Replacing this configuration with the alternative one defines another dessin $D'\subset S$.
We say that~$D'$ is obtained from $D$ by {\it passing a nodal point through an inflexion point}.
Two dessins $D$, $D'$ are {\it very weakly equivalent} if there exists a finite sequence of dessins $D=D_0,D_1,\dots, D_n=D'$ such that $D_{i+1}$ is either weakly equivalent to $D_{i}$ or it is obtained from $D_i$ by passing a nodal point through an inflexion point.
\end{df}

\begin{figure}[h]
\begin{center}
\begin{tabular}{cc}
\includegraphics[width=1.5in]{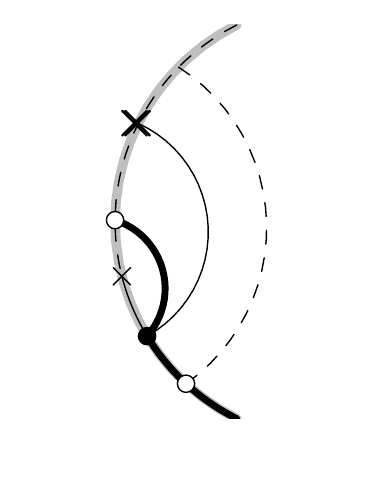}
&
\includegraphics[width=1.5in]{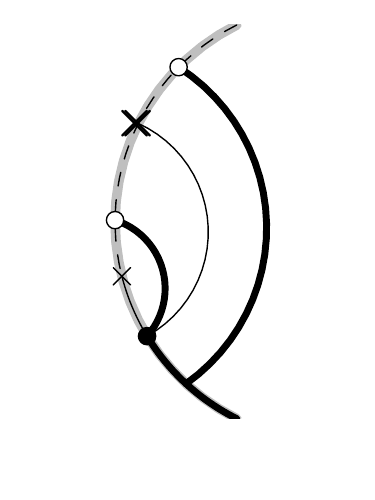}
\\
\includegraphics[width=2.5in]{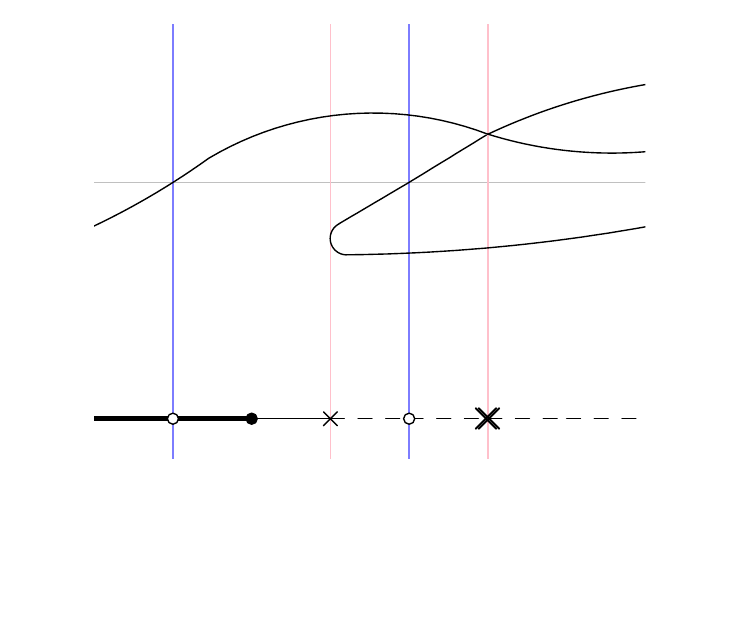}
&
\includegraphics[width=2.5in]{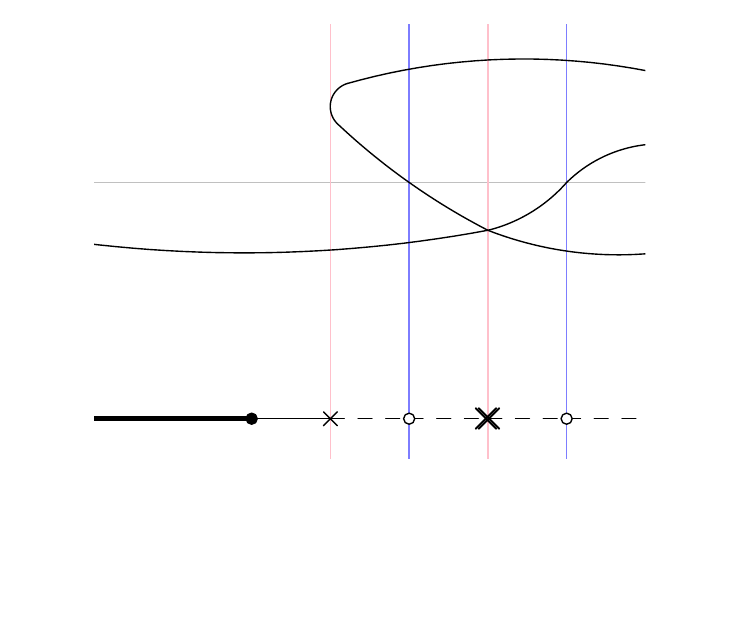}
\end{tabular} 
\end{center}
\caption{Passing a nodal point through an inflexion point. \label{fig:xzigzag}}
\end{figure}

\begin{df}
Let $D\subset S$ be a real dessin and let $v_{T}$ a marked nodal {\tv} of $D$.
Assume there is a neighborhood of $v_{T}$ in which $D$ has a configuration of vertices and edges as one in Figure~\ref{fig:marking}. Then, replacing this configuration with the corresponding alternative one defines another dessin $D'\subset S$.
We say that $D'$ is obtained from $D$ by {\it passing a singular fiber through a tangent line}.
Two marked dessins $(D,v_{T_1},v_{T_2})$, $(D',v_{T_1}',v_{T_2}')$ are {\it equivalent} if there exists a finite sequence of dessins $D=D_0,D_1,\dots, D_n=D'$ such that $D_{i+1}$ is either very weakly equivalent to $D_{i}$ or it is obtained from $D_i$ by passing a singular fiber through a tangent line.
\end{df}

\begin{figure}[h]
\begin{center}
\begin{subfigure}{\linewidth}
\centering
\begin{tabular}{cc}
\includegraphics[width=1.5in]{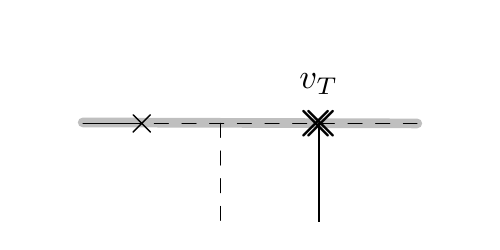}
&
\includegraphics[width=1.5in]{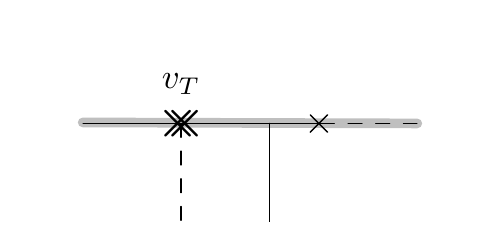}
\\
\includegraphics[width=2in]{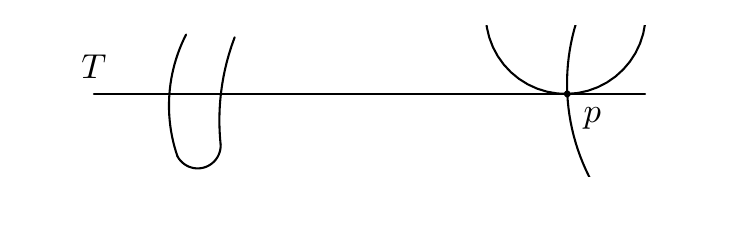}
&
\includegraphics[width=2in]{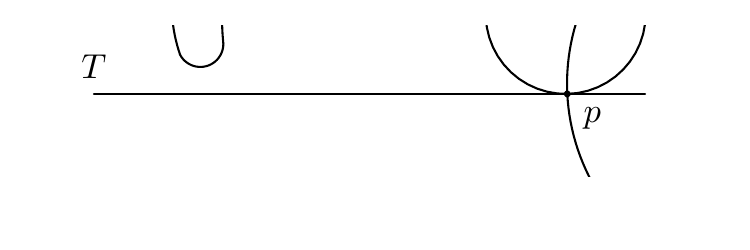}
\end{tabular} 
\caption{Passing a simple tangent through the tangent line $T$.}
\end{subfigure}

\begin{subfigure}{\linewidth}
\centering
\begin{tabular}{cc}
\includegraphics[width=1.5in]{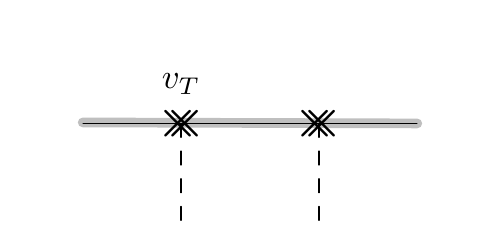}
&
\includegraphics[width=1.5in]{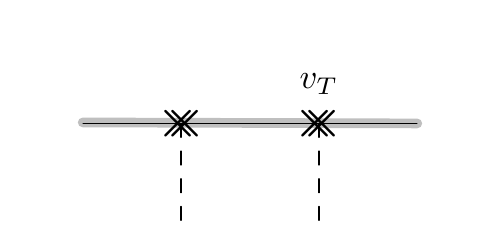}
\\
\includegraphics[width=2in]{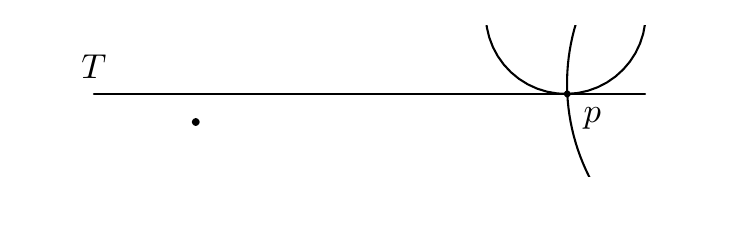}
&
\includegraphics[width=2in]{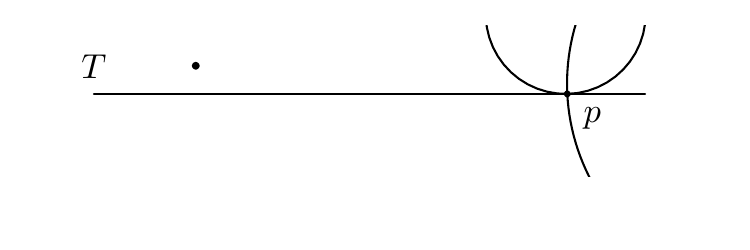}
\end{tabular} 
\caption{Passing an elliptic nodal point of the curve through the tangent line $T$.}
\end{subfigure}

\begin{subfigure}{\linewidth}
\centering
\begin{tabular}{cc}
\includegraphics[width=1.5in]{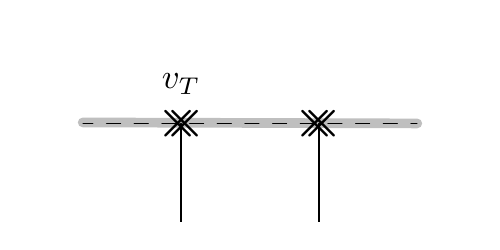}
&
\includegraphics[width=1.5in]{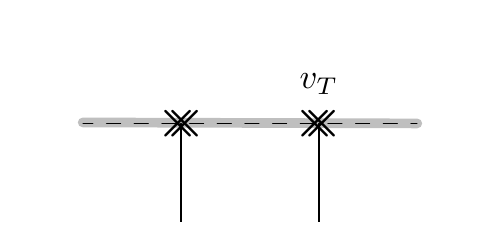}
\\
\includegraphics[width=2in]{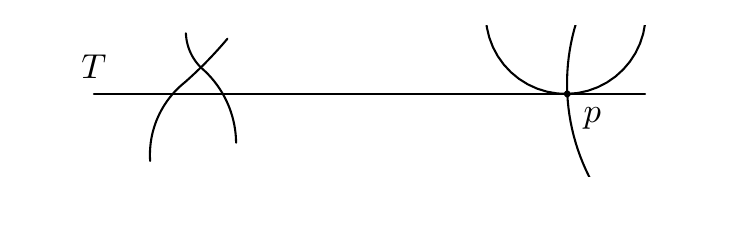}
&
\includegraphics[width=2in]{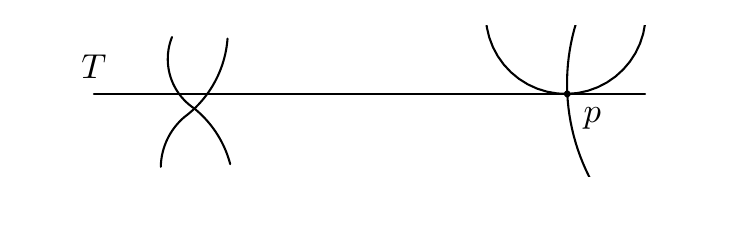}
\end{tabular} 
\caption{Passing a hyperbolic nodal point of the curve through the tangent line $T$.}
\end{subfigure}
\end{center}
\caption{Passing a singular fiber through a tangent line.\label{fig:marking}}
\end{figure}

In Figure~\ref{fig:marking} we have the possible equisingular deformations of the curve when the tangent line $T$ is a bitangent of $X$ or when the tangent line $T$ intersects one additional nodal point (hyperbolic or elliptic).

When the base point $p\in X$ is an inflection point of one branch of $\R X$,
passing a marked nodal point through the inflexion point $p$ defines a relation corresponding to the equisingular perturbations of the curve, as shown in Figure~\ref{fig:xmark}.

\begin{figure}[h]
\begin{center}
\begin{tabular}{cc}
\includegraphics[width=1.5in]{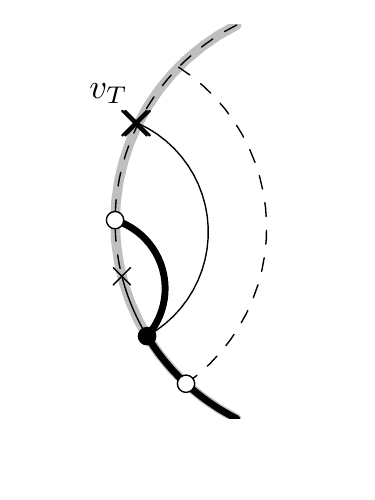}
&
\includegraphics[width=1.5in]{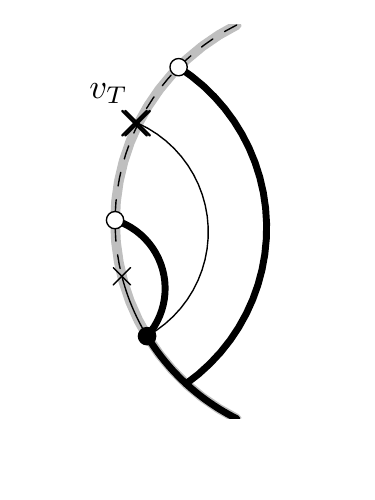}
\\
\includegraphics[width=2.5in]{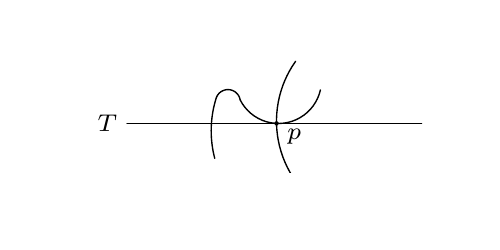}
&
\includegraphics[width=2.5in]{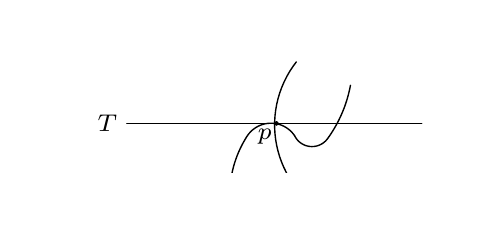}
\end{tabular} 
\end{center}
\caption{Passing a marked nodal point through an inflexion point. \label{fig:xmark}}
\end{figure}

\begin{thm}
There is a one-to-one correspondence between the set of rigid isotopy classes of marked nodal rational curves of degree~$5$ in $\RPP$ and the set of equivalence classes of marked toiles of degree~$9$ having~$6$ nodal {\tvs} and exactly one set of type~I labelling.
\end{thm}

\begin{proof}
The combinatorial restrictions on the dessins imply they represent irreducible curves of genus~$0$ in~$\RPP$.
The correspondence between equivalence classes 
 is a consequence of Theorem \ref{th:correpondance2}
and the fact that passing a nodal point through an inflexion point and passing a singular fiber through a tangent line are equivariant equisingular deformations that do not change the rigid isotopy class.
\end{proof}

\begin{df}
Given a nodal type~$\mathrm{I}$ dessin $D$, we call the type~$\mathrm{I}$ {\it perturbation} of $D$ the dessin obtained by perturbing every non-marked isolated nodal {\tv} into an oval and every non-isolated nodal vertex either into an inner simple {\tv} if its labeling is $1$ or into two real simple {\tvs} if its label is $\overline{1}$.
\end{df}

\subsection{Maximally perturbable curves}

\begin{df} 
A real algebraic curve $C$ in ${\mathbb R}P^2$ is \emph{maximally perturbable} if there is a real perturbation $C_{0}$ of~$C$ such that $C_0$ is an $M$-curve.
\end{df}

Recall that a non-singular $M$-quintic in $\RPP$ has six ovals in a convex position (see Figure~\ref{fig:nonsingM}). They form a hexagon in the following way:
given two points $p_1, p_2\in\RPP\setminus \R X$, we can define the \emph{segment} $[p_1,p_2]$ between $p_1$ and $p_2$ as the connected component of $L\setminus\{p_1,p_2\}$ with even intersection with $\R X$, where $L$ is the line passing through $p_1$ and $p_2$.
As a Corollary of B\'ezout theorem, there do not exist points $p_1, p_2, \dots, p_6$, each at the interior of a different oval, such that 
$[p_1,p_2]\cup[p_2,p_3]\cup\dots\cup[p_5,p_1]$ forms a contractible circle in $\RPP$ whose interior contain $p_6$.

If we fix a point~$p$ at the interior of an oval, the pencil of lines passing through~$p$ induces an order on the other five ovals. Varying the fix point among the ovals, we can assign to every oval two neighboring ones, defining a reversible cyclic order on the six ovals.
If a type~$\mathrm{I}$ maximally perturbable curve~$C$ has an elliptic nodal point, the oval produced by it in a type~$\mathrm{I}$ perturbation of~$C$ must respect its relative position with respect to its neighboring ovals, and so does the elliptic nodal point of the original curve.

\begin{figure}[h] 
\centering
\includegraphics[width=2in]{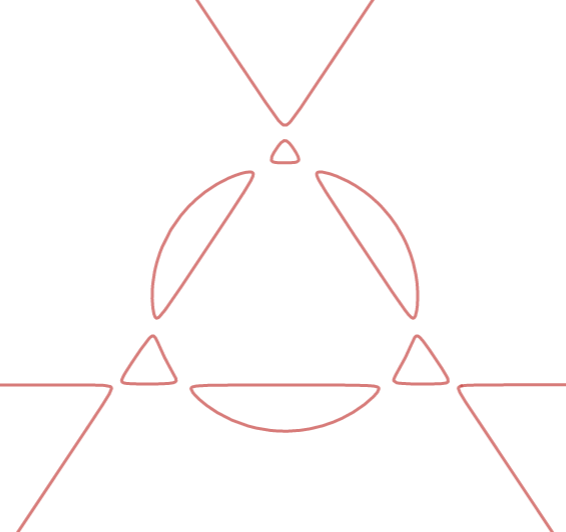}
\caption{The rigid isotopy class of non-singular $M$-quintic curves.}
\label{fig:nonsingM}
\end{figure}

\begin{lm} \label{lm:maxnodes}
If a nodal rational curve in $\RPP$ is maximally perturbable, then it is maximal.
\end{lm}

\begin{proof}
Let~$C$ be a real nodal rational plane curve of degre $d$. 
In the type~$\mathrm{I}$ perturbation~$\widetilde{C}$ of the curve~$C$ every oval comes either from an elliptic nodal point or from an oval attached by a hyperbolic nodal point to a chain of ovals, which in the case when $d$ is odd, is subsequently attached to the pseudoline.
Therefore, the number $l$ of connected components of~$\widetilde{C}$ is at most the sum $e+h+1$. Since~$\widetilde{C}$ is maximal, it has $ \frac{(d-1)(d-2)}{2}+1$ connected components of the real point set. Thus, the total Betti number $$\displaystyle b_{*}(\R C)=2+e+h\geq 2+\frac{(d-1)(d-2)}{2}=b_{*}(\C C).$$ By the Smith-Thom inequality these quantities are in fact equal and the curve~$C$ is maximal.
\end{proof}

\begin{coro}
Let~$C$ be a maximal nodal rational curve of degree~$5$ in $\RPP$. Assume~$C$ has a hyperbolic nodal point $p\in\R C$. Then, the corresponding proper trigonal curve~$\overline{C}\subset\Sigma_3$ is maximal.
\end{coro}

\begin{proof}
Since~$C$ is a maximal nodal rational curve of degree~$5$, we have that $ e+h=6$.
Let $e'$ and $h'$ the numbers of elliptic and hyperbolic nodal points of the curve $\R \overline{C}$.
Due to the fact that the base point~$p$ is hyperbolic, the nodal points $n_{T_1}$ and $n_{T_2}$ corresponding to the tangent lines of~$C$ at $p$ are real. The base point~$p$ disappears on the strict transform, hence $e'+h'=e+(h-1)+2=7$.
Since $\overline{C}$ is obtained from~$C$ by birational transformations, it is rational. Thus, the total Betti number of the real point set $b_{*}(\R C)=2+e'+h'=9$. On the other hand, since a non-singular trigonal curve lying on $\Sigma_3$ has genus $7$, the total Betti number of the complex point set $b_{*}(\C C)$ is equal to $9$.
\end{proof}

\begin{lm}\label{lm:maxpert}
Let~$C$ be a maximally perturbable nodal rational curve of degree~$5$ in $\RPP$. Assume that~$C$ has a hyperbolic nodal point $p\in\R C$ belonging to the pseudoline of a type~$\mathrm{I}$ perturbation of~$C$. Then, the type~$\mathrm{I}$ perturbation of the marked toile~$(D_C,v_{T_1},v_{T_2})$ associated to $(C,p)$ is a non-singular maximal toile.
\end{lm}

\begin{proof}
Let us consider the curve $C_0$ obtained by a type~$\mathrm{I}$ perturbation of every nodal point of~$C$ except for the base point~$p$. The curve $C_0$ has five ovals and such ovals appears in the type~$\mathrm{I}$ perturbation of the marked toile $(D_C,v_{T_1},v_{T_2})$ associated to $(C,p)$.
If one of the tangent $T_i$ produces an isolated nodal {\tv} $v_{T_i}$, it is perturbed to an oval. Since the number of ovals of a type~$\mathrm{I}$ curve has the same parity as a maximal curve, the number of ovals of $(D_C,v_{T_1},v_{T_2})$ is seven.
Otherwise the tangent lines $T_1$ and $T_2$ produce non-isolated nodal {\tvs}.
Since the original trigonal curve $\overline{C}\subset\Sigma_3$ is rational, the hyperbolic nodes $n_{T_1}$ and $n_{T_2}$ belong to the long component. Topologically, the long component is a circle with two simple crossings, whose type~$\mathrm{I}$ perturbation has $3$ components. Therefore, the type~$\mathrm{I}$ perturbation of the marked toile $(D_C,v_{T_1},v_{T_2})$ has one long component and seven ovals.
Since all singularities of the original curve~$C$ are real and the base point~$p$ is real, the toile~$D_C$ has no inner singular vertices.
\end{proof}

\begin{df} 
Given a type~$\mathrm{I}$ marked toile $(D,v_1,v_2)$, we call a {\it true oval} a dotted segment of $D$ such that after a type~$\mathrm{I}$ perturbation of every vertex other than $v_1$ and $v_2$, the connected components of $S\setminus\{v_1,v_2\}$ have an even number of white vertices.
\end{df}

\begin{coro} \label{coro:lpert}
Let~$C$ be a nodal real rational curve of degree~$5$ in $\RPP$ whose type~$\mathrm{I}$ perturbation has $l$ ovals. Assume that~$C$ has a hyperbolic nodal point $p\in\R C$ belonging to the pseudoline of the type~$\mathrm{I}$ perturbation of~$C$. Then, the type~$\mathrm{I}$ perturbation of the marked toile $(D_C,v_{T_1},v_{T_2})$ associated to $(C,p)$ has $l+1$ ovals, among which $l-1$ are true ovals.
\end{coro}

\begin{proof}
As in the proof of Lemma~\ref{lm:maxpert}, there are $l-1$ ovals in the type~$\mathrm{I}$ perturbation of the marked dessin $D_C$ arising from ovals of the the curve $C_0$ obtained by realizing a type~$\mathrm{I}$ perturbation of every nodal point of~$C$ except for the base point~$p$. The remaining two ovals come either from an isolated nodal point $v_i$ or from a non-isolated nodal point $v_i$ representing the self-intersection of the long component or the self-intersection of an oval.
\end{proof}

Due to the previous lemma, we use the classification of smooth trigonal $M$-curves. We state a version of the theorem suitable for the content of this work. The theorem can be found in full statement in \cite{DIK}.

\begin{thm}[\cite{DIK}]
Any smooth real trigonal $M$-curve $X\subset\Sigma_3$ has a corresponding dessin $D\subset\mathbb{D}^2$ which has a canonical decomposition as the gluing of three cubic dessins of type~$\mathrm{I}$ along dotted generalized cuts corresponding to zigzags on the cubics.
\end{thm}

The previous theorem states that any non-singular maximal  toile $D$ of degree~$9$ is decomposed as the gluing of three type~$\mathrm{I}$ cubics $A$, $B$ and $A'$ along edges belonging to a zigzag. If the decomposition takes the form $A-B-A'$, where $D_1 - D_2$ indicates that the dessin $D_1$ is glued to the dessin $D_2$, then the cubic $B$ is the cubic $\mathrm{I}_2$ and the cubics $A$ and $A'$ are either $\mathrm{I}_1$ or $\mathrm{I}_2$.

\begin{lm}\label{lm:maxdec}
Let $(D_C,v_{T_1},v_{T_2})$ be the marked toile associated to a maximally perturbable nodal rational curve~$C$ of degree~$5$ in $\RPP$ with a hyperbolic nodal point $p\in\R C$ belonging to the pseudoline.
Then, the marked toile $(D_C,v_{T_1},v_{T_2})$ is equivalent to a marked toile $(D',v_1',v_2')$ whose type~$\mathrm{I}$ perturbation is the gluing of cubics $\mathrm{I}_1- \mathrm{I}_2- \mathrm{I}_1$.
\end{lm}

\begin{proof} 
Due to Lemma~\ref{lm:maxpert}, the type~$\mathrm{I}$ perturbation of the toile $D_C$ is $A-\mathrm{I}_2 - A'$ where $A$ and $A'$ are either $\mathrm{I}_1$ or $\mathrm{I}_2$.
If the cubic $A$ equals $\mathrm{I}_2$, let $Z$ be the zigzag other than the one where the gluing takes place.
If the zigzag $Z$ comes from a zigzag belonging to the toile $D_C$, then the destruction of the zigzag $Z$ followed by an elementary move of type $\bullet$-in produces a marked toile whose type~$\mathrm{I}$ perturbation has a cubic $\mathrm{I}_1$ in the place of $A$.
If the zigzag $Z$ comes from perturbing a non-isolated nodal {\tv} $v$, then passing a nodal point through an inflexion point followed by an elementary move of type $\bullet$-in produces a marked toile whose type~$\mathrm{I}$ perturbation has a cubic $\mathrm{I}_1$ in the place of $A$.
A symmetric argument applied to the cubic $A'$ proves the statement.
\end{proof}


Henceforth, let $v_e$ and $v_h$ be the numbers of isolated and non-isolated marked nodal {\tvs} of a marked dessin $(D,v_1,v_2)$.

\begin{prop} \label{prop:maxhyp}
 Let~$C$ be a maximally perturbable nodal rational curve of degree~$5$ in $\RPP$ such that the curve~$C$ has at least one hyperbolic nodal point. Then, the curve~$C$ belongs to one of the different rigid isotopy classes represented in Figures~\ref{fig:b01} to~\ref{fig:b06}.
\end{prop}

\begin{proof} 

Since~$C$ is a maximally perturbable nodal rational curve, by Lemma~\ref{lm:maxnodes} it is maximal, hence $e+h=6$. Since~$C$ is a rational curve with hyperbolic nodal points, we can choose $p\in\R C$ a hyperbolic nodal point of~$C$ belonging to the pseudoline.
Then, by Lemma~\ref{lm:maxdec}, the marked toile $(D_C,v_{T_1},v_{T_2})$ is equivalent to a marked toile $(D',v_1',v_2')$ whose type~$\mathrm{I}$ perturbation is the gluing of cubics $\mathrm{I}_1- \mathrm{I}_2- \mathrm{I}_1$, with associated trigonal curve~$C_M$.
In the real point set of the trigonal curve $\R C_M$ there is a set of ovals and possible hyperbolic nodes that can be created indicated from the dessin. 
We search to enumerate all equivalence classes of marked dessin $(D',v_1',v_2')$. For that, let us choose $e+v_e$ ovals of the curve $C_M$ to contract, producing that amount of isolated nodal {\tvs} among which $v_e$ are marked. Then, all the remaining ovals must be attached among them or to the long component in a way that there is only one irreducible component with a number of non-isolated {\tvs} equals to $h+v_h-1$, among which $v_h$ are marked.
Then, enumerating all equivalence classes of marked dessins satisfying the aforementioned restrictions and realizing the birational transformations in order to recover the associated curves $\R C\subset\RPP$ lead to the plane curves shown on Figures~\ref{fig:b01} to~\ref{fig:b06}.
\end{proof}

\begin{figure}[h] 
\centering
\begin{tabular}{ccc}
\includegraphics[width=1.2in]{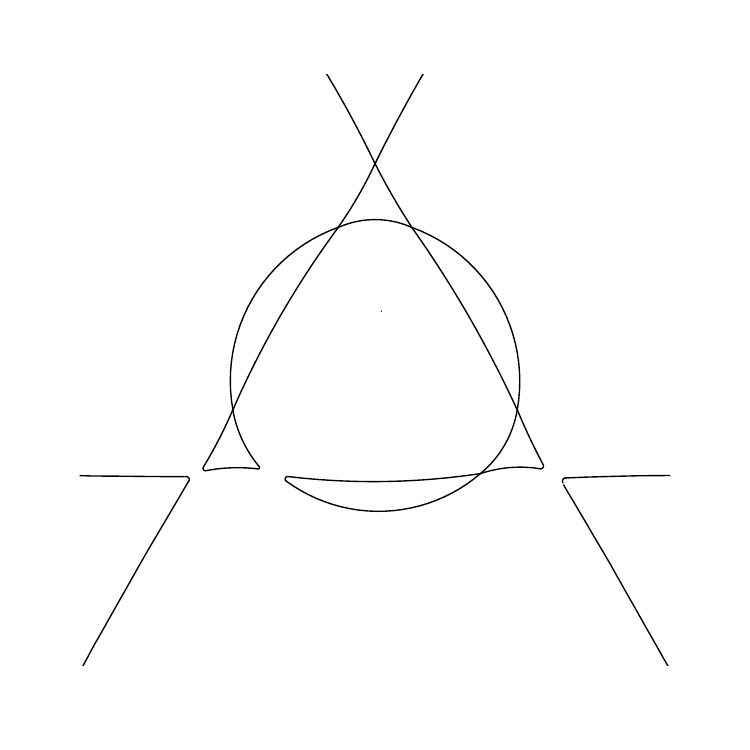}
&
\includegraphics[width=1.2in]{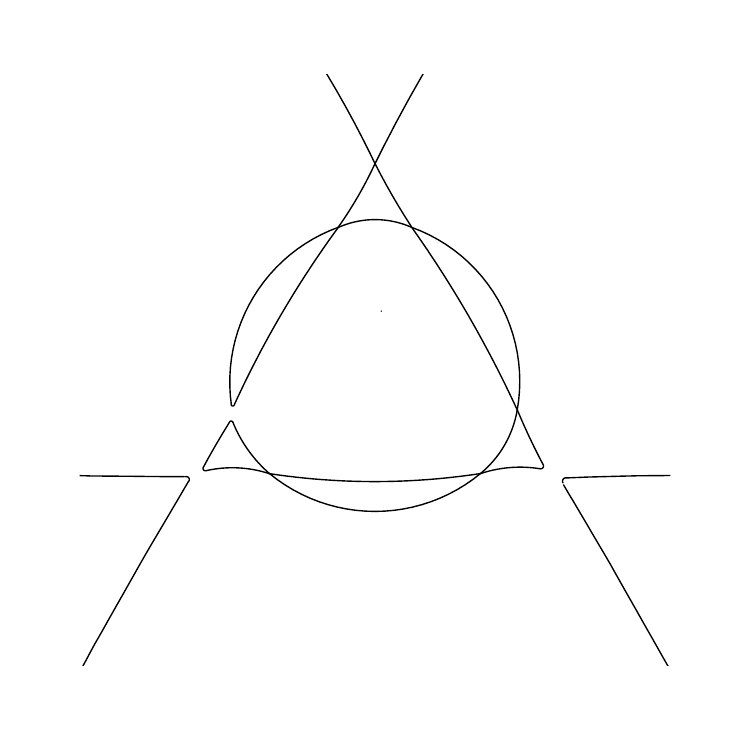}
&
\includegraphics[width=1.2in]{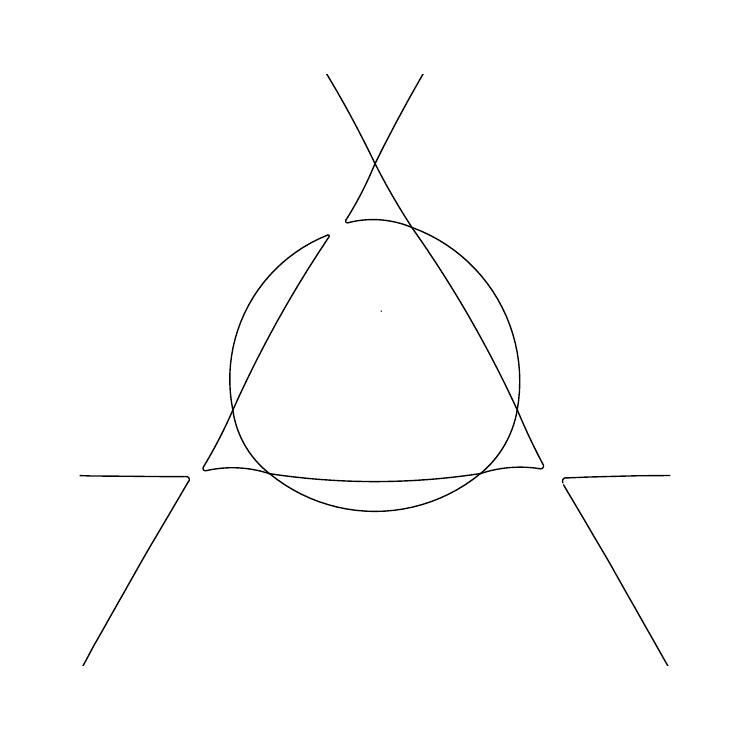}
\\
\includegraphics[width=1.2in]{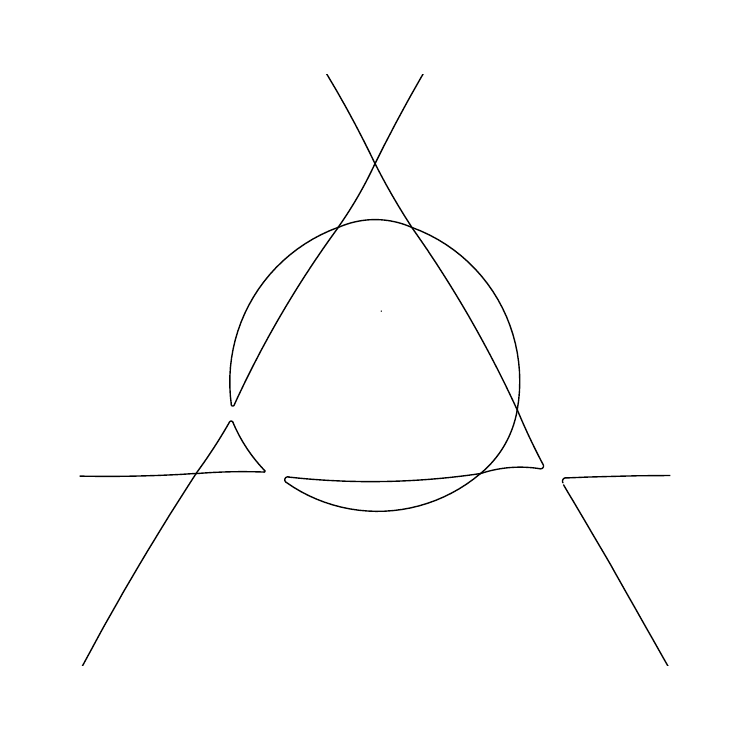}
&
\includegraphics[width=1.2in]{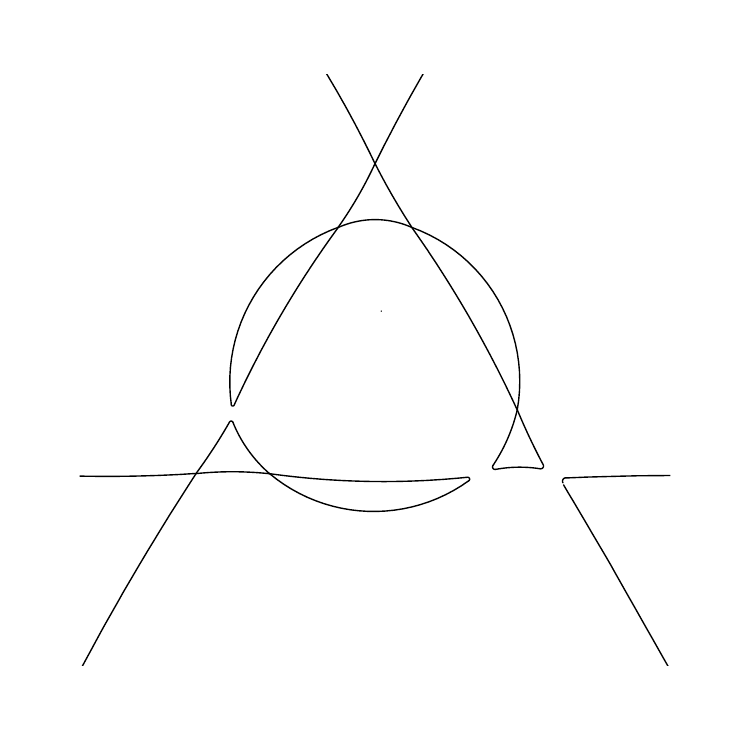}
&
\includegraphics[width=1.2in]{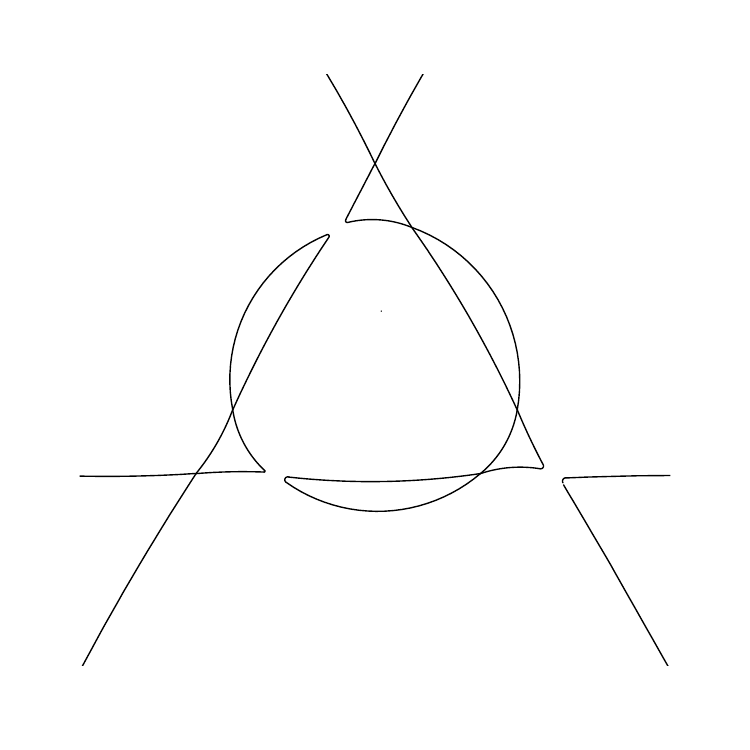}
\\
\includegraphics[width=1.2in]{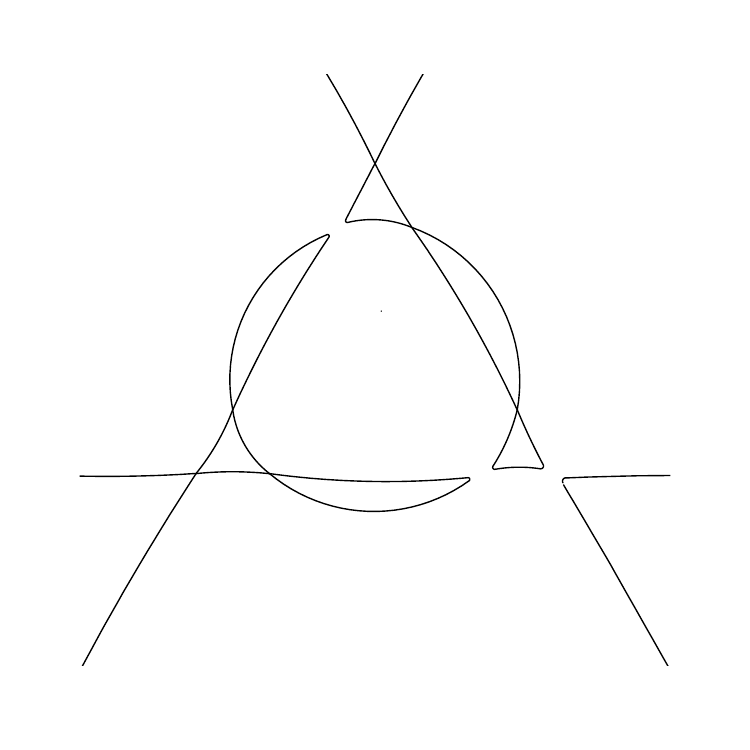}
&
\includegraphics[width=1.2in]{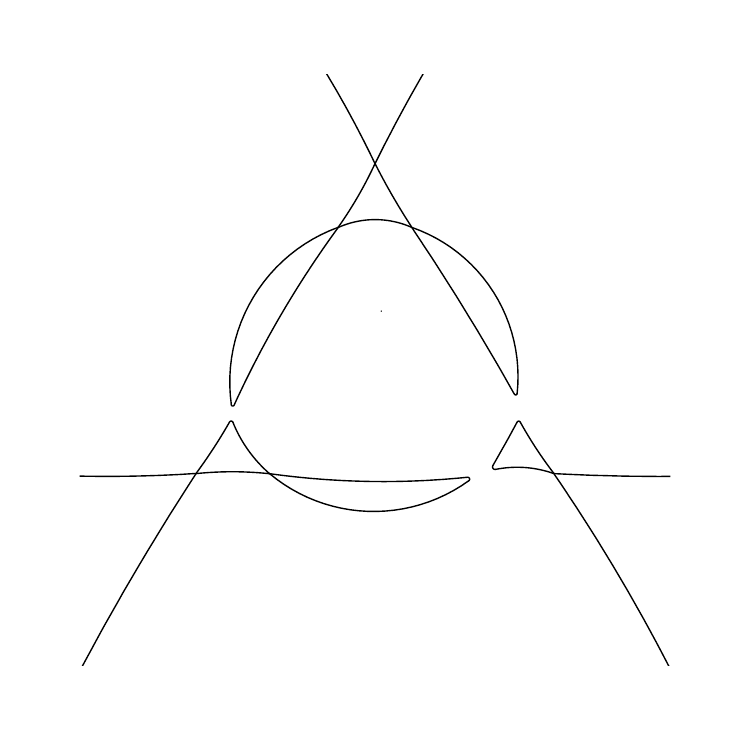}
&
\includegraphics[width=1.2in]{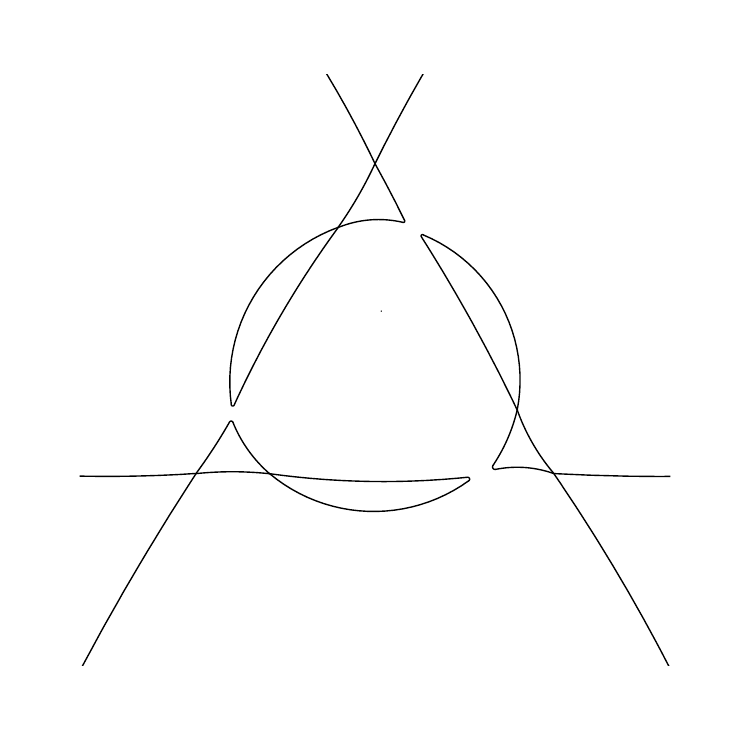}
\\
\end{tabular}
\caption{Rigid isotopy classes of maximally perturbable nodal rational curves of degree~$5$ in $\RPP$ without isolated nodes.}
\label{fig:b01}
\end{figure}

\begin{figure}[h] 
\centering
\begin{tabular}{ccc}
\includegraphics[width=1.1in]{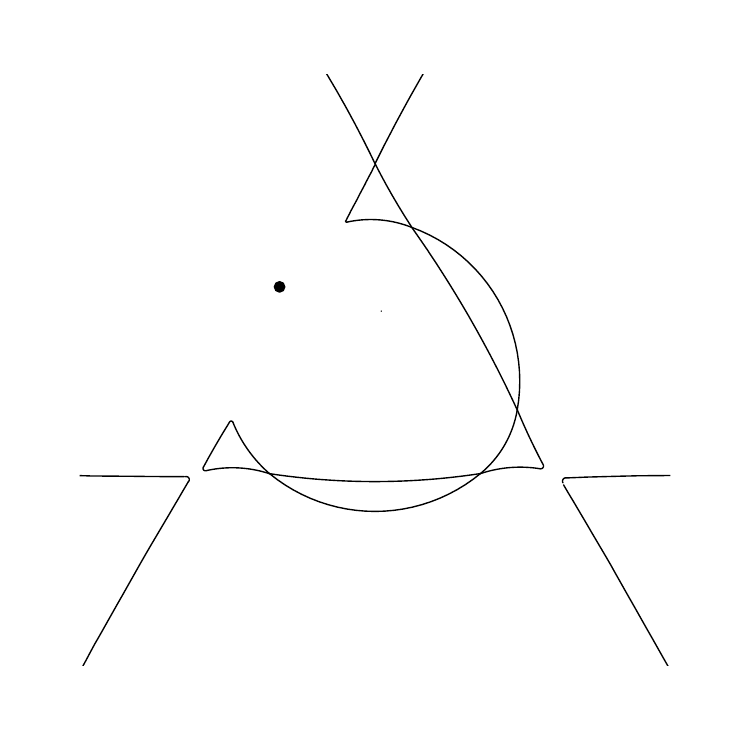}
&
\includegraphics[width=1.1in]{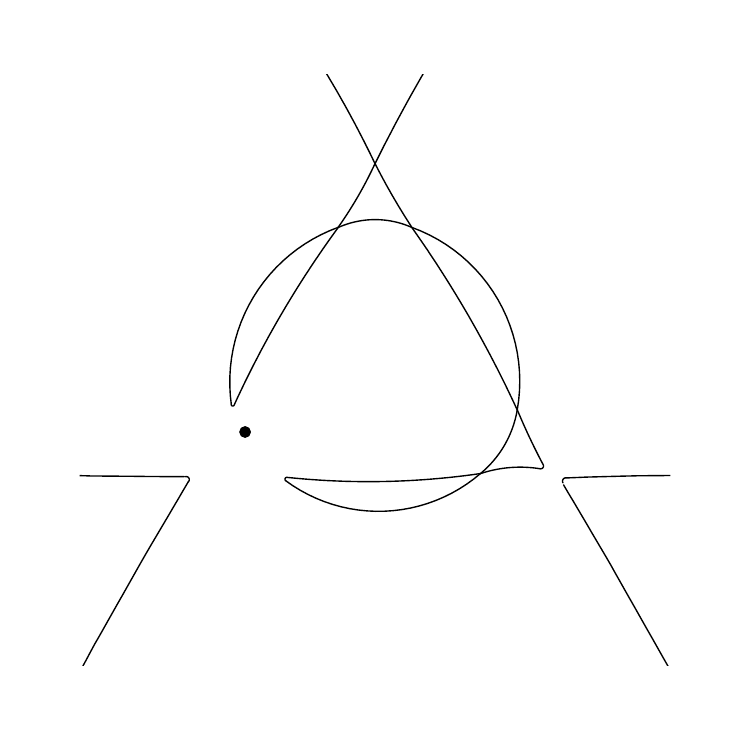}
&
\includegraphics[width=1.1in]{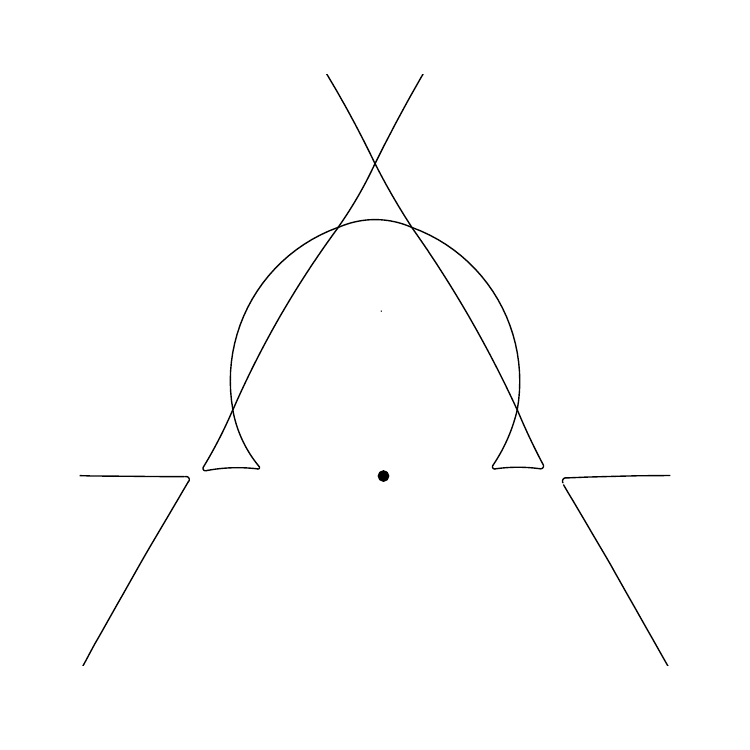}
\end{tabular}\\
\begin{tabular}{cccc}
\includegraphics[width=1.1in]{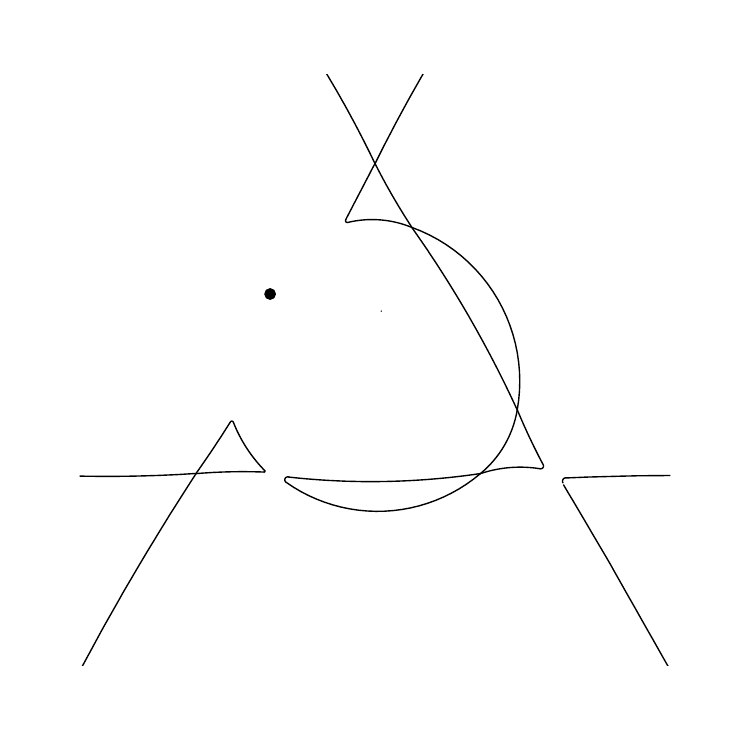}
&
\includegraphics[width=1.1in]{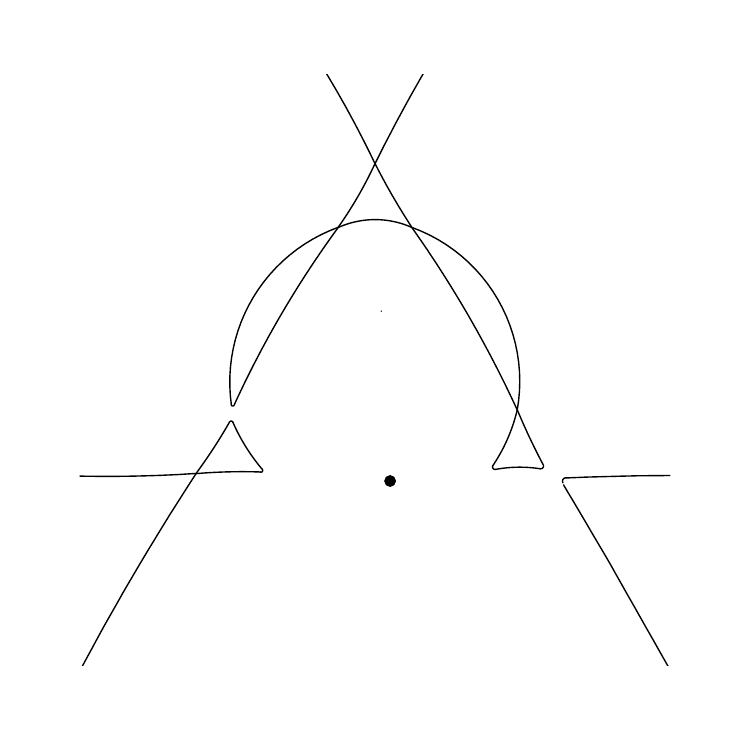}
&
\includegraphics[width=1.1in]{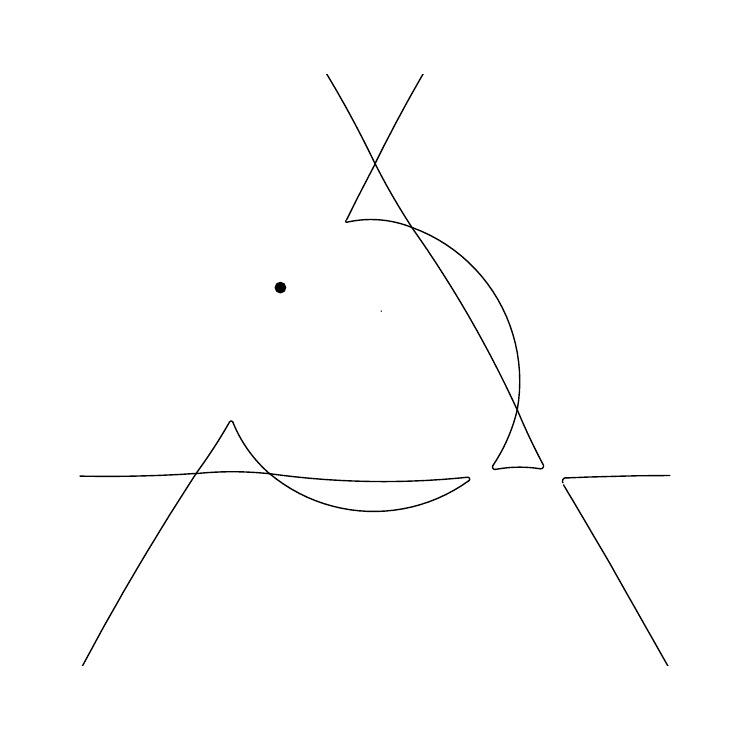}
&
\includegraphics[width=1.1in]{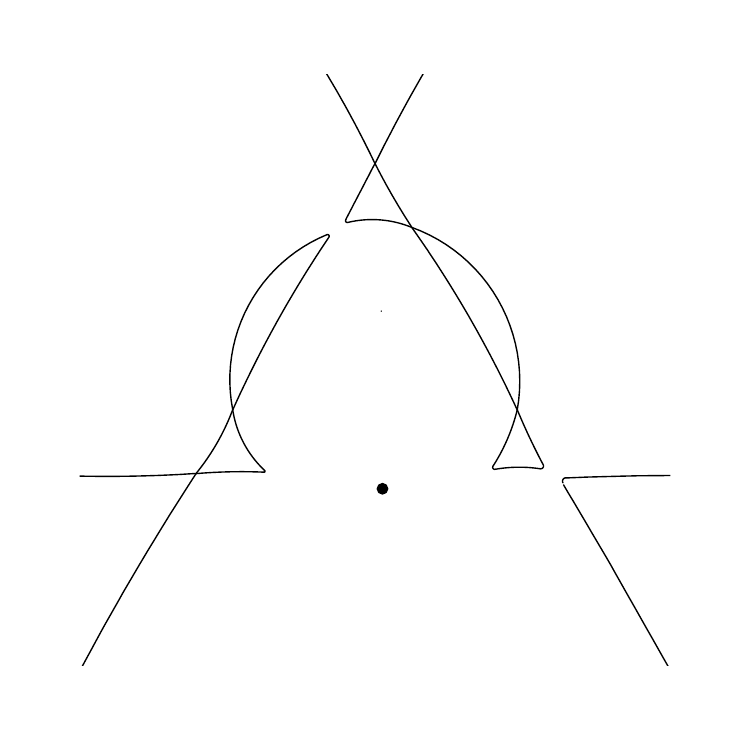}
\\
\includegraphics[width=1.1in]{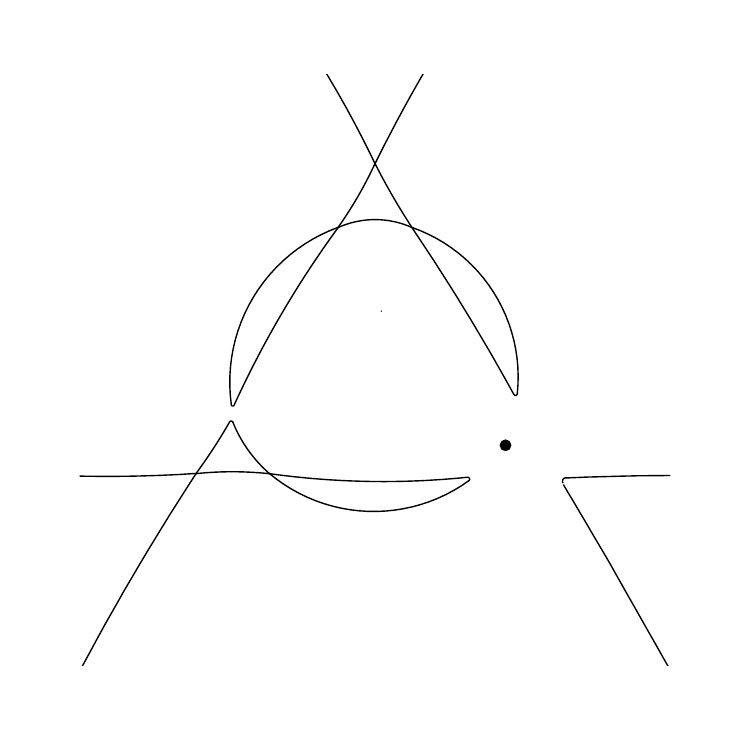}
&
\includegraphics[width=1.1in]{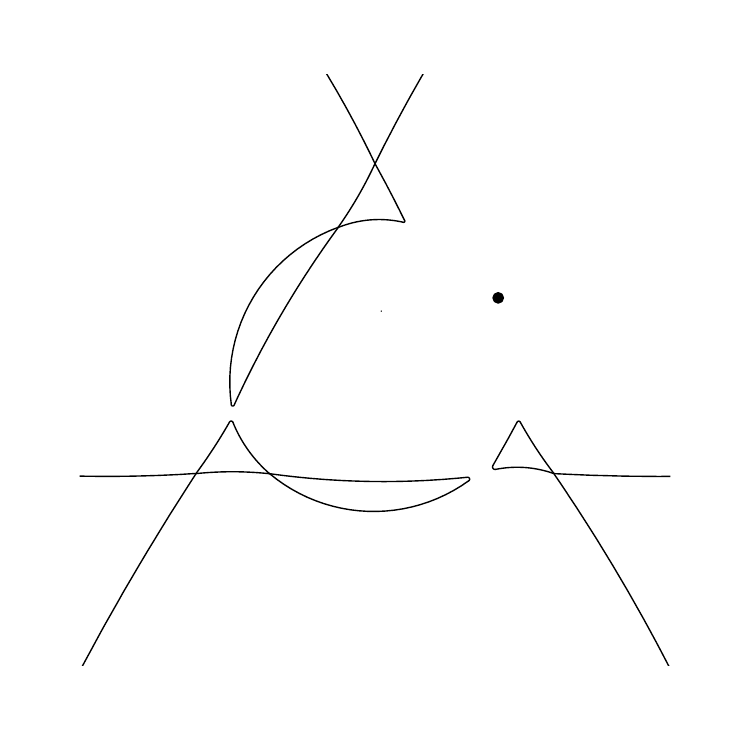}
&
\includegraphics[width=1.1in]{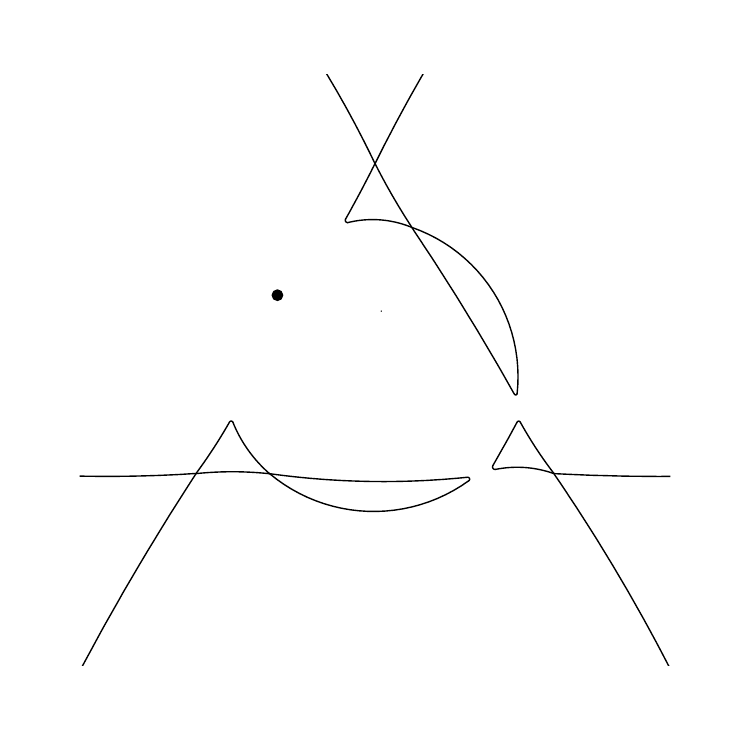}
&
\includegraphics[width=1.1in]{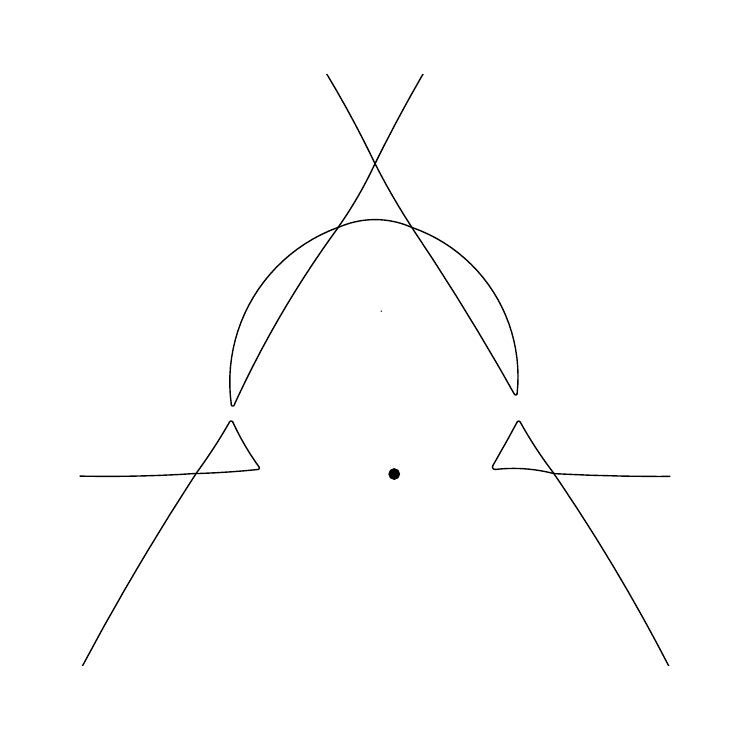}
\\
\end{tabular}
\caption{Rigid isotopy classes of maximally perturbable nodal rational curves of degree~$5$ in~$\RPP$ with exactly one isolated node.}
\label{fig:b02}
\end{figure}

\begin{figure}[h] 
\centering
\begin{tabular}{ccc}
\includegraphics[width=1.1in]{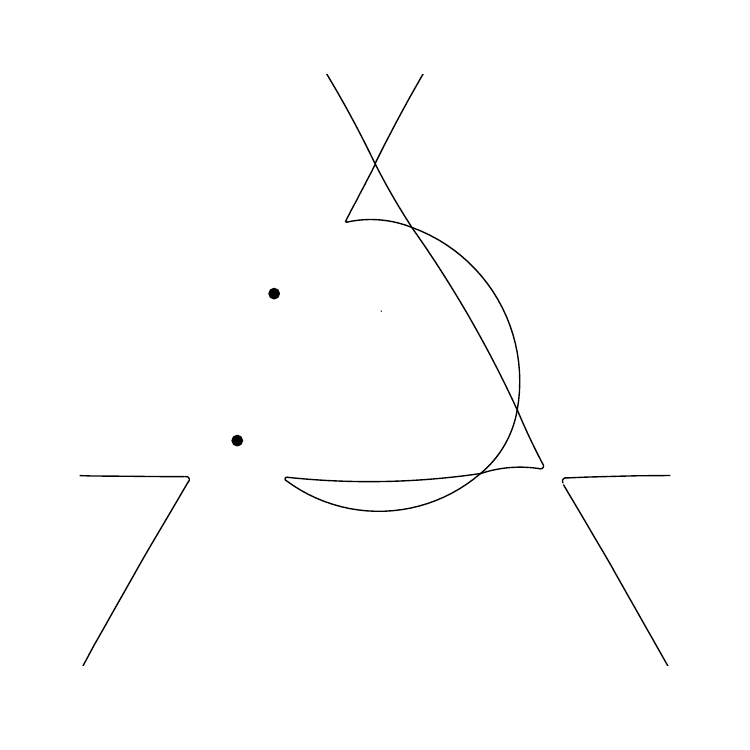}
&
\includegraphics[width=1.1in]{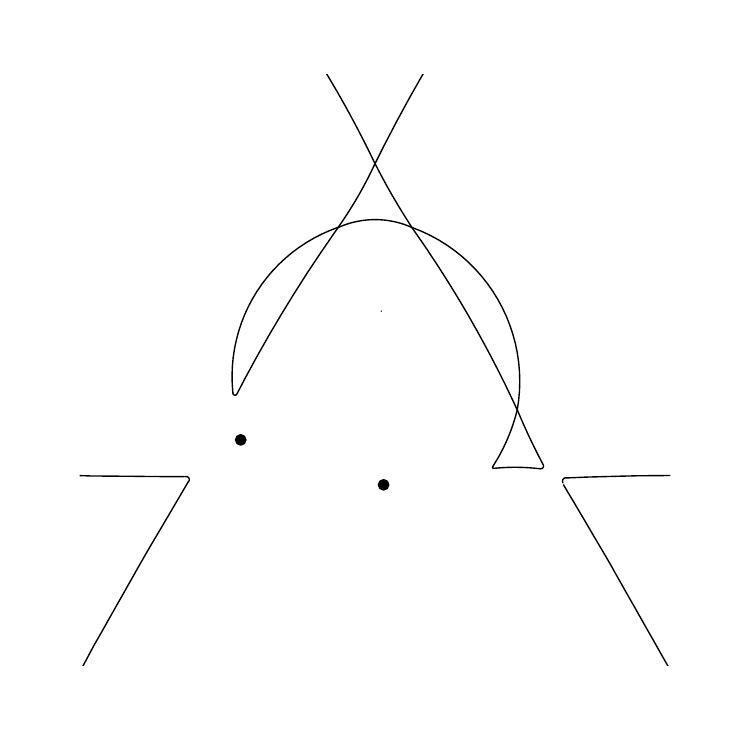}
&
\includegraphics[width=1.1in]{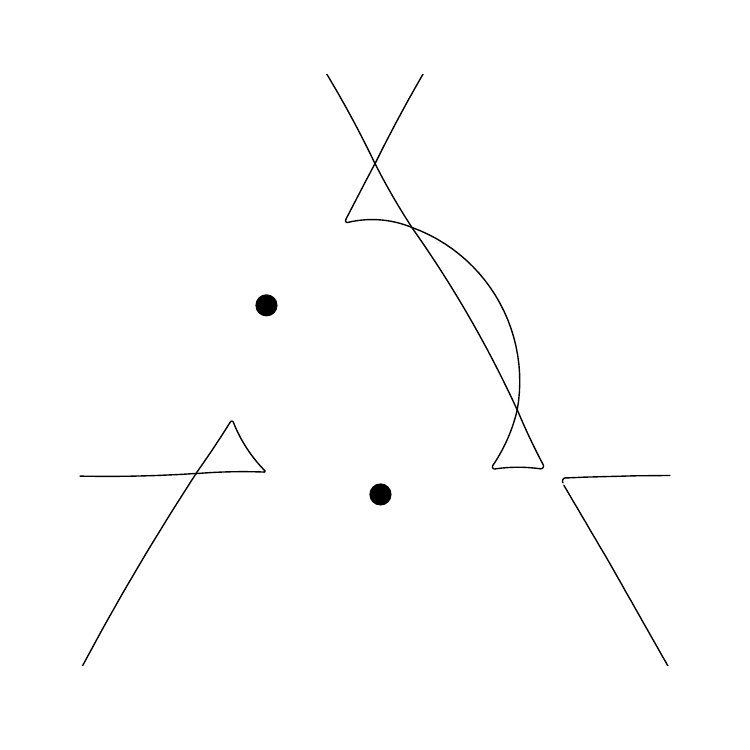}
\end{tabular}\\
\begin{tabular}{cccc}
\includegraphics[width=1.1in]{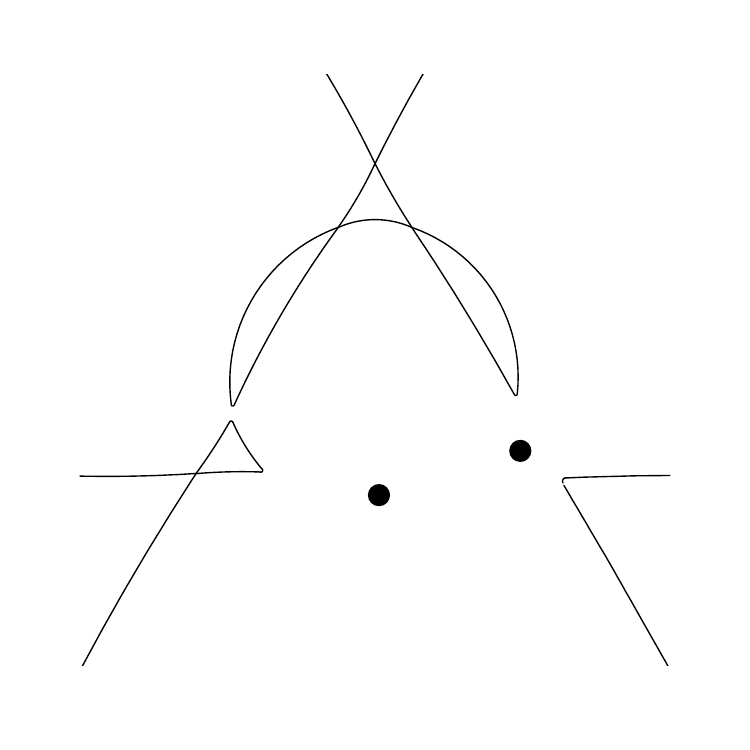}
&
\includegraphics[width=1.1in]{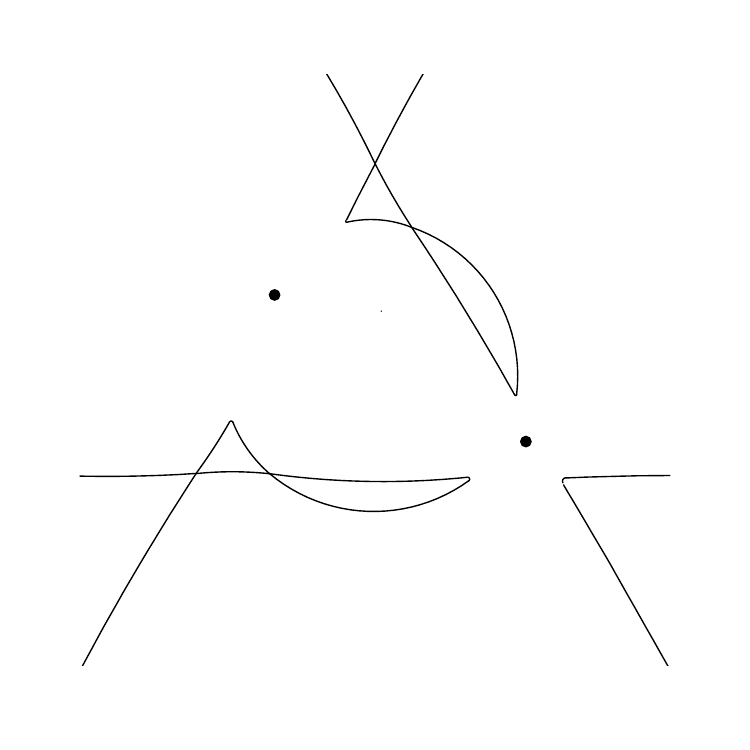}
&
\includegraphics[width=1.1in]{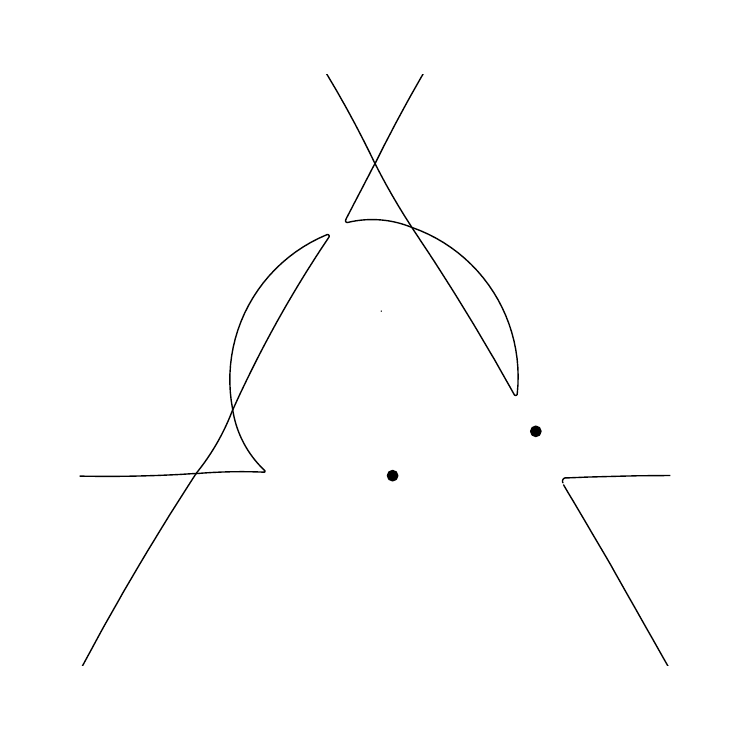}
&
\includegraphics[width=1.1in]{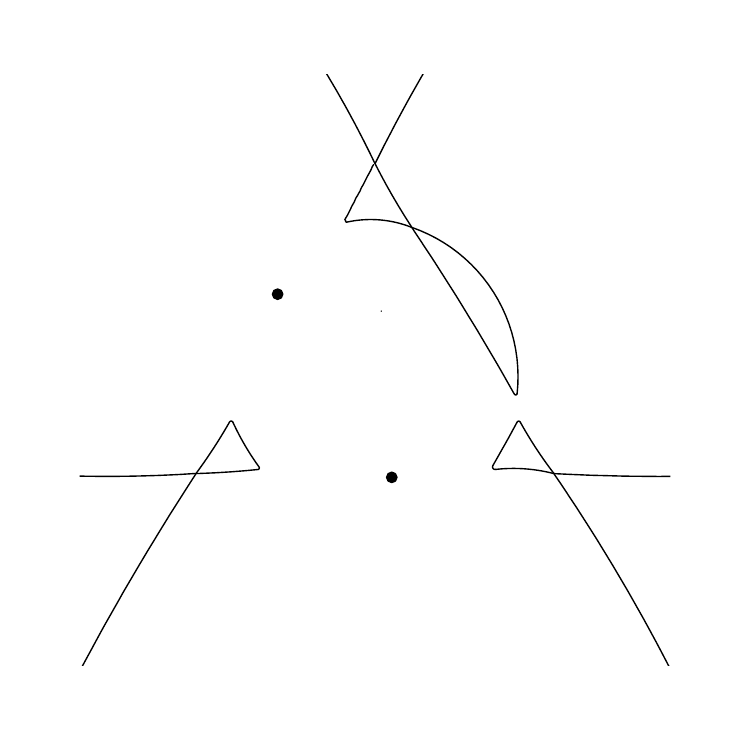}
\\
\end{tabular}
\caption{Rigid isotopy classes of maximally perturbable nodal rational curves of degree~$5$ in~$\RPP$ with exactly two isolated nodes.}
\label{fig:b03}
\end{figure}

\begin{figure}[h] 
\centering
\begin{tabular}{cc}
\includegraphics[width=1.2in]{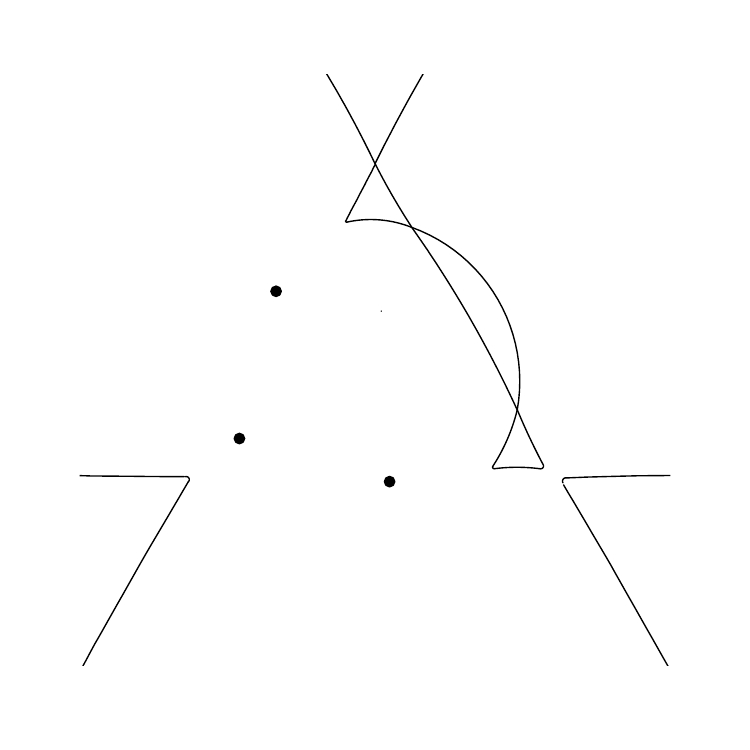}
&
\includegraphics[width=1.2in]{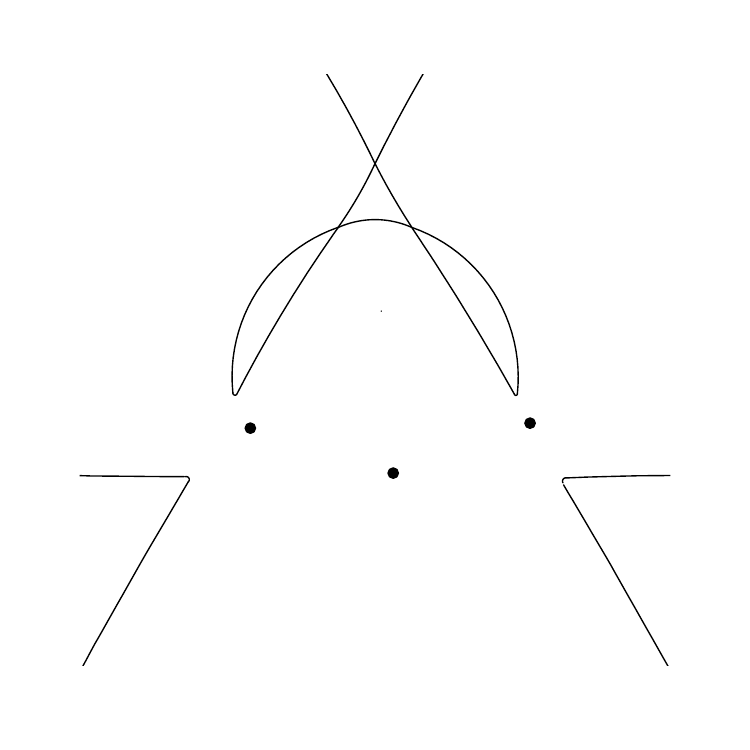}
\\
\end{tabular}\\
\begin{tabular}{ccc}
\includegraphics[width=1.2in]{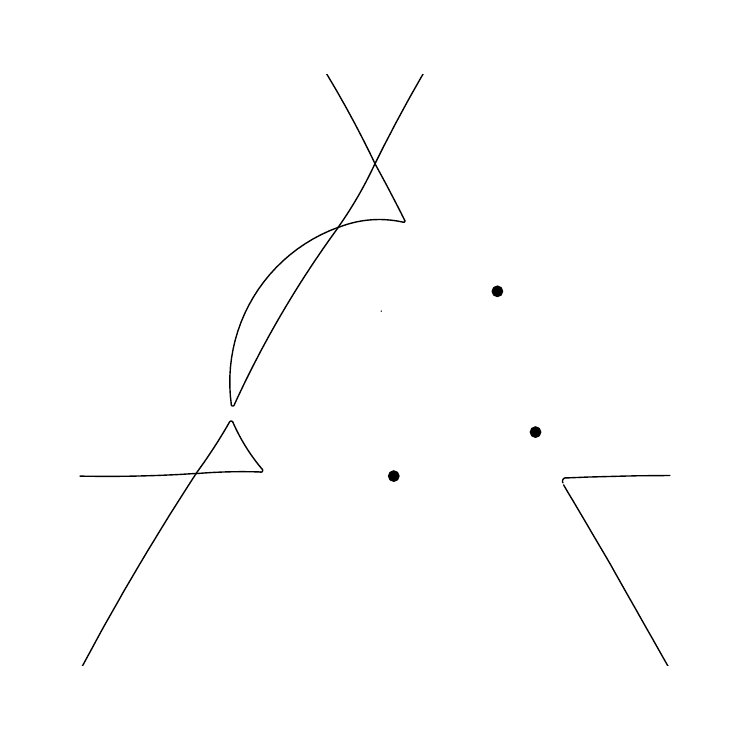}
&
\includegraphics[width=1.2in]{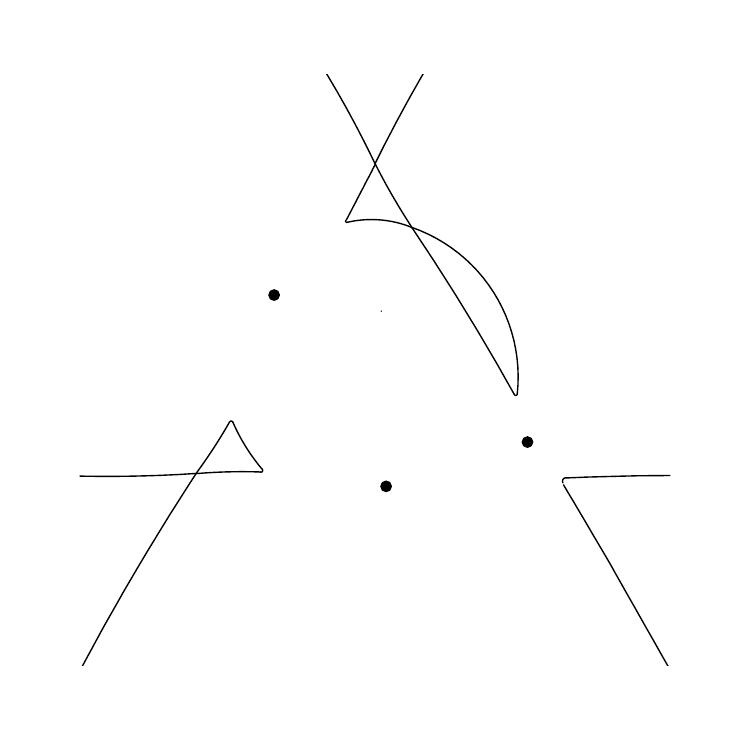}
&
\includegraphics[width=1.2in]{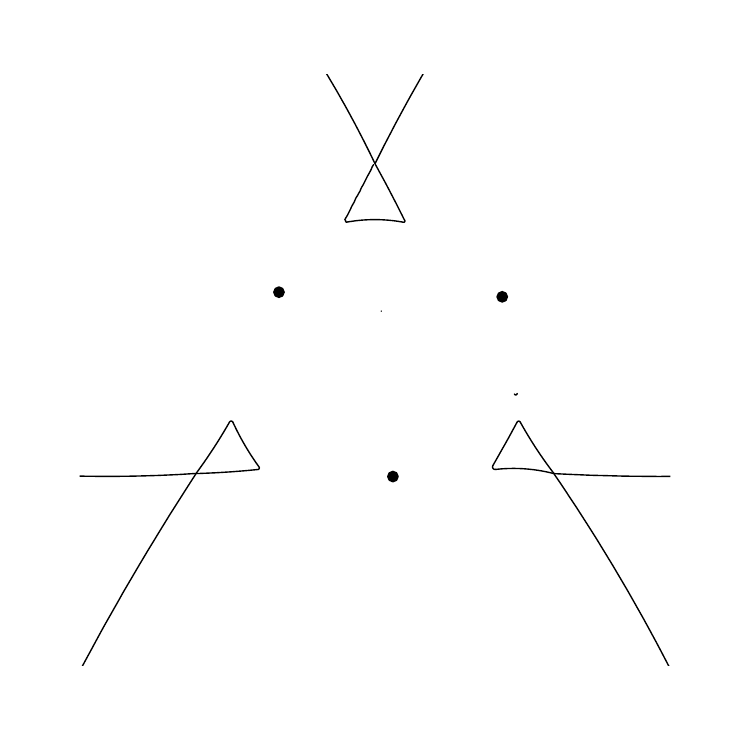}
\\
\end{tabular}
\caption{Rigid isotopy classes of maximally perturbable nodal rational curves of degree~$5$ in~$\RPP$ with exactly three isolated nodes.}
\label{fig:b04}
\end{figure}

\begin{figure}[h] 
\centering
\begin{tabular}{cc}
\includegraphics[width=1.2in]{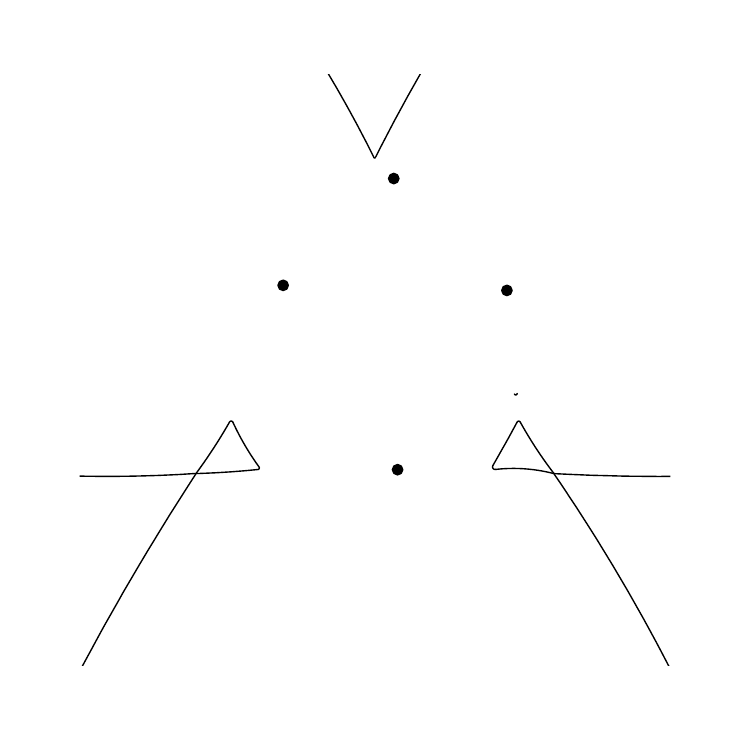}
&
\includegraphics[width=1.2in]{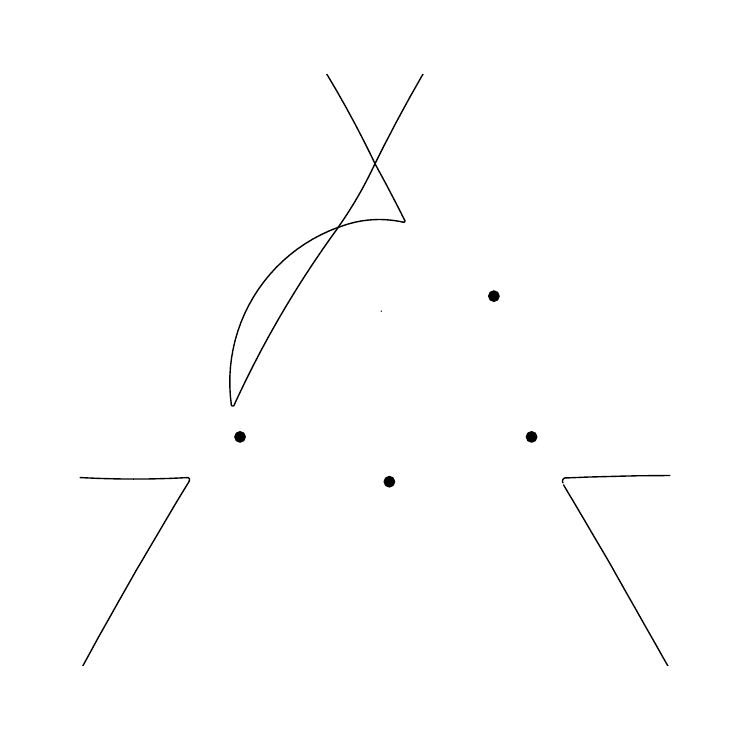}
\\
\end{tabular}
\caption{Rigid isotopy classes of maximally perturbable nodal rational curves of degree~$5$ in~$\RPP$ with exactly four isolated nodes.}
\label{fig:b05}
\end{figure}

\begin{figure}[h] 
\centering
\includegraphics[width=1.2in]{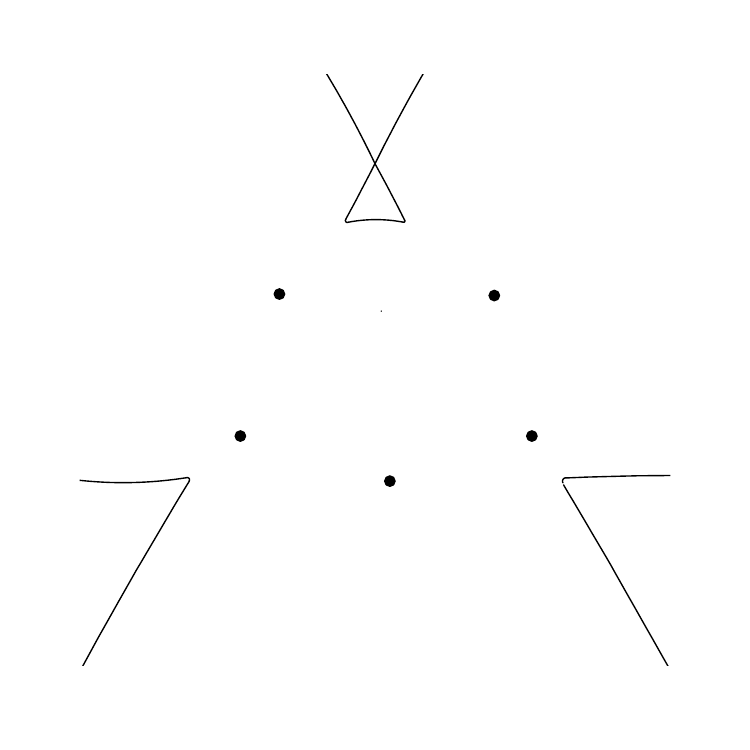}
\caption{The rigid isotopy class of nodal rational maximally perturbable curves of degree~$5$ in~$\RPP$ with exactly five isolated nodes.}
\label{fig:b06}
\end{figure}

\begin{prop}\label{prop:maxell}
 Let~$C$ be a maximally perturbable nodal rational curve of degree~$5$ in $\RPP$ such that the curve~$C$ has no hyperbolic nodal points. Then, there is only one rigid isotopy class of such a curve, represented in Figure~\ref{fig:maxellip}.
\end{prop}

\begin{proof}
Since~$C$ is a maximally perturbable nodal rational curve, by Lemma~\ref{lm:maxnodes} it is maximal, hence $e=6$.
Let $p\in\R C$ an elliptic nodal point of the curve~$C$ and let $(D_C,v_1,v_2)$ be the marked toile associated to $(C,p)$.
Since the base point~$p$ is elliptic, the nodes $n_{T}$ and $n_{\overline{T}}$ corresponding to the tangent lines of~$C$ at $p$ are a pair or complex conjugated nodal points. Therefore, the markings $v_1$ and $v_2$ are an inner nodal {\tvs} of $D_C$.
The toile $D_C$ has $5$ isolated nodal {\tvs} corresponding to the elliptic nodal points different from $p$.
Due to Proposition~\ref{prop:decomptoi}, since the toile $D_C$ has only isolated nodal {\tvs} it can be decomposed as the gluing of three cubic toiles by dotted {\gc} or by a solid axe. Since the toile $D_C$ is of type~$\mathrm{I}$, so are the glued cubic and the corresponding {\gc s} must be coherent with the labeling of the regions.
Then, after enumerating the equivalence classes of marked toiles with the aforementioned restrictions, it turns out that there are two equivalence classes having representatives~$D_+$ and~$D_-$, respectively (see Figure~\ref{fig:maxellip}).
The toile~$D_+$ has a type~$\mathrm{I}$ perturbation having~$3$ negative ovals and~$2$ positive ovals. Moreover, the toile~$D_+$ has a type~$\mathrm{I}$ non-singular perturbation.
The toile~$D_-$ has a type~$\mathrm{I}$ perturbation having~$2$ negative ovals and~$3$ positive ovals.
Therefore, we can conclude that the toile~$D_+$ (resp., $D_-$) corresponds to a pair $(C,p)$ such that the elliptic nodal point~$p$ can be perturbed to a positive (resp., negative) oval.
Hence, there is only one class of rigid isotopy of the curve~$C$. 
\end{proof}

\begin{figure}[h] 
\centering
\begin{subfigure}{\linewidth}
\centering 
\includegraphics[scale=0.9]{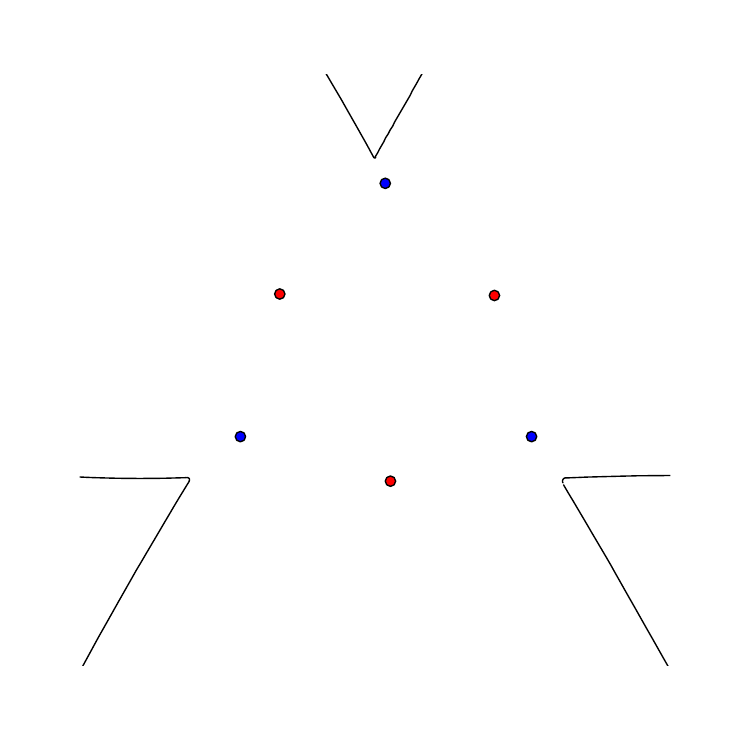}
\caption{The rigid isotopy class of maximally perturbable nodal rational curves of degree~$5$ in~$\RPP$ with exactly six isolated nodes.}
\label{fig:maxellip}
\end{subfigure}
\begin{subfigure}{0.4\linewidth} 
\centering
\includegraphics[width=1.7in]{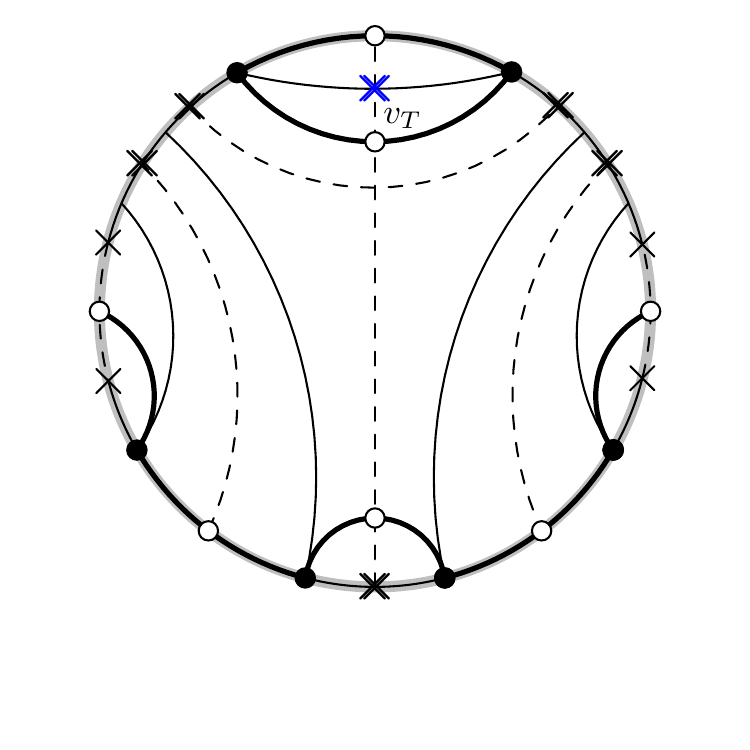}
\caption{The marked dessin $D_-$, associated to a negative elliptic nodal point as the base point.}
\end{subfigure}
\begin{subfigure}{0.4\linewidth} 
\centering
\includegraphics[width=1.7in]{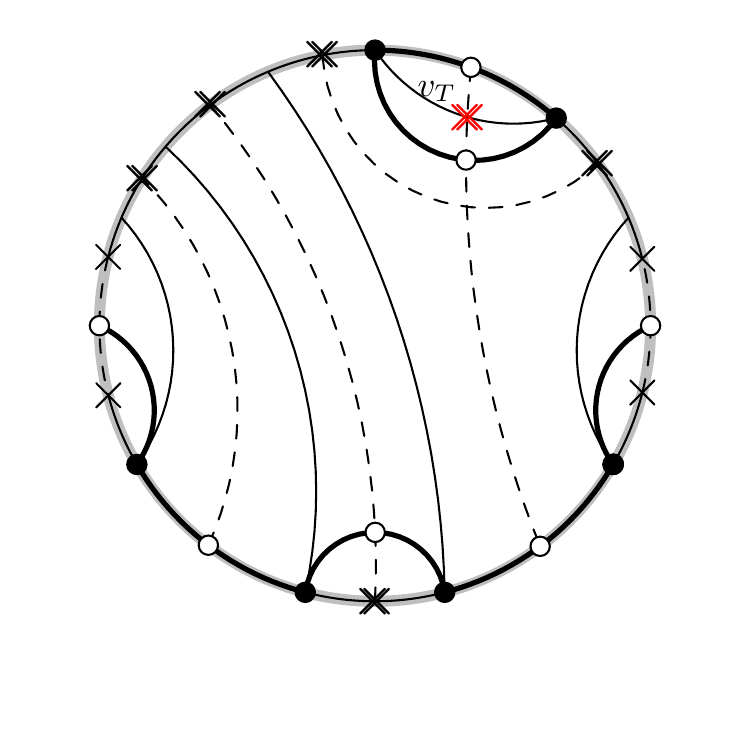}
\caption{The marked dessin $D_+$, associated to a positive elliptic nodal point as the base point.}
\end{subfigure}
\caption{}
\end{figure}

\begin{thm}[Rigid isotopy classification of maximally perturbable real quintic rational curves]
 If~$C$ is a maximally perturbable nodal rational curve of degree~$5$ in~$\RPP$, then, the curve~$C$ belongs to one of the different rigid isotopy classes represented in Figures~\ref{fig:b01} to~\ref{fig:maxellip}.
\end{thm}

\begin{proof}
The case when $h\geq1$ is considered in Proposition~\ref{prop:maxhyp} and the case when $h=0$ is considered in Proposition~\ref{prop:maxell}.
\end{proof}

\subsection{$(M-2)$-perturbable curves}

\begin{df} \label{df:Mdpert}
A real type~$\mathrm{I}$ nodal plane curve~$C$ of degree~$d$ in $\RPP$ is \emph{$(M-s)$-perturbable} if its type~$\mathrm{I}$ perturbation 
has $\frac{(d-1)(d-2)}{2}+1-s$ connected components of the real point set. 
\end{df}

Recall that a non-singular $(M-2)$-quintic $C_0$ has four ovals, which may or not be in a convex position, depending on whether $\sigma(C_0)=2$ or $\sigma(C_0)=0$, respectively. 
In the case when $\sigma(C_0)=2$, the curve $C_0$ has four ovals forming a quadrangle. As before, if we fix a point~$p$ at the interior of an oval, the pencil of lines passing for~$p$ induces an order on the ovals. Varying the fix point, we can assign to every oval two neighboring ones.
In the case when $\sigma(C_0)=0$, the curve $C_0$ has four ovals which are not in convex position, namely, three of the ovals form a triangle which does not intersect the pseudoline in which lies the forth oval.
If we fix a point~$p$ at the interior of the central oval and an orientation of the central oval, the pencil of lines passing through~$p$ induces a cyclic order on the three exterior ovals depending on the orientation given to the central oval.
In both cases, if a type~$\mathrm{I}$ $(M-2)$-perturbable curve~$C$ has an elliptic nodal point, the oval produce by it in a type~$\mathrm{I}$ perturbation of~$C$ must respect its relative position with respect to its neighboring ovals, and so does the elliptic nodal point of the original curve (see Figure~\ref{fig:nonsingM2}).

\begin{figure}[h] 
\centering
\begin{tabular}{cc}
\includegraphics[width=1.8in]{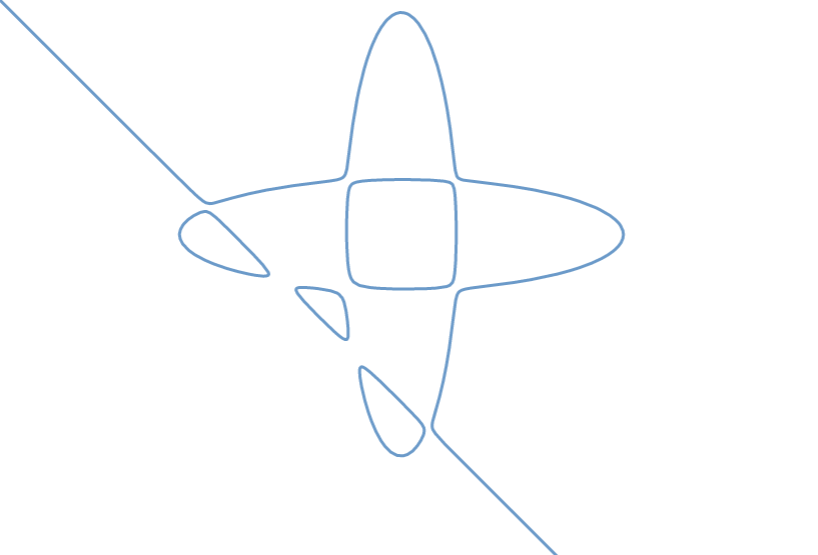}
&
\includegraphics[width=1.2in]{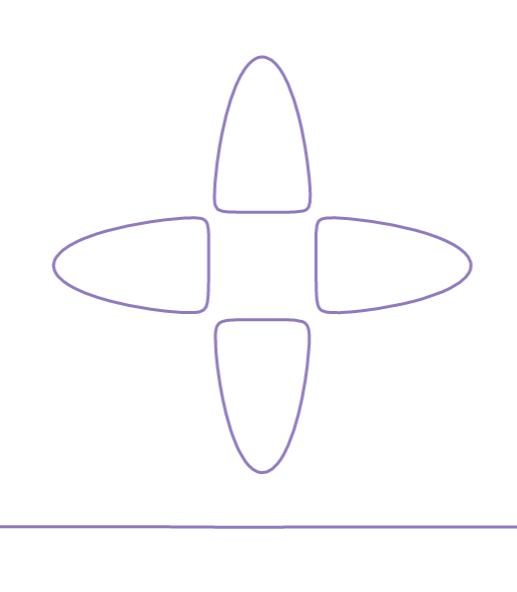}
\end{tabular}
\caption{Rigid isotopy classes of non-singular $(M-2)$-quintic curves in $\RPP$.}
\label{fig:nonsingM2}
\end{figure}

\begin{lm}\label{lm:h3}
Let~$C$ be a nodal rational $M$-curve of degree~$5$ in $\RPP$. If~$C$ is $(M-2)$-perturbable, then $h\geq3$.
\end{lm}

\begin{proof} 
For a nodal rational $M$-curve of degree~$5$ in $\RPP$, we have that $e+h=6$. Since~$C$ is $(M-2)$-perturbable, its type~$\mathrm{I}$ perturbation has $4$ ovals and a pseudoline component. If $e\geq4$, since each elliptic nodal point gives rise to an oval, then $e=4$. If $h=1$ or $2$, then the real point set $\R C$ has a circle with $h$ self-intersections and its $\mathrm{I}$ perturbation produces $h$ extra ovals, contradicting the fact that~$C$ is $(M-2)$-perturbable.
\end{proof}

\begin{thm}
Let~$C$ be a nodal rational $M$-curve of degree~$5$ in $\RPP$. If~$C$ is $(M-2)$-perturbable, its rigid isotopy class is determined by its isotopy class.
\end{thm}

\begin{proof} 
Pick a hyperbolic nodal point $p\in\R C$ belonging to the pseudoline.
The corresponding marked toile $(D_C,v_1,v_2)$ associated to $(C,p)$ is a type~$\mathrm{I}$ toile with $7$ real nodal {\tvs}, among which two are markings.
By Corollary~\ref{coro:lpert}, the marked toile $(D_C,v_1,v_2)$ has a type~$\mathrm{I}$ perturbation with $3$ true ovals. 
Then, enumerating all equivalence classes of marked toiles satisfying the aforementioned restrictions and realizing the birational transformations in order to recover the associated curves $\R C\subset\RPP$ lead to the plane curves shown on Figures~\ref{fig:MM2I} and~\ref{fig:MM2II}.
\end{proof}

\begin{figure}[ht]
\centering
\begin{tabular}{ccccc}
\includegraphics[width=1in]{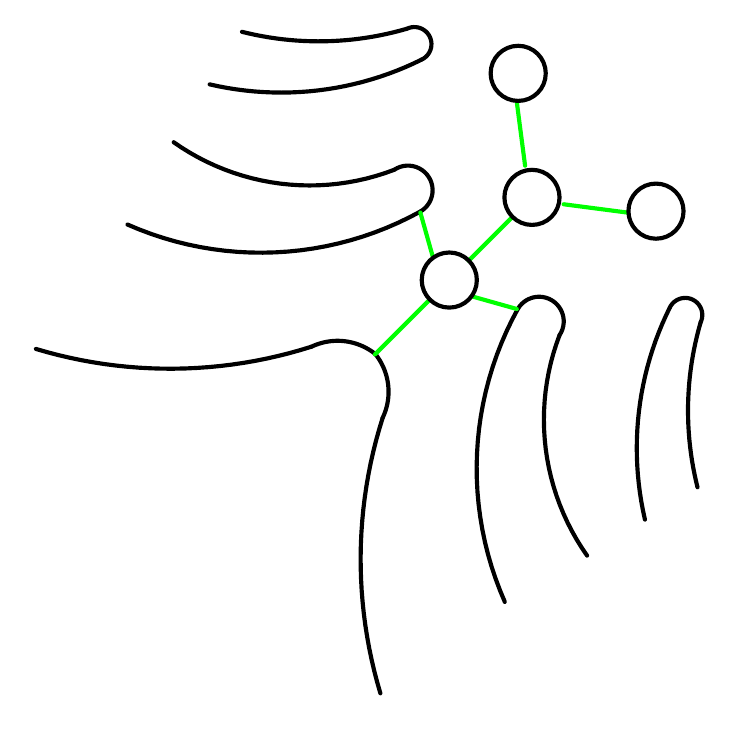}&
\includegraphics[width=1in]{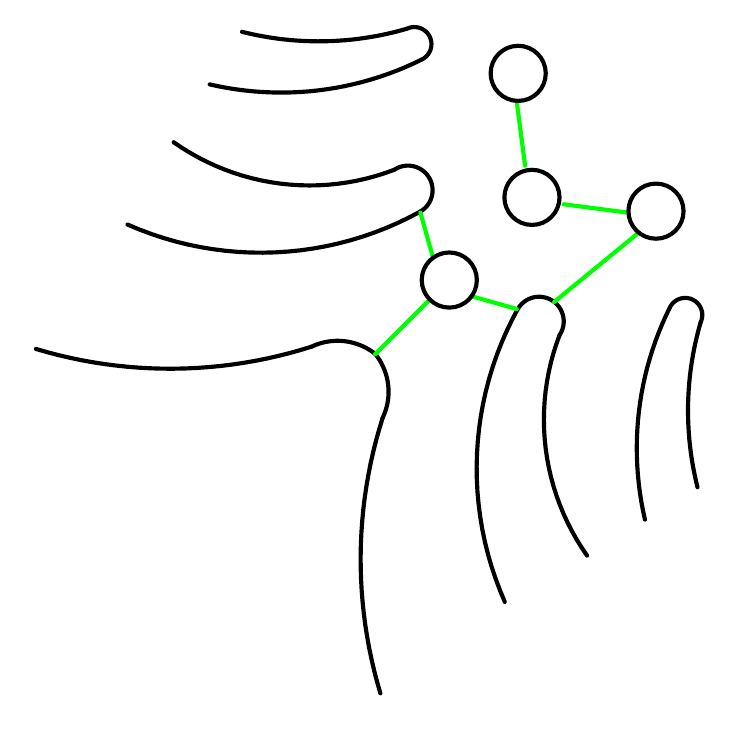}&
\includegraphics[width=1in]{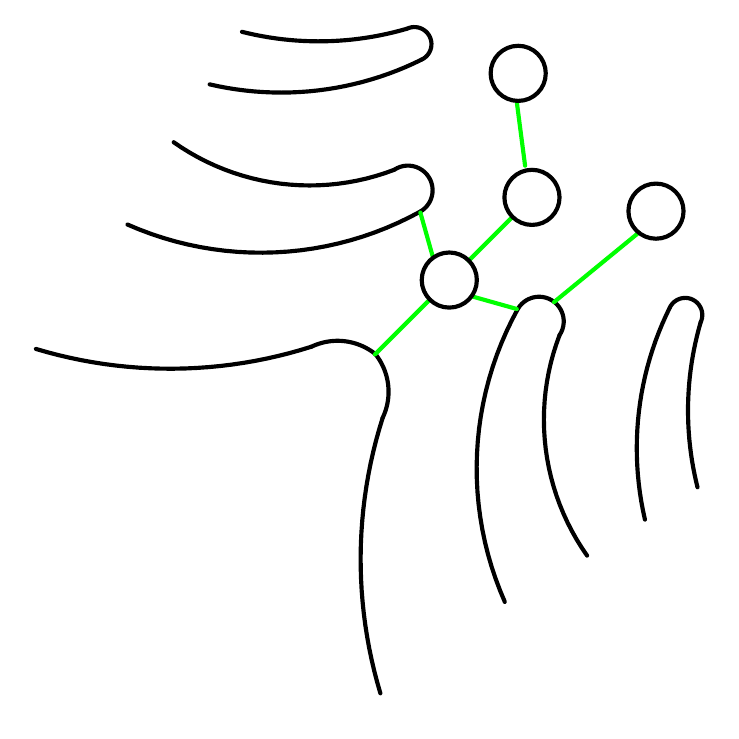}&
\includegraphics[width=1in]{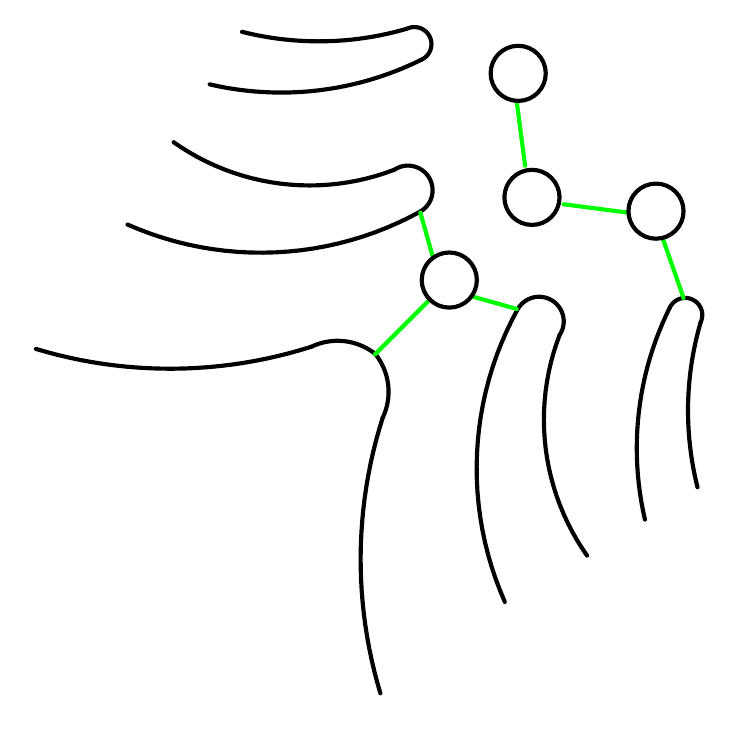}&
\includegraphics[width=1in]{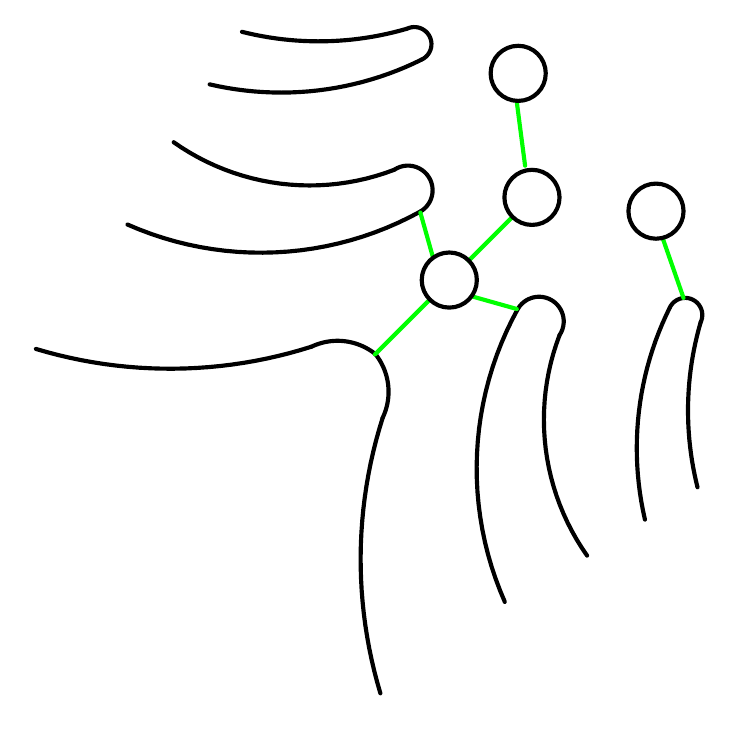}\\
\includegraphics[width=1in]{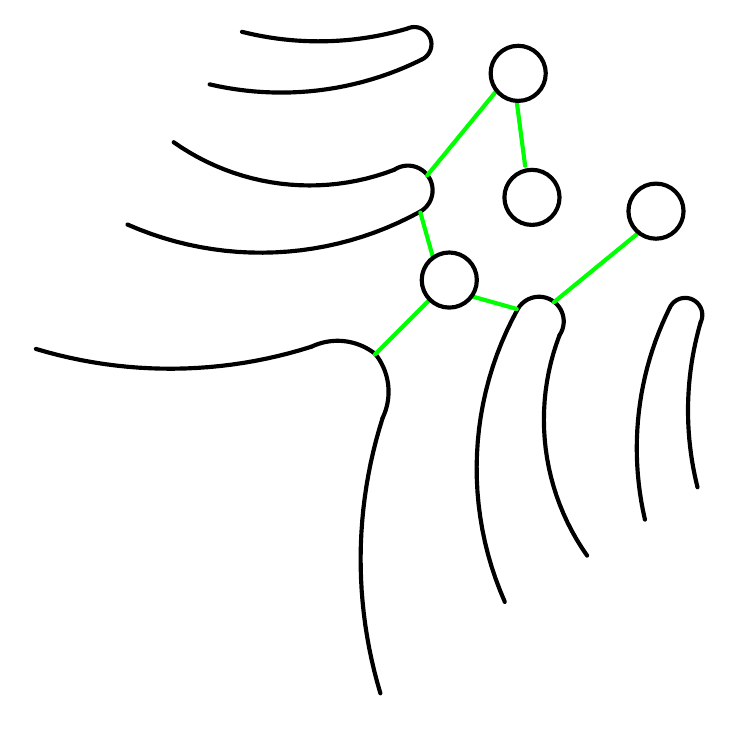}&
\includegraphics[width=1in]{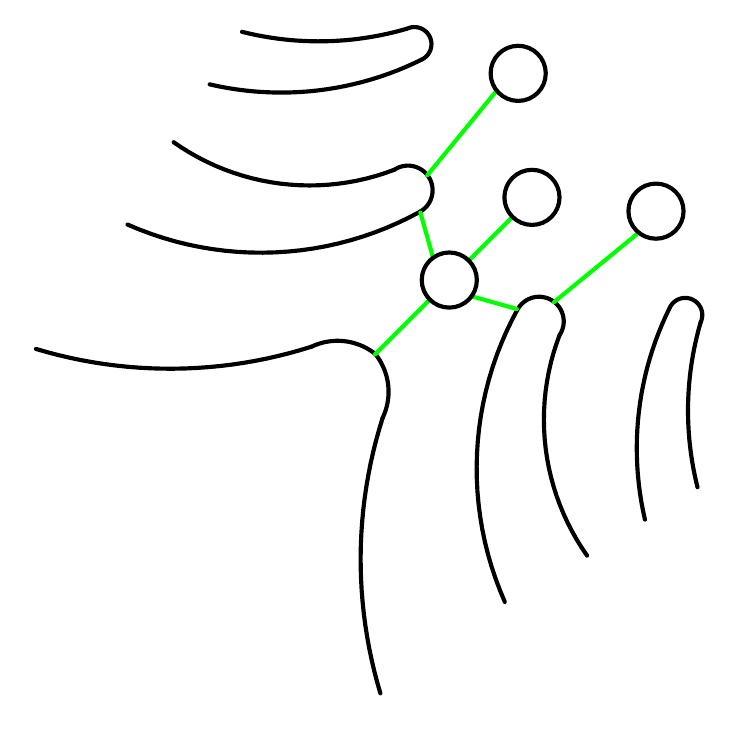}&
\includegraphics[width=1in]{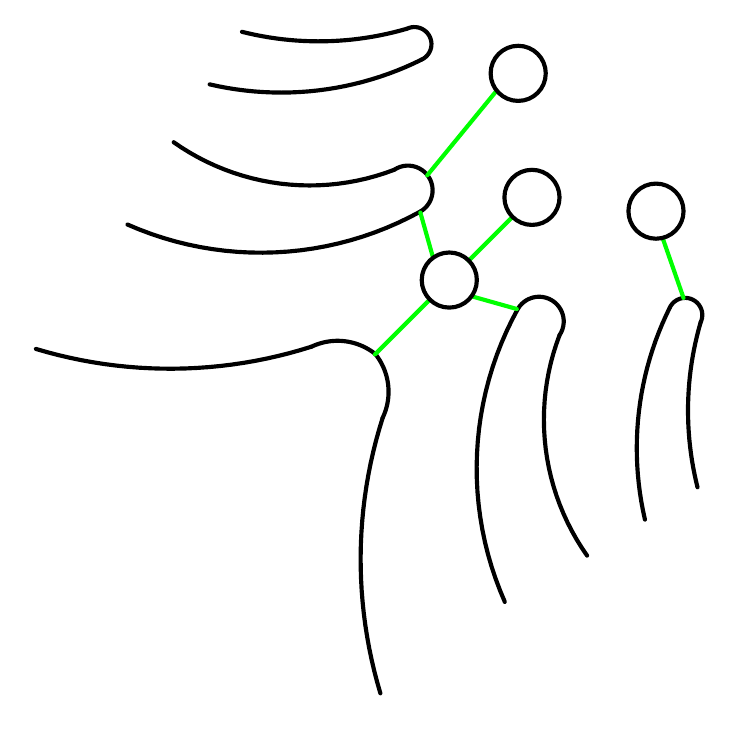}&
\includegraphics[width=1in]{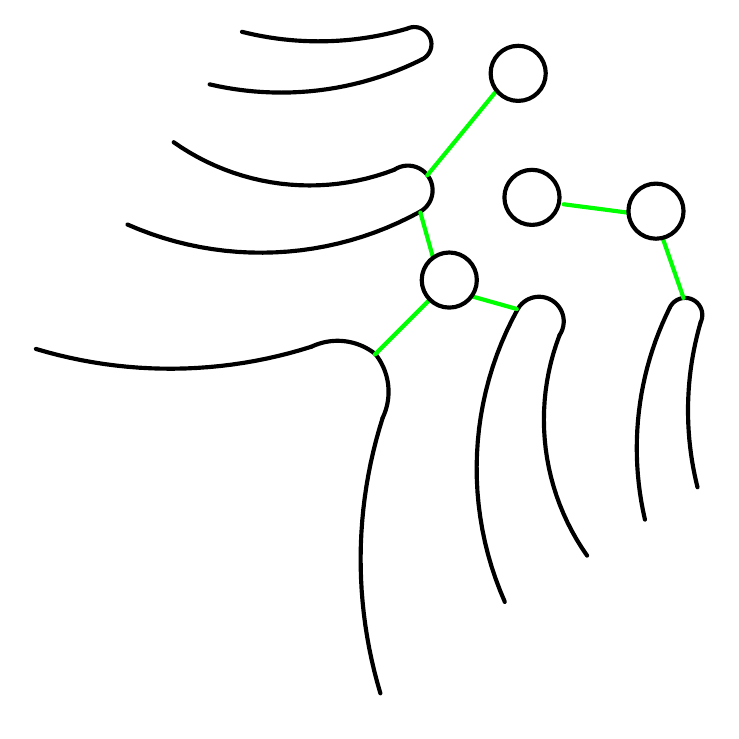}&
\includegraphics[width=1in]{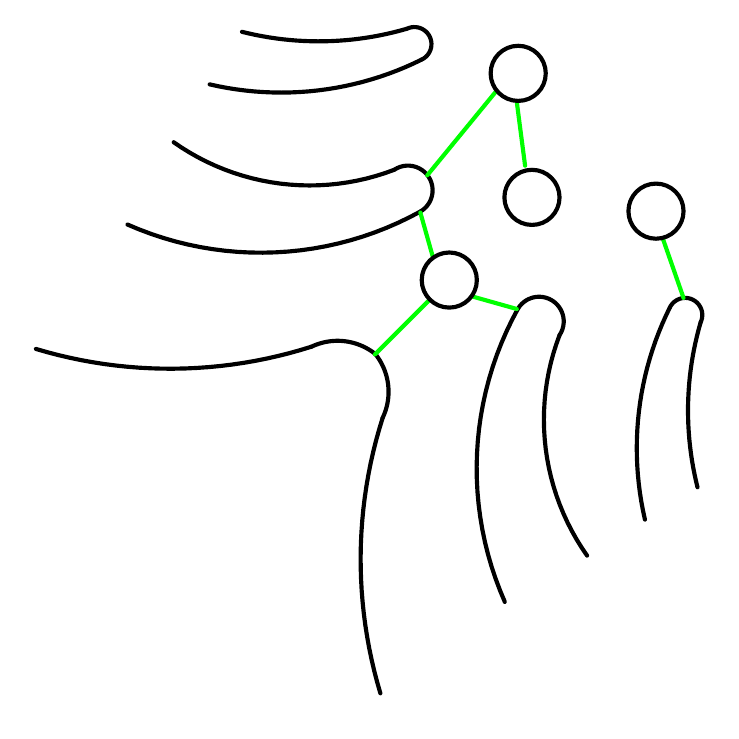}\\
\includegraphics[width=1in]{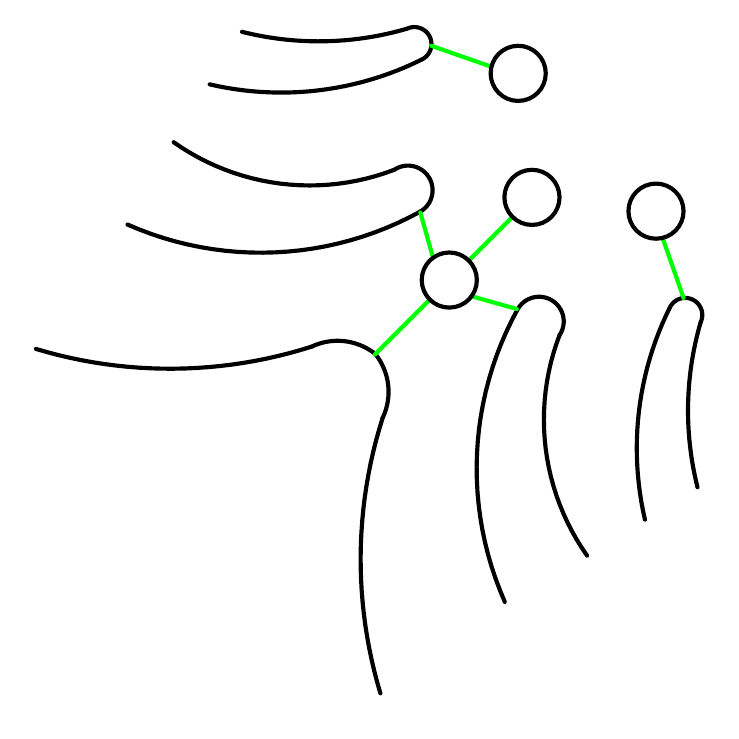}&
\includegraphics[width=1in]{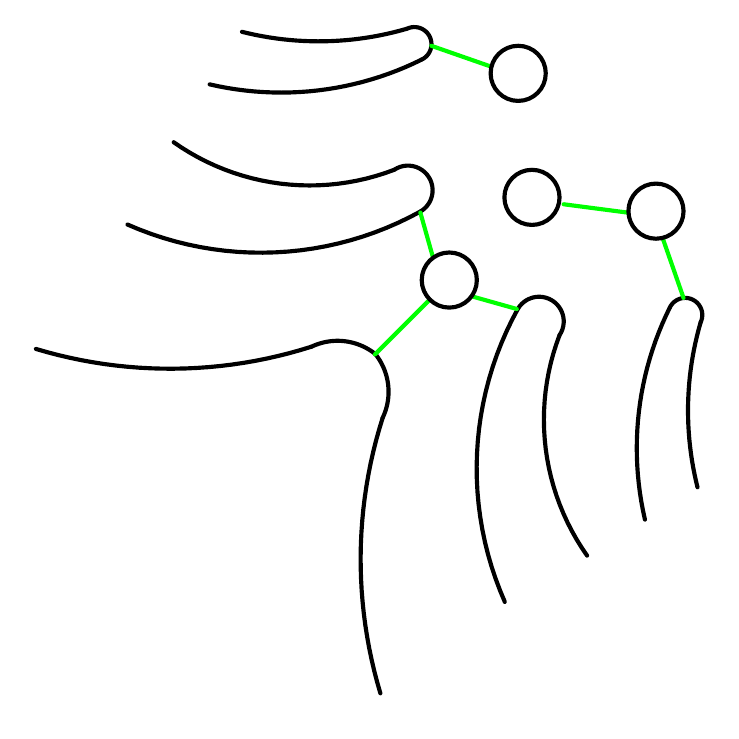}&
\includegraphics[width=1in]{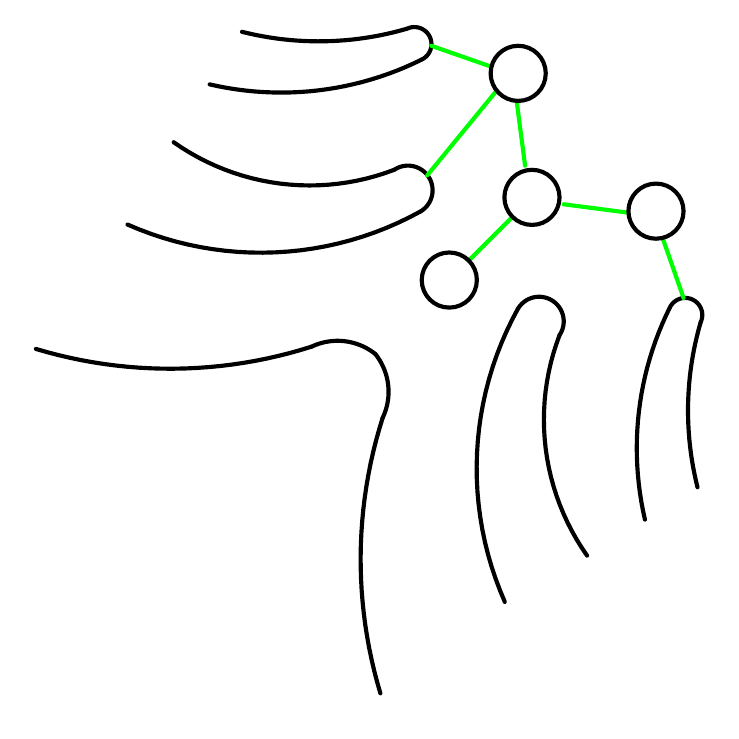}&
\includegraphics[width=1in]{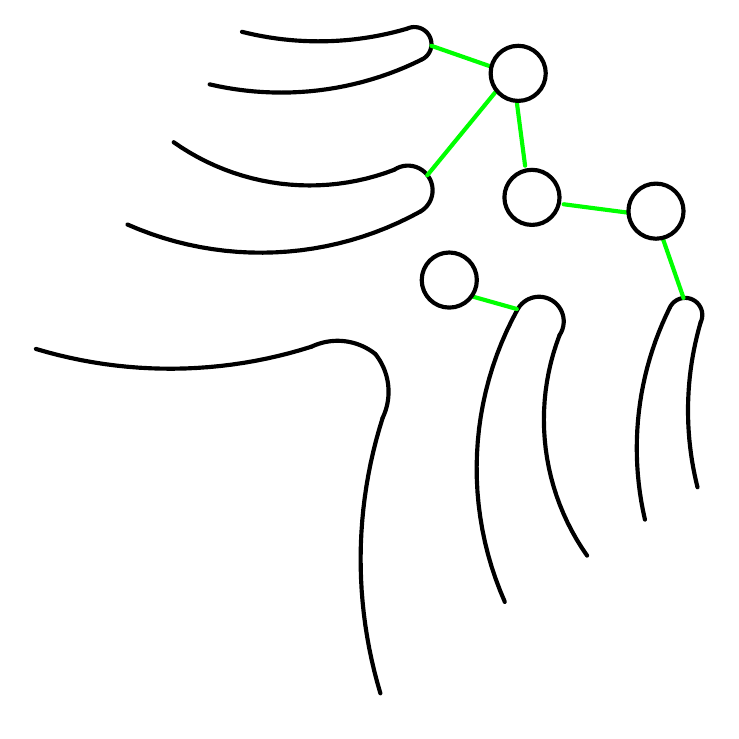}&
\includegraphics[width=1in]{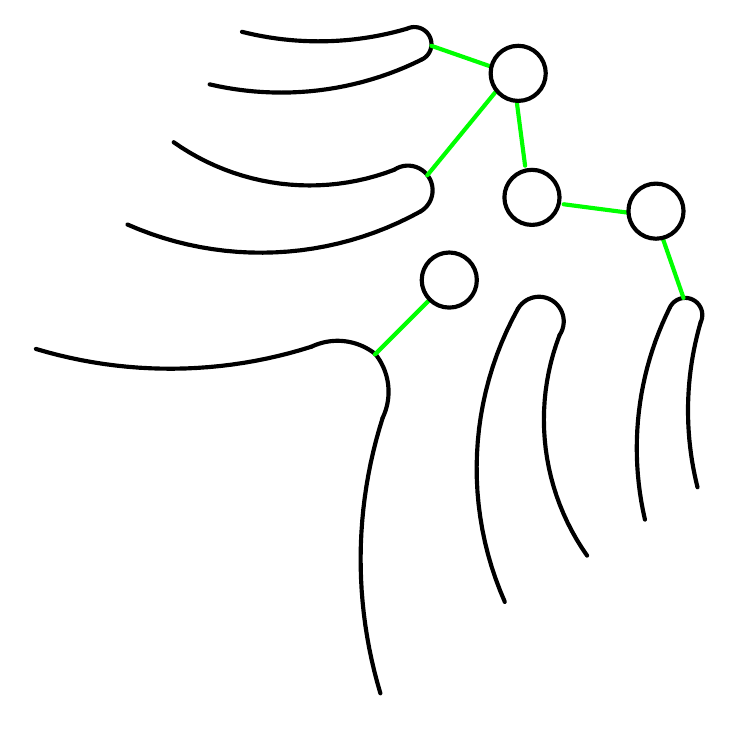}\\
\includegraphics[width=1in]{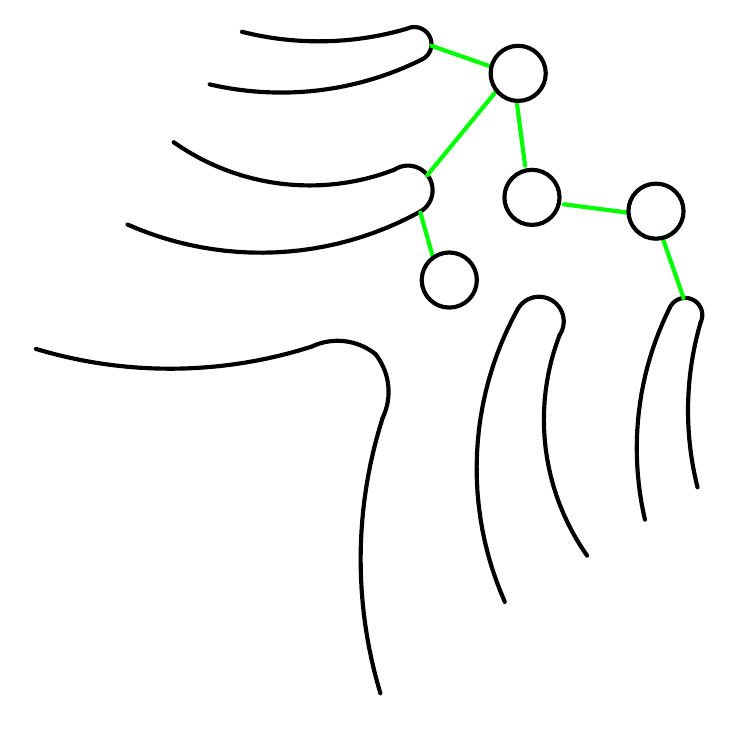}&
\includegraphics[width=1in]{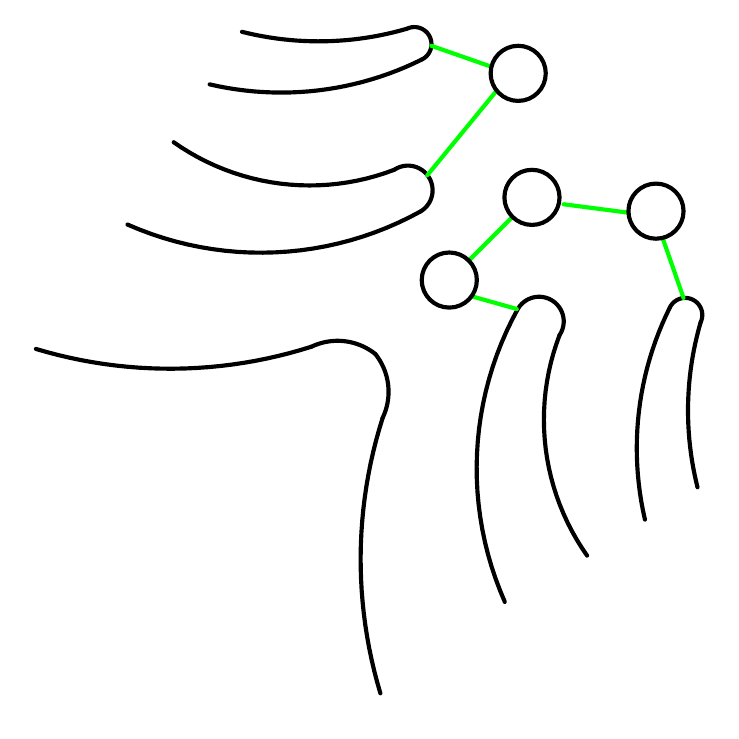}&
\includegraphics[width=1in]{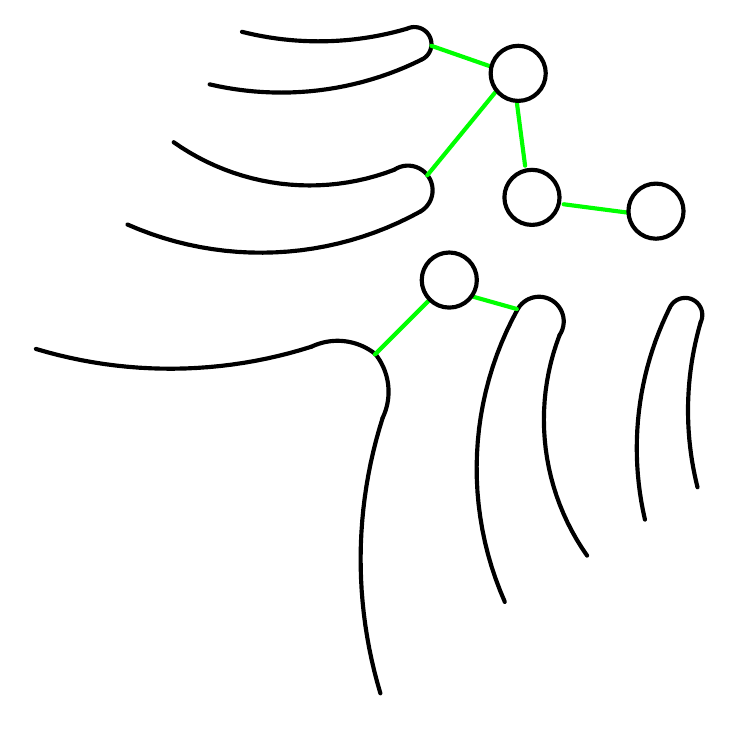}&
\includegraphics[width=1in]{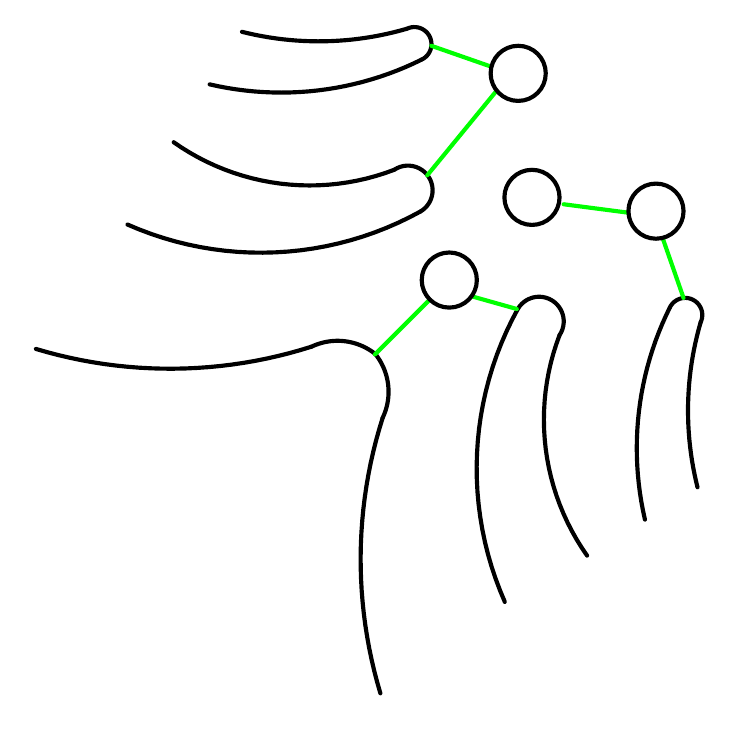}&
\includegraphics[width=1in]{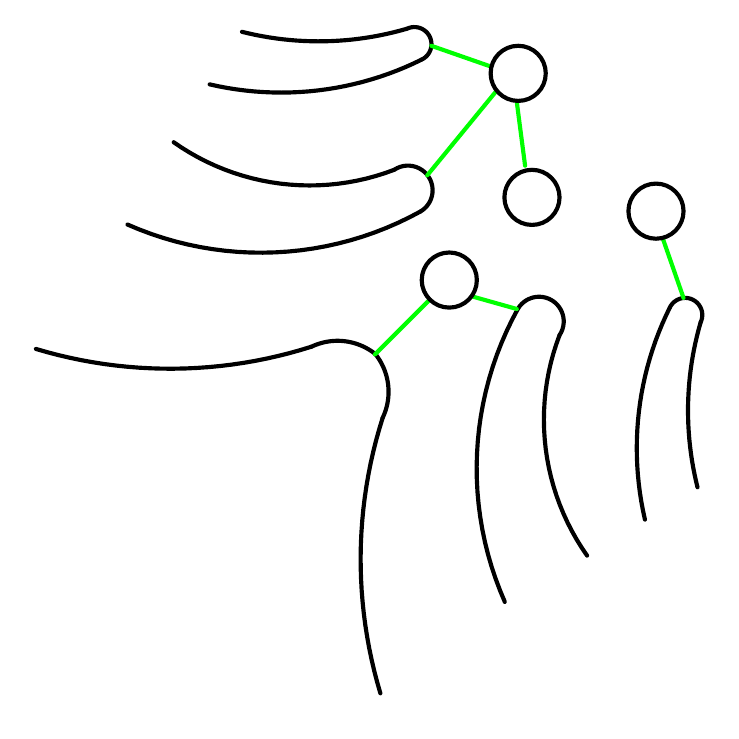}
\end{tabular}
\caption{Rigid isotopy classes of~$(M-2)$-perturbable nodal rational~$M$-curves in~$\RPP$. (Part I).}
\label{fig:MM2I}
\end{figure}

\begin{figure}[h]
\centering
\begin{tabular}{ccccc}
\includegraphics[width=1in]{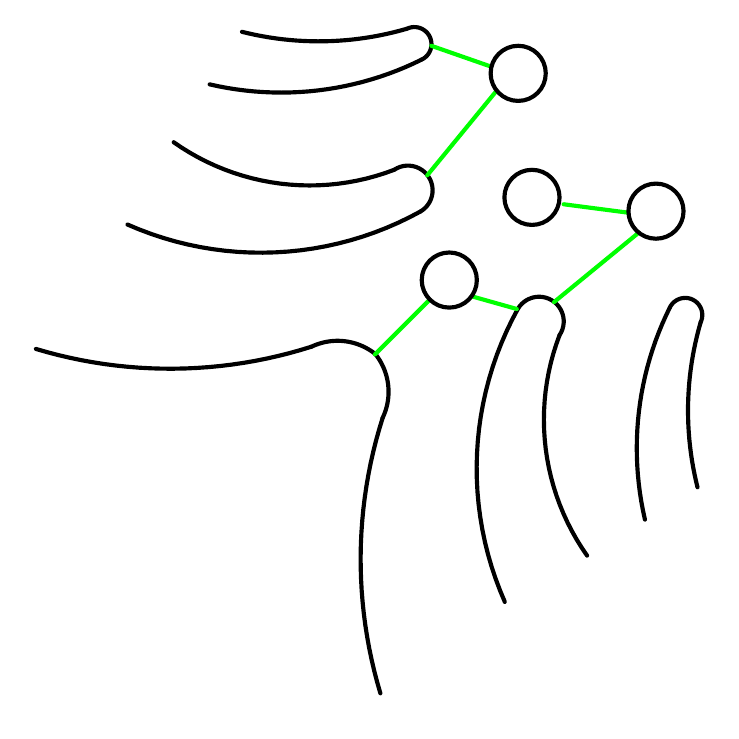}&
\includegraphics[width=1in]{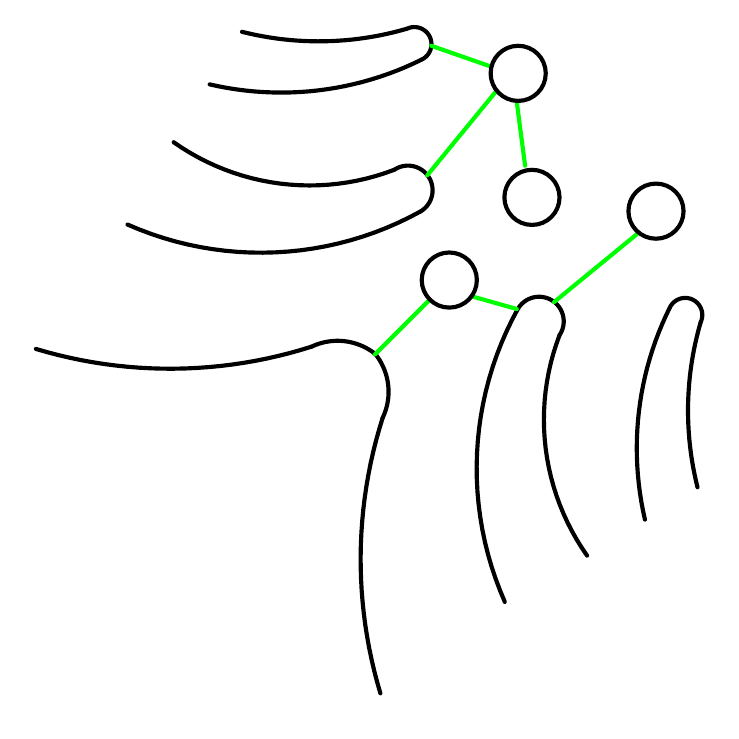}&
\includegraphics[width=1in]{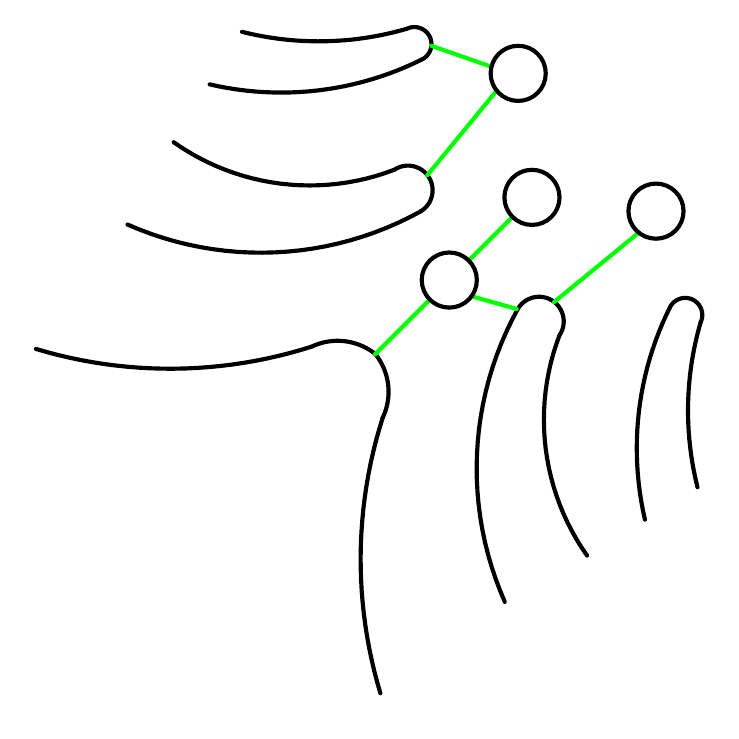}&
\includegraphics[width=1in]{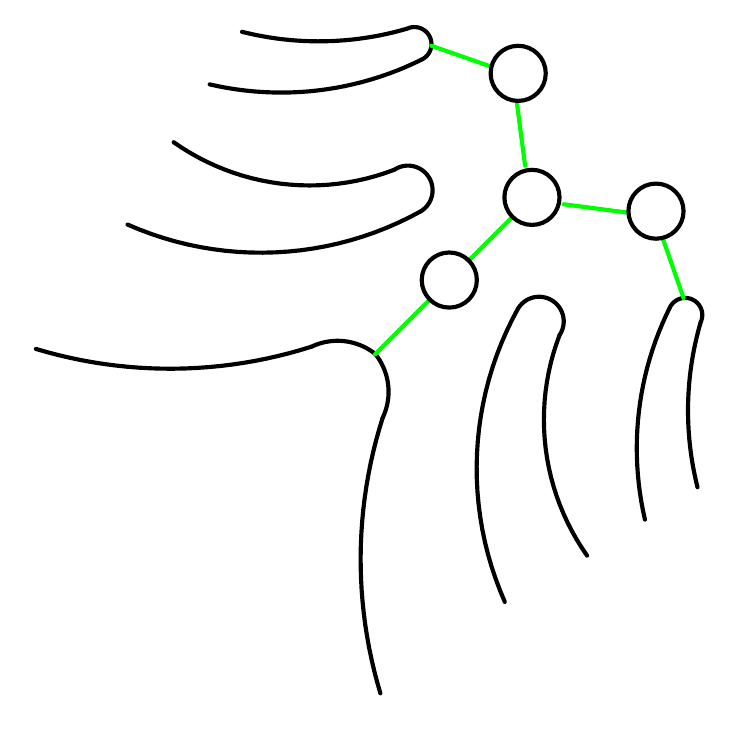}&
\includegraphics[width=1in]{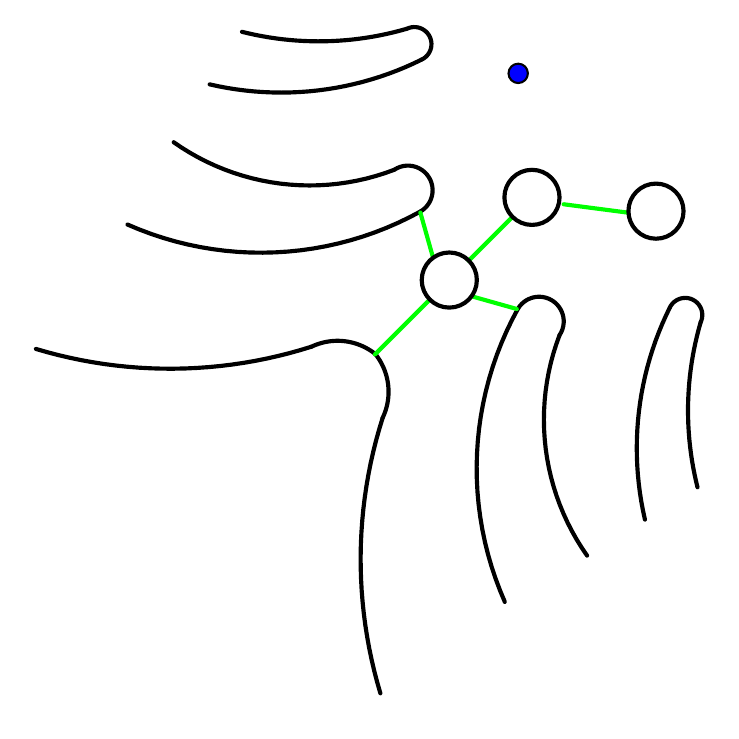}\\
\includegraphics[width=1in]{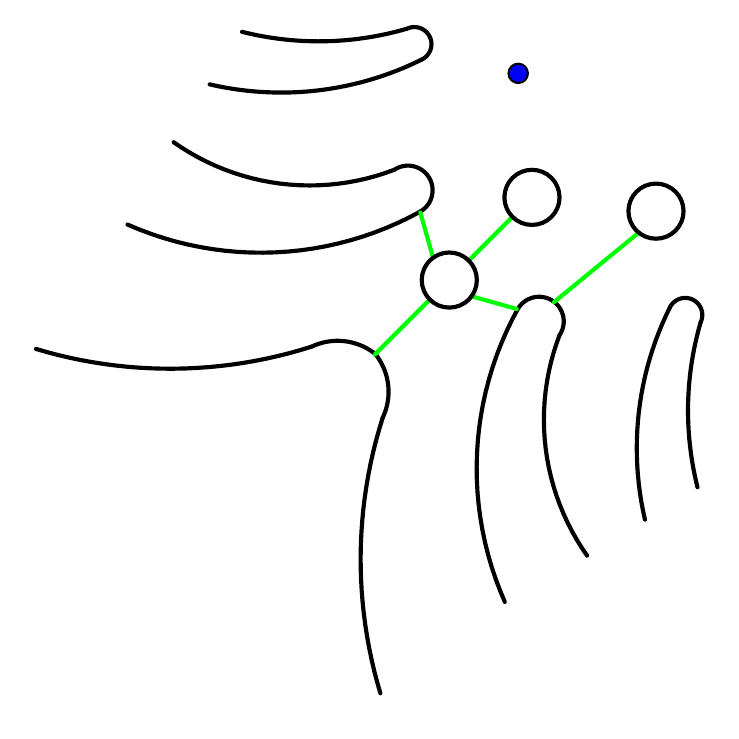}&
\includegraphics[width=1in]{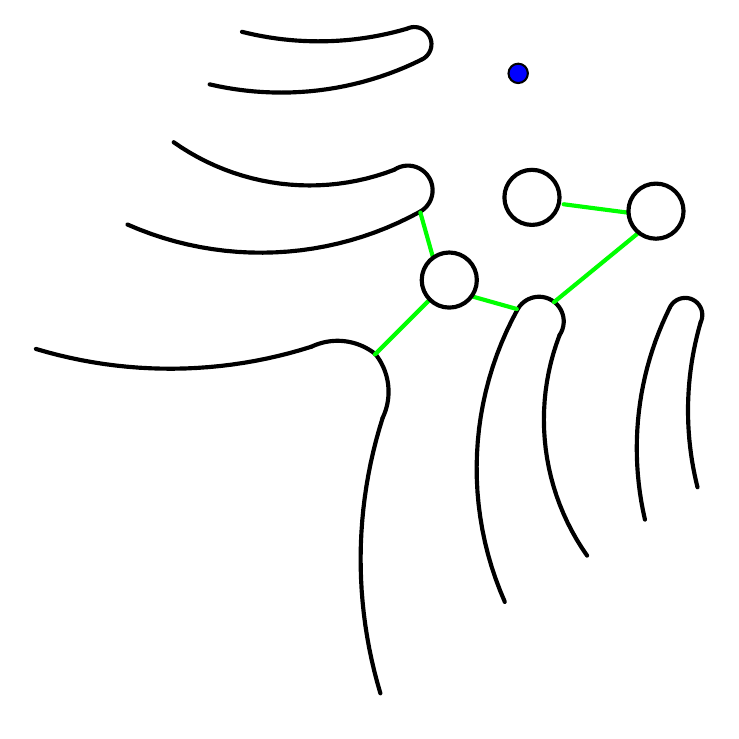}&
\includegraphics[width=1in]{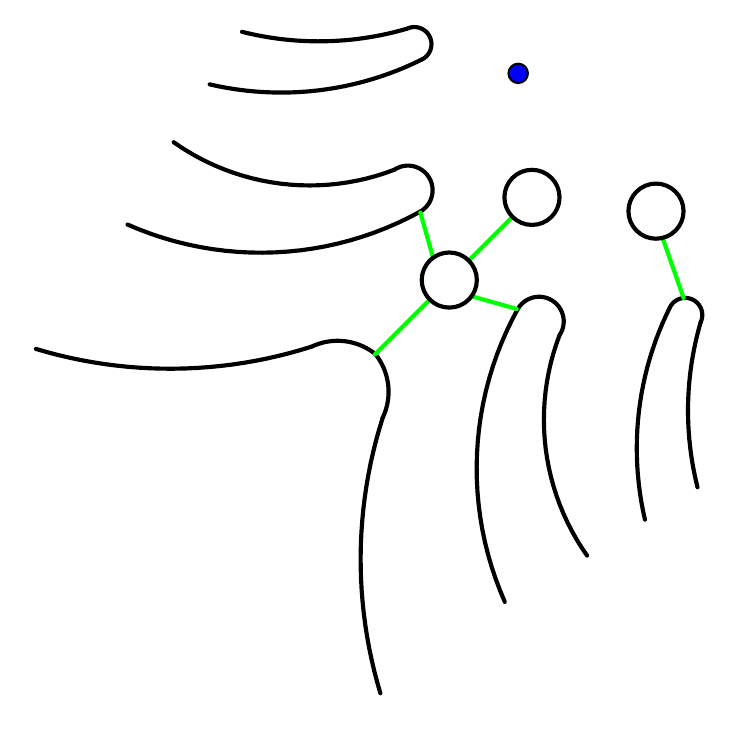}&
\includegraphics[width=1in]{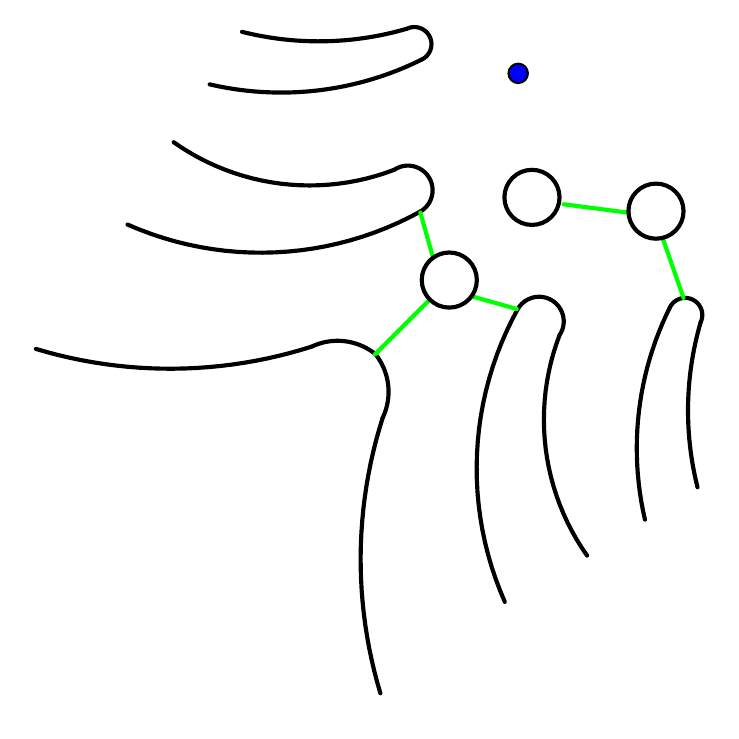}&
\includegraphics[width=1in]{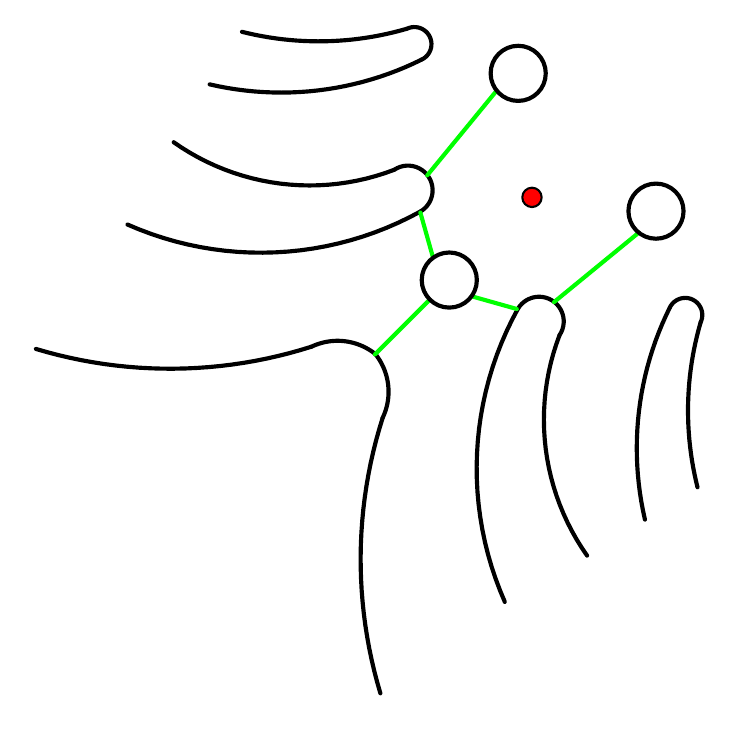}\\
\end{tabular}\\
\begin{tabular}{cccc}
\includegraphics[width=1in]{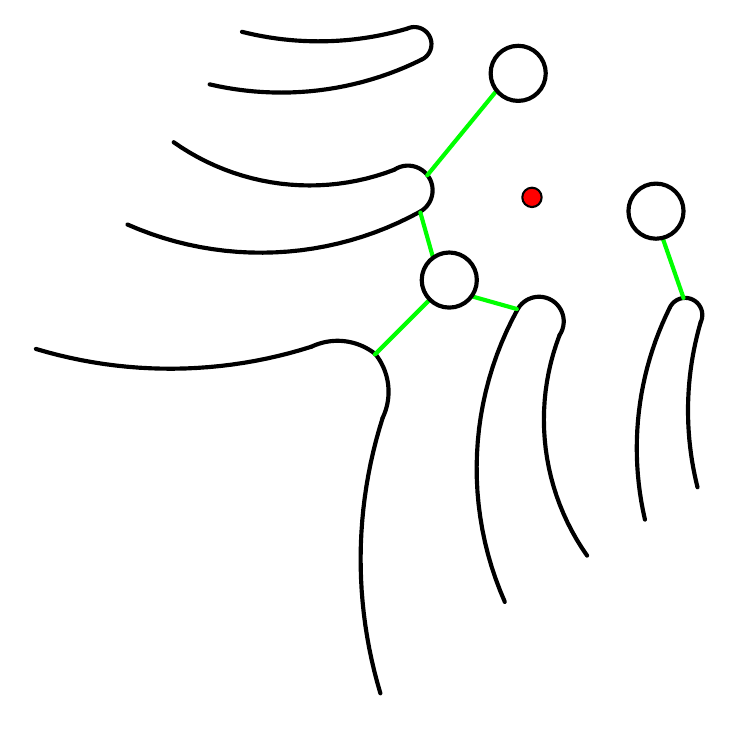}&
\includegraphics[width=1in]{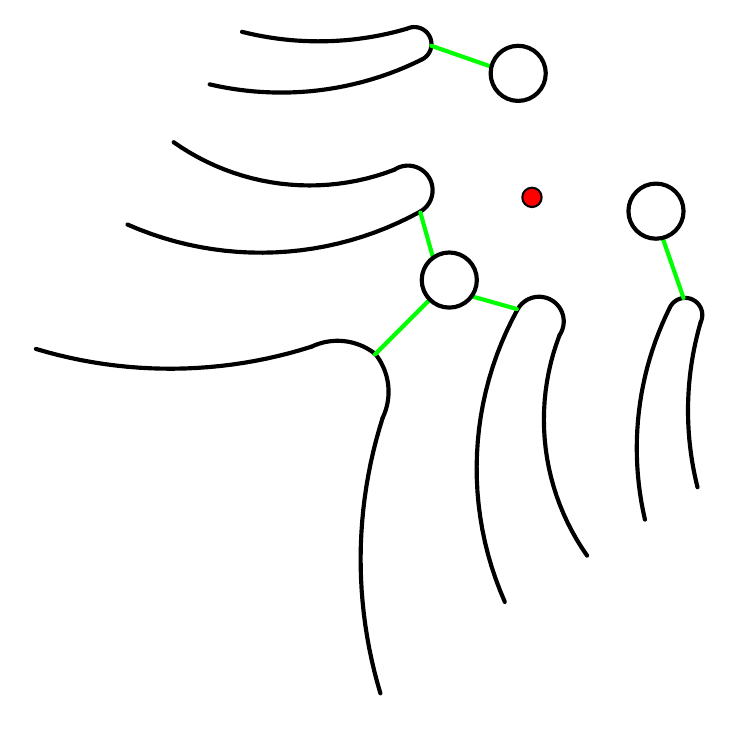}&
\includegraphics[width=1in]{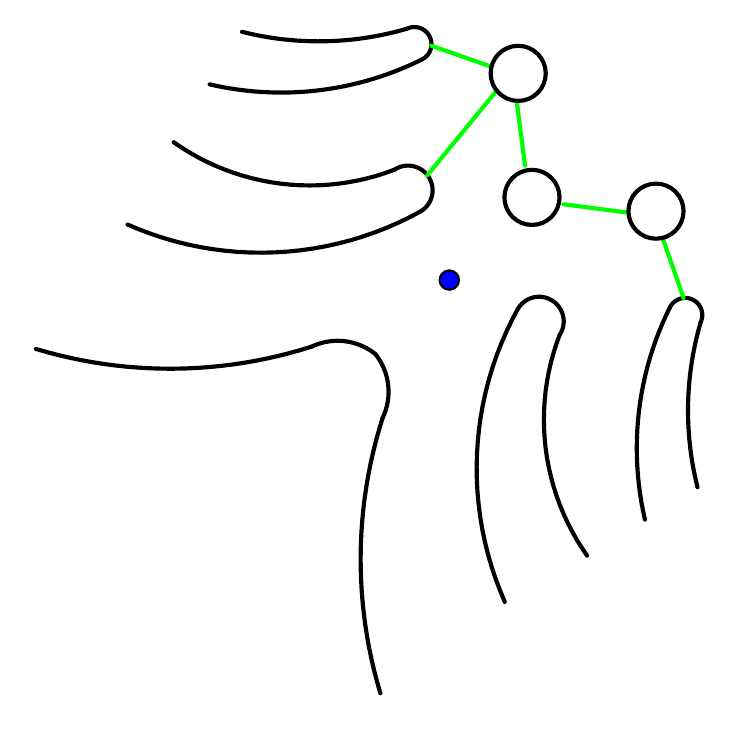}&
\includegraphics[width=1in]{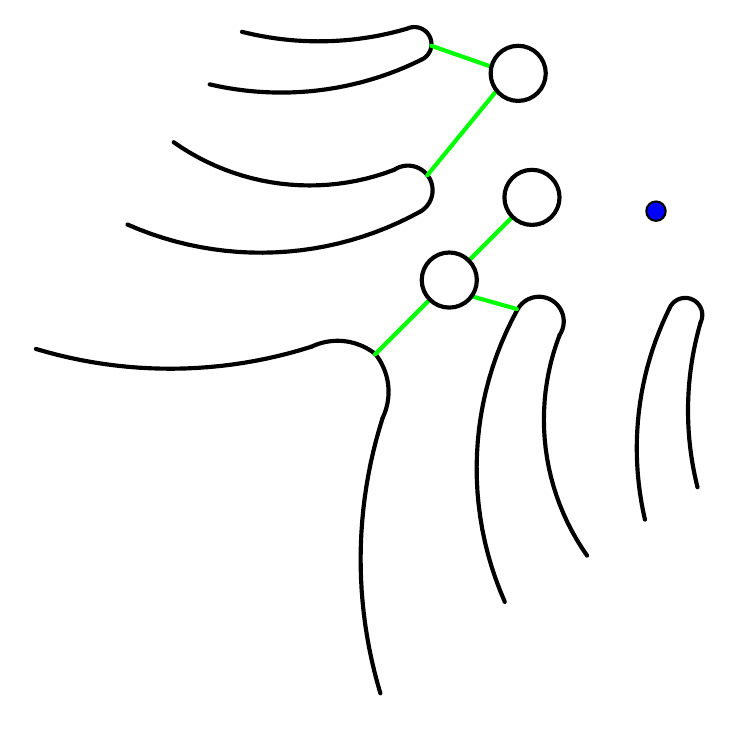}\\
\includegraphics[width=1in]{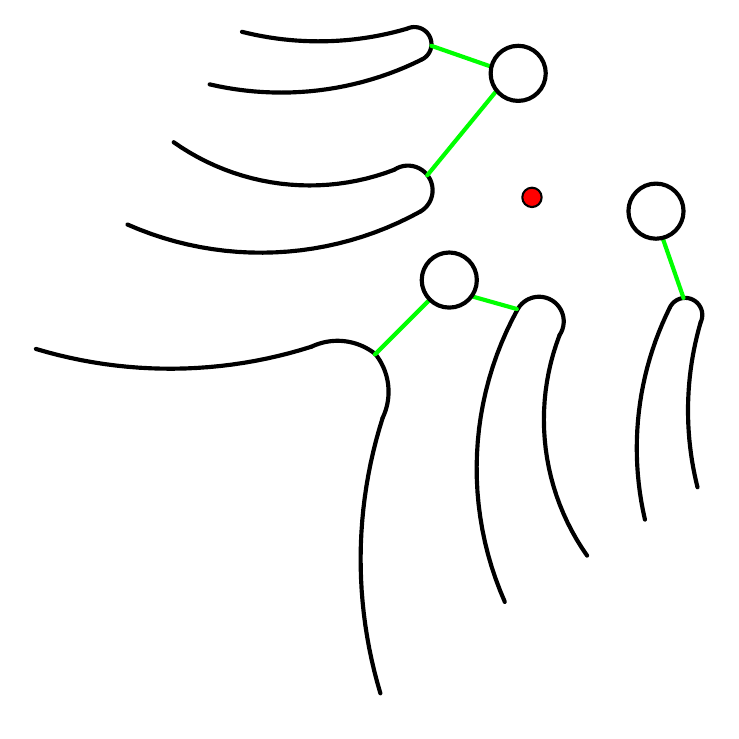}&
\includegraphics[width=1in]{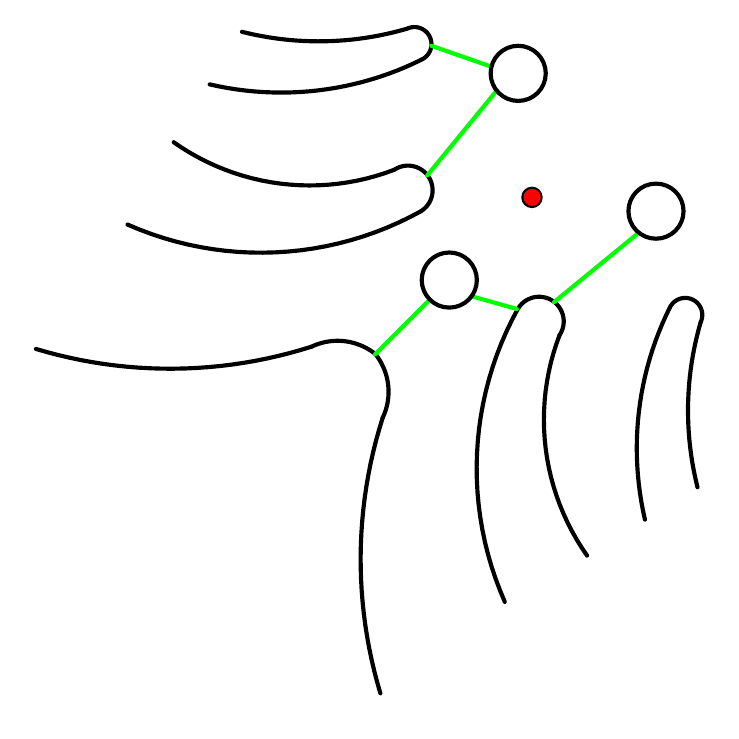}&
\includegraphics[width=1in]{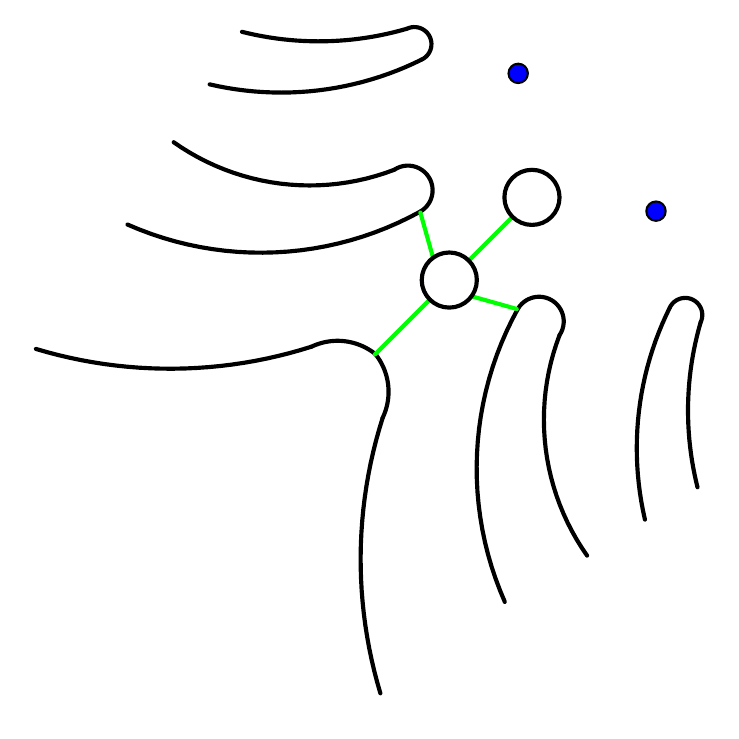}&
\includegraphics[width=1in]{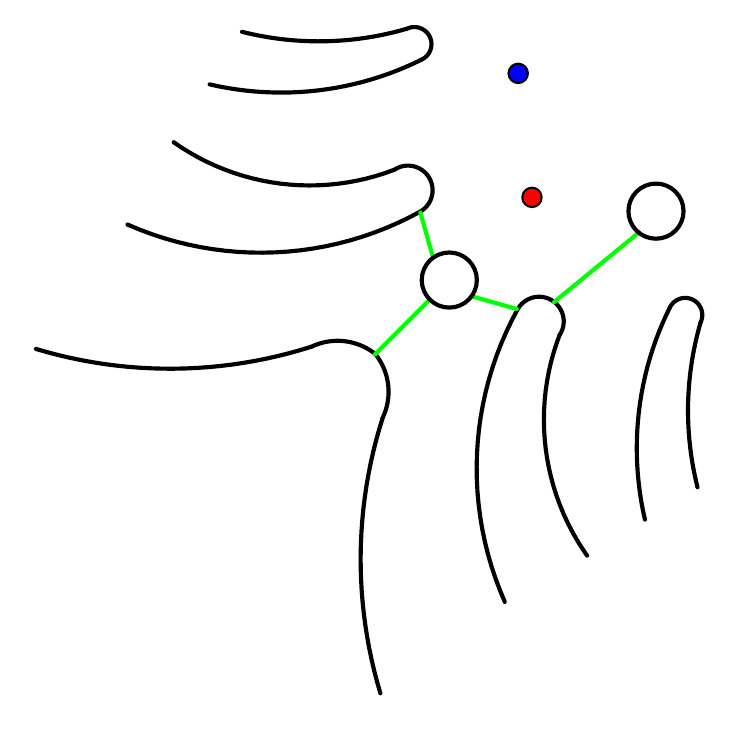}\\
\end{tabular}\\
\begin{tabular}{ccc}
\includegraphics[width=1in]{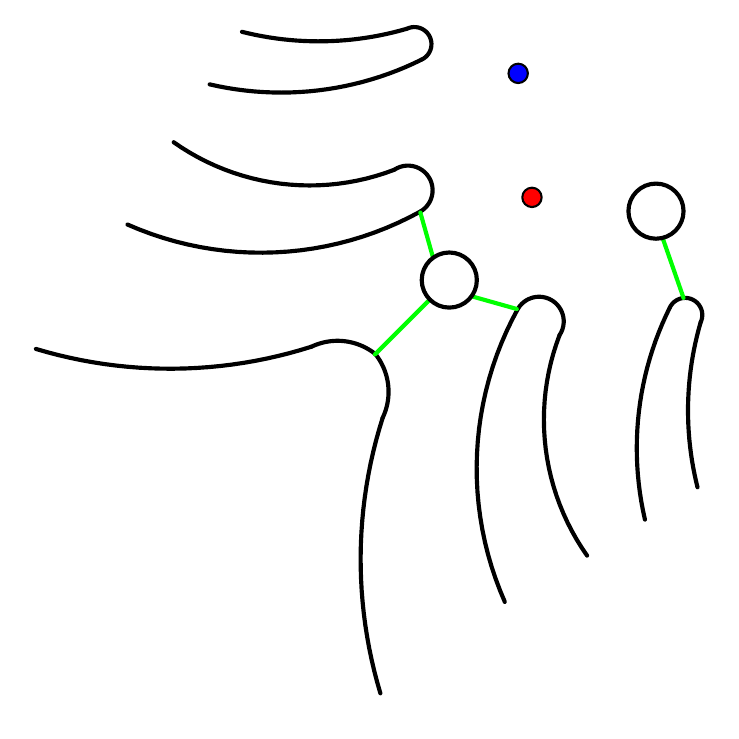}&
\includegraphics[width=1in]{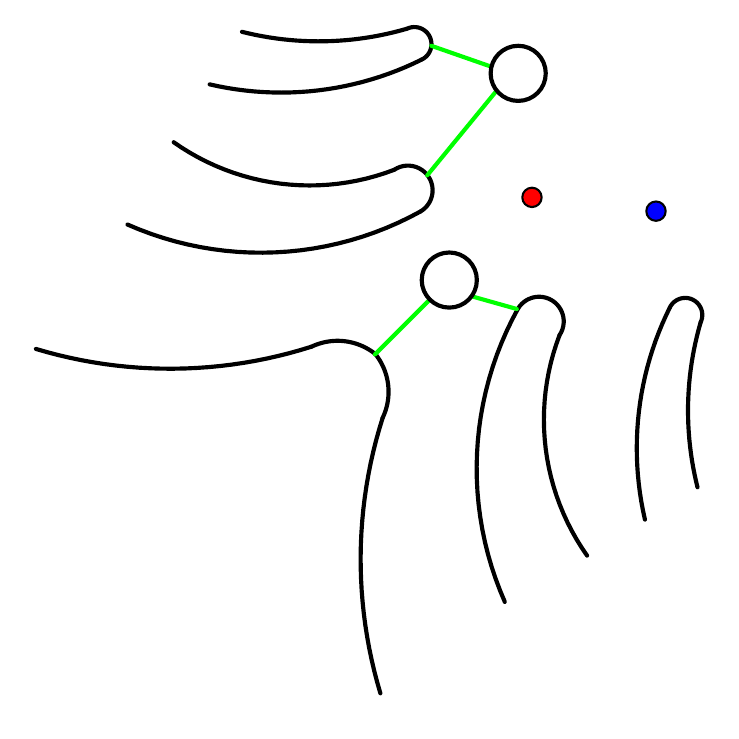}&
\includegraphics[width=1in]{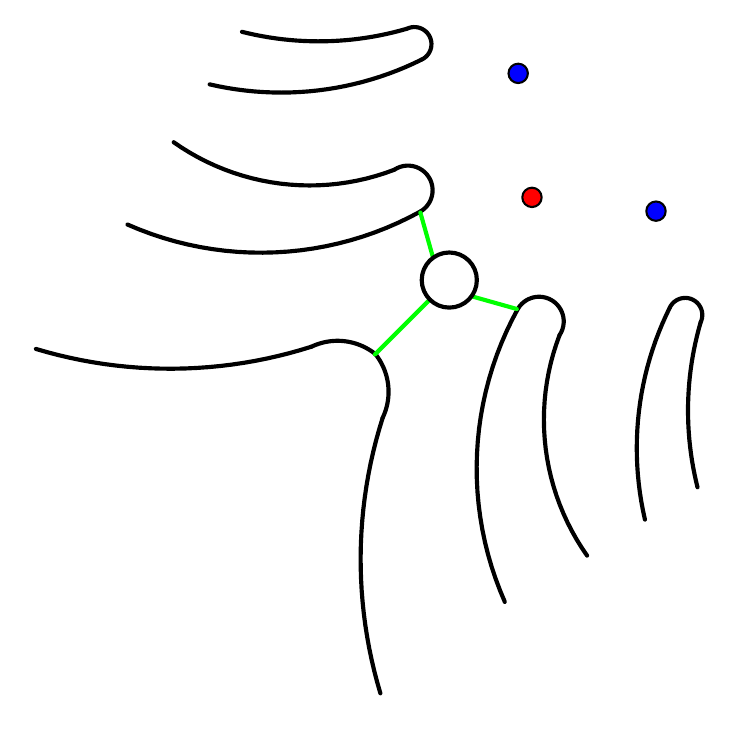}
\end{tabular}
\caption{Rigid isotopy classes of~$(M-2)$-perturbable nodal rational~$M$-curves in~$\RPP$. (Part II).}
\label{fig:MM2II}
\end{figure}

Given a type~$\mathrm{I}$ nodal curve~$C$ of odd degree in $\RPP$, denote by $h_p$ the number of hyperbolic nodal points connecting the pseudoline to an oval in a type~$\mathrm{I}$ perturbation of~$C$. 

\begin{prop}
Let~$C$ be a nodal rational $(M-2)$-curve of degree~$5$ in~$\RPP$. If~$C$ is $(M-2)$-perturbable
and $\sigma(C)=2$, then, the curve~$C$ belongs to one of the different rigid isotopy classes represented in Figure~\ref{fig:M2M2s}.
\end{prop}

\begin{proof}
Since~$C$ is a nodal rational $(M-2)$-curve of degree~$5$ in $\RPP$, we have $h+e=4$. 
Let us begin with the case when $h\geq1$. Then, the marked toile $(D_C,v_1,v_2)$ associated to $(C,p)$ is a type~$\mathrm{I}$ toile with $5$ real nodal {\tvs}, among which two are markings, and a negative inner nodal {\tv}. 
By Corollary~\ref{coro:lpert}, the marked toile $(D_C,v_1,v_2)$ has a type~$\mathrm{I}$ perturbation with $3$ true ovals. 
Then, enumerating all equivalence classes of marked toiles satisfying the aforementioned restrictions and realizing the birational transformations in order to recover the associated curves $\R C\subset\RPP$ lead to the plane curves shown on Figure~\ref{fig:M2M2s}, as small perturbations of reducible degree~$5$ curves having a line as component.

Remark that in the case when $h_p=2$, the two negative ovals on a type~$\mathrm{I}$ perturbation of the curve are opposite in the quadrangle formed by the ovals.
Let~$n_1$ and $n_2$ be the hyperbolic nodal points adding up to $h_p$ and let $L$ be the line passing through~$n_1$ and $n_2$.
There is a unique connected component of $L\setminus\{n_1,n_2\}$ that do not intersect~$C$. Define this connected component as the segment $\overline{n_1 n_2}$. Analogously, there is a unique segment $S$ of the pseudoline between $n_1$ and $n_2$ that do not intersect the line $L$.
Then, the segments $\overline{n_1 n_2}$ and $S$ bound a disk in $\RPP$ that contains a nodal point $q$ of~$C$ that can be distinguished. This construction is valid for any curve in the rigid isotopy class of~$C$.

In the case when $h=0$, by Rokhlin's complex orientation formula for nodal curves,
the curve~$C$ has exactly $2$ positive elliptic nodal points and $2$ negative elliptic nodal points. 
Pick a negative elliptic nodal point $p\in\R C$
as a base point. Then, the marked toile $(D_C,v_1,v_2)$ associated to $(C,p)$ is a type~$\mathrm{I}$ toile with $3$ real isolated nodal {\tvs}, among which exactly $1$ is negative and the remaining $2$ are positive, and two inner nodal {\tvs}, among which one is the nodal {\tv} $v_1=v_2$ and the remaining one is a negative inner nodal {\tv}.
Then, enumerating all possible toiles satisfying these conditions {\it via} the decompositions given in Proposition \ref{prop:decomptoi}, there is only one 
equivalence class of marked dessins 
satisfying these restrictions. 
\end{proof} 

\begin{figure}[h] 
\centering
\includegraphics[width=4in]{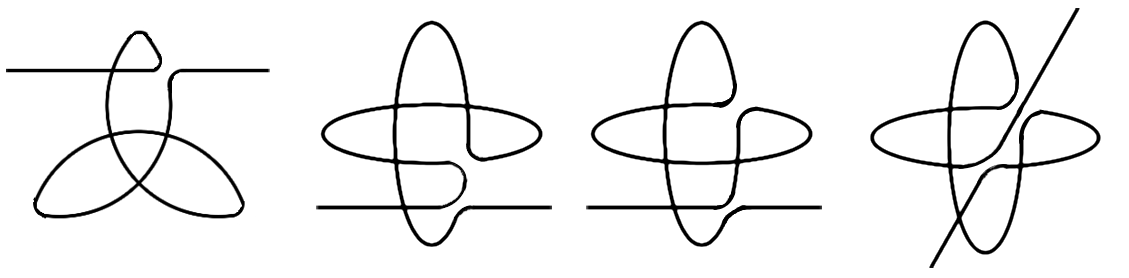}
\\
\includegraphics[width=4in]{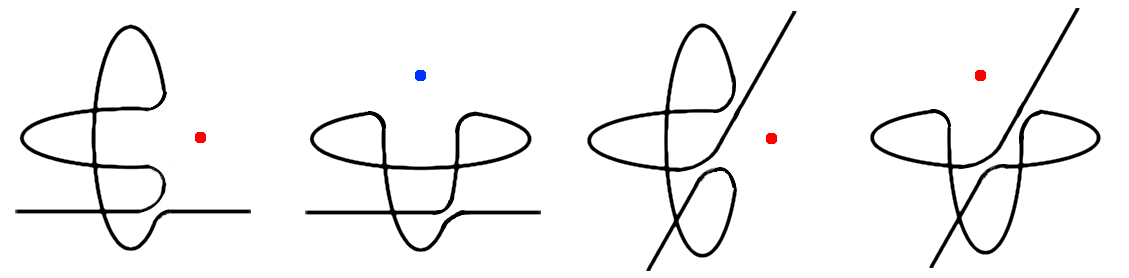}
\\
\includegraphics[width=4in]{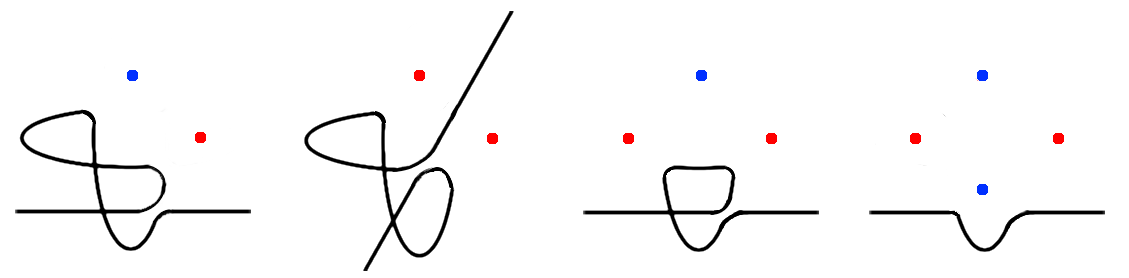}
\caption{The rigid isotopy class of~$(M-2)$-perturbable nodal rational~$(M-2)$-curves of degree~$5$ in~$\RPP$ with~$\sigma=2$. \label{fig:M2M2s}}
\end{figure}

\begin{prop}
Let~$C$ be a nodal rational $(M-2)$-curve of degree~$5$ in $\RPP$. If~$C$ is $(M-2)$-perturbable, $\sigma(C)=0$ and $h_p=1$, then the rigid isotopy class of~$C$ is determined by its isotopy class.
\end{prop} 

\begin{proof} 
Since~$C$ is a nodal rational $(M-2)$-curve of degree~$5$ in $\RPP$, we have $h+e=4$.
Let $p\in\R C$ be the hyperbolic nodal point connecting the pseudoline to an oval in a type~$\mathrm{I}$ perturbation of~$C$. 
Then, the marked toile $(D_C,v_1,v_2)$ associated to $(C,p)$ is a type~$\mathrm{I}$ toile with $5$ real nodal {\tvs}, among which two are markings, and with a positive inner nodal {\tv}.
By Corollary~\ref{coro:lpert}, the marked toile $(D_C,v_1,v_2)$ has a type~$\mathrm{I}$ perturbation with $3$ true ovals. 
Then, enumerating all equivalence classes of marked toiles satisfying the aforementioned restrictions and realizing the birational transformations in order to recover the associated curves $\R C\subset\RPP$ 
lead to the plane curves shown in Figure~\ref{fig:M2M2hp1}.
\end{proof}

\begin{figure}[h] 
\centering
\includegraphics[width=4in]{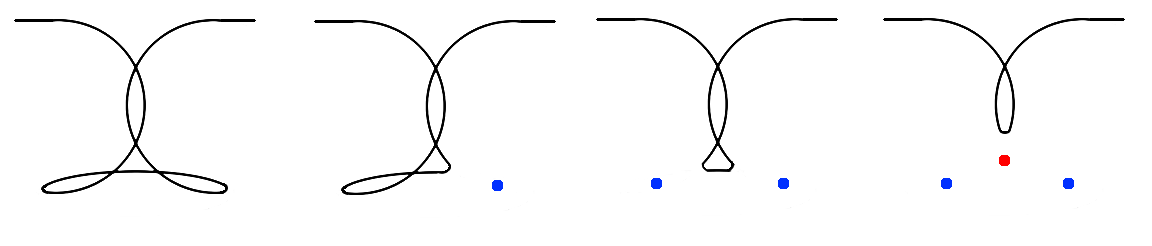}
\caption{Rigid isotopy class of $(M-2)$-perturbable nodal rational~$(M-2)$-curves of degree~$5$ in~$\RPP$ with~$h_p=1$ and $\sigma=0$. \label{fig:M2M2hp1}}
\end{figure}

\begin{prop}
Let~$C$ be a nodal rational $(M-2)$-curve of degree~$5$ in $\RPP$. If~$C$ is $(M-2)$-perturbable, $\sigma(C)=0$ and $h_p=2$, then, the curve~$C$ belongs to one of the different rigid isotopy classes represented in Figure~\ref{fig:M2M2hp2}.
\end{prop}

\begin{proof} 
Pick a hyperbolic nodal point $p\in\R C$ connecting 
the pseudoline to an oval in a type~$\mathrm{I}$ perturbation of~$C$. 
Then, the marked toile $(D_C,v_1,v_2)$ associated to $(C,p)$ is a type~$\mathrm{I}$ toile with $5$ real nodal {\tvs}, among which two are markings, and with a positive inner nodal {\tv}.
By Corollary~\ref{coro:lpert}, 
the marked toile $(D_C,v_1,v_2)$ has a type~$\mathrm{I}$ perturbation with $3$ true ovals. 
Then, enumerating all equivalence classes of marked toiles satisfying the aforementioned and
realizing the birational transformations in order to recover the associated curves $\R C\subset\RPP$ lead to the plane curves shown in Figure~\ref{fig:M2M2hp2}.

In this setting, there are exactly two rigid isotopy classes per isotopy class. In order to distinguish them, we use the line $L$ passing by the complex conjugated pair of nodal points of~$C$. This line intersects the curve~$C$ in a unique real point~$p_0$. Let~$p'\neq p$ be the other hyperbolic nodal point 
connecting the pseudoline to an oval in a type~$\mathrm{I}$ perturbation of~$C$,
and let $C_0$ be the type~$\mathrm{I}$ perturbation of~$C$. 
Then, the triple $(p,p',p_0)$ induces an orientation on the pseudoline of the curve~$C$ that can be extended to a complex orientation of the whole curve, determining a cyclic order of the ovals of a type~$\mathrm{I}$ perturbation of~$C$.
This orientation is a rigid isotopy invariant of the marked curve $(C,p)$.
Then, every isotopy class endowed with the orientation given by $(p,p',p_0)$ determines a unique rigid isotopy class.
\end{proof}

\begin{figure}[h] 
\centering
\includegraphics[width=4in]{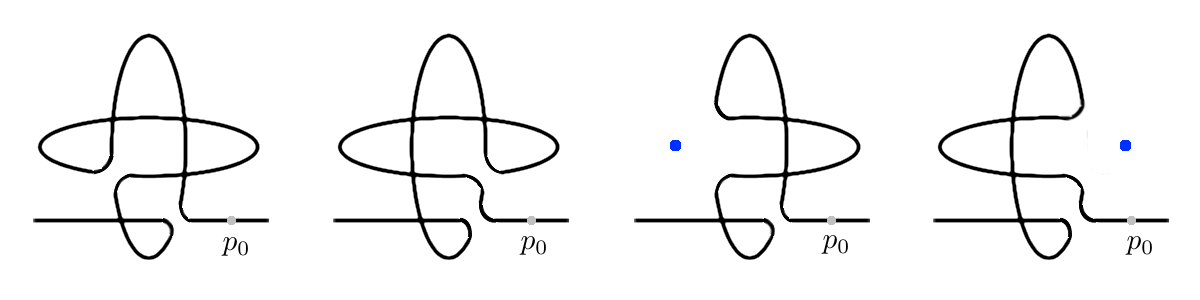}
\\
\includegraphics[width=2in]{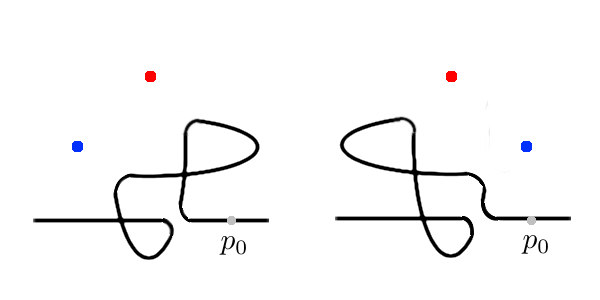}
\caption{Rigid isotopy classes of $(M-2)$-perturbable nodal rational~$(M-2)$-curves of degree~$5$ in~$\RPP$ with~$h_p=2$ and $\sigma=0$. \label{fig:M2M2hp2}}
\end{figure}

\begin{prop}
Let~$C$ be a nodal rational $(M-2)$-curve of degree~$5$ in $\RPP$. If~$C$ is $(M-2)$-perturbable, $\sigma(C) = 0$ and $h_p = 3$, then the rigid isotopy class of~$C$ is determined by its isotopy class.
\end{prop} 

\begin{proof} 
Pick a hyperbolic nodal point $p\in\R C$
connecting the pseudoline to an oval in 
a type~$\mathrm{I}$ perturbation of~$C$. 
Then, the marked toile $(D_C,v_1,v_2)$ associated to $(C,p)$ is a type~$\mathrm{I}$ toile with $5$ real nodal {\tvs}, among which two are markings, and with a positive inner nodal {\tv}.
By Corollary~\ref{coro:lpert}, the marked toile $(D_C,v_1,v_2)$ has a type~$\mathrm{I}$ perturbation with $3$ true ovals. 
Since $h_p=3$ for the curve~$C$, the toile $D_C$ has exactly $2$ non-marked non-isolated {\tvs} connecting the long component to an oval in the type~$\mathrm{I}$ perturbation of $D_C$.
Then, enumerating all equivalence classes of marked toiles satisfying the aforementioned restrictions and realizing the birational transformations in order to recover the associated curves $\R C\subset\RPP$ lead to the plane curves shown in Figure~\ref{fig:M2M2hp3}.
\end{proof}

\begin{figure}[h] 
\centering
\includegraphics[width=4in]{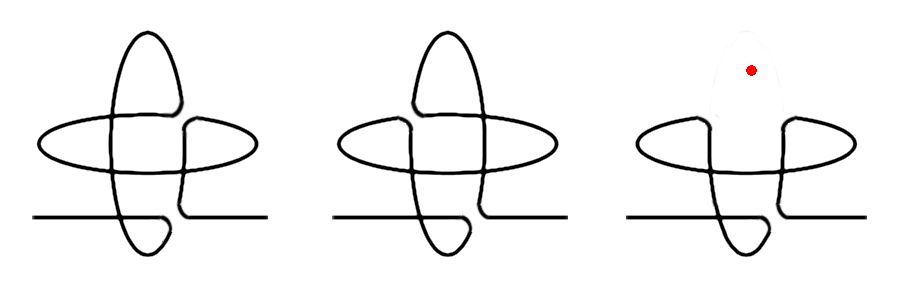}
\caption{Rigid isotopy class of $(M-2)$-perturbable nodal rational plane~$(M-2)$-curves of degree~$5$ with~$h_p=3$ and $\sigma=0$. \label{fig:M2M2hp3}}
\end{figure}

\begin{thm}
Let~$C$ be a nodal rational $(M-2)$-curve of degree~$5$ in $\RPP$. If~$C$ is $(M-2)$-perturbable and $\sigma(C)=0$, then, the curve~$C$ belongs to one of the different rigid isotopy classes represented in Figures~\ref{fig:M2M2hp1} to~\ref{fig:M2M2hp0}.
\end{thm}

\begin{proof} 
It remains to consider curves with $h=0$, and hence $e=4$.
By Rokhlin’s complex orientation formula for nodal curves,
the curve~$C$ has exactly $1$ positive elliptic nodal 
point and $3$ negative elliptic nodal points. 
Pick the positive elliptic nodal point $p\in\R C$ of~$C$ as a base point. Then, the marked toile $(D_C,v_1,v_2)$ associated to $(C,p)$ is a type~$\mathrm{I}$ toile with $3$ real isolated nodal {\tvs}, which are all negative, and two inner nodal {\tvs}, among which one is the nodal {\tv} $v_1=v_2$ and the remaining one is a positive inner nodal {\tv}.
Then, enumerating all equivalence classes of marked toiles leads to the fact that there is only one equivalence class of marked dessins satisfying these restrictions. 
\end{proof}

\begin{figure}[h] 
\centering
\includegraphics[width=1in]{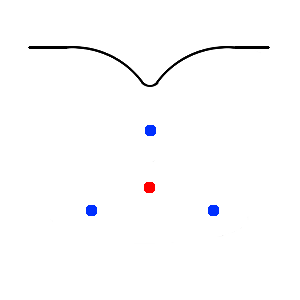}
\caption{Rigid isotopy class of $(M-2)$-perturbable nodal rational plane~$(M-2)$-curves of degree~$5$ with~$h=0$ and $\sigma=0$.}
\label{fig:M2M2hp0}
\end{figure}

\subsection{$(M-4)$-perturbable curves}

\begin{lm}\label{lm:h1}
Let~$C$ be a nodal rational curve of degree~$5$ in $\RPP$. If~$C$ is $(M-4)$-perturbable and $\sigma(C)=0$, then $h\geq1 $.
\end{lm}

\begin{proof} 
The fact that $\sigma(C)=0$ implies that the curve~$C$ can be perturbed to a non-singular type~$\mathrm{I}$ curve $C'$. The curve $C'$ being a non-singular $(M-4)$-curve of type~$\mathrm{I}$ has a nest of two ovals.
If $h=0$, then every nodal point of~$C$ is elliptic, producing a simple oval in $C'$ and no nest can be produced in this manner.
\end{proof}

\begin{thm}
Let~$C$ be a nodal rational $M$-curve of degree~$5$ in $\RPP$. If~$C$ is $(M-4)$-perturbable, its rigid isotopy class is determined by its isotopy class.
\end{thm}

\begin{proof} 
Pick a hyperbolic nodal point $p\in\R C$ belonging to the pseudoline.
Then, the marked toile $(D_C,v_1,v_2)$ associated to $(C,p)$ is a type~$\mathrm{I}$ toile with $7$ real nodal {\tvs}. 
By Corollary~\ref{coro:lpert}, the marked toile $(D_C,v_1,v_2)$ has a type~$\mathrm{I}$ perturbation with $1$ true oval. 
Then, enumerating all equivalence classes of marked toiles satisfying the aforementioned restrictions leads to the fact that every isotopy class in Figure~\ref{fig:MM4} corresponds to exactly one rigid isotopy class.
\end{proof}

\begin{figure}[h] 
\centering
\includegraphics[width=4in]{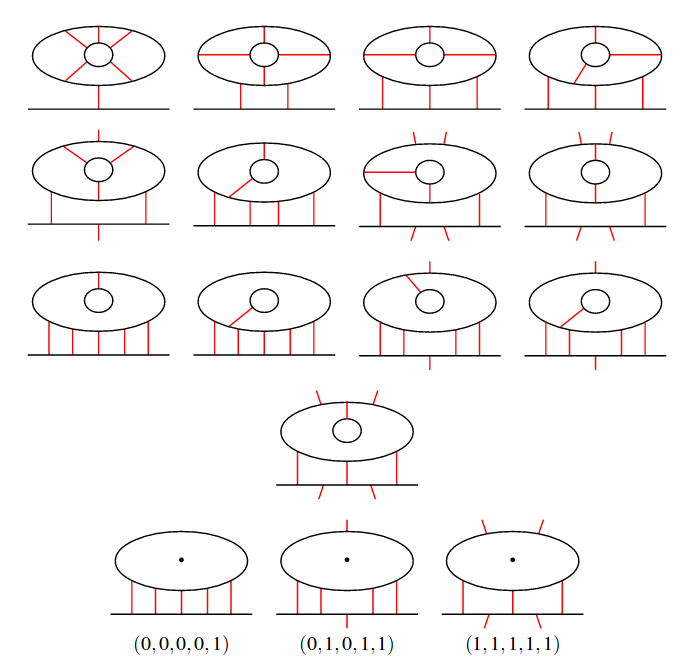}
\caption{Rigid isotopy classes of $(M-4)$-perturbable nodal rational~$M$-curves of degree~$5$ in~$\RPP$. Taken from \cite{IMR}.}
\label{fig:MM4}
\end{figure}


\begin{lm}\label{lm:hn1}
Let~$C$ be a nodal rational $(M-2)$-curve of degree~$5$ in $\RPP$. If~$C$ is $(M-4)$-perturbable, then $h\geq 1$.
\end{lm}

\begin{proof} 
If $h=0$, the type~$\mathrm{I}$ perturbation of~$C$ has two ovals arising from elliptic nodes, in which case $h+e=2$, contradicting the fact that for a nodal rational plane curve $(M-2)$-curve of degree~$5$ we have $h+e=4$.
\end{proof}

\begin{thm}
Let~$C$ be a nodal rational $(M-2)$-curve of degree~$5$ in $\RPP$. If~$C$ is $(M-4)$-perturbable, its rigid isotopy class is determined by its isotopy class and the number $\sigma(C)$.
\end{thm}

\begin{proof} 
Pick a hyperbolic nodal point $p\in\R C$ belonging to the pseudoline.
Then, the marked toile $(D_C,v_1,v_2)$ associated to $(C,p)$ is a type~$\mathrm{I}$ toile with $5$ real nodal {\tvs}, among which two are markings, and with an inner nodal {\tv}.
By Corollary~\ref{coro:lpert}, the marked toile $(D_C,v_1,v_2)$ has a type~$\mathrm{I}$ perturbation with $1$ true oval. 
Then, enumerating all equivalence classes of marked toiles satisfying the aforementioned restrictions leads leads to the fact that every isotopy class in Figure~\ref{fig:M2M4} corresponds to exactly one rigid isotopy class.
\end{proof}

\begin{figure}[h] 
\centering
\includegraphics[width=4in]{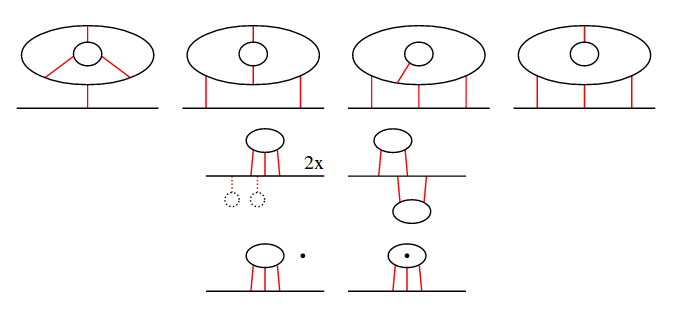}
\caption{Rigid isotopy classes of $(M-4)$-perturbable nodal rational~$(M-2)$-curves of degree~$5$ in~$\RPP$. Taken from \cite{IMR}.}
\label{fig:M2M4}
\end{figure}


\begin{thm}
Let~$C$ be a nodal rational $(M-4)$-curve of degree~$5$ in~$\RPP$. If~$C$ is $(M-4)$-perturbable, its rigid isotopy class is determined by its isotopy class and the sign of its elliptic nodal points.
\end{thm}

\begin{proof}
Let us consider all the possible values that $\sigma(C)$ can take.

{\it Case 1:} let us start with the case when $\sigma(C)=0$.
By Lemma~\ref{lm:h1}, we have $h\geq1$. Pick a hyperbolic nodal point $p\in\R C$ belonging to the pseudoline. Then, the marked toile $(D_C,v_1,v_2)$ associated to $(C,p)$ is a type~$\mathrm{I}$ toile with $3$ real nodal {\tvs}, among which two are markings, and with $2$ positive inner nodal {\tvs}.
By Corollary~\ref{coro:lpert}, the marked toile $(D_C,v_1,v_2)$ has a type~$\mathrm{I}$ perturbation with $1$ true oval. 
Then, enumerating all equivalence classes of marked toiles satisfying the aforementioned restrictions leads leads to the fact that every isotopy class in Figure~\ref{fig:M4M4s0} corresponds to exactly one rigid isotopy class.

{\it Case 2:}
In the case when $\sigma(C)=2$, let us suppose $h\geq1$.
Pick a hyperbolic nodal point $p\in\R C$ belonging to the pseudoline. Then, the marked toile $(D_C,v_1,v_2)$ associated to $(C,p)$ is a type~$\mathrm{I}$ toile with $3$ real nodal {\tvs}, among which two are markings, and with $1$ positive inner nodal {\tv} and $1$ negative inner nodal {\tv}.
By Corollary~\ref{coro:lpert}, the marked toile $(D_C,v_1,v_2)$ has a type~$\mathrm{I}$ perturbation with $1$ true oval. 

Otherwise, we have $h=0$, and consequently $e=2$.
By Rokhlin’s complex orientation formula for nodal curves the curve~$C$ has exactly $2$ negative elliptic nodal points. 
Let us pick $p\in\R C$ a negative elliptic nodal point of~$C$ as a base point. Then, the marked toile $(D_C,v_1,v_2)$ associated to $(C,p)$ is a type~$\mathrm{I}$ toile with $1$ real isolated nodal {\tv}, which is negative, and $3$ inner nodal {\tvs}, among which one is the nodal {\tv} $v_1=v_2$ and the remaining two are one positive and one negative inner nodal {\tvs}.
Then, enumerating all equivalence classes of marked toiles satisfying the aforementioned restrictions leads leads to the fact that every isotopy class in Figure~\ref{fig:M4M4s2} corresponds to exactly one rigid isotopy class.

{\it Case 3:}
In the case when $\sigma(C)=4$, let us suppose $h\geq1$.
Pick a hyperbolic nodal point $p\in\R C$ belonging to the pseudoline. Then, the marked toile $(D_C,v_1,v_2)$ associated to $(C,p)$ is a type~$\mathrm{I}$ toile with $3$ real nodal {\tvs}, among which two are markings, and with $2$ negative inner nodal {\tvs}.
By Corollary~\ref{coro:lpert}, the marked toile $(D_C,v_1,v_2)$ has a type~$\mathrm{I}$ perturbation with $1$ true oval. 

Otherwise we have that $h=0$, and consequently $e=2$.
By Rokhlin’s complex orientation formula for nodal curves the curve~$C$ has exactly $1$ positive elliptic nodal point and$1$ negative elliptic nodal point. 
Let us pick $p\in\R C$ the positive elliptic nodal point of~$C$ as a base point. Then, the marked toile $(D_C,v_1,v_2)$ associated to $(C,p)$ is a type~$\mathrm{I}$ toile with $1$ real isolated nodal {\tv}, which is negative, and $3$ inner nodal {\tvs}, among which one is the nodal {\tv} $v_1=v_2$ and the remaining two are negative inner nodal {\tvs}.
Then, enumerating all equivalence classes of marked toiles satisfying the aforementioned restrictions leads leads to the fact that every isotopy class in Figure~\ref{fig:M4M4s4} corresponds to exactly one rigid isotopy class.
\end{proof}

\begin{figure}[h] 
\centering
\begin{subfigure}{\linewidth}
\centering 
\includegraphics[width=3in]{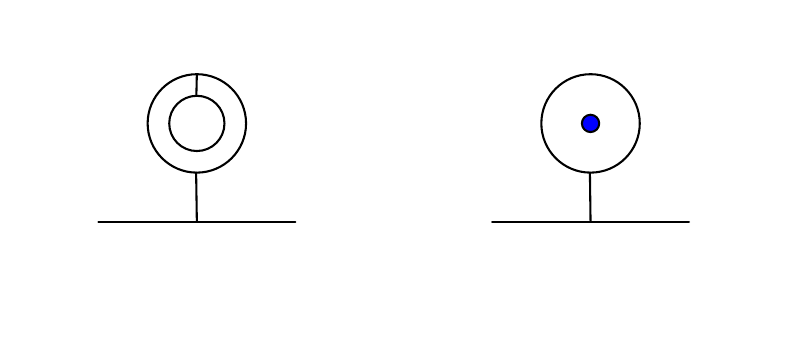}
\caption{Rigid isotopy classes of nodal rational $(M-4)$-perturbable curves of degree~$5$ in~$\RPP$ with~$\sigma=0$.}
\label{fig:M4M4s0}
\end{subfigure}
\begin{subfigure}{\linewidth} 
\centering
\includegraphics[width=4in]{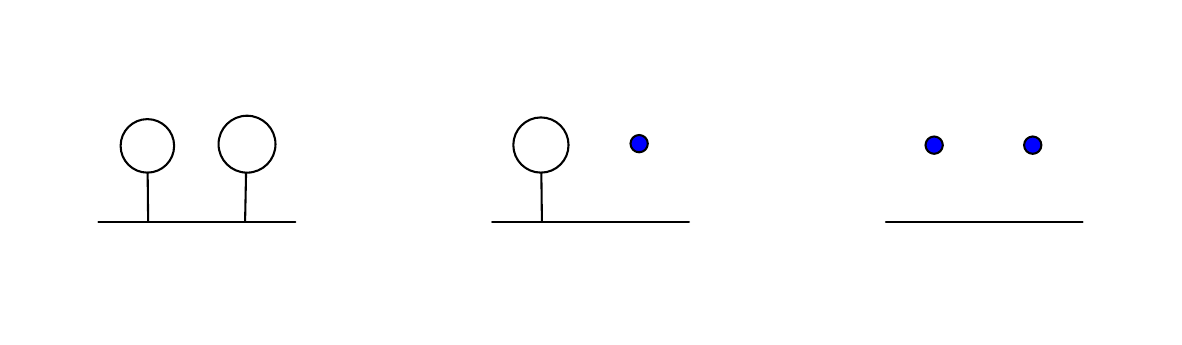}
\caption{Rigid isotopy classes of nodal rational $(M-4)$-perturbable curves of degree~$5$ in~$\RPP$ with~$\sigma=2$.}
\label{fig:M4M4s2}
\end{subfigure}
\begin{subfigure}{\linewidth} 
\centering
\includegraphics[width=4in]{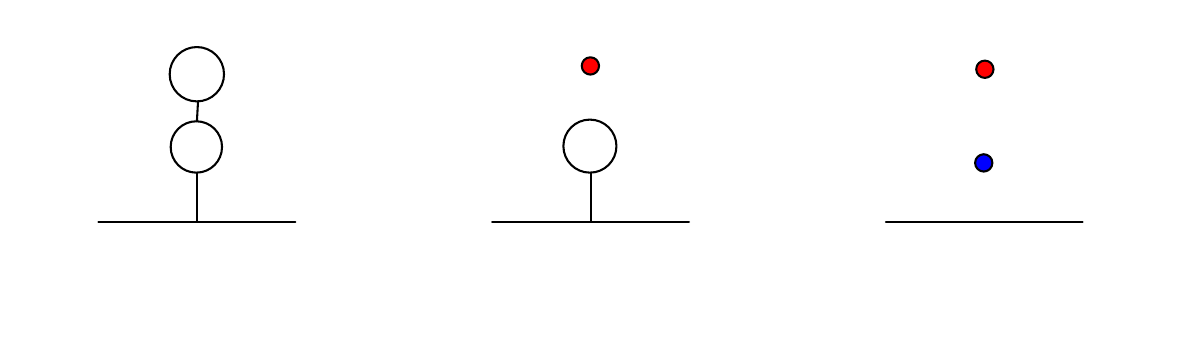}
\caption{Rigid isotopy classes of nodal rational $(M-4)$-perturbable curves of degree~$5$ in~$\RPP$ with~$\sigma = 4$.}
\label{fig:M4M4s4}
\end{subfigure}
\end{figure}

\subsection{$(M-6)$-perturbable curves}

\begin{thm}
There is a unique rigid isotopy class of rational $(M-6)$-curves of degree~$5$ in $\RPP$.
\end{thm}

This rigid isotopy class is shown in Figure~\ref{fig:M6}.

\begin{proof}
Let $C_0$ be a rational 
$(M-6)$-curve 
of degree~$5$ in~$\RPP$.
We have $h = e = 0$, and therefore $c(C_0) = 6$. 
Due to the B\'ezout theorem, the six nodal points $p_1, p_2,\dots,p_6$ of $C_0$ are in general position. In particular, they do not belong to a conic.
Let us consider the linear system formed by the curves of degree~$5$ with singularities at $p_1, p_2,\dots,p_6$. 
This is a plane in the space of quintic curves, which has dimension~$20$.
In this linear system, the curves which are more degenerated than $C_0$ form a codimension~$2$ subset. 
Indeed, if a curve~$C$ in this linear system has an extra nodal point, then this nodal point must be real. 
In this case, such a curve~$C$ is a reducible curve whose components are a rational cubic and a conic with empty real point set, which contradicts the fact that $p_1, p_2,\ldots,p_6$ do not belong to a conic.
Therefore, a degeneration of $C_0$ in this linear system must have at least two additional nodal points or additional singularities of other kind.
Hence, in this linear system, every degeneration is path-connected to $C_0$. 

On the other hand, among the degenerations, there is a reducible curve $C_0'$ formed by a real line and two complex conjugated conics such that the real line passes through the pair of complex conjugated points $p_5$ and $p_6$, one conic passes through the points $p_1, \ldots, p_5$, and the complex conjugated conic passes through the points $p_1, \ldots , p_4, p_6$. 
Thus, the statement of the theorem follows from the fact that any two real reducible curves of this kind
(each of the curves being formed by a real line and two complex conjugated conics without real intersection points) 
can be connected within the class of such real reducible curves.
\end{proof}

\begin{figure}[h]
\centering
\includegraphics[width=1.5in]{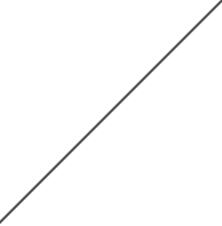}
\caption{The rigid isotopy class of nodal rational $(M-6)$-perturbable curves of degree~$5$ in~$\RPP$.}
\label{fig:M6}
\end{figure}

\bibliographystyle{plain} 
\bibliography{bibliothese} 
\end{document}